\documentclass[11pt, a4paper, twoside, openright]{report}

\usepackage[a4paper,width=145mm,top=25mm,bottom=25mm, headheight=14pt]{geometry}

\usepackage{fancyhdr}
\usepackage{amsfonts}
\usepackage{amsmath}
\usepackage{amssymb}
\usepackage{amsthm}
\usepackage[enableskew]{youngtab}
\usepackage{relsize}
\usepackage{tikz}
\usetikzlibrary{patterns}
\usetikzlibrary{decorations.pathreplacing}
\usetikzlibrary{arrows}
\usepackage{enumitem}
\usetikzlibrary{decorations.pathreplacing,angles,quotes}
\usepackage{mathtools}    
\usepackage{longtable}

\usepackage{caption}
\usepackage{subcaption}
\numberwithin{equation}{section}

\newtheorem{thm}{Theorem}[section]
\newtheorem{prop}[thm]{Proposition}
\newtheorem{lem}[thm]{Lemma}
\newtheorem{cor}[thm]{Corollary}
\newtheorem{conj}[thm]{Conjecture}
\newtheorem*{conj*}{Conjecture}
\newtheorem*{lem*}{Lemma}
\newtheorem*{recipe*}{Recipe}
\theoremstyle{remark}
\newtheorem{rem}[thm]{Remark}
\newtheorem{defn}[thm]{Definition}
\newtheorem{ex}[thm]{Example}
\newtheorem*{rem*}{Remark}
\newtheorem*{defn*}{Definition}
\newtheorem*{ex*}{Example}

\newcommand{\R}{\mathbb{R}}
\newcommand{\C}{\mathbb{C}}
\newcommand{\Q}{\mathbb{Q}}
\newcommand{\Z}{\mathbb{Z}}

\newcommand{\calA}{\mathcal{A}}
\newcommand{\calB}{\mathcal{B}}
\newcommand{\calC}{\mathcal{C}}
\newcommand{\calD}{\mathcal{D}}
\newcommand{\calE}{\mathcal{E}}
\newcommand{\calF}{\mathcal{F}}
\newcommand{\calP}{\mathcal{P}}

\newcommand{\calR}{\mathcal{R}}
\newcommand{\calS}{\mathcal{S}}
\newcommand{\calT}{\mathcal{T}}
\newcommand{\calU}{\mathcal{U}}
\newcommand{\calV}{\mathcal{V}}

\newcommand{\calX}{\mathcal{X}}
\newcommand{\calY}{\mathcal{Y}}

\newcommand{\power}{\mathfrak{p}}
\newcommand{\complete}{\mathfrak{h}}
\newcommand{\schur}{\mathfrak{s}}
\newcommand{\elementary}{\mathfrak{e}}
\newcommand{\monomial}{\mathfrak{m}}
\DeclareMathOperator{\sort}{sort}
\newcommand{\sorteq}{\overset{\sort}{=}}
\newcommand{\defeq}{\overset{\triangle}{=}}

\let\Re=\undefined
\DeclareMathOperator{\Re}{Re}
\let\Im=\undefined
\DeclareMathOperator{\Im}{Im}
\DeclareMathOperator{\height}{ht}
\DeclareMathOperator{\Sym}{Sym}
\DeclareMathOperator{\error}{error}
\DeclareMathOperator{\main}{main}
\DeclareMathOperator{\LHS}{LHS}
\DeclareMathOperator{\RHS}{RHS}
\DeclareMathOperator{\abs}{abs}
\DeclareMathOperator{\ribbon}{ribbon}
\DeclareMathOperator{\sub}{sub}
\DeclareMathOperator{\imaginary}{i}
\DeclareMathOperator{\Cor}{Cor}

\begin{document}

\begin{titlepage}
\phantom{.}
\vspace{1cm}
    \begin{center}
        \large  
        \begingroup
            \LARGE{\textbf{Combinatorial Structures in Random Matrix Theory Predictions for $\boldsymbol{L}$-Functions}} \\ \bigskip
        \endgroup
        
\rule{12cm}{0.4pt}
\vspace{2cm}

	Dissertation\\
\medskip
	zur\\
\medskip
        Erlangung der naturwissenschaftlichen Doktorw\"urde \\ \medskip (Dr.\ sc.\ nat.) \\
\medskip
	vorgelegt der\\
\medskip
        Mathematisch-naturwissenschaftlichen Fakult\"at \\
\medskip
	der\\
\medskip
        Universit\"at Z\"urich \\ \bigskip\bigskip\bigskip
\medskip
	von\\
\medskip
	Helen Riedtmann \\
\medskip
von \\
\medskip
Basel BS

\vspace{2cm}

Promotionskommission\\
\medskip
Prof.\ Dr.\ Valentin F\'eray (Leitung der Dissertation)\\
\medskip
Prof.\ Dr.\ Ashkan Nikeghbali \\
\medskip
Prof.\ Dr.\ Joachim Rosenthal \\
\vspace{1cm}
        Z\"urich, 2018

        \vfill                      

    \end{center}        
\end{titlepage}   

\begingroup
\pagestyle{plain}

\chapter*{Zusammenfassung}
Unsere Resultate k\"onnen als eine Anwendung algebraischer Kombinatorik auf Zufallsmatrizen betrachtet werden. Die Motivation f\"ur diese Anwendung ist, dass die Theorie der Zufallsmatrizen uns erm\"oglicht, die statistische Verhaltensweise der ber\"uhmten Riemann $\zeta$-Funktion (und $L$-Funk\-tio\-nen im Allgemeinen) vorherzusagen. Diese Vorhersagekraft von Zufallsmatrizen wurde von Montgomery (bez\"uglich der Nullstellen von $L$-Funktionen) und von Keating und Snaith (bez\"uglich der Werte von $L$-Funktionen) entdeckt. Unsere Resultate lassen sich in drei Teilbereiche untergliedern.

Die ersten Resultate behandeln eine neue Operation auf Partitionen, die wir \emph{overlap} (oder zu Deutsch \"Uber\-schnei\-dung) nennen. Wir beweisen zwei \"Uber\-schnei\-dungs\-gleichungen f\"ur sogenannte Littlewood-Schurfunktionen. Littlewood-Schurfunktionen sind eine Verallgemeinerung von Schurfunktionen, die von Littlewood eingef\"uhrt wurde: Die zur Partition $\lambda$ geh\"orige Littlewood-Schurfunktion $LS_\lambda(\calX; \calY)$ ist ein Polynom in den Variabeln $\calX \cup \calY$, welches sowohl in $\calX$ als auch in $\calY$ symmetrisch ist. Die erste \"Uber\-schnei\-dungs\-gleich\-ung stellt $LS_\lambda(\calX; \calY)$ als eine Summe \"uber Teilmengen von $\calX$ dar; die zweite \"Uber\-schnei\-dungsgleichung stellt $LS_\lambda(\calX; \calY)$ als eine Summe \"uber Paare von Partitionen dar, deren \"Uber\-schnei\-dung gleich der Partition $\lambda$ ist. Beide \"Uber\-schnei\-dungs\-gleich\-ungen folgen aus einer Anwendung des Laplaceschen Entwicklungssatzes, dank einer Formel von Moens und Van der Jeugt, die Littlewood-Schurfunktionen im Wesentlichen als Determinanten darstellt. Des Weiteren zeigen wir zwei graphische Beschreibungen aller Paare von Partitionen, deren \"Uber\-schnei\-dung eine gegebene Partition $\lambda$ ist.

Das zweite Resultat ist eine asymptotische Formel f\"ur Integrale von gemischten Br\"uchen charakteristischer Polynome \"uber die unit\"are Gruppe, wobei ein gemischter Bruch ein Produkt von Br\"uchen und/oder logarithmischen Ableitungen ist. Unser Beweis dieser Formel verallgemeinert Bump and Gamburds elegante kombinatorische Herleitung von Conrey, Forrester und Snaiths Formel f\"ur Integrale von Br\"uchen charakteristischer Polynome \"uber die unit\"are Gruppe. Diese Verallgemeinerung st\"utzt sich auf die folgenden drei kombinatorischen Restultate: die erste \"Uber\-schnei\-dungs\-gleich\-ung, eine neue Variante der Murnaghan-Nakayama-Regel und eine Idee, die aus dem Knotenoperatorformalismus entlehnt ist.

Das dritte und letzte Resultat nennen wir eine explizite Formel f\"ur Eigenwerte einer zuf\"alligen unit\"aren Matrix. In der Zahlentheorie nennt man eine Formel, die einen Zusammenhang zwischen einer Summe \"uber die Nullstellen einer $L$-Funktion und einer Summe \"uber die Primzahlen herstellt, eine explizite Formel. Unsere explizite Formel f\"ur Eigenwerte ist ein asymptotischer Ausdruck f\"ur eine Summe \"uber die Nullstellen eines zuf\"alligen charakteristischen Polynoms der unit\"aren Gruppe, deren Beweis den Beweis von Rudnick und Sarnaks expliziter Formel f\"ur eine recht allgemeine Klasse von $L$-Funktionen widerspiegelt. Wir beschliessen diese Arbeit mit einem Ausblick, wie dieser Ansatz zu neuen zahlentheoretischen Beweisen f\"uhren k\"onnte.

\chapter*{Abstract}
Our results can be viewed as applications of algebraic combinatorics in random matrix theory. These applications are motivated by the predictive power of random matrix theory for the statistical behavior of the celebrated Riemann $\zeta$-function (and $L$-functions in general), which was discovered by Montgomery (with regard to zeros of $L$-functions) and by Keating and Snaith (with regard to values of $L$-functions). Our results can be divided into three parts.

The first results revolve around a new operation on partitions, which we call overlap. We prove two overlap identities for so-called Littlewood-Schur functions. Littlewood-Schur functions are a generalization of Schur functions, whose study was  introduced by Littlewood. More concretely, the Littlewood-Schur function $LS_\lambda(\calX; \calY)$ indexed by the partition $\lambda$ is a polynomial in the variables $\calX \cup \calY$ that is symmetric in both $\calX$ and $\calY$ separately. The first overlap identity represents $LS _\lambda(\calX; \calY)$ as a sum over subsets of $\calX$, while the second overlap identity essentially represents $LS_\lambda(\calX; \calY)$ as a sum over pairs of partitions whose overlap equals $\lambda$. Both identities are derived by applying Laplace expansion to a determinantal formula for Littlewood-Schur functions due to Moens and Van der Jeugt. In addition, we give two visual characterizations for the set of all pairs of partitions whose overlap is equal to a partition $\lambda$.

The second result is an asymptotic formula for averages of mixed ratios of characteristic polynomials over the unitary group, where mixed ratios are products of ratios and/or logarithmic derivatives. Our proof of this formula is a generalization of Bump and Gamburd's elegant combinatorial proof of Conrey, Forrester and Snaith's formula for averages of ratios of characteristic polynomials over the unitary group. The generalization relies on three combinatorial results, namely the first overlap identity, a new variant of the Murnaghan-Nakayama rule and an idea from vertex operator formalism.

The third and last result is what we call an explicit formula for eigenvalues of a random unitary matrix. In number theory, a formula that establishes a relation between a sum over the zeros of some $L$-function and a sum over the primes is called an explicit formula. Our explicit formula for eigenvalues is an asymptotic expression for a sum over zeros of a random characteristic polynomial from the unitary group, whose proof mirrors Rudnick and Sarnak's proof of an explicit formula for a fairly general class of $L$-functions. We conclude this thesis by explaining how this approach might lead to new number theoretic proofs.

\chapter*{Acknowledgements}
I would like to thank Paul-Olivier Dehaye, who was my advisor during the first years of my PhD. I could not have wished for a better advisor. Paul generously shared his vision that combinatorics might be the key to a better understanding of the mysterious behavior of $L$-functions, and then let me find my own path in this fascinating field. I am particularly grateful to him for not losing patience with me, although it took me a while to find a path I wanted to follow. He was always available when I was stuck -- be that in my research or (in one memorable instance) in real life. His advice, which he continued to give even after having left academia, frequently nudged me into directions that proved fruitful. I am honored to have been one of Paul's numerous strange hobbies.

\medskip

I would also like to thank Valentin F\'eray, who was my advisor during the last years of my PhD, for taking me under his wing when I was suddenly in need of a new advisor. The monthly meetings with Valentin and Paul ensured that I did not stray from my path (for too long). In addition, I am obliged to Valentin for reading and commenting on my first efforts to write down my results. He taught me how to present mathematical research, while always respecting my own style. Without his help, this thesis would be hard(er) to read.

\medskip

I am grateful to Pierre-Lo\"ic M\'eliot and Jonathan Novak for having accepted to review this thesis. It was a great pleasure to read their positive reports.

\medskip

My thanks to my research group, Benedikt, Dario, Jacopo, Jehanne, Marko, Ma\-thil\-de (who feels like a part of the group although she technically is not), Nahid (who used to belong to my first research group), Ra\'ul and Valentin, for countless interesting discussions during our lunch break in a multitude of different languages. Hence, it is not only `thank you', but also `herzlichen Dank', `merci a tous' and `much\'isimas gracias'. My special thanks to Jehanne for inviting me to dinner after I had supervised an exam in her stead, and for everything this invitation led to.

\medskip

I would also like to thank other great people I have met during my time at the institute (including those I might have forgotten to mention): Silvia for confessing to me that she was also contemplating dropping math after our very first linear algebra lesson, which gave me the confidence to go the the second lesson, and for being a great companion to watch trashy movies with. Simone for always being willing to interrupt her work to have a restorative cup of tea with me and for our cooking adventures. Tovo for making our office a pleasant place to work and for affording me a glimpse into the culture of Madagascar. Joachim for his unfailing confidence in me and his invaluable career advice.

\medskip

Special thanks to Marcel for always being encouraging, while never putting any additional pressure on me. He would always make sure that I let go in the evenings by taking me dancing or by battling armies of undead with me.

\medskip

I am more grateful to my mother than I am able to express in words. Here I want to thank her for supporting me throughout my studies, in spite of her misgivings about me following in her footsteps. I promise to take a different path from here onwards.

\tableofcontents

\chapter*{List of Symbols}
\begin{longtable}{p{.17\textwidth} p{.65\textwidth} r}
\textbf{Symbol} & \textbf{Meaning} & \textbf{Page} \\
\hline
 & & \\
\endhead
$\abs(\calX)$ & sequence of absolute values of the elements in $\calX$ & \pageref{symbol_abs} \\
 & & \\
$c^\lambda_{\mu \nu}$ & Littlewood-Richardson coefficient & \pageref{symbol_LR_coeff} \\
$C_n(K)$ & $C_n(K) \sorteq (n - j + 1: j \not\in K)$ & \pageref{symbol_C_n(K)} \\
$\chi_g$ & characteristic polynomial of $g$ & \pageref{symbol_char_pol} \\
$\Cor_n(L, f)$ & $n$-correlation of zeros of L associated to $f$ & \pageref{symbol_cor_L} \\
$\Cor_n(U(N), f)$ & $n$-correlation of eigenvalues of a random unitary matrix of size $N$ associated to $f$ & \pageref{symbol_cor_U(N)} \\
 & & \\
$\Delta(\calX)$ & first $\Delta$-function & \pageref{symbol_first_Delta_function} \\
$\Delta(\calX; \calY)$ & second $\Delta$-function & \pageref{symbol_second_Delta_function} \\
$\frac{\partial}{\partial \power_r}$ & $r$-th derivation (power sum) operator & \pageref{symbol_k-th_derivation_operator} \\
$\frac{\partial}{\partial \power_\lambda}$ & $\lambda$-th derivation (power sum) operator & \pageref{symbol_lambda-th_operator} \\
 & & \\
$\elementary_r$ & $r$-th elementary symmetric polynomial/function & \pageref{symbol_elementary_poly} / \pageref{symbol_symmetric_function} \\
$\elementary_\lambda$ & $\lambda$-th elementary symmetric polynomial/function & \pageref{symbol_lambda-th_symmetric_poly} / \pageref{symbol_symmetric_function} \\
$\elementary(\calX)$ & product of the elements in $\calX$ & \pageref{3_eq_e(X)_defn} \\
 & & \\
$\complete_r$ & $r$-th complete symmetric polynomial/function & \pageref{symbol_complete_poly} / \pageref{symbol_symmetric_function} \\
$\complete_\lambda$ & $\lambda$-th complete symmetric polynomial/function & \pageref{symbol_lambda-th_symmetric_poly} / \pageref{symbol_symmetric_function} \\
$H(\pi)$ & sequence associated to the staircase walk $\pi$ & \pageref{symbol_sequence_associated_to_pi} \\
$\height(\lambda \setminus \kappa)$ & height of the ribbon $\lambda \setminus \kappa$ & \pageref{symbol_height_of_a ribbon} \\
 & & \\
$\imaginary$ & a square root of $-1$ & \\
$I_k(L, T)$ & $2k$-th moment of $L$ & \pageref{symbol_moment_L} \\
$I_k(U(N))$ & $2k$-th moment of a random characteristic polynomial from $U(N)$ & \pageref{symbol_moment_char_pol} \\
 & & \\
$\lambda$ ($\kappa$, $\mu$, $\nu$) & partition & \pageref{symbol_partition} \\
$|\lambda|$ & size of the partition $\lambda$ &  \pageref{symbol_size_of_partition} \\
$\lambda'$ & conjugate of the partition $\lambda$ & \pageref{symbol_conjugate_partition} \\
$\lambda \setminus \kappa$ & skew diagram & \pageref{symbol_skew_diagram} \\
$\tilde{\lambda}$ & complement of the partition $\lambda$ & \pageref{symbol_complement_of_partition} \\
$l(\lambda)$ & length of the partition $\lambda$ & \pageref{symbol_length_of_partition} \\
$\Lambda$ & completed $L$-function & \pageref{1_eq_defn_completed_L_funct} \\
$\Lambda_g$ & completed characteristic polynomial of $g$ & \pageref{1_defn_completed_char_pol} \\
$l(\calX)$ & length of the sequence $\calX$ & \pageref{symbol_length_of_sequence} \\
$LS_\lambda$ & Littlewood-Schur function associated to $\lambda$ & \pageref{symbol_LS_function} \\
$\LHS$ & left-hand side of the equation under consideration & \pageref{symbol_LHS/RHS} \\
 & & \\
$m_i(\lambda)$ & multiplicity of $i$ as a part of the partition $\lambda$ & \pageref{symbol_multiplicity} \\
$\monomial_\lambda$ & monomial symmetric polynomial/function indexed by $\lambda$ &  \pageref{symbol_monomial_poly} / \pageref{symbol_symmetric_function} \\
$\mu(\pi)$ & partition associated to the staircase walk $\pi$ & \pageref{symbol_partition_associated_to_pi_mu} \\
 & & \\
$[n]$ & $[n] = (1, 2, \dots, n)$ & \pageref{symbol_n_in_square_brackets} \\
$N(T)$ & number of non-trivial zeros with height in $[0, T]$ & \pageref{symbol_N(T)} \\
$\nu(\pi)$ & partition associated to the staircase walk $\pi$ & \pageref{symbol_partition_associated_to_pi_nu} \\
 & & \\
$\omega$ & involution on $\Sym$ & \pageref{symbol_involution_on_Sym} \\
$O_\calP$ & big-$O$ notation whose implicit constant depends on $\calP$ & \pageref{symbol_big_o_notation} \\
 & & \\
$\power_r$ & $r$-th power sum symmetric polynomial/function & \pageref{symbol_power_poly} / \pageref{symbol_symmetric_function} \\
 & $r$-th product (power sum) operator & \pageref{symbol_k-th_product_operator} \\
$\power_\lambda$ & $\lambda$-th power sum symmetric polynomial/function & \pageref{symbol_lambda-th_symmetric_poly} / \pageref{symbol_symmetric_function} \\
 & $\lambda$-th product (power sum) operator & \pageref{symbol_lambda-th_operator} \\
$\mathfrak{P}(m,n)$ & set of the staircase walks $\pi$ in an $m \times n$ rectangle & \pageref{symbol_staircase_walks_in_rectangle} \\
 & & \\
$\rho_n$ & $\rho_n = (n - 1, \dots, 1, 0)$ & \pageref{symbol_rho_n} \\
$\rho^\alpha$ & finite length specialization & \pageref{symbol_finite_length_specializations} \\
$\rho^\beta$ & finite length dual specialization & \pageref{symbol_finite_length_specializations} \\
$\calR(g)$ & multiset of eigenvalues of the unitary matrix $g$ & \pageref{symbol_R(g)} \\
$\RHS$ & right-hand side of the equation under consideration & \pageref{symbol_LHS/RHS} \\
 & & \\
$\schur_\lambda$ & Schur function associated to $\lambda$ & \pageref{symbol_schur_function} \\
$\sorteq$ & equality of sequences up to reordering & \pageref{symbol_sorteq} \\
$\sub_{n}(\lambda, M)$ & subpartition of $\lambda$ corresponding to $M \subset [n]$ & \pageref{symbol_subpartition} \\
$\Sym$ & ring of symmetric functions & \pageref{symbol_ring_of_symmetric_functions} \\
$\Sym(\calX)$ & ring of symmetric polynomials in the variables $\calX$ & \pageref{symbol_ring_of_symmetric_polynomials} \\
 & & \\
$U(N)$ & unitary group of degree $N$ & \pageref{symbol_U(N)} \\
 & & \\
$V(\pi)$ & sequence associated to the staircase walk $\pi$ & \pageref{symbol_sequence_associated_to_pi} \\
 & & \\
$\calX$ ($\calS$, $\calT$, $\calY$) & set of variables \textit{i.e.}\ sequence & \pageref{symbol_sequence} \\
$\calX^{-1}$ & sequence of multiplicative inverses of the elements in $\calX$ & \\
$\calX_K$ & subsequence of $\calX$ indexed by $K$ & \pageref{symbol_subsequence_indexed_by_K} \\
$\calX \setminus \calY$ & complement of the subsequence $\calY$ of $\calX$ & \pageref{symbol_complement_of_subsequence} \\
$x^+$ & positive part of $x$ & \pageref{symbol_positive_part} \\
 & & \\
$z_\lambda$ & $z_\lambda = \prod_{i \geq 1} i^{m_i(\lambda)} m_i(\lambda)!$ & \pageref{1_eq_cauchy_id} \\
$\zeta$ & the Riemann $\zeta$-function & \pageref{symbol_zeta} \\
 & & \\
$\subset$ & subset/subsequence & \pageref{symbol_subset} / \pageref{symbol_subsequence} \\
$+$ & sum of partitions/sequences & \pageref{symbol_sum_of_partitions} / \pageref{symbol_sum_of_sequences} \\
$\cup$ & union of partitions/sequences & \pageref{symbol_union_of_partitions} / \pageref{symbol_union_of_sequences} \\
$\cup_{m,n}$ & union of sequences of fixed lengths & \\
$\star_{m, n}$ & $(m,n)$-overlap & \pageref{symbol_overlap} \\
$\defeq$ & equality by definition & \pageref{symbol_defeq} \\
$\overset{k}{\to}$ & the skew diagram of the partition to the right without the partition to the left is a $k$-ribbon & \pageref{symbol_is_k_ribbon}
\end{longtable}

\chapter{Introduction} \label{1_cha_intro}
\endgroup
\noindent The aim of this introduction is to present the number theoretic motivation behind our combinatorial results. In Section~\ref{1_sec_NT_and_RMT}, we present conjectural (but widely accepted) connections between the theory of $L$-functions and random matrix theory. Section~\ref{1_sec_NT_and_alg_com} is a discussion of some combinatorial work that has been stimulated by these connections, which includes an overview of our results in Section~\ref{1_sec_my_results}.

Before we delve into the connection between the theory of $L$-functions and algebraic combinatorics, let us fix some general notation: throughout this thesis, we will use the symbol \label{symbol_defeq} $\defeq$ to denote an equality between the quantities on its left-hand side and its right-hand side which defines the quantity on its left-hand side. Furthermore, \label{symbol_LHS/RHS} $\LHS$/$\RHS$ always denote the left-hand/right-hand side of the equality under consideration.
\vspace{10pt}
\section{Number theory and random matrix theory} \label{1_sec_NT_and_RMT}
In the 19th century, Riemann conjectured that the zeros of the $\zeta$-function all lie on the critical line, \textit{i.e.}\ that all nontrivial zeros of the $\zeta$-function have real part equal to $1/2$. This conjecture is the famous Riemann hypothesis. Today, it is widely believed (but far from proven) that all so-called $L$-functions satisfy the Riemann hypothesis.

If we suppose that the Riemann hypothesis is true for some $L$-function, then all of its zeros are of the form $1/2 + \imaginary t$ for some height $t \in \R$. However, the Riemann hypothesis makes no prediction about the distribution of these heights. In 1973, Montgomery studied one statistical property of the distribution of the heights of the zeros of the $\zeta$-function, their so called pair correlation. He conjectured (and partially proved) that it is equal to the pair correlation of the eigenvalues of a random unitary matrix of large size. Taking the conjectured relationship between zeros of $L$-functions and eigenvalues of random unitary matrices of large size as a starting point, Keating and Snaith discovered a connection between the distribution of the \emph{values} of $L$-functions on the critical line and the \emph{values} of the characteristic polynomial of a random unitary matrix of large size. In 2000, they formulated a conjecture how the so-called moments of the two functions are related. 

\bigskip
In Section~\ref{1_sec_L_functs}, we give an axiomatic definition for $L$-functions, following Selberg, and show that the Riemann $\zeta$-function is indeed an $L$-function. Section~\ref{1_sec_on_correlations} is dedicated to the $n$-correlation of both zeros of $L$-functions and eigenvalues of unitary matrices. In Section~\ref{1_sec_on_moments}, we present Keating and Snaith's moment conjecture and then discuss some ideas (due to Conrey, Farmer, Keating, Rubinstein and Snaith) that were inspired by it in Section~\ref{1_sec_prop_char_pol}. In fact, all of the work which we will present in Section~\ref{1_sec_NT_and_alg_com} has also been stimulated by Keating and Snaith's discovery. In Section~\ref{1_sec_RMT}, we touch upon a few applications of random matrix theory that are not related to number theory.

\subsection[$L$-functions]{$\boldsymbol{L}$-functions} \label{1_sec_L_functs}
There does not exist one generally accepted definition for $L$-functions. In ``Old and new conjectures and results about a class of Dirichlet series'', Selberg proposes an axiomatic approach to $L$-functions \cite{selberg}. He focuses on four crucial analytic properties to characterize $L$-functions. In particular, the Riemann $\zeta$-function -- the oldest and most famous $L$-function -- can be shown to satisfy all of Selberg's conditions.

\subsubsection{The Selberg class}
The Selberg class is a class of Dirichlet series, which was introduced by Selberg \cite{selberg}. We shall call members of the Selberg class $L$-functions; however, it is important to keep in mind that there are other (conjecturally equivalent) definitions for $L$-functions, such as the notion of $L$-function used in \cite{rudnick1996}.

In \cite{selberg}, Selberg considers Dirichlet series with complex coefficients $a_n$
\begin{align*}
L(s) = \sum_{n \geq 1} \frac{a_n}{n^s}
\end{align*}
that are absolutely convergent for $\sigma > 1$ where $s = \sigma + \imaginary t$, which possess the following four additional properties:
\begin{enumerate}
\item (analytic continuation) The function $L(s)$ has a meromorphic continuation to the entire complex plane, so that $(s - 1)^m L(s)$ is an entire function of finite order for some integer $m \geq 0$.
\item (functional equation) The function $L(s)$ has a functional equation of the form
\begin{align} \label{1_eq_functional_eq_for_L_funct}
\Lambda(s) = \overline{\Lambda(1 - \bar{s})}
\end{align}
where
\begin{align} \label{1_eq_defn_completed_L_funct}
\Lambda(s) = \varepsilon Q^s \prod_{i = 1}^k \Gamma(\lambda_i s + \mu_i) L(s)
\end{align}
and 
\begin{align*}
|\varepsilon| = 1 \text{, }\; Q > 0 \text{, }\; \lambda_i > 0 \text{, }\; \Re \mu_i \geq 0
\end{align*}
are constants. Throughout this chapter, $\Gamma$ denotes the gamma function. We call $\Lambda(s)$ the completed $L$-function associated to $L(s)$, $\varepsilon$ its root number and $m = 2 \sum \lambda_i$ its degree. For convenience of notation, we sometimes abbreviate the factor $\varepsilon Q^s \prod_{i = 1}^k \Gamma(\lambda_i s + \mu_i)$ by $\gamma(s)$. Notice that the equation in \eqref{1_eq_functional_eq_for_L_funct} implies that $\Lambda(s)$ is real at the values $s = 1/2 + \imaginary t$.
\item (Ramanujan conjecture) In the context of the Selberg class, Ramanujan \emph{conjecture} is really a misnomer as it is an assumption rather than a conjecture. In any case, it states that $a_1 = 1$ and $a_n = O_\delta \left(n^\delta \right)$ for each fixed $\delta > 0$. 
\item (Euler product) For $\sigma > 1$,
\begin{align*}
\log L(s) = \sum_{n \geq 1} \frac{b_n}{n^s}
\end{align*}
with $b_n = 0$ unless $n$ is of the form $p^r$ where $p$ is a prime and $r$ a positive integer. In addition, $b_n = O \left( n^\theta \right)$ for some $\theta < 1/2$. In consequence, the completed $L$-function may also be written as a product over the primes: if $\sigma > 1$,
\begin{align*}
\Lambda(s) = \varepsilon Q^s \prod_{i = 1}^k \Gamma(\lambda_i s + \mu_i) \prod_{p \text{ prime}} \exp \left( \sum_{r \geq 1} \frac{b_{p^r}}{p^{rs}} \right)
.
\end{align*}
In fact, the $\Gamma$-factor should be seen as the factor corresponding to the ``prime at infinity'', following Tate's thesis \cite{tate}. However, this point of view seems not be relevant for the connection to random matrix theory discussed in this section.
\end{enumerate}
In spite of this axiomatic definition, all known examples of $L$-functions can be constructed from natural arithmetic objects, such as characters, automorphic forms and automorphic representations \cite[p.~37]{CFKRS05}.

The zeros of any $L$-function can be divided into two classes: the trivial zeros which are located at the poles of the functions $s \mapsto \Gamma(\lambda_i s + \mu_i)$ for $i = 1, \dots, k$, and the rest, which are called nontrivial zeros. The nontrivial zeros all lie in some vertical strip $1 - A \leq \sigma \leq A$ with $A \geq 1/2$. The generalized Riemann hypothesis is the conjecture that $A$ is equal to $1/2$, \textit{i.e.}\ that all nontrivial zeros of any $L$-function lie on the critical line given by $\sigma = 1/2$.

Let \label{symbol_N(T)} $N(T)$ denote the number of nontrivial zeros $s = \sigma + \imaginary t$ with height $t \in [0,T]$ for some positive $T$. It is a fact that no information is lost through restricting our attention to non-negative heights $t$. As a consequence of Cauchy's argument principle, there is a constant $c$ so that
\begin{align} \label{1_eq_for_N(T)}
N(T) = \frac{m}{2 \pi} T \left( \log T + c \right) + O(\log T) = \frac{m}{2 \pi} T \log T + O(T)
\end{align}
as $T \to \infty$ \cite[p.~48]{selberg}. 

\subsubsection[The Riemann $\zeta$-function]{The Riemann $\boldsymbol{\zeta}$-function}
The simplest example of an $L$-function is the Riemann $\zeta$-function, which is defined by \label{symbol_zeta}
\begin{align*}
\zeta(s) = \sum_{n \geq 1} \frac{1}{n^s}
\end{align*}
for $\sigma > 1$. It is obvious from its definition that $\zeta(s)$ satisfies the Ramanujan conjecture. In addition, a straightforward computation verifies that 
\begin{align*}
\zeta(s) = \frac{s}{s-1} - s\int_1^{\infty}\frac{x-[x]}{x^{s+1}}dx,
\end{align*}
which meromorphically continues $\zeta(s)$ to the half plane given by $\sigma > 0$, with a simple pole at $s = 1$. In ``Ueber die Anzahl der Primzahlen unter einer gegebenen Gr\"osse'' \cite{riemann}, Riemann completes the $\zeta$-function to
\begin{align*}
\xi(s) = \pi^{-s/2} \Gamma(s/2) \zeta(s)
,
\end{align*} 
and shows the following functional equation:
\begin{align*}
\xi(s) = \xi(1 - s) = \overline{\xi(1 - \bar{s})}
.
\end{align*} 
This functional equation allows us to continue $\zeta(s)$ to the entire complex plane. Finally, the Euler product of the $\zeta$-function, which states that for $\sigma > 1$ 
\begin{align*}
\prod_{p \text{ prime}} \frac{1}{1 - p^{-s}} = \prod_{p \text{ prime}} \left( \sum_{r \geq 0} p^{-rs}\right) = \sum_{n \geq 1} n^{-s} = \zeta(s)
,
\end{align*}
encodes the fact that every positive integer can be uniquely written as a product of primes. 
We therefore conclude that the Riemann $\zeta$-function is indeed an element of the Selberg class.

The trivial zeros of $\zeta(s)$ are located at the strictly negative even integers. Its infinitely many nontrivial zeros all lie in the strip $0 \leq \sigma \leq 1$. Indeed, the Euler product shows that $\zeta(s)$ is not equal to zero if $\sigma > 1$, which entails that no \emph{nontrivial} zeros are located at $\sigma < 0$, according to the functional equation. In 1859, Riemann conjectured that all nontrivial zeros
of the $\zeta$-function lie on the critical line, \textit{i.e.}\ the line consisting of all complex numbers of the form $s = 1/2+ \imaginary t$, thus giving birth to the celebrated Riemann hypothesis \cite{riemann}.

\subsection{On correlations} \label{1_sec_on_correlations}
In 1973 Montgomery conjectured that pairs of zeros of the Riemann $\zeta$-function behave like  pairs of eigenvalues of random unitary matrices \cite{montgomery73}. More precisely, Montgomery writes: ``F.~J.~Dyson has drawn my attention to the fact that the eigenvalues of a random complex Hermitian or unitary matrix of large order have precisely the same pair correlation'' as the zeros of the $\zeta$-function \cite[p.~184]{montgomery73}. He gives a partial proof of his conjecture, which relies on an explicit formula that relates the zeros of the $\zeta$-function to the prime numbers. Montgomery's discovery has led to a new branch of research, which delves into the conjectured connection between the zeros of $L$-functions and eigenvalues of random unitary matrices.

First, we give a rough sketch of the classic proof for $n$-correlation of eigenvalues. Second, we discuss partial proofs for $n$-correlation of zeros of $L$-functions, due to Montgomery \cite{montgomery73} and later in more generality Rudnick and Sarnak \cite{rudnick1996}. Third and last, we present an alternative approach to $n$-correlation that can be applied to both the Riemann $\zeta$-function (conjecturally) and to random matrices (rigorously), due to Conrey and Snaith \cite{CS}.

\subsubsection{$\boldsymbol{n}$-correlation of eigenvalues (based on Gaudin's lemma)}
In this section, we briefly introduce $n$-correlation of eigenvalues of random unitary matrices, following \cite{ConreyNotes}. 

A unitary matrix of size $N$ is a complex $N \times N$ matrix whose conjugate transpose is also its inverse. The group of all unitary matrices of size $N$ is customarily denoted $U(N)$. In symbols, \label{symbol_U(N)}
\begin{align*}
U(N) = \{g \in \text{Mat}(\C, N \times N): {\bar g}^t = g^{-1}\}
\end{align*}
where $\text{Mat}(\C, N \times N)$ is the set of all complex $N \times N$ matrices. It is a simple exercise in linear algebra to show that the eigenvalues of a unitary matrix $g$ all have absolute value 1. In consequence, we may write any eigenvalue of $g$ in the form $e^{\imaginary \theta}$ for some $\theta \in [0, 2 \pi)$, which we call an eigenangle of $g$. If $\theta_1, \dots, \theta_N$ are the eigenangles of $g \in U(N)$, its normalized eigenangles are given by
\begin{align*}
\tilde \theta_1 = \frac{N \theta_1}{2 \pi}, \dots, \tilde \theta_N = \frac{N \theta_N}{2 \pi}
.
\end{align*}
The normalization is chosen so that $\tilde \theta_1, \dots, \tilde \theta_N \in [0,N)$ have mean spacing 1. 

The $n$-correlation of the eigenvalues of a unitary matrix of size $N$ measures the correlation between differences of $n$ normalized eigenangles; more concretely, for any $(n - 1)$-dimensional box $Q = [-q, q]^{n - 1} \subset \R^{n - 1}$ it provides an estimate for the following quantity: 
\begin{align*}
\frac{1}{N} \# \left\lbrace 1 \leq j_1, \dots, j_n \leq N \text{ pairwise distinct}: \left( \tilde \theta_{j_1} - \tilde \theta_{j_2}, \dots, \tilde \theta_{j_{n - 1}} - \tilde \theta_{j_n} \right) \in Q \right\rbrace
.
\end{align*}
The reason for the division by $N$ is that the number of indices contained in the set to be counted would grow linearly with $N$ if the normalized eigenangles were uniformly distributed in the interval $[0,N)$. Informally, the $n$-correlation of a random unitary matrix of large size is obtained by first averaging the above quantity over $U(N)$ and then letting $N$ go to $\infty$.  In order to take formal averages over $U(N)$, we specify that any integral over $U(N)$ is to be understood as an integral with respect to the corresponding Haar measure, which is defined on page \pageref{1_page_Haar_measure}. In practice, it is more convenient to use smooth test functions rather than indicator functions of boxes, which leads to the following definition:

\begin{defn} [$n$-correlation of eigenvalues]
Let $f: \R^n \to \C$ be a smooth function with the following properties:
\begin{itemize}
\item $f$ is symmetric, \textit{i.e.}\ $f(x_{\sigma(1)}, \dots, x_{\sigma(n)}) = f(x_1, \dots, x_n)$ for permutations $\sigma \in S_n$;
\item $f$ is translation invariant, \textit{i.e.}\ $f(x_1 + t, \dots, x_n + t) = f(x_1, \dots, x_n)$ for $t \in \R$;
\item $f$ decays rapidly as $\sum_{1 \leq i \leq n} |x_i| \to \infty$ in the hyperplane $\sum_{1 \leq i \leq n} x_i = 0$. How fast $f(x_1, \dots, x_n)$ converges to 0, should be specified in the prerequisites of any statement about $n$-correlation, but we will not concern ourselves with this.
\end{itemize}
The $n$-correlation of eigenvalues of a random unitary matrix of size $N$ (associated to the function $f$) is defined by \label{symbol_cor_U(N)}
\begin{align*}
\Cor_n(U(N), f) = \frac{1}{N} \int_{U(N)} \sideset{}{'}\sum_{1 \leq j_1, \dots, j_n \leq N} f\left(\frac{N \theta_{j_1}}{2 \pi}, \dots, \frac{N \theta_{j_n}}{2 \pi} \right) dg
\end{align*}
where $\sideset{}{'}\sum$ indicates that the sum is over pairwise distinct indices.
\end{defn}

The three properties of the test function $f$ ensure that $\Cor_n(U(N), f)$ recovers what we seek to measure: the symmetry of $f$ mirrors the symmetries of the box $Q$; the translation invariance means that $f$ is a function of differences; and combining the last two properties entails that $\Cor_n(U(N), f)$ can be regarded as counting clusters of size $n$ in the set of eigenangles. 

\begin{thm} [$n$-correlation of eigenvalues] If $f$ is a suitable test function, then 
\begin{align} \begin{split} \label{1_thm_n-correlation_eigenvalues_eq}
& \lim_{N \to \infty} \Cor_n(U(N), f) \\
& \hspace{15pt} = \int_{\R^n} f(x_1, \dots, x_n) \det \left( \frac{\sin(\pi (x_i - x_j))}{\pi (x_i - x_j)} \right)_{1 \leq i,j \leq n} \delta \left( x_1 + \dots + x_n \right) dx_1 \dots dx_n
\end{split}
\end{align}
where $\delta$ is the Dirac mass at zero. More concretely, the integral $$\int_{\R^n} g(x_1, \dots, x_n) \delta(x_1 + \dots + x_n) dx_1 \dots dx_n$$ stands for the integral of $g$ over the hyperplane $\{(x_1, \dots, x_n) \in \R^n: x_1 + \dots + x_n = 0\}$.
\end{thm}

This neat formula for the $n$-correlation of eigenvalues is a fairly direct consequence of the following equality of integrals, which in its turn is an application of Gaudin's lemma \cite[p.~126]{ConreyNotes}: for symmetric Haar measurable functions $g$,
\begin{multline*}
\int_{U(N)} \sideset{}{'}\sum_{1 \leq j_1, \dots, j_n \leq N} g\left(\theta_{j_1}, \dots, \theta_{j_n} \right) dg \\ = \frac{1}{(2\pi)^n} \int_{[0, 2\pi]^n} g(y_1, \dots, y_n) \det \left( \frac{\sin(N(y_i - y_j)/2)}{\sin((y_i - y_j)/2)} \right)_{1 \leq i,j \leq n} dy_1 \dots dy_n
.
\end{multline*}

\subsubsection{$\boldsymbol{n}$-correlation of zeros of $\boldsymbol{L}$-functions (based on an explicit formula)}
Let us suppose that the generalized Riemann hypothesis is true. Under this assumption, the non-trivial zeros of any $L$-function all lie on a line, which allows us to study their $n$-correlation. More concretely, we fix an $L$-function and order its zeros according to their height (with multiplicities). If 
$0 \leq t_1 \leq \dots \leq t_N$ are the $N$ smallest non-negative heights, we set
\begin{align*}
\tilde t_1 = \frac{m}{2 \pi} t_1 \log(t_1), \dots, \tilde t_N = \frac{m}{2 \pi} t_N \log(t_N)
.
\end{align*}
These normalized heights have mean spacing 1, owing to the fact that 
\begin{align*}
N(T) \sim \frac{m}{2\pi} T \log T
\end{align*} 
as $T \to \infty$. In analogy to the case of random unitary matrices, we thus define $n$-correlation of zeros of $L$-functions as follows: 

\begin{defn} [$n$-correlation of zeros of $L$-functions]
Let $f: \R^n \to \C$ be a smooth function with the following properties:
\begin{itemize}
\item $f$ is symmetric, \textit{i.e.}\ $f(x_{\sigma(1)}, \dots, x_{\sigma(n)}) = f(x_1, \dots, x_n)$ for permutations $\sigma \in S_n$;
\item $f$ is translation invariant, \textit{i.e.}\ $f(x_1 + t, \dots, x_n + t) = f(x_1, \dots, x_n)$ for $t \in \R$;
\item $f$ decays rapidly as $\sum_{1 \leq i \leq n} |x_i| \to \infty$ in the hyperplane $\sum_{1 \leq i \leq n} x_i = 0$.
\end{itemize}
If $0\leq t_1 \leq t_2 \leq \dots$ are the non-negative heights of the zeros of some $L$-function $L$, then the $n$-correlation of zeros of $L$ (associated to $f$) is defined by \label{symbol_cor_L}
\begin{align*}
\Cor_n(L, f) = \lim_{N \to \infty} \frac{1}{N} \sideset{}{'}\sum_{1 \leq j_1, \dots, j_n \leq N} f\left( \frac{m}{2\pi} t_{j_1} \log (t_{j_1}), \dots, \frac{m}{2\pi} t_{j_n} \log (t_{j_n}) \right) 
\end{align*}
where $\sideset{}{'}\sum$ indicates that the sum is over pairwise distinct indices.
\end{defn}

As we have mentioned before, Montgomery conjectured and partially proved that the pair correlation of the eigenvalues of a random matrix in $U(N)$ converges to the pair correlation of the zeros of the Riemann $\zeta$-function as $N$ goes to infinity. In fact, he showed that the 2-correlation of zeros of $\zeta$ is equal to the right-hand side in \eqref{1_thm_n-correlation_eigenvalues_eq} (for $n = 2$) \emph{if the Fourier transform of the test function $f$ is supported on $$\left\lbrace (\xi_1, \xi_2) \in \R^2: \sum_{1 \leq i \leq 2} |\xi_i| < 2 \right\rbrace.$$} Moreover, Montgomery conjectured that the $n$-correlation of the zeros of $\zeta$ is equal to the right-hand side in \eqref{1_thm_n-correlation_eigenvalues_eq} for \emph{all} suitable test functions $f$ \cite[p.~184-185]{montgomery73}. Thanks to extensive numerical computations by Odlyzko \cite{odlyzko87, odlyzko89}, Montgomery's conjecture is now widely accepted (but still far from proven). In \cite{rudnick1996}, Rudnick and Sarnak generalized Montgomery's pair correlation theorem to $n$-correlation of the zeros of a fairly general class of $L$-functions. We should mention that Rudnick and Sarnak's result is independent of the generalized Riemann hypothesis, in spite of the fact that their motivation relies on the truth of the hypothesis. However, the results admit a nicer formulation under the assumption of the Riemann hypothesis:

\begin{thm} [$n$-correlation of zeros of $L$-functions, \cite{rudnick1996}] \label{1_thm_n-correlation_zeros_of_L-functs} Assume that the Riemann hypothesis holds for some $L$-function $L$ (according to the definition used in \cite{rudnick1996}). If $f$ is a suitable test function whose Fourier transform is supported on $\left\lbrace (\xi_1, \dots, \xi_n) \in \R^n: \sum_{1 \leq i \leq n} \left| \xi_i \right| < 2/m \right\rbrace$ where $m$ is the degree of $L$, then
\begin{align} \begin{split} \label{1_thm_n-correlation_zeros_of_L-functs_eq}
& \lim_{N \to \infty} \Cor_n(L, f) \\
& \hspace{15pt} = \int_{\R^n} f(x_1, \dots, x_n) \det \left( \frac{\sin(\pi (x_i - x_j))}{\pi (x_i - x_j)} \right)_{1 \leq i,j \leq n} \delta \left( x_1 + \dots + x_n \right) dx_1 \dots dx_n
\end{split}
\end{align}
where $\delta$ is the Dirac mass at zero.
\end{thm}

As ``all work on zeros of $L$-functions'' \cite[p.~274]{rudnick1996}, Theorem~\ref{1_thm_n-correlation_zeros_of_L-functs} relies on an explicit formula, which relates sums over zeros of $L$-functions to sums over prime numbers. In fact, the first explicit formula of this kind can be found in Riemann's paper \cite{riemann}.

\subsubsection{An explicit formula}
In this section, we reproduce Rudnick and Sarnak's proof for the explicit formula used in \cite{rudnick1996}, which we will compare to our proof of what we call an explicit formula for eigenvalues in Section~\ref{5_sec_explicit_formulae}.

As we have hinted at before, there are various (conjecturally equivalent) notions of $L$-functions. In \cite{rudnick1996}, Rudnick and Sarnak do not work with the Selberg class, but with $L$-functions $L(s, \pi)$ that are attached to an irreducible cuspidal automorphic representation $\pi$ of $GL_m$ over $\Q$. To date, it is not known whether these $L$-functions are members of the Selberg class or not. In particular, the notation introduced below does not match the notation used in Section~\ref{1_sec_L_functs}. In this context, the exact definition of cuspidal automorphic representations is irrelevant. For our purposes, it is sufficient to know that $L(s, \pi)$ has a Euler product and a functional equation:
\begin{enumerate}
\item (functional equation) The completed $L$-function is given by \begin{align*}
\Lambda(s, \pi) = L(s, \pi) \prod_{i = 1}^m \Gamma_\R \left( s + \mu_\pi(i) \right) 
\end{align*} 
where $\Gamma_\R(s) = \pi^{-s/2} \Gamma(s/2)$ and $\Re(\mu_\pi(i)) > -1/2$. If $\Lambda(s, \tilde \pi)$ is the completed $L$-function attached to the contragredient $\tilde \pi$ associated to $\pi$, then \label{1_page_functional_equation_RS}
\begin{align*} 
\Lambda(s, \pi) = \tau(\pi) Q_\pi^{-s} \Lambda(1 - s, \tilde \pi)
\end{align*}
where $\tau(\pi) \in \C \setminus \{0\}$ and $Q_\pi > 0$.
\item (Euler product) For $\Re s > 3/2$, the logarithmic derivative of $L(s, \pi)$ is of the following form:
\begin{align*}
\frac{L'}{L} (s, \pi) = - \sum_{n \geq 1} \frac{b_\pi(n)}{n^s}
\end{align*}
with $b_\pi(n) = 0$ unless $n = p^r$ where $p$ is a prime and $r$ a positive integer.
\end{enumerate}

\begin{thm} [explicit formula, \cite{rudnick1996}] \label{1_thm_explicit_formula} Let $g$ be a smooth compactly supported function. Let $H(s) = \int_{-\infty}^{+\infty} g(u) e^{\imaginary su} du$ and $h(s) = H((s - 1/2)/i)$ for $s \in \C$. If the symbol $\sum_{\rho_\pi}$ stand for the sum over all nontrivial zeros of $L(s, \pi)$, then 
\begin{align}
\begin{split} \label{1_thm_explicit_formula_eq}
\sum_{\rho_\pi} h(\rho_\pi) - \delta(\pi) \left[ h(0) + h(1) \right] ={} & \frac{1}{2 \pi} \int_{-\infty}^{\infty} h \left( 1/2 + \imaginary r \right) \Bigg( \log Q_\pi \\
& \hspace{30pt} + \sum_{i = 1}^m \bigg[ \frac{\Gamma_\R'}{\Gamma_\R} (1/2 + \mu_\pi(i) + \imaginary r) \\
& \hspace{30pt} + \frac{\Gamma_\R'}{\Gamma_\R} (1/2 + \overline{\mu_\pi(i)} - \imaginary r) \bigg] \Bigg) dr \\
& - \sum_{n \geq 1} \left( \frac{b_\pi(n)}{\sqrt{n}} g(\log n) + \frac{\overline{b_\pi(n)}}{\sqrt{n}} g(-\log n) \right)
\end{split}
\end{align}
where $\delta(\pi) = 1$ if $\pi$ corresponds to the Riemann $\zeta$-function and $\delta(\pi) = 0$ otherwise.
\end{thm}

As mentioned above, this formula relates a sum over the nontrivial zeros of $L(s, \pi)$ to a sum over the prime numbers. Indeed, the left-hand side in \eqref{1_thm_explicit_formula_eq} is essentially a sum over the zeros, while the last line on the right-hand side can be viewed as a sum over the primes given that $b_\pi(n)$ vanishes unless $n = p^r$ for some prime $p$ and some positive integer $r$.


\begin{proof}[Sketch of Proof] Recall that in the vertical strip delimited by $-1 \leq \Re s \leq 2$ the function $\Lambda'(s, \pi)/ \Lambda(s, \pi)$ has simple poles at the non-trivial zeros of $L(s, \pi)$ with residues the multiplicity of the zero (and in the case of $\zeta(s)$ a simple pole with residue $-1$ at $s = 0,1$). We consider the following difference of integrals:
\begin{align*}
\text{Zeros}(\pi) \defeq \frac{1}{2\pi \imaginary} \int_{\Re s = 2} \frac{\Lambda'(s, \pi)}{\Lambda(s, \pi)} h(s) ds - \frac{1}{2\pi \imaginary} \int_{\Re s = -1} \frac{\Lambda'(s, \pi)}{\Lambda(s, \pi)} h(s) ds
.
\end{align*}
Given that $h(s)$ is rapidly decreasing in $\Im s$ and entire, we may view this difference of integrals as a closed contour integral. Hence, the residue theorem allows us to infer that $\text{Zeros}(\pi)$ is equal to the left-hand side in \eqref{1_thm_explicit_formula_eq}.

In a next step we show that $\text{Zeros}(\pi)$ is also equal to the right-hand side in \eqref{1_thm_explicit_formula_eq}. To this end, we apply the functional equation to the logarithmic derivative that is integrated along $\Re s = -1$:
\begin{align*}
\text{Zeros}(\pi) ={} & \frac{1}{2\pi \imaginary} \int_{\Re s = 2} \frac{\Lambda'(s, \pi)}{\Lambda(s, \pi)} h(s) ds - \frac{1}{2\pi \imaginary} \int_{\Re s = -1} \left[ -\log Q_\pi - \frac{\Lambda'(1 - s, \tilde{\pi})}{\Lambda(1 - s, \tilde{\pi})}\right] \! h(s) ds
.
\intertext{The change of variables $s \mapsto 1 - s$ yields}
\text{Zeros}(\pi) ={} & \frac{1}{2\pi \imaginary} \int_{\Re s = 2} \frac{\Lambda'(s, \pi)}{\Lambda(s, \pi)} h(s) ds + \frac{1}{2\pi \imaginary} \int_{\Re s = 2} \left[ \log Q_\pi + \frac{\Lambda'(s, \tilde{\pi})}{\Lambda(s, \tilde{\pi})}\right] h(1 - s) ds
.
\end{align*}
Let us show the required expression for the first integral:
\begin{align*}
\frac{1}{2\pi \imaginary} \int_{\Re s = 2} \frac{\Lambda'(s, \pi)}{\Lambda(s, \pi)} h(s) ds ={} & \frac{1}{2\pi \imaginary} \int_{\Re s = 2} \left[ \sum_{i = 1}^m \frac{\Gamma_\R'}{\Gamma_\R}(s + \mu_\pi(i)) + \frac{L'(s, \pi)}{L(s, \pi)} \right] h(s) ds
\end{align*}
For the first term, shifting the contour of integration to $\Re s = 1/2$ gives
\begin{multline*}
\frac{1}{2\pi \imaginary} \int_{\Re s = 2} \sum_{i = 1}^m \frac{\Gamma_\R'}{\Gamma_\R}(s + \mu_\pi(i)) h(s) ds 
\\ = \frac{1}{2\pi} \int_{-\infty}^{+\infty} \sum_{i = 1}^m \frac{\Gamma_\R'}{\Gamma_\R} (1/2 + \mu_\pi(i) + \imaginary r) h(1/2 + \imaginary r) dr
.
\end{multline*}
For the second term, applying the fact that the Euler product provides an explicit expression for $L'(s, \pi)/ L(s, \pi)$ sufficiently far to the right of the critical line results in
\begin{align*} 
\frac{1}{2\pi \imaginary} \int_{\Re s = 2} \frac{L'(s, \pi)}{L(s, \pi)} h(s) ds ={} & - \sum_{n \geq 1} \frac{b_\pi(n)}{\sqrt n} \frac{1}{2 \pi} \int_{-\infty}^{+\infty} H(r) e^{-\imaginary r \log n} dr \\
={} & - \sum_{n \geq 1} \frac{b_\pi(n)}{\sqrt n} g(\log n)
.
\end{align*}
Apply the same computation to the second integral in our expression for $\text{Zeros}(\pi)$, using that $\mu_{\tilde{\pi}}(i) = \overline{\mu_\pi(i)}$ and $b_{\tilde{\pi}}(n) = \overline{b_\pi(n)}$. Collecting the terms gives the right-hand side in \eqref{1_thm_explicit_formula_eq}.
\end{proof}

In summary, the explicit formula is an application of the Euler product and the functional equation; in other words, it is an application of an explicit expression for $L'(s)/L(s)$ that holds sufficiently far to the right of the critical line and a symmetry between the values of $L(s)$ to the right and to the left of the critical line.

\subsubsection[$n$-correlation of zeros of $L$-functions and eigenvalues]{$\boldsymbol{n}$-correlation of zeros of $\boldsymbol{L}$-functions and eigenvalues}
On the one hand, we have seen that for test functions $f$ whose Fourier transform has sufficiently small support 
\begin{align*}
\lim_{N \to \infty} \Cor_n(U(N), f) = \Cor_n(L, f)
;
\end{align*}
in other words, the right-hand sides of the equalities in \eqref{1_thm_n-correlation_eigenvalues_eq} and \eqref{1_thm_n-correlation_zeros_of_L-functs_eq} are identical. On the other hand, the proofs of the two equalities do not have much in common: in the case of eigenvalues, the classical proof is based on Gaudin's lemma, \textit{i.e.}\ the properties of Haar measure on the unitary group; whereas in the case of zeros of $L$-functions, the proof is based on an explicit formula, \textit{i.e.}\ a link between zeros and primes. In consequence, it is hard to grasp the exact nature of the connection between eigenvalues of a random unitary matrix of large size and the zeros of any $L$-function. However, understanding the nature of this connection might be a possible approach towards a proof of the Riemann hypothesis: ``[Montgomery's conjecture] fits well with the view that there is a linear operator (not yet discovered) whose eigenvalues characterize the zeros of the zeta function'' \cite[p.~184]{montgomery73}. In addition the conjecture suggests that ``if there is a linear operator whose eigenvalues characterize the zeros of the zeta function, we might expect that it is complex Hermitian or unitary'' \cite[p.~184]{montgomery73}, in which case its eigenvalues would lie on the unit circle, \textit{i.e.}\ a line.

In \cite{CS}, Conrey and Snaith present a unified approach to $n$-correlation which is applicable to both the eigenvalues of a random unitary matrix and the zeros of the Riemann $\zeta$-function. Their approach yields a formula that gives all of the lower-order terms in the $n$-correlation of both the zeros of the $\zeta$-function (conjecturally) and the eigenvalues of a random unitary matrix (rigorously). Moreover, the structures of the two formulas for $n$-correlation are identical. Their derivation is based on a ratio conjecture for the $\zeta$-function and a ratio theorem for averages of characteristic polynomials from $U(N)$, respectively.

Comparing \cite{CS} and \cite{rudnick1996}, the main disadvantage of Rudnick and Sarnak's method is that there is no (obvious) way to translate if into a random matrix theory framework, as it is based on a connection to prime numbers. On the other hand, the principal disadvantage of Conrey and Snaith's approach is that it is conjectural. Moreover, there is (almost) no hope of somehow making their heuristics rigorous, given that (to our knowledge) all mathematically rigorous work on the zeros of $L$-functions is based on an explicit formula, while the derivation in \cite{CS} is based on a ratio conjecture.

\subsection{On moments} \label{1_sec_on_moments}
As we have discussed in Section~\ref{1_sec_on_correlations}, there is a widely accepted conjecture that the distribution of the heights of the non-trivial zeros of the Riemann $\zeta$-function (or indeed any $L$-function) and the distribution of the eigenangles of a random unitary matrix of large size are the same (when normalized so that their mean spacing equals 1). Recall that the eigenvalues of a matrix $g \in U(N)$ are by definition equal to the zeros of its characteristic polynomial. Therefore, the distribution of the values (on the critical line) taken by an $L$-function might be related to the distribution of the values (on the unit circle) taken by the characteristic polynomial of a random unitary matrix of large size. This groundbreaking idea motivated Keating and Snaith to study the moments of characteristic polynomials of a random matrix in $U(N)$ \cite{KS00zeta}, which might thus be expected to converge to a quantity that is related to the moments of $L$-functions as $N \to \infty$. This study led to the discovery of a (conjectured) relationship between moments of the $\zeta$-function and moments of characteristic polynomials from the unitary group, which has stimulated a great deal of subsequent work. We will present a selection thereof after having discussed \cite{KS00zeta} in a bit more detail. 

Let us define the characteristic polynomial of $g \in U(N)$ (or any other invertible matrix $g$) as \label{symbol_char_pol}
\begin{align*}
\chi_g: \C \to \C; z \mapsto \det \left( I - z g^{-1} \right)  
\end{align*}
where $I$ denotes the identity matrix of the appropriate size.

\begin{defn} [moments of a random characteristic polynomial from $U(N)$] Let $k$ be a non-negative integer. The $2k$-th moment of a random characteristic polynomial from $U(N)$ is given by \label{symbol_moment_char_pol}
\begin{align*}
I_k(U(N)) = \int_{U(N)} \left| \chi_g (1) \right|^{2k} dg
.
\end{align*}
\end{defn}

In \cite{KS00zeta}, Keating and Snaith use the Selberg integral to  derive the following expression for $I_k(U(N))$: for any non-negative integer $k$,
\begin{align} \label{1_eq_moments_of_char_pol}
I_k(U(N)) = \prod_{0 \leq j \leq N - 1} \frac{j! (j + 2k)!}{(j + k)!^2}
.
\end{align}
This equality entails that as $N \to \infty$ the leading term in $I_k(U(N))$ is given by
\begin{align} \label{1_eq_leading_coeff_moments_of_char_pol}
f_k(U(N)) \defeq \lim_{N \to \infty} \frac{1}{N^{k^2}} I_k(U(N)) = \prod_{j = 0}^{k - 1} \frac{j!}{(j + k)!}
.
\end{align}

Let us turn to how Keating and Snaith arrive at a conjecture for the moments of the Riemann $\zeta$-function, taking the equality in \eqref{1_eq_leading_coeff_moments_of_char_pol} as a starting point.

\begin{defn} [moments of $L$-functions] Let $L(s)$ be an $L$-function and $k$ a non-negative integer. The $2k$-th moment of $L$ is given by \label{symbol_moment_L}
\begin{align*}
I_k(L, T) = \frac{1}{T} \int_0^T \left| L(1/2 + \imaginary t) \right|^{2k} dt
.
\end{align*} 
\end{defn}

In 1918, Hardy and Littlewood proved an asymptotic estimate for the second moment $I_1(\zeta, T)$ of the $\zeta$-function as $T \to \infty$. Ingham showed a similar estimate for $I_2(\zeta, T)$ in 1926. To this day, no asymptotic formula for any higher moment has been found. However, there is the following folklore conjecture about the leading term in $I_k(\zeta, T)$ as $T \to \infty$: for any non-negative integer $k$, there exists a constant $f_k$ with the property that
\begin{align*}
\lim_{T \to \infty} \frac{1}{(\log T)^{k^2}} I_k(\zeta, T) = f_k a_k
\end{align*}
where the arithmetic factor $a_k$ is some product over the primes, whose exact definition can be found in \cite[p.~59]{KS00zeta}. Note that $I_k(U(N))$ is of order $N^{k^2}$, while $I_k(\zeta, T)$ is conjectured to be of order $(\log T)^{k^2}$. This suggests an asymptotic correspondence between $N$ and $\log T$. In fact, this correspondence is also obtained by identifying the mean density of the eigenangles (which is equal to $N / 2\pi$) with the means density of the zeros at height $T$ (which is equal to $\log (T / 2\pi) / 2\pi$ by \eqref{1_eq_for_N(T)}) \cite[p.~59]{KS00zeta}. Further note that no value of $f_k$ was suggested for general non-negative integers $k$. In 1998, Keating and Snaith noticed that $$f_k = f_k(U(N))$$ for $k = 1,2$, leading them to conjecture that this equality holds for all non-negative integers $k$. Their conjecture is corroborated by independent conjectures for the values taken by $f_3$ and $f_4$ as well as numerical evidence \cite{KS00zeta}. According to Keating and Snaith's conjecture, as $T \to \infty$ the leading term in the $2k$-th moment of the $\zeta$-function (and presumably other $L$-functions) thus splits into a product of two factors: an arithmetic factor $a_k$ and a random matrix theory factor $f_k$. We will discuss the behavior of the lower order terms in the next section and again in Section~\ref{1_sec_POD12}.

\subsubsection[A recipe for conjecturing the moments of $L$-functions]{A recipe for conjecturing the moments of $\boldsymbol{L}$-functions}
This section is devoted to what we call the CFKRS recipe for conjecturing the lower order terms in the moments of $L$-functions due to Conrey, Farmer, Keating, Rubinstein and Snaith \cite{CFKRS05}. 

The CFKRS recipe is based on a so-called approximate functional equation, which is best explained by writing the functional equation listed in Section~\ref{1_sec_L_functs} in asymmetric form: if $L(s)$ is a member of the Selberg class, then 
\begin{align*}
L(s) = X(s) \overline{L(1 - \bar{s})}
\end{align*}
where $X(s) = \overline{\gamma(1 - \bar{s})} / \gamma(s)$. The approximate functional equation states that for all real numbers $x$, $y$ and for all $s$ in the critical
strip whose imaginary part equals $Cxy$ (where $C$ is some constant that depends on the parameters in the functional equation), it holds that
\begin{align*}
L(s) = \sum_{1 \leq n \leq x} \frac{a_n}{n^s} + X(s) \sum_{1 \leq n \leq y} \frac{\overline{a_n}}{n^{1 - s}} + \error
\end{align*}
with some bounded error. An informal justification of this equality is that if $x$ is large, the right-hand side resembles the Dirichlet series representation of $L(s)$, and if $x$ is small (and $y$ thus comparatively large), the right-hand side resembles the Dirichlet series representation of $X(s) \overline{L(1 - \bar s)}$.

We now give an outline of the CFKRS recipe for conjecturing the $2k$-th moment of an $L$-function proposed in \cite[p.~52ff]{CFKRS05}:
\begin{enumerate}
\item Take the product of $2k$-shifted $Z$-functions: the $Z$-function associated to an $L$-function $L(s)$ is given by
\begin{align*}
Z(s) = X(s)^{-1/2} L(s)
.
\end{align*}
We thus consider a sequence of pairwise distinct real shifts $\calD = (\delta_1, \dots, \delta_{2k})$ and define
\begin{align*}
Z \left( \frac{1}{2} + \imaginary t, \calD \right) = Z \left( \frac{1}{2} + \imaginary t + \delta_1 \right) \cdots Z \left( \frac{1}{2} + \imaginary t + \delta_{2k} \right)
.
\end{align*}
It is worth noting that
\begin{align*}
\lim_{\calD \to 0} \frac{1}{T} \int_0^T Z \left( \frac{1}{2} + \imaginary t, \calD \right) dt = I_k(L, T) 
\end{align*}
because $Z(1/2 + \imaginary t)$ is real and $|X(1/2 + \imaginary t)| = 1$ (for $t \in \R$).
This limit is the quantity we are interested in but the shifts are necessary in order to see the structure of these integrals, and to avoid poles of high order.
\item Replace each $L$-function that appears in $Z \left( 1/2 + \imaginary t, \calD \right)$ by the approximate functional equation, ignoring the error term as well as the bounds on the two sums in the main term. Formally multiply out the resulting expression to obtain a sum of $2^{2k}$ terms.
\item Keep the $\binom{2k}{k}$ terms for which the product of $X$ factors (that come from the definition of $Z(s)$ and the approximate functional equation) is not rapidly oscillating --which is a reasonable simplification because integrals of rapidly oscillating functions are vanishingly small. More concretely, fix one of the $2^{2k}$ terms, \textit{i.e.}\ fix two subsets $\calA, \calB \subset \calD$ that partition $\calD$, then the $X$ factor in this term is equal to
\begin{align*}
X \left(\frac{1}{2} + \imaginary t, \calA, \calB \right) = \prod_{\alpha \in \calA} X \left( \frac{1}{2} + \imaginary t + \alpha \right)^{-1/2} \prod_{\beta \in \calB} X \left( \frac{1}{2} + \imaginary t + \beta \right)^{1/2}
.
\end{align*}
Keeping in mind that $X(s)$ is defined as a rational function of $\Gamma$-functions, one can show that $X \left(1/2 + \imaginary t, \calA, \calB \right)$ is not rapidly oscillating if and only if $|\calA| = k = |\calB|$, in which case
\begin{align*}
X \left(\frac{1}{2} + \imaginary t, \calA, \calB \right) = \left( \frac{Q^{2/m} t}{2} \right)^{\frac{m}{2} \left( \sum_{\alpha \in \calA} \alpha - \sum_{\beta \in \calB} \beta \right)} \left( 1 + O \left( \frac{1}{t} \right) \right)
\end{align*}
as $t \to \infty$. Conclude this step, by using this asymptotic equality to simplify the non-oscillating $X$ factors. These heuristics result in the following simplified expression for $Z(1/2 + \imaginary t)$:
\begin{align*}
g_t \defeq {} & \sum_{I, J} \left( \frac{Q^{2/m} t}{2} \right)^{\frac{m}{2} \left( \sum_{i \in I} \delta_i - \sum_{j \in J} \delta_j \right)} \prod_{i \in I} \left[ \sum_{n \geq 1} \frac{a_n}{n^{1/2 + \imaginary t + \delta_i}} \right] \prod_{j \in J}  \left[ \sum_{n \geq 1} \frac{\overline{a_n}}{n^{1/2 - \imaginary t - \delta_j}}\right]
\end{align*}
where the sum is over subsets $I,J$ of $[2k]$ with $I \cup J = [2k]$ and $|I| = k = |J|$.
\item In each of the remaining $\binom{2k}{k}$ terms of $g_t$, discard everything except the  so-called diagonal sum, and call the resulting expression $M(1/2 + \imaginary t, \calD)$: expanding the products over $I$ and $J$ that appear in $g_t$ yields
\begin{align*}
g_t = {} & \sum_{I, J} \left( \frac{Q^{2/m} t}{2} \right)^{\frac{m}{2} \left( \sum_{i \in I} \delta_i - \sum_{j \in J} \delta_j \right)} \\
& \times \hspace{20pt} \sum_{n_1, \dots, n_{2k} \geq 1} \left[ \prod_{i \in I} \frac{a_{n_i}}{n_i^{1/2 + \delta_i}} \right] \left[ \prod_{j \in J} \frac{\overline{a_{n_j}}}{n_j^{1/2 - \delta_j}} \right] \left[ \frac{\prod_{j \in J} n_j}{\prod_{i \in I} n_i} \right]^{\imaginary t}
.
\end{align*}
For each pair of subsets $I$, $J$ the diagonal sum (\textit{i.e.}\ the part of $g_t$ that contributes to $M(1/2 + \imaginary t, \calD)$) runs through all sequences of positive integers $(n_1, \dots, n_{2k})$ so that $\prod_{i \in I} n_i = \prod_{j \in J} n_j$. Hence,
\begin{align*}
M \left( \frac{1}{2} + \imaginary t, \calD \right) = {} & \sum_{I, J} \left( \frac{Q^{2/m} t}{2} \right)^{\frac{m}{2} \left( \sum_{i \in I} \delta_i - \sum_{j \in J} \delta_j \right)} \\
& \times \hspace{20pt} \sum_{n_1, \dots, n_{2k} \geq 1} \left[ \prod_{i \in I} \frac{a_{n_i}}{n_i^{1/2 + \delta_i}} \right] \left[ \prod_{j \in J} \frac{\overline{a_{n_j}}}{n_j^{1/2 - \delta_j}} \right] 
\end{align*}
where the two sums are over subsets $I,J$ of $[2k]$ with $I \cup J = [2k]$ and $|I| = k = |J|$ and over sequences $(n_1, \dots, n_{2k})$ of $2k$ positive integers so that $\prod_{i \in I} n_i = \prod_{j \in J} n_j$, respectively.
A naive justification for discarding the off-diagonal sums is that
\begin{align*}
\frac{1}{T} \int_0^T \left[ \frac{\prod_{j \in J} n_j}{\prod_{i \in I} n_i} \right]^{\imaginary t} dt = \begin{cases} 1 &\text{if } \prod_{i \in I} n_i = \prod_{j \in J} n_j, \\ o(1) &\text{otherwise.}\end{cases}
\end{align*}
The reason why we call this justification naive is that the number of off-diagonal terms exceeds the number of diagonal terms by far. However, the authors of \cite{CFKRS05} stress that the steps of their recipe cannot be justified.
\item The conjecture is that
\begin{align*}
\frac{1}{T} \int_0^T Z \left( \frac{1}{2} + \imaginary t, \calD \right) dt = \frac{1}{T} \int_0^T M \left( \frac{1}{2} + \imaginary t, \calD \right) dt + O \left( T^{-1/2 + \varepsilon} \right)
.
\end{align*}
\end{enumerate}
In \cite{CFKRS05}, the authors carry out involved ad hoc computations in order to \emph{prove} the following reformulation of the integral on the right-hand side of their conjectured equality:

\begin{conj} [moment conjecture, \cite{CFKRS05}]
\label{1_conj_CFKRSconj} 
Let $L(s)$ be an $L$-function and $k$ a non-negative integer. Then for all $\varepsilon > 0$ 
\begin{align*}
I_k(L, T) = \frac{1}{T} \int_0^T P_k \left(L, m \log \frac{Q^{2/m} t}{2\pi}\right)dt + O \left( T^{-1/2 + \varepsilon} \right)\end{align*}
where the function $t \mapsto P_k(L, t)$ is given by a rather complicated $2k$-fold residue (defined in \cite[p.~63]{CFKRS05}). Recall that the parameters $Q$ and $m$ appear in the functional equation of $L$, and are defined in \eqref{1_eq_defn_completed_L_funct}.
\end{conj}

Although the ``justifications'' for Conjecture~\ref{1_conj_CFKRSconj} and Keating and Snaith's moment conjecture are dissimilar, the two conjectures concur. Indeed, it is shown in \cite[p.~66-71]{CFKRS05} that $P_k(\zeta, t)$ is a polynomial of degree $k^2$ with leading coefficient $a_k f_k$, which means that Conjecture~\ref{1_conj_CFKRSconj} is a refinement of Keating and Snaith's moment conjecture. Furthermore, the formula in \eqref{1_eq_moments_of_char_pol} for the $2k$-th moment of a random characteristic polynomial from $U(N)$ can be recast so that $I_k(U(N))$ and the conjectured expression for $I_k(L, T)$ display the same structure \cite{CFKRS03}, the primary difference between the two expressions being that $I_k(U(N))$ carries additional arithmetic information. 

In the case of the Riemann $\zeta$-function, Conjecture~\ref{1_conj_CFKRSconj} is corroborated by extensive numerical computations. In fact, the same authors have developed an algorithm to obtain meromorphic expressions in $k$ for the coefficients of the polynomial $P_k(\zeta, t)$ in \cite{CFKRS08}. In addition, there already are a few numerical verifications in \cite[p.~91]{CFKRS05}.

\subsection{Properties of characteristic polynomials} \label{1_sec_prop_char_pol}
The philosophy behind the conjectures presented in \cite{KS00zeta} suggests that there is a deep connection between $L$-functions and characteristic polynomials from the unitary group. In this section, we will see that characteristic polynomials even satisfy properties that are analogous to the four characterizing properties of the Selberg class -- with one crucial exception. 

In \cite{CFKRS05} Conrey, Farmer, Keating, Rubinstein and Snaith observe that the (conjectured) properties of $L$-functions that we have listed in Section~\ref{1_sec_L_functs} possess natural analogues in the characteristic polynomials of unitary matrices -- except for the Euler product. According to \cite[p.~39]{CFKRS05}, writing the characteristic polynomial in the expanded form
\begin{align*}
\chi_g(z) = \sum_{n = 0}^N a_n z^n
\end{align*} 
corresponds to representing $L$-functions by Dirichlet series. Let us discuss the analogues of the four characterizing properties of $L$-functions:
\begin{enumerate}
\item (analytic continuation) Given that $\chi_g(z)$ is a polynomial, it is an entire function of finite order. 
\item (functional equation) For unitary matrices $g$ that satisfy $\det(-g) \neq -1$, we introduce the completed characteristic polynomial:
\begin{align} \label{1_defn_completed_char_pol}
\Lambda_g(z) ={} & \det(-g)^{1/2} z^{-N/2} \chi_g(z)
.
\end{align}
Notice that while the characteristic polynomial $\chi_g$ is an entire function, $\Lambda_g$ might only be defined on $\C \setminus \R_-$. Further observe that $\det(-g)^{1/2}$ corresponds to the root number \cite[p.~39]{CFKRS05}. As in the case of completed $L$-functions, the completed characteristic polynomial is designed to satisfy a symmetry relation: for $z \in \C \setminus \R_-$, \label{1_proof_of_funct_eq_for_char_pol}
\begin{align*}
\Lambda_g(z) ={} & \det(-g)^{1/2} z^{-N/2} \det \left( I - z g^{-1} \right) \\
={} & \det(-g)^{-1/2} z^{N/2} \det \left(-z^{-1} g \right) \det \left( I - z g^{-1} \right) \\
={} & \det \left(-g^{-1} \right)^{1/2} z^{N/2} \det \left(-z^{-1} g + I \right)  \\
={} & \Lambda_{g^{-1}} \left( z^{-1}\right) = \overline{\Lambda_g \left( \bar{z}^{-1} \right)}
\end{align*}
where the last equality is due to the assumption that the matrix $g$ be unitary. We see that completed characteristic polynomials are symmetric with respect to the transformation $z \mapsto z^{-1}$, while completed $L$-functions are symmetric with respect to the transformation $z \mapsto 1 - z$. We remark that the transformation $z \mapsto z^{-1}$ fixes the points $\pm 1$ and maps the unit circle to itself, while the transformation $z \mapsto 1 - z$ fixes the (projective) points $1/2$ and $1/2 + \imaginary \infty$ and maps the critical line to itself. Hence, the unit circle is the analogue of the critical line.
\item (Ramanujan conjecture) It is not clear (to me) what the exact analogue should be in this context. However, the fact that the Ramanujan conjecture ensures that the Dirichlet series converges to the right of the line $s = 1 + \imaginary t$ makes it plausible that the coefficients of the characteristic polynomial satisfy any reasonable analogous condition, given that any polynomial converges on the entire complex plane. Moreover, the condition that the first Dirichlet coefficient be normalized to 1 (\textit{i.e.}\ $a_1 = 1$) should correspond to the constant coefficient of the characteristic polynomial being equal to 1 (\textit{i.e.}\ $a_0 = 1$).
\item (Euler product) It is not surprising that the Euler product does not have a natural analogue in the characteristic polynomials of unitary matrices, given that it links $L$-functions to prime numbers. There is no hope of modeling the arithmetic aspect of $L$-functions by means of characteristic polynomials. In Section~\ref{5_sec_explicit_formulae}, we will propose a possible substitute for this missing analogue to the Euler product.
\end{enumerate}
The zeros of the characteristic polynomial $\chi_g$ are the eigenvalues of the matrix $g$. In fact, the characteristic polynomial is equal to the following product:
\begin{align*}
\chi_g(z) = \prod_{\rho \in \calR(g)} \left( 1 - \rho^{-1}z \right)
\end{align*}
where \label{symbol_R(g)} $\calR(g)$ stands for the multiset of the eigenvalues of $g$. As $g \in U(N)$, its eigenvalues all have absolute value 1. We conclude that the zeros of any characteristic polynomial from the unitary group all lie on the unit circle, which corresponds to the critical line. In other words, the Riemann hypothesis is true.

Another property that $L$-functions and characteristic polynomials from the unitary group have in common is that they both satisfy an approximate functional equation. In fact, this commonality is the motivation behind the CFKRS recipe for conjecturing the lower order terms in the moments of $L$-functions. According to \cite[p.~40]{CFKRS05}, the following property of the characteristic polynomial $\chi_g(z)$ of a matrix $g \in U(N)$ is analogous to the approximate functional equation: 
\begin{align*}
\chi_g(z) = \sum_{0 \leq n \leq (N + 1)/2} a_n z^{n} + \overline{ \det\left(-g \right)} z^N \sum_{0 \leq n \leq N/2} \overline{a_n} z^{-n}
.
\end{align*}
It should even be possible to use this approximate functional equation for characteristic polynomials to compute an expression for $I_k(U(N))$ following the CFKRS recipe \cite[p.~40]{CFKRS05}. However, to date nobody seems to have carried out this project.

\subsection{A note on random matrix theory} \label{1_sec_RMT}
In Section~\ref{1_sec_NT_and_RMT}, we have only focused on one aspect of random matrix theory, namely the relatively recent discovery that it is a powerful tool for predicting the behavior of $L$-functions. This section is a brief overview of other aspects of random matrix theory, which is a field in its own right that predates Montgomery's discovery in 1973. We follow \cite{Mehta} in our presentation. 

Random matrix theory originated in mathematical statistics in the 1930s but it did not attract much attention at the time. The study of random matrices revolves around the following question: given a random matrix of large size, what can one say about the behavior of its eigenvalues? It turns out that the same question is relevant for understanding nuclear reactions. In consequence, random matrices were an intensely studied topic in nuclear physics during the 1950s. In more recent years questions concerning the behavior of eigenvalues of random matrices have found applications in various other (seemingly) disparate fields, such as the conductivity in disordered metals, the enumeration of permutations having certain particularities, quantum gravity, theoretical neuroscience, \textit{etc}. 

In summary, what we have presented as random matrix theory predictions in number theory is really an intersection of two fields.

\section{Number theory and symmetric function theory} \label{1_sec_NT_and_alg_com}
In the preceding section we have discussed the predictive power of random matrix theory for the behavior of $L$-functions. This section is about applying symmetric function theory in random matrix theory, and thus also in the theory of $L$-functions (at least conjecturally). With the help of symmetric functions, Bump and Gamburd produce shorter and more elegant proofs for some of the results presented in Section~\ref{1_sec_NT_and_RMT}, such as Keating and Snaith's formula for the moments of a random characteristic polynomial from the unitary group. The methods used  in \cite{bump06} are the subject of Section~\ref{1_sec_BG}. In \cite{dehaye12}, Dehaye applies similar methods to the CFKRS recipe and obtains a neater expression for the conjectured lower order terms of the moments of the Riemann $\zeta$-function. We will give an outline of Dehaye's approach in Section~\ref{1_sec_POD12}. In Section~\ref{1_sec_my_results}, we will give a quick overview of our results, which are inspired by Bump and Gamburd's combinatorial approach to number theoretically motivated problems in random matrix theory.

\subsection{Ratios of characteristic polynomials from the unitary group} \label{1_sec_BG}
In this section, we reproduce Bump and Gamburd's combinatorial derivation of formulas for averages of products/ratios of characteristic polynomials from the unitary group. As we do not expect the reader to be proficient in symmetric function theory, we introduce the required combinatorial background along the way. A more systematic introduction to symmetric function theory can be found in Chapter~\ref{1_cha_algebraic_combinatorics}.

\subsubsection{A product formula}
The primary goal of this section is to present Bump and Gamburd's combinatorial proof of the following theorem \cite[p.~238-239]{bump06}. 
\begin{thm} [product theorem, adapted from \cite{bump06}] \label{1_thm_products_BG} Let $\calA$ and $\calB$ be two finite sets of numbers in $\C \setminus \{0\}$ that contain $n$ and $m$ elements, respectively. Then
\begin{align}
\begin{split} \label{1_eq_prods_thm_1}
\int_{U(N)} \prod_{\alpha \in \calA} \chi_g(\alpha) \prod_{\beta \in \calB} \chi_{g^{-1}}(\beta) dg ={} & \prod_{\beta \in \calB} \beta^N \schur_{\left\langle N^m \right\rangle} \left( \calA \cup B^{-1} \right) 
.
\end{split} \\
\intertext{If the numbers in $\calA \cup \calB^{-1}$ are pairwise distinct (where $\calB^{-1}$ is shorthand for the set $\left\lbrace \beta^{-1}: \beta \in \calB \right\rbrace$), then}
\begin{split} \label{1_eq_prods_thm_2}
\int_{U(N)} \prod_{\alpha \in \calA} \chi_g(\alpha) \prod_{\beta \in \calB} \chi_{g^{-1}}(\beta) dg ={} & \prod_{\beta \in \calB} \beta^N \sum_{\calS, \calT} \frac{\prod_{s \in \calS} s^{n + N}}{\Delta(\calS; \calT)}
\end{split}
\end{align}
where the sum runs over all subsets $\calS$, $\calT$ of $\calA \cup \calB^{-1}$ containing $m$ and $n$ elements, respectively, so that $\calS \cup \calT = \calA \cup \calB^{-1}$, and $\Delta(\calS; \calT) = \prod_{\substack{s \in \calS, t \in \calT}} (s - t)$.
\end{thm}
The second expression for the average of products of characteristic polynomials was first derived in \cite{CFKRS03} (without using any combinatorial methods). The symbol $\schur_{\left\langle N^m \right\rangle}$ in the first expression stands for the so-called Schur function associated to the partition $\left\langle N^m \right\rangle$, which we define in what follows. A partition is a finite sequence of non-increasing non-negative integers. For a partition $\lambda = (\lambda_1, \dots, \lambda_n)$, the elements $\lambda_i$ are called its parts, the sum of its parts is called its size (and denoted $|\lambda|$) and the number of its positive parts is called its length. The symbol $\left\langle N^m \right\rangle$ is a shorthand for the partition of length $m$ whose parts are all equal to $N$. For every set of variables $\calX = (x_1, \dots, x_n)$, define the Schur function associated to $\lambda$ 
\begin{align} \label{1_NT_part_eq_defn_Schur}
\schur_\lambda(\calX) = \frac{\det \left( x_i^{\lambda_j + n - j} \right)_{1 \leq i,j \leq n}}{\prod_{1 \leq i < j \leq n} (x_i - x_j)}
.
\end{align}
If $\calX$ contains more elements than $\lambda$ parts, append zeros to $\lambda$; if the lengths of $\lambda$ exceeds the number of variables in $\calX$, set $\schur_\lambda(\calX) = 0$. One easily verifies that 
\begin{align} \label{1_NT_part_Schur_indexed_by_rectangle}
\schur_{\langle m^n \rangle}(x_1, \dots, x_n) = \prod_{1 \leq i \leq n} x_i^m
.
\end{align}
In general, the Schur function $\schur_\lambda(\calX)$ is a symmetric homogeneous polynomial of degree $|\lambda|$, where symmetric means that $\schur_\lambda(\calX)$ is invariant under permutations of the variables in $\calX$. In fact, the set of all Schur functions $\schur_\lambda(\calX)$ associated to a partition $\lambda$ of size $n$ forms a basis for $\Sym^n(\calX)$, the vector space of the symmetric homogeneous polynomials in $\calX$ of degree $n$ (for any integer $n \geq 0$). The following two properties of Schur functions play a central role in Bump and Gamburd's combinatorial derivation of the product theorem:
\begin{itemize}
\item The dual Cauchy identity states that
\begin{align*}
\sum_\lambda \schur_\lambda(\calX) \schur_{\lambda'}(\calY) = \prod_{\substack{x \in \calX \\ y \in \calY}} (1 + xy)
\end{align*}
where $\lambda'$ is the conjugate of the partition $\lambda$. The definition of the conjugate partition can be found on page \pageref{symbol_conjugate_partition} but here we only need that $\langle m^n \rangle' = \langle n^m \rangle$ for all non-negative integers $m$ and $n$.
\item If $\calR(g)$ denotes the multiset of the eigenvalues of $g \in U(N)$, then
\begin{align*}
\int_{U(N)} \schur_\lambda(\calR(g)) \overline{\schur_\kappa(\calR(g))} dg = \begin{dcases} 1 & \text{if } \lambda = \kappa \text{ and } l(\lambda) \leq N, \\ 0 & \text{otherwise,} \end{dcases}
\end{align*}
\textit{i.e.}\ Schur functions are essentially orthonormal.
\end{itemize}

\begin{proof}[Proof of Theorem~\ref{1_thm_products_BG}] In a first step, reformulate the product of characteristic polynomials to be integrated:
\begin{align*}
\prod_{\alpha \in \calA} \chi_g(\alpha) \prod_{\beta \in \calB} \chi_{g^{-1}}(\beta) ={} & \prod_{\alpha \in \calA} \det \left(I - \alpha g^{-1} \right) \prod_{\beta \in \calB} \det \left( I - \beta g \right) \\
={} & \prod_{\beta \in \calB} \det(-\beta g) \prod_{x \in \calA \cup \calB^{-1}} \det \left( I - x g^{-1} \right).
\intertext{If we let $\calR(g)$ be the multiset of eigenvalues of $g \in U(N)$, then}
\prod_{\alpha \in \calA} \chi_g(\alpha) \prod_{\beta \in \calB} \chi_{g^{-1}}(\beta) ={} & \left( \prod_{\beta \in \calB} (-\beta)^N \right) \left( \prod_{\rho \in \calR(g)} \rho^m \right) \left( \prod_{x \in \calA \cup \calB^{-1}} \prod_{\rho \in \calR(g)} (1 - x \bar{\rho}) \right)
.
\intertext{Use the equality in \eqref{1_NT_part_Schur_indexed_by_rectangle} and the dual Cauchy identity to express the right-hand side in terms of Schur functions:}
\prod_{\alpha \in \calA} \chi_g(\alpha) \prod_{\beta \in \calB} \chi_{g^{-1}}(\beta) ={} & \left( \prod_{\beta \in \calB} (-\beta)^N \right) \schur_{\left\langle m^N \right\rangle} (\calR(g)) \sum_\lambda \schur_{\lambda'} \left(-\left(\calA \cup \calB^{-1} \right) \right) \overline{\schur_\lambda(\calR(g))}
.
\end{align*}
where we use the notation that $-\calA = \{-\alpha: \alpha \in \calA\}$. Hence, Schur orthogonality implies that
\begin{align*}
\int_{U(N)} \prod_{\alpha \in \calA} \chi_g(\alpha) \prod_{\beta \in \calB} \chi_{g^{-1}}(\beta) dg ={} &  \left( \prod_{\beta \in \calB} (-\beta)^N \right) \sum_\lambda \Bigg[ \schur_{\lambda'} \left(-\left(\calA \cup \calB^{-1} \right) \right) \\
& \times \int_{U(N)} \schur_{\left\langle m^N \right\rangle} (\calR(g)) \overline{\schur_\lambda(\calR(g))} dg \Bigg]
\\
={} & \left( \prod_{\beta \in \calB} (-\beta)^N \right) \schur_{\left\langle m^N \right\rangle'} \left(-\left( \calA \cup \calB^{-1} \right) \right)
\\
={} & \left( \prod_{\beta \in \calB} \beta^N \right) \schur_{\left\langle N^m \right\rangle} \left(\calA \cup \calB^{-1} \right)
.
\end{align*}
where the last equality is due to the fact that $\schur_\lambda(\calX)$ is homogeneous of degree $|\lambda|$. This shows the equality in \eqref{1_eq_prods_thm_1}, from which the equality in \eqref{1_eq_prods_thm_2} follows by an application of Lemma~\ref{1_lem_BG_laplace_Schur} stated below this proof.
\end{proof}

\begin{lem} [adapted from \cite{bump06}] \label{1_lem_BG_laplace_Schur} Let $\calX$ be a set of $m + n$ pairwise distinct numbers in $\C$ and let $\lambda = (\lambda_1, \dots, \lambda_{m + n})$ be a partition of length at most $m + n$. If we set $\mu = (\lambda_1 + n, \dots, \lambda_m + n)$ and $\nu = (\lambda_{m + 1}, \dots, \lambda_{m + n})$, then
\begin{align} \label{1_lem_BG_laplace_Schur_eq}
\schur_{\lambda}(\calX) = \sum_{\calS, \calT} \frac{\schur_\mu (\calS) \schur_\nu (\calT)}{\Delta(\calS; \calT)}
\end{align}
where the sum runs over all subsets $\calS$, $\calT$ of $\calX$ containing $m$ and $n$ elements, respectively, so that $\calS \cup \calT = \calX$.
\end{lem}

Bump and Gamburd's proof of this lemma is based on Laplace expansion. Readers who want to refresh their memory will find this classical result from linear algebra on page \pageref{3_section_lapalace_expansion_for_LS}. 

\begin{proof} For any pair of subsets $K$, $L \subset [m + n]$ containing $m$ and $n$, respectively, so that $K \cup L = [m + n]$, let $\sigma_{K, L}$ be the unique permutation in $S_{m + n}$ given by the conditions that $\sigma_{K, L}(K) = [m]$ (and thus $\sigma_{K, L}(L) = \{m + 1, \dots, m + n\}$) and that $\sigma_{K, L}$ respects the relative order of the elements in $K$ and $L$. Expanding the determinant in the numerator along the $m$ left-most columns results in
\begin{align*}
\schur_\lambda (\calX) ={} & \frac{\det \left( x_i^{\lambda_j + m + n - j} \right)_{1 \leq i,j \leq m + n}}{\prod_{1 \leq i < j \leq m + n} (x_i - x_j)} 
\displaybreak[2] \\
 ={} & \sum_{K, L} \frac{\varepsilon(\sigma_{K, L}) \det \left( x_k^{\lambda_j + m + n - j} \right)_{k \in K, 1 \leq j \leq m} \det \left( x_l^{\lambda_j + m + n - j} \right)_{l \in L, m + 1 \leq j \leq m + n}}{\prod_{1 \leq i < j \leq m + n} (x_i - x_j)}
\intertext{where the sum is over subsets $K$, $L \subset [m + n]$ containing $m$ and $n$ elements, respectively, so that $K \cup L = [m + n]$. In order to remove the sign of $\sigma_{K, L}$ from the numerator, we permute the variables in the denominator:}
\schur_\lambda (\calX) ={} & \sum_{K, L} \frac{\det \left( x_k^{\lambda_j + m + n - j} \right)_{k \in K, 1 \leq j \leq m} \det \left( x_l^{\lambda_j + m + n - j} \right)_{l \in L, m + 1 \leq j \leq m + n}}{\prod_{1 \leq i < j \leq m + n} (x_{\sigma_{K, L}(i)} - x_{\sigma_{K, L}(j)})}
\displaybreak[2] \\
={} & \sum_{K, L} \frac{\det \left( x_k^{\lambda_j + m + n - j} \right)_{k \in K, 1 \leq j \leq m} \det \left( x_l^{\lambda_j + n - j} \right)_{l \in L, 1 \leq j \leq n}}{\prod_{i,j \in K: i < j} (x_i - x_j) \prod_{i,j \in L: i < j} (x_i - x_j) \prod_{k \in K, l \in L} (x_k - x_l)}
\displaybreak[2] \\
={} & \sum_{K, L} \frac{\schur_\mu (x_k: k \in K) \schur_\nu (x_l: l \in L)}{\prod_{k \in K, l \in L} (x_k - x_l)}
. \qedhere
\end{align*}
\end{proof}

One indication for the power of applying symmetric function theory in random matrix theory is that Keating and Snaith's formula for the $2k$-th moment of a random characteristic polynomial from $U(N)$ follows immediately from Bump and Gamburd's combinatorial reformulation of the product theorem \cite[p.~239]{bump06}:

\begin{cor} [moments of a random characteristic polynomial from $U(N)$, \cite{KS00zeta}] Let $k$ be a non-negative integer. Then
\begin{align} \label{1_eq_moments_BG}
I_k(U(N)) ={} & \prod_{0 \leq j \leq N - 1} \frac{j!(j + 2k)!}{(j + k)!^2}
.
\end{align}
\end{cor}

\begin{proof} Applying the equality in \eqref{1_eq_prods_thm_1} with $m = k = n$ and $\calA = \{1,\dots, 1\} = \calB$ results in
\begin{align*}
I_k(U(N)) ={} & \int_{U(N)} |\chi_g(1)|^{2k} dg = \int_{U(N)} \chi_g(1)^k \overline{\chi_g(1)^k} dg = \int_{U(N)} \chi_g(1)^k \chi_{g^{-1}}(1)^k dg
\\
={} & \schur_{\left\langle N^k \right\rangle} \left( 1^{2k} \right)
\end{align*}
where we use $1^{2k}$ to denote a set consisting of $2k$ times the number $1$. Now the equality in \eqref{1_eq_moments_BG} is a direct consequence of the Weyl dimension formula, which expresses $\schur_\lambda \left( 1^n \right)$ as a function in the parts of $\lambda$.
\end{proof}

\subsubsection{A ratio formula}
In \cite[p.~245-246]{bump06}, Bump and Gamburd use essentially the same approach to show the following formula for averages of ratios of characteristic polynomials, the main difference being that they work with Littlewood-Schur functions instead of Schur functions. 

\begin{thm} [ratio theorem, adapted from \cite{original_ratios}] \label{1_thm_ratios_BG} Let $\calA$, $\calB$, $\calC$ and $\calD$ be four finite sets of numbers in $\C \setminus \{0\}$ that contain $m$, $n$, $p$ and $q$ elements, respectively. Let the elements in $\calA \cup \calB^{-1}$ be pairwise distinct and let the elements in $\calC \cup \calD$ be less than 1 in absolute value. If $p + q \leq N$, then
\begin{align}
\begin{split} \label{1_thm_ratios_BG_eq}
& \hspace{-15pt} \int_{U(N)} \frac{\prod_{\alpha \in \calA} \chi_g(\alpha) \prod_{\beta \in \calB} \chi_{g^{-1}}(\beta)}{\prod_{\delta \in \calD} \chi_g(\delta) \prod_{\gamma \in \calC} \chi_{g^{-1}} (\gamma)} dg \\
={} & \prod_{\beta \in \calB} (-\beta)^N \sum_{\calS, \calT} \prod_{s \in \calS} (-s)^{m - q + N} \frac{\Delta(\calD; \calS)}{\Delta(\calT; \calS)} \prod_{\substack{\gamma \in \calC \\ \delta \in \calD}} (1 - \gamma \delta)^{-1} \prod_{\substack{t \in \calT \\ \gamma \in \calC}} (1 - t \gamma)
\end{split}
\end{align}
where the sum runs over subsets $\calS$, $\calT \subset \calA \cup \calB^{-1}$ containing $n$ and $m$ elements, respectively, so that $\calS \cup \calT = \calA \cup \calB^{-1}$.
\end{thm}

In Chapter~\ref{4_cha_mixed_ratios}, we generalize Bump and Gamburd's derivation of Theorem~\ref{1_thm_ratios_BG} in order to obtain a formula for averages of so-called mixed ratios of polynomials from unitary matrices. What we call mixed ratios are simply products of logarithmic derivatives and ratios. Furthermore, Theorem~\ref{1_thm_ratios_BG} is the ratio theorem on which Conrey and Snaith's unified approach to $n$-correlation in \cite{CS} is partially based. We will not reproduce Bump and Gamburd's proof of the ratio theorem, as it is conceptually identical to their proof of the product theorem but requires more intricate computations: their first step is to express the integrand on the left-hand side in \eqref{1_thm_ratios_BG_eq} as a linear combination of products of two Schur functions in the variables $\calR(g)$ and $\overline{\calR(g)}$, respectively, where $\calR(g)$ stands for the multiset of eigenvalues of the matrix $g \in U(N)$. In a second step, they use Schur orthogonality to compute the integral on the left-hand side in \eqref{1_thm_ratios_BG_eq}. Their third and last conceptual step consists of simplifying the resulting expression with the help of a generalization of Lemma~\ref{1_lem_BG_laplace_Schur} to Littlewood-Schur functions.

We only reproduce how Bump and Gamburd generalize Lemma~\ref{1_lem_BG_laplace_Schur} to Littlewood-Schur functions, which is the only non-classical property of Littlewood-Schur functions that their proof relies on. Our reason for reproducing their proof is to convince the reader that Bump and Gamburd's derivations of Lemmas~\ref{1_lem_BG_laplace_Schur} and \ref{1_lem_BG_laplace_LS} (\textit{i.e.}\ the lemma for Schur functions and its generalization to Littlewood-Schur functions) are quite different. Littlewood-Schur functions are a generalization of Schur functions that can be characterized by the following recursive property, which relies on the notion of a vertical strip. For partitions $\lambda$ and $\kappa$, we say that $\lambda \setminus \kappa$ is a vertical strip if $0 \leq \lambda_i - \kappa_i \leq 1$ for all $i \geq 1$. For any partition $\lambda$,
\begin{align*}
LS_\lambda(\calX; \emptyset) ={} & \schur_\lambda(\calX)
\intertext{and}
LS_\lambda(\calX; \calY \cup \{y_0\}) ={} & \sum_{\substack{\kappa: \\ \lambda \setminus \kappa \text{ is a vertical strip}}} LS_\kappa(\calX; \calY) y_0^{|\lambda| - |\kappa|}
.
\end{align*}
It is worth noting that the Littlewood-Schur function are doubly symmetric polynomials, \textit{i.e.}\ $LS_\lambda(\calX; \calY)$ is a polynomial in the variables $\calX \cup \calY$ that is symmetric in both $\calX$ and $\calY$ separately, which is not immediate from this definition. In Chapters~\ref{1_cha_algebraic_combinatorics} and \ref{2_cha_det_defn_LS} we will see more common definitions for Littlewood-Schur functions but this is the property that Bump and Gamburd's derivation of the following lemma is based on.

\begin{lem} [adapted from \cite{bump06}] \label{1_lem_BG_laplace_LS}  Let $\calX$ and $\calY$ be sets of $m + n$ and $p$ complex numbers, respectively, so that the elements in $\calX$ are pairwise distinct. Let $\lambda$ be a partition so that $\lambda_m \geq \lambda_{m + 1} + p$. If we set $\mu = (\lambda_1 + n, \dots, \lambda_m + n)$ and $\nu = (\lambda_{m + 1}, \dots)$, then
\begin{align*}
LS_{\lambda} (\calX; \calY) ={} & \sum_{\calS, \calT} \frac{LS_\mu(\calS; \calY) LS_\nu(\calT; \calY)}{\Delta(\calS; \calT)}
\end{align*}
where the sum runs over all subsets $\calS, \calT$ of $\calX$ containing $m$ and $n$ elements, respectively, so that $\calS \cup \calT = \calX$.
\end{lem}

\begin{proof} This proof is an induction on $p$. The base case $p = 0$ is Lemma~\ref{1_lem_BG_laplace_Schur}. For the induction step, suppose that the statement holds for a set $\calY$ that contains $p$ elements and consider $\calY \cup \{y_{p + 1}\}$:
\begin{align*}
LS_\lambda (\calX; \calY \cup \{y_{p + 1}\}) ={} & \sum_{\substack{\kappa: \\ \lambda \setminus \kappa \text{ is a vertical strip}}} LS_\kappa(\calX; \calY) y_{p + 1}^{|\lambda| - |\kappa|}
.
\intertext{Notice that each partition $\kappa$ that appears in this sum satisfies
$$\kappa_m \geq \lambda_m - 1 \geq \lambda_{m + 1} + p + 1 - 1 \geq \kappa_{m + 1} + p.$$ Hence, if we set $\varphi = (\kappa_1 + n, \dots, \kappa_m + n)$ and $\psi = (\kappa_{m + 1}, \dots)$, then the induction hypothesis allows us to infer that}
LS_\lambda (\calX; \calY \cup \{y_{p + 1}\}) ={} & \sum_{\substack{\kappa: \\ \lambda \setminus \kappa \text{ is a vertical strip}}} \sum_{\calS, \calT} \frac{LS_\varphi(\calS; \calY) LS_\psi(\calT; \calY)}{\Delta(\calS; \calT)} y_{p + 1}^{|\lambda| - |\kappa|}
.
\end{align*}
The assumption that $\lambda_m \geq \lambda_{m + 1} + p + 1 > \lambda_{m + 1}$ entails that as $\kappa$ runs through the partitions such that $\lambda \setminus \kappa$ is a vertical strip, $\varphi$, $\psi$ run through the pairs of partitions such that $\mu \setminus \varphi$ and $\nu \setminus \psi$ are vertical strips (where $\mu = (\lambda_1 + n, \dots, \lambda_m + n)$ and $\nu = (\lambda_{m + 1}, \dots)$). Moreover, $|\lambda| - |\kappa| = |\mu| - |\varphi| + |\nu| - |\psi|$, which concludes the proof.
\end{proof} 

In Chapter~\ref{3_cha_overlap_ids}, we use a determinantal definition for Littlewood-Schur functions together with the strategy underlying Bump and Gamburd's proof of Lemma~\ref{1_lem_BG_laplace_Schur} to show what we call overlap identities for Littlewood-Schur functions. It turns out that Lemma~\ref{1_lem_BG_laplace_LS} is the simplest case of the first overlap identity with stronger assumptions on the partition $\lambda$.

\subsection[A combinatorial approach to the moments of the $\zeta$-function]{A combinatorial approach to the moments of the $\boldsymbol{\zeta}$-function} \label{1_sec_POD12}

Arguably the most convincing support for the CFKRS recipe for conjecturing the moments of $L$-functions (presented in Section~\ref{1_sec_on_moments}) is numerical data that matches Conjecture~\ref{1_conj_CFKRSconj} in case of the Riemann $\zeta$-function. Therefore, it would be of great interest to obtain further numerical confirmation. In order to run additional numerical tests, precise explicit expressions for all of the coefficients of the polynomial $P_k(\zeta, t)$ are needed. However, the formula for $P_k(\zeta, t)$ given in \cite{CFKRS05} is rather implicit as it describes the polynomial through a $2k$-fold residue. This issue is partially resolved in \cite{CFKRS08}, in which the authors describe an algorithm for computing the coefficients of $P_k(\zeta, t)$. This algorithm is useful numerically but does not provide explicit expressions for the coefficients. In \cite{dehaye12}, Dehaye uses symmetric function theory to derive neat formulas for all of the coefficients of $P_k(\zeta, t)$, which would also facilitate the numerical computations carried out in \cite{CFKRS08}. 

For integers $r$, $k \geq 0$ with $r \leq k^2$, we define the coefficients $c_r(\zeta, k) = c_r(k)$ through
\begin{align*}
P_k(\zeta, t) = c_0(k) t^{k^2} + c_1(k) t^{k^2 - 1} + \dots + c_{k^2} (k)
.
\end{align*}
The main result in \cite{dehaye12} gives the following combinatorial description for $c_r(k)$:

\begin{thm} [\cite{dehaye12}] \label{1_thm_POD_coeffs}
For integers $r$, $k \geq 0$ with $r \leq k^2$, the coefficient $c_r(k)$ satisfies the equation
\begin{align*}
c_r(k) = \frac{1}{(k^2 - r)!} \sum_{\substack{\kappa, \lambda \text{ partitions} \\ |\kappa| + |\lambda| = r}} d_{\kappa \lambda} \dim(\lambda, S_k(\kappa))
\end{align*}
where $\dim(\mu, \nu)$ denotes the number of standard tableaux of skew shape $\mu / \nu$ (defined on page \pageref{symbol_skew_diagram}), $S_k(\kappa)$ is an explicitly given partition and $d_{\kappa \lambda}$ is a coefficient of arithmetic significance.
\end{thm}
Recall that the leading coefficient $c_0(k)$ is the product of an arithmetic factor $a_k$ and a random matrix theory factor $f_k$ -- although $f_k$ might be more aptly called a combinatorial factor in this context. Theorem~\ref{1_thm_POD_coeffs} shows that the coefficients of the lower order terms should also be viewed as a combination of arithmetic terms (\textit{i.e.}\ $d_{\kappa \lambda}$) and combinatorial terms (\textit{i.e.}\ $\dim(\lambda, S_k(\kappa)) / (k^2 - r)!$) which is the result of a more intricate interaction than simply taking the product. It is worth mentioning that the restriction to the Riemann $\zeta$-function is not inherent in Dehaye's method. In fact, considering $c_r(L, k)$ instead of $c_r(\zeta, k)$ would only affect the arithmetic terms $d_{\kappa \lambda}$ and thus leave the general structure of the formula intact.

The remainder of this section is devoted to a more detailed description of the method used in \cite{dehaye12}. In Section~\ref{1_sec_BG}, we have encountered one family of symmetric polynomials, namely the Schur functions. In order to describe Dehaye's method we introduce another family of symmetric polynomials: for any integer $r \geq 1$ and any set of variables $\calX = (x_1, \dots, x_n)$ define the $r$-th complete symmetric polynomial
\begin{align*}
\complete_r(\calX) = \sum_{1 \leq i_1 \leq i_1 \leq \dots \leq i_r \leq n} x_{i_1} x_{i_2} \cdots x_{i_r}
\end{align*}
and set $\complete_0(\calX) = 1$. In \cite{dehaye12}, Dehaye follows the CFKRS recipe for conjecturing the $2k$-th moment of the Riemann $\zeta$-function (\textit{i.e.}\ he executes the five steps given in Section~\ref{1_sec_on_moments}) and obtains:
\begin{align*}
M \left( \frac 12 + \imaginary t \right) = \left ( \frac{t}{2 \pi} \right )^{\frac 12 \sum_{\delta \in \calD} \delta} \sum_{\calA, \calB} \left ( \frac{t}{2 \pi} \right )^{-\sum_{\beta \in \calB} \beta } \prod_{p \text{ prime}} \left ( \sum_{r \geq 0} \complete_r \left(p^{-\calA} \right) \complete_r \left(p^{\calB}\right) p^{-r} \right)
\end{align*}
where the sum is over subsets $\calA$, $\calB$ of the shifts $\calD$ so that $|\calA| = k = |\calB|$ and $\calA \cup \calB = \calD$. In addition, $p^{-\calA} = \{p^{-\alpha}: \alpha \in \calA\}$ and $p^\calB = \{p^\beta: \beta \in \calB\}$.

At this point, the translation into the language of symmetric functions has merely achieved a neater presentation. However, Dehaye goes on to show that for all integers $n \geq 0$, there is a homogeneous polynomial $g_n(\calA; \calB)$ (in the set of variables $\calA \cup \calB$) of degree $n$ which is symmetric in both $\calA$ and $\calB$ separately such that
\begin{align*}
\prod_{p \text{ prime}} \left ( \sum_{r \geq 0} \complete_r \left(p^{-\calA} \right) \complete_r \left(p^{\calB}\right) p^{-r} \right) ={} & \sum_{n \geq 0} \frac{g_n(\calA; \calB)}{\Delta(\calA; \calB)}
.
\intertext{Given that the set of all Schur functions $\schur_\lambda(\calA)$ (or $\schur_\lambda(\calB)$) associated to a partition $\lambda$ of size $n$ forms a basis for $\Sym^n(\calA)$ (or $\Sym^n(\calB)$) thus entails that}
\prod_{p \text{ prime}} \left ( \sum_{r \geq 0} \complete_r \left(p^{-\calA} \right) \complete_r \left(p^{\calB}\right) p^{-r} \right) ={} & \sum_{\mu, \nu} \frac{c_{\mu \nu} \schur_\mu(\calA) \schur_\nu(\calB)}{\Delta(\calA; \calB)}
\end{align*}
for some coefficients $c_{\mu \nu} \in \C$. In fact, the coefficient $c_{\mu \nu}$ turns out to be equal to $(-1)^{|\nu|} d_{\mu \nu}$ (the coefficient of arithmetic significance in Theorem~\ref{1_thm_POD_coeffs}). To sum up, the following asymptotic equality holds under the assumption of the CFKRS recipe: as $T \to \infty$,
\begin{multline} \label{1_eq_in_proof_POD_for_shifted_moment}
\frac{1}{T} \int_0^T Z \left( \frac{1}{2} + \imaginary t, \calD \right) dt \sim \\ \sum_{\mu, \nu} (-1)^{|\nu|} d_{\mu \nu} \frac{1}{T} \int_0^T \left ( \frac{t}{2 \pi} \right )^{\frac 12 \sum_{\delta \in \calD} \delta} \sum_{\calA, \calB} \left ( \frac{t}{2 \pi} \right )^{-\sum_{\beta \in \calB} \beta} \frac{\schur_\mu(\calA) \schur_\nu(\calB)}{\Delta(\calA; \calB)} dt
\end{multline}
where the sum is over subsets $\calA$, $\calB$ of the shifts $\calD$ such that $|\calA| = k = |\calB|$ and $\calA \cup \calB = \calD$. Let us denote this sum by $K(t, \mu, \nu)$:
\begin{align*}
K(t, \mu, \nu) = \sum_{\calA, \calB} \left ( \frac{t}{2 \pi} \right )^{-\sum_{\beta \in \calB} \beta} \frac{\schur_\mu(\calA) \schur_\nu(\calB)}{\Delta(\calA; \calB)}
.
\end{align*}
Given that the exponent of $\left ( t/ 2 \pi \right )$ is symmetric in the set variables $\calB$ it is possible (and actually easy, owing to the Pieri rule, which we state on page \pageref{1_thm_pieri_rule}) to write $(t / 2 \pi)^{-\sum_{\beta \in \calB} \beta} \schur_\nu(\calB)$ as an infinite linear combination of Schur functions in the set of variables $\calB$. It is important to keep in mind that the coefficients of this linear combination depend on $t$. Concretely,
\begin{align*}
K(t, \mu, \nu) = \sum_\kappa c_{\nu, t}(\kappa) \sum_{\calA, \calB} \frac{\schur_\mu(\calA) \schur_\kappa(\calB)}{\Delta(\calA; \calB)}
.
\end{align*}

In \cite{dehaye12}, Dehaye observes that the resulting sum over $\calA$, $\calB$ is reminiscent of the right-hand side in \eqref{1_lem_BG_laplace_Schur_eq}. Recall that Bump and Gamburd derive the identity in \eqref{1_lem_BG_laplace_Schur_eq} by expanding the determinant in the numerator of $\schur_\lambda(\calX)$ along the $m$ left-most columns. Dehaye generalizes the identity in \cite{bump06} by expanding the determinant in question along any fixed set of $m$ columns:

\begin{lem} [\cite{dehaye12}] \label{1_lem_POD_laplace} Let $\calX$ be a set of $m + n$ pairwise distinct numbers in $\C$ and let $\mu = (\mu_1, \dots, \mu_m)$, $\nu = (\nu_1, \dots, \nu_n)$ be two partitions. Suppose that there exists a partition $\lambda = (\lambda_1, \dots, \lambda_{m + n})$ so that
\begin{multline} \label{1_eq_lem_POD_laplace_condition}
\{\lambda_i + m + n - i: 1 \leq i \leq m + n\} \\
= \{\mu_i + m - i: 1 \leq i \leq m \} \cup \{\nu_i + n - i: 1 \leq i \leq n\},
\end{multline}
then
\begin{align} \label{1_eq_lem_POD_laplace_indentity}
\varepsilon(\mu, \nu) \schur_{\lambda}(\calX) = \sum_{\calS, \calT} \frac{\schur_\mu (\calS) \schur_\nu (\calT)}{\Delta(\calS; \calT)}
\end{align}
where the sum runs over all subsets $\calS$, $\calT$ of $\calX$ containing $m$ and $n$ elements, respectively, so that $\calS \cup \calT = \calX$, and $\varepsilon(\mu, \nu) = \pm 1$ is some sign. If no partition satisfies the condition in \eqref{1_eq_lem_POD_laplace_condition}, the right-hand side in \eqref{1_eq_lem_POD_laplace_indentity} is equal to zero.
\end{lem}
We will revisit this lemma in Chapter~\ref{3_cha_overlap_ids}. In fact, if two partitions $\mu$ and $\nu$ satisfy the condition in \eqref{1_eq_lem_POD_laplace_condition} for some partition $\lambda$, we will call $\lambda$ the \label{1_page_overlap_defn} $(m,n)$-overlap of $\mu$ and $\nu$, for which we will introduce the notation $\lambda = \mu \star_{m,n} \nu$. In addition, we will call Lemma~\ref{1_lem_POD_laplace} the first overlap identity for Schur functions, which is a special case of the so-called first overlap identity for Littlewood-Schur functions. Coming back to Dehaye's proof of Theorem~\ref{1_thm_POD_coeffs}, we see that Lemma~\ref{1_lem_POD_laplace} implies that $K(t, \mu, \nu)$ can be written as linear combination of Schur functions $\schur_\lambda(\calD)$ whose coefficients depend on $t$. In fact, only the coefficient corresponding to the empty partition matters. Indeed, the $2k$-th moment of the $\zeta$-function is equal to
\begin{align*}
\lim_{\calD \to 0} \frac{1}{T} \int_0^T Z \left( \frac{1}{2} + \imaginary t, \calD \right) dt,
\end{align*}
entailing that we will eventually need to let shifts in $\calD$ go the zero in \eqref{1_eq_in_proof_POD_for_shifted_moment}. Moreover, the fact that $\schur_\lambda(\calD)$ is a homogeneous polynomial of degree $|\lambda|$ implies that $$\lim_{\calD \to 0} \schur_\lambda(\calD) = 0$$ unless $\lambda$ is the empty partition. Therefore, the following remark about pairs of partitions that satisfy the condition in \eqref{1_eq_lem_POD_laplace_condition} for $\lambda = \emptyset$ is the last missing ingredient in Dehaye's proof of Theorem~\ref{1_thm_POD_coeffs}.
\begin{rem} [{\cite[p.~3]{mac}}] \label{1_rem_POD_visual_empty_partition} Let $\mu = (\mu_1, \dots, \mu_m)$ and $\nu = (\nu_1, \dots, \nu_n)$ be two partitions. Then
\begin{align*}
\{m + n - i: 1 \leq i \leq m + n\}
= \{\mu_i + m - i: 1 \leq i \leq m \} \cup \{\nu_i + n - i: 1 \leq i \leq n\}
\end{align*}
if and only if the conjugate partition of $\nu$ is given by $\nu' = (n - \mu_m, \dots, n - \mu_1)$.
\end{rem} 
This condition possesses a beautiful graphical interpretation, which becomes apparent by viewing partitions as Ferrers diagrams. In Chapter~\ref{3_cha_overlap_ids}, we will extend this visualization of the condition in \eqref{1_eq_lem_POD_laplace_condition} to all partitions $\lambda$. 

\subsection{Overview of the new results in this thesis}
\label{1_sec_my_results}
In this section we give a brief overview of our results. More detailed descriptions can be found in the introductions to Chapters~\ref{3_cha_overlap_ids} and \ref{4_cha_mixed_ratios}.

Our first results are two overlap identities for Littlewood-Schur functions. (Recall that we have introduced the overlap of two partitions on page \pageref{1_page_overlap_defn}.) What we call the first overlap identity is an equality of the following form: let $\calX$, $\calY$ be sets of complex numbers so that $\calX$ consists of $m + n$ pairwise distinct variables, then
\begin{align*}
LS_\lambda(\calX; \calY) = \sum_{\calS, \calT} \frac{\varepsilon(\mu, \nu) LS_\mu(\calS; \calY) LS_\nu(\calT; \calY)}{\Delta(\calS; \calT)}
\end{align*}
where the sum runs over all subsets $\calS$, $\calT$ of $\calX$ containing $m$ and $n$ elements, respectively, so that $\calS \cup \calT = \calX$. Here, the indexing partitions $\mu$, $\nu$ are constructed from a pair of partitions whose $(m, n - k)$-overlap equals $(\lambda_1, \dots, \lambda_{m + n - k})$ where $k$ denotes the so-called \textit{index} of the partition $\lambda$. Finally, $\varepsilon(\mu, \nu) = \pm 1$ is some sign. The first overlap identity (which is formally stated in Theorem~\ref{3_thm_laplace_expansion_LS_cut_before_index}) generalizes both Lemma~\ref{1_lem_POD_laplace} (due to Dehaye) and Lemma~\ref{1_lem_BG_laplace_LS} (due to Bump and Gamburd).

While the first overlap identity represents the Littlewood-Schur function $LS_\lambda(\calX; \calY)$ as a sum over pairs of subsets of $\calX$, the second overlap identity essentially represents the Littlewood-Schur function indexed by the partition $\lambda$ as a sum over pairs of partitions whose overlap equals $\lambda$. To avoid introducing new notation, we only state its specialization to Schur functions, which also seems to be new: let $\calS$ and $\calT$ be sets consisting of $m$ and $n$ variables, respectively, and let $\calX$ be the disjoint union of $\calS$ and $\calT$. If $\Delta(\calS; \calT) \neq 0$, then
\begin{align*}
\schur_\lambda(\calX) ={} & \sum_{\substack{\mu, \nu: \\ \mu \star_{m,n} \nu = \lambda}} \frac{\varepsilon(\mu, \nu) \schur_\mu(\calS) \schur_\nu(\calT)}{\Delta(\calS; \calT)}
\end{align*}
where $\varepsilon(\mu, \nu)$ is some sign. In the full statement for Littlewood-Schur functions (which is given in Theorem~\ref{3_thm_lapalace_expansion_new_LS}), there is an additional sum over all ways to split the second set of variables $\calY$ into two subsets. 

As this equality expresses $\schur_\lambda(\calX)$ as a sum over pairs of partitions whose $(m,n)$-overlap equals $\lambda$, simple characterizations of this set facilitate applications of the second overlap identity. Therefore, we give two visual characterization for pair of partitions with the same overlap, which both generalize the equivalence stated in Remark~\ref{1_rem_POD_visual_empty_partition}.

Both overlap identities are based on a determinantal formula for Littlewood-Schur functions. In \cite{vanderjeugt}, Moens and Van der Jeugt introduce a formula for Littlewood-Schur functions whose structure mirrors that of the definition for Schur functions given in \eqref{1_NT_part_eq_defn_Schur}. More precisely, they describe $LS_\lambda(\calX; \calY)$ as a product of three factors: a factor that only depends on the variables in $\calX$ and $\calY$, a determinant and a sign. Recall that Dehaye derives the equality in \eqref{1_eq_lem_POD_laplace_indentity} by expanding the determinant in the definition for the Schur function $\schur_\lambda(\calX)$ along subsets of the columns. We obtain the first overlap identity for Littlewood-Schur functions by expanding the determinant in Moens and Van der Jeugt's formula for the Littlewood-Schur function $LS_\lambda(\calX; \calY)$ along specific subsets of the columns. Expanding the determinant in the determinantal formula for Littlewood-Schur functions along specific subsets of the \emph{rows} results in the second overlap identity for Littlewood-Schur functions. 

\bigskip 

Our second result is an asymptotic formula for the following average of mixed ratios of characteristic polynomials:
\begin{align} \label{1_eq_mixed_ratios}
\int_{U(N)} \frac{\prod_{\alpha \in \calA} \chi_g(\alpha) \prod_{\beta \in \calB} \chi_{g^{-1}}(\beta)}{\prod_{\delta \in \calD} \chi_g(\delta) \prod_{\gamma \in \calC} \chi_{g^{-1}} (\gamma)} \prod_{\varepsilon \in \calE} \frac{\chi'_g(\varepsilon)}{\chi_g(\varepsilon)} \prod_{\varphi \in \calF} \frac{\chi'_{g^{-1}}(\varphi)}{\chi_{g^{-1}}(\varphi)} dg
.
\end{align}
In fact, we describe this integral over the unitary group $U(N)$ as a main term plus a remainder. Under some assumptions on the sets of variables $\calA$, $\calB, \dots, \calF$, we can prove that the remainder converges to 0 exponentially fast as $N$ goes to infinity. By setting $\calA, \dots, \calD$ equal to the empty set, this result implies a simple combinatorial expression for the leading term in Conrey and Snaith's formula for products of logarithmic derivatives \cite{CS}. In addition, it generalizes Bump and Gamburd's product/ratio theorem (stated in Theorem~\ref{1_thm_products_BG}/Theorem~\ref{1_thm_ratios_BG}). In fact, our proof is a generalization of Bump and Gamburd's combinatorial method, which we have outlined in Section~\ref{1_sec_BG}. It is based on three additional combinatorial results: 
\begin{enumerate}
\item In order to express the integrand in \eqref{1_eq_mixed_ratios} as an infinite linear combination of products of two Schur functions in the variables $\calR(g)$ and $\overline{\calR(g)}$, respectively, we do not only use a Cauchy identity (as in Bump and Gamburd's proof of the product and ratio theorems) but also \emph{the classical Murnaghan-Nakayama rule for Schur functions}. A statement of this equality on Schur functions is given in Section~\ref{2_sec_det_proof_of_MN_for_Schur}. Once the integrand is of the required form, we follow Bump and Gamburd in computing the integral over $U(N)$ by employing the fact that Schur functions are essentially orthonormal.
\item After having computed the integral, we simplify the resulting expression with the help of the equality in Lemma~\ref{1_lem_BG_laplace_LS}. This simplification is exactly analogous to the simplification in Bump and Gamburd's proof of the ratio theorem, except that in our generalized setting the precondition of Lemma~\ref{1_lem_BG_laplace_LS} on the indexing partition $\lambda$ is not satisfied. However, Lemma~\ref{1_lem_BG_laplace_LS} is also the simplest case of \emph{the first overlap identity for Littlewood-Schur functions}, which holds under sufficiently weak assumptions.
\item In the last step, we need to simplify a term that does not appear in the simpler case considered by Bump and Gamburd. To this end we consider two ``power sum'' operators on the ring of symmetric functions, which should be thought of as the vector space of all symmetric polynomials, or equivalently, the vector space spanned by the Schur functions $\schur_\lambda$ as $\lambda$ runs through the partitions. In essence, we express our combinatorial formula as composed power sum operators acting on the Schur function associated to the empty partition, in order to derive a simpler formula. 

We have taken the idea of simplifying combinatorial expressions by recasting them as operators acting on $\schur_\emptyset$ from the so-called \emph{vertex operator formalism}. We briefly explain this method, following \cite{Steeptilings}: the vertex operators $\Gamma_+(t)$ and $\Gamma_-(t)$ on the formal vector space spanned by the partitions are defined by
\begin{align*}
\Gamma_+(t)[\lambda] = \sum_{\substack{\kappa: \\ \kappa \setminus \lambda \text{ is a horizontal strip}}} t^{|\lambda| - |\kappa|} [\kappa]
\intertext{}
\Gamma_-(t)[\lambda] = \sum_{\substack{\kappa: \\ \lambda \setminus \kappa \text{ is a horizontal strip}}} t^{|\kappa| - |\lambda|} [\kappa]
\end{align*}
where $t$ is a formal variable and $\lambda \setminus \kappa$ is called a horizontal strip if $0 \leq \lambda'_i - \kappa'_i \leq 1$ for all $i \geq 1$. Given that the Schur functions form a basis of the ring of symmetric functions which is indexed by partitions, we may regard the two vertex operators as operators on the ring of symmetric functions. The following two properties are the reason why the vertex operators are such a useful tool:
\begin{itemize}
\item by definition, $\Gamma_-(t)[\emptyset] = [\emptyset]$;
\item they satisfy the following commutation relations:
\begin{align*}
\Gamma_-(t) \Gamma_+(u) = \frac{1}{1 - tu} \Gamma_+(t) \Gamma_-(u),
\end{align*}
while vertex operators with the same sign commute.
\end{itemize}
This means that we can simplify expressions of the type $\Gamma_\pm(t_1) \dots \Gamma_\pm(t_1) [\emptyset]$ by moving the vertex operators with negative sign to the right. It turns out that the power sum operators satisfy similar properties, making it possible to apply the same trick.
\end{enumerate} 

\bigskip

Our third results consist of two statements that admit number theoretic interpretations. The formula for mixed ratios allows us to infer the following asymptotic expression for averages of products of completed characteristic polynomials: let $\calE$ and $\calF$ be sets of numbers in $\C \setminus \{0\}$ having absolute value strictly less than 1. If $\calE$ and $\calF$ contain $m$ and $n$ variables, respectively, then
\begin{align} \label{1_eq_average_of_log_ders}
\begin{split}
& \int_{U(N)} \prod_{\varepsilon \in \calE} \varepsilon \frac{\Lambda'_g(\varepsilon)}{\Lambda_g(\varepsilon)} \prod_{\varphi \in \calF} \varphi \frac{\Lambda'_{g^{-1}}(\varphi)}{\Lambda_{g^{-1}}(\varphi)} dg = \sum_\lambda \left( -\frac N2 \right)^{m + n - 2l(\lambda)}
z_\lambda \monomial_\lambda(\calE) \monomial_\lambda(\calF) + \error
\end{split}
\end{align}
where $z_\lambda = \prod_{i \geq 1} i^{m_i(\lambda)} m_i(\lambda)!$ and $\monomial_\lambda$ is the monomial symmetric polynomial indexed by $\lambda$ (defined on page \pageref{symbol_monomial_poly}). We use this equality to derive what we call an \emph{explicit formula for eigenvalues}, which establishes an asymptotic expression for a sum over zeros of a random characteristic polynomial from the unitary group (stated in Theorem~\ref{4_thm_explicit_formula}). Its principal feature is that its proof mirrors Rudnick and Sarnak's proof of the explicit formula for zeros of $L$-functions reproduced in Section~\ref{1_sec_on_correlations}, making the connection between the two worlds stronger. Recall that the explicit formula for zeros of $L$-functions is based on the functional equation and the Euler product, \textit{i.e.}\ an explicit expression for the logarithmic derivative $L' / L$ that holds sufficiently far to the right of the critical line. In analogy, our proof of the explicit formula for eigenvalues is based on the functional equation (for characteristic polynomials) and the equality in \eqref{1_eq_average_of_log_ders}, \textit{i.e.}\ an explicit expression for the average of the logarithmic derivative that holds inside the unit circle.

\section{Structure of this thesis}
Chapter~\ref{1_cha_algebraic_combinatorics} is a brief introduction to symmetric function theory with focus on Schur and Littlewood-Schur functions. In Chapter~\ref{2_cha_det_defn_LS}, we present Moens and Van der Jeugt's determinantal formula for Littlewood-Schur functions as well as one application thereof, namely an elementary proof for the generalization of the Murnaghan-Nakayama rule to Littlewood-Schur functions. Chapters~\ref{3_cha_overlap_ids} and \ref{4_cha_mixed_ratios} contain the main results of this thesis: In Chapter~\ref{3_cha_overlap_ids}, we introduce the notion of overlapping two partitions and derive the two overlap identities for Littlewood-Schur functions. We also discuss two visual characterizations for the set of all pairs of partitions with the same overlap. In Chapter~\ref{4_cha_mixed_ratios}, we first present a combinatorial method for computing averages of mixed ratios of characteristic polynomials from the unitary group. In a second part, we apply our method to products of logarithmic derivatives of completed characteristic polynomials, which results in a natural combinatorial expression that holds inside the unit disc. From this expression we then derive an explicit formula for eigenvalues. In Chapter~\ref{5_cha_conclusions}, we consider our combinatorial results from the point of view of number theoretic applications.

\chapter{Symmetric and Doubly Symmetric Functions} \label{1_cha_algebraic_combinatorics}
In this chapter we give the combinatorial background required for the subsequent chapters. More concretely, this chapter is a brief introduction to Schur and Littlewood-Schur functions, which are symmetric and doubly symmetric, respectively. 

Some of the definitions and properties stated in this chapter will be re-introduced at the beginning of Chapters~\ref{3_cha_overlap_ids} and \ref{4_cha_mixed_ratios}. We apologize to chronological readers for this repetitiveness. To compensate, we will tell linear readers which sections they may safely skip.

\section{Schur functions} \label{1_sec_Schur_functions}
In the subsequent chapters we will mainly work with a generalization of Schur functions -- the so-called Littlewood-Schur functions, which are the subject of Section~\ref{1_sec_LS_functions}. However, the classic Schur functions remain an important special case. Moreover, our combinatorial approach to mixed ratios, which is the primary focus of Chapter~\ref{4_cha_mixed_ratios}, relies on the orthogonality of Schur functions with respect to two inner products that arise in completely different contexts.

Schur functions are symmetric functions that are indexed by partitions. We first discuss the combinatorial objects that are required to define Schur functions combinatorially, namely partitions and Young tableaux. Then we present a framework for Schur functions, \textit{i.e.}\ we formally introduce the ring of symmetric functions. In the next section, we give a combinatorial as well as a determinantal definition for Schur functions. In the last section, we consider two reasons why Schur functions are the most natural basis for the ring of symmetric functions: orthogonality and their connection to representation theory.

\subsection{Partitions and tableaux}
A \label{symbol_partition} partition is a non-increasing finite sequence $\lambda = (\lambda_1, \dots, \lambda_n)$ of non-negative integers, called parts. If two partitions only differ by a string of zeros, we regard them as equal. The number of non-zero parts of $\lambda$ is its length, denoted \label{symbol_length_of_partition} $l(\lambda)$. The size of a partition $\lambda$ is the sum of its parts, denoted \label{symbol_size_of_partition} $|\lambda|$. For any integer $i \geq 1$, \label{symbol_multiplicity} $m_i(\lambda)$ is the number of parts of $\lambda$ that are equal to $i$. It is sometimes convenient to use a notation for partitions that
makes multiplicities explicit:
$$\lambda = \left\langle 1^{m_1(\lambda)} 2^{m_2(\lambda)} \dots i^{m_i(\lambda)} \dots \right\rangle.$$

It is often helpful to visualize partitions by means of Ferrers diagrams. The Ferrers diagram of a partition $\lambda$ is defined as the set of points $(i, j) \in \Z \times \Z$ such that $1 \leq i \leq \lambda_j$. When drawing diagrams, we replace dots by square boxes. As an example, the Ferrers diagram of the partition $\lambda = (5,5,2)$ is given by
\begin{center}
\begin{tikzpicture}
\node(mu) at (-0.5, 0.75) {$\lambda =$};
\draw[step=0.5cm, thin] (0, 0) grid (1, 0.5);
\draw[step=0.5cm, thin] (0, 0.5) grid (2.5, 1.5);
\end{tikzpicture}
\end{center}
Given two partitions $\kappa$ and $\lambda$, we say that $\kappa$ is a subset of $\lambda$ if the diagram of $\kappa$ is a subset of the diagram of $\lambda$, which we denote by \label{symbol_subset} $\kappa \subset \lambda$. We remark that $\kappa \subset \lambda$ if and only if $\lambda_i \geq \kappa_i$ for all $i \geq 1$. The conjugate of a partition $\lambda$ is the partition \label{symbol_conjugate_partition} $\lambda'$ whose diagram is the transpose of the diagram of $\lambda$, \textit{e.g.}\ we see that the conjugate of $(5,5,2)$ is $(3,3,2,2,2)$. Formally the $i$-th part of $\lambda'$ is defined as the number of parts of $\lambda$ that are greater or equal to $i$. In particular, $\lambda'_1 = l(\lambda)$.

We define two binary operations on partitions, namely addition and union. Let $\kappa$ and $\lambda$ be partitions. The sum of $\kappa$ and $\lambda$ is the partition \label{symbol_sum_of_partitions} $\kappa + \lambda$ given by
$$(\kappa + \lambda)_i = \kappa_i + \lambda_i.$$
The union of $\kappa$ and $\lambda$ is the sequence \label{symbol_union_of_partitions} $\kappa \cup \lambda = \left(\kappa_1, \dots, \kappa_{l(\kappa)}, \lambda_1, \dots, \lambda_{l(\lambda)}\right)$, which need not be a partition. Our notion of union is unconventional. Usually $\kappa \cup \lambda$ is defined as the partition whose parts are those of $\lambda$ and $\mu$, \emph{arranged in descending order}, such as in \cite[p.~6]{mac}. If the union $\kappa \cup \lambda$ is a partition, then
$$\left( \kappa \cup \lambda \right)' = \kappa' + \lambda'.$$
We justify this classical duality by regarding the partitions in question as diagrams: on the one hand, the $i$-th row of $\left( \kappa \cup \lambda \right)'$ is equal to the $i$-th column of $\kappa \cup \lambda$. As the diagram of $\kappa \cup \lambda$ is obtained by taking the rows of $\lambda$ and putting them below the rows of $\kappa$, the $i$-th column of $\kappa \cup \lambda$ is the sum of the $i$-th columns of $\kappa$ and $\lambda$. On the other hand, the $i$-th row of $\kappa' + \lambda'$ is the sum of the $i$-th rows of $\kappa'$ and $\lambda'$, which is also equal to the sum of the $i$-th columns of $\kappa$ and $\lambda$.

\bigskip
A Ferrers diagram can be viewed as a collection of empty boxes whose shape determines a partition $\lambda$. If we fill the boxes of a Ferrers diagram with some symbols, we obtain a Young tableau:

\begin{defn} [Young tableau] Let $\lambda$ be a partition of size $n$. A Young tableau of shape $\lambda$ is obtained by filling the boxes of the Ferrers diagram of $\lambda$ with the numbers $1, 2, \dots, n$ so that every number appears exactly once.
\end{defn}
Let us illustrate this definition with a Young tableau of shape $(5,5,2)$: 
\begin{center}
\begin{tikzpicture}
\draw[step=0.5cm, thin] (0, 0) grid (1, 0.5);
\draw[step=0.5cm, thin] (0, 0.5) grid (2.5, 1.5);
\node at (0.25, 0.25) {\small{10}};
\node at (0.75, 0.25) {\small{3}};
\node at (0.25, 0.75) {\small{6}};
\node at (0.75, 0.75) {\small{7}};
\node at (1.22, 0.75) {\small{12}};
\node at (1.75, 0.75) {\small{1}};
\node at (2.25, 0.75) {\small{9}};
\node at (0.25, 1.25) {\small{11}};
\node at (0.75, 1.25) {\small{4}};
\node at (1.25, 1.25) {\small{5}};
\node at (1.75, 1.25) {\small{8}};
\node at (2.25, 1.25) {\small{2}};
\end{tikzpicture}
\end{center}

Our focus is on generalized Young tableaux, given that what we will call the combinatorial definition for Schur functions describes the Schur function associated to a partition $\lambda$ as a sum over generalized semistandard Young tableaux of shape $\lambda$.

\begin{defn} [generalized Young tableau] Let $\lambda$ be a partition. A generalized Young tableau of shape $\lambda$ (or a generalized $\lambda$-tableau) is obtained by filling the boxes of the Ferrers diagram of $\lambda$ with positive integers, allowing repetitions. The content of a generalized $\lambda$-tableau $T$ is the sequence of non-negative integers $\mu = (\mu_1, \mu_2, \dots)$ where $\mu_i$ is given by the number of $i$'s that appear in $T$.
\end{defn}

\begin{defn} [semistandard Young tableau] A generalized Young tableau is semistandard if its rows are weakly increasing and its columns are strictly increasing.
\end{defn}

To give an explanation for this terminology, we remark that a standard Young tableau is a semistandard tableau that is not generalized, which entails that both its rows and its columns are strictly increasing. 

The following example shows two generalized Young tableaux of shape $(5,5,2)$ and with content $(2, 0, 3, 4, 2, 0, 1)$. The tableau to the left is semistandard, while the tableau to the right is not.
\begin{equation} \label{1_ex_semistandard_tableaux}
\begin{tikzpicture}
\draw[step=0.5cm, thin] (0, 0) grid (1, 0.5);
\draw[step=0.5cm, thin] (0, 0.5) grid (2.5, 1.5);
\node at (0.25, 0.25) {\small{5}};
\node at (0.75, 0.25) {\small{5}};
\node at (0.25, 0.75) {\small{3}};
\node at (0.75, 0.75) {\small{4}};
\node at (1.25, 0.75) {\small{4}};
\node at (1.75, 0.75) {\small{4}};
\node at (2.25, 0.75) {\small{7}};
\node at (0.25, 1.25) {\small{1}};
\node at (0.75, 1.25) {\small{1}};
\node at (1.25, 1.25) {\small{3}};
\node at (1.75, 1.25) {\small{3}};
\node at (2.25, 1.25) {\small{4}};
\end{tikzpicture}
\hspace{2cm}
\begin{tikzpicture}
\draw[step=0.5cm, thin] (0, 0) grid (1, 0.5);
\draw[step=0.5cm, thin] (0, 0.5) grid (2.5, 1.5);
\node at (0.25, 0.25) {\small{7}};
\node at (0.75, 0.25) {\small{4}};
\node at (0.25, 0.75) {\small{1}};
\node at (0.75, 0.75) {\small{4}};
\node at (1.25, 0.75) {\small{4}};
\node at (1.75, 0.75) {\small{4}};
\node at (2.25, 0.75) {\small{5}};
\node at (0.25, 1.25) {\small{1}};
\node at (0.75, 1.25) {\small{3}};
\node at (1.25, 1.25) {\small{3}};
\node at (1.75, 1.25) {\small{3}};
\node at (2.25, 1.25) {\small{4}};
\end{tikzpicture}
\end{equation}

The notion of a horizontal strip allows us to give an alternative description of semistandard tableaux, which will lead to identities for Schur functions:
\begin{defn} [horizontal and vertical strips] Let $\kappa$ and $\lambda$ be partitions. If $\kappa$ is a subset of $\lambda$, then the corresponding skew diagram is the set of boxes \label{symbol_skew_diagram} $\lambda \setminus \kappa$ that are contained in $\lambda$ but not in $\kappa$. If $0 \leq \lambda'_i - \kappa'_i \leq 1$ for all $i \geq 1$, the the skew diagram $\lambda \setminus \kappa$ is called a horizontal strip; if $0 \leq \lambda_i - \kappa_i \leq 1$ for all $i \geq 1$, the the skew diagram $\lambda \setminus \kappa$ is called a vertical strip. From a more visual point of view, a horizontal (resp.\ vertical) strip is characterized by the condition that it has at most one box in each column (resp.\ row).

The size of a strip (be it horizontal or vertical) is the number of its boxes. We sometimes call a strip of size
$r$ an $r$-strip.
\end{defn}
In the following drawing the shaded region below is the diagram of the horizontal strip $\lambda \setminus \kappa$ for the partitions $\lambda = (5,4,1,1)$ and $\kappa = (4, 2, 1)$:
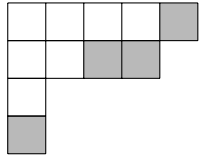
\begin{center}
\begin{tikzpicture}
\fill[black!27.5!white] (1, 0.5) rectangle (2, 1);
\fill[black!27.5!white] (2, 1) rectangle (2.5, 1.5);
\fill[black!27.5!white] (0, -0.5) rectangle (0.5, 0);
\draw[step=0.5cm, thin] (0, 0) grid (0.5, 0.5);
\draw[step=0.5cm, thin] (0, 0.5) grid (2, 1);
\draw[step=0.5cm, thin] (0, 0.999) grid (2.5, 1.5);
\draw[step=0.5cm, thin] (0, -0.5) grid (0.5, 0);
\end{tikzpicture}
\end{center}
We \label{1_page_link_horizontal_strip_semistandard_tableau} observe that semistandard Young tableaux possess the property that for each positive integer $i$, the boxes filled by $i$ form a horizontal strip. Indeed, the condition that the columns be strictly increasing ensures that there is at most one box filled by the number $i$ in every column, while the condition that both the rows and columns be increasing ensures that the collection of boxes filled by $j \leq i$ (or $j < i$) is the Ferrers diagram of some partition. Given a generalized tableau $T$, let us use $T(i)$ as a shorthand for the collection of boxes filled by $j \leq i$. In fact, the tableau $T$ is semistandard if and only if for each $i \geq 1$, $T(i)$ is a partition and $T(i + 1) \setminus T(i)$ is a horizontal strip. 

\subsection{The ring of symmetric functions}
In the preceding paragraph we have introduced the combinatorial objects that are necessary in order to give the combinatorial definition for Schur functions. In this paragraph we establish the algebraic framework for Schur functions, namely the ring of symmetric functions. We follow \cite{mac} in our presentation.

Before we formally define the ring of symmetric functions, we give some examples of symmetric polynomials, given that they are the concrete counterpart to the rather abstract notion of symmetric functions.
\begin{defn} [monomial symmetric polynomials] \label{symbol_monomial_poly} Let $\calX = (x_1, \dots, x_n)$ be a set of variables and let $\lambda$ be a partition. If $l(\lambda) > n$, then the monomial symmetric polynomial $\monomial_\lambda(\calX)$ is identically zero; otherwise, 
$$\monomial_\lambda(\calX) = \sum_{\alpha} x_1^{\alpha_1} \cdots x_n^{\alpha_n}
$$
where the sum runs over all \emph{distinct} permutations $\alpha = (\alpha_1, \dots, \alpha_n)$ of $\lambda = (\lambda_1, \dots, \lambda_n)$. We remark that this definition makes use of the convention that any partition of length less than $n$ may be viewed as a sequence of length exactly $n$ by appending zeros. 
\end{defn}

To illustrate effect of restricting the sum to distinct permutations of the indexing partition $\lambda$, consider the following examples of monomial polynomials:
\begin{align*}
\monomial_{(5, 2, 1)}(x_1, x_2, x_3) ={} & x_1^5 x_2^2 x_3^1 + x_1^5 x_2^1 x_3^2 + x_1^2 x_2^5 x_3^1 + x_1^2 x_2^1 x_3^5 + x_1^1 x_2^5 x_3^2 + x_1^1 x_2^2 x_3^5
\intertext{and}
\monomial_{(2, 2, 1)}(x_1, x_2, x_3) ={} & x_1^2 x_2^2 x_3^1 + x_1^2 x_2^1 x_3^2 + x_1^1 x_2^2 x_3^2
.
\end{align*}

The polynomials defined above are called symmetric because they are invariant under permutations of the elements of $\calX$. The following definition lists three other commonly used families of symmetric polynomials.

\begin{defn} (power sums, elementary and complete symmetric polynomials) Let $r$ be a positive integer and let $\calX$ be a set of variables. 
\begin{enumerate}
\item \label{symbol_elementary_poly} The $r$-th elementary symmetric polynomial $\elementary_r(\calX)$ is given by $\monomial_{\left(1^r\right)} (\calX)$, which is equal to the sum of all products of $r$ variables with distinct indices. We use the convention that $\elementary_0(\calX) = 1$.
\item \label{symbol_complete_poly} The $r$-th complete symmetric polynomial $\complete_r(\calX)$ is equal to $\sum_{\lambda: |\lambda| = r} \monomial_\lambda(\calX)$. We use the convention that $\complete_0(\calX) = 1$.
\item \label{symbol_power_poly} The $r$-th power sum $\power_r(\calX)$ is defined by $\power_{(r)}(\calX) = \sum_{x \in \calX} x^r$.
\end{enumerate}
We extend the definitions listed here to symmetric polynomials indexed by a partition: if $\lambda$ is a partition, we set \label{symbol_lambda-th_symmetric_poly} $$f_\lambda(\calX) = \prod_{i \geq 1} f_i(\calX)^{m_i(\lambda)}$$ where $f$ stand for $\elementary$, $\complete$ or $\power$.
\end{defn}

We remark that the power sum, the monomial, the elementary and the complete symmetric polynomial indexed by a partition $\lambda$ are homogeneous of degree $|\lambda|$. 

For theoretical considerations, it is often more convenient to work with symmetric functions instead of symmetric polynomials, given that they are not dependent on a set of variables. Let us give a rigorous construction for the ring of symmetric functions: consider a (countably) infinite set of variables $\calX = (x_1, x_2, \dots)$, and let $\calX_{[n]}$ denote the set consisting of the variables $(x_1, \dots, x_n)$ for every integer $n \geq 1$. The symmetric polynomials in the variables $\calX_{[n]}$ form a subring of $\C \left[ \calX_{[n]} \right]$, which we call the ring of symmetric polynomials in the variables $\calX_{[n]}$ (\label{symbol_ring_of_symmetric_polynomials} $\Sym \left(\calX_{[n]} \right)$). Arguably the simplest basis of $\Sym \left(\calX_{[n]} \right)$ is given by the monomial symmetric polynomials $\monomial_\lambda\left( \calX_{[n]}\right)$ as $\lambda$ ranges over all partitions of length less than $n$. We remark that the ring of symmetric polynomials inherits a natural grading from the polynomial ring:
$$\Sym \left(\calX_{[n]} \right) = \bigoplus_{k \geq 1} \Sym^k \left(\calX_{[n]} \right)$$
where $\Sym^k \left(\calX_{[n]} \right)$ consists of the homogeneous symmetric polynomials of degree $k$ (and the zero polynomial); or equivalently, $\Sym^k \left(\calX_{[n]} \right)$ is the span of $\monomial_\lambda\left( \calX_{[n]} \right)$ as $\lambda$ runs through the partitions of size $k$ and length at most $n$. In particular, the monomial symmetric polynomials indexed by the partitions of size $k$ form a basis of $\Sym^k \left(\calX_{[n]} \right)$, whenever $k \leq n$. In order to get rid of the dependence on variables, let us consider the following ring homomorphism: for $m \geq n \geq 1$,
$$\rho_{m,n}: \C \left(\calX_{[m]} \right) \to \C \left(\calX_{[n]} \right)$$
which is given by $$\rho_{m,n}(x_i) = \begin{dcases} x_i &\text{if } i \leq n \\ 0 &\text{if } i > m. \end{dcases}$$
Restricting the domain of this homomorphism to homogeneous symmetric polynomials in $\calX_{[m]}$ of degree $k$ results in a linear map
$$\rho^k_{m,n}: \Sym^k \left(\calX_{[m]} \right) \to \Sym^k \left(\calX_{[n]} \right),$$
given that for any partition $\lambda$ of size $k$ and length at most $m$, the basis element $\monomial_\lambda \left( \calX_{[m]}\right) \in \Sym^k \left(\calX_{[m]} \right)$ is mapped to $\monomial_\lambda \left( \calX_{[n]}\right) \in \Sym^k \left(\calX_{[n]} \right)$. We observe that the monomial symmetric polynomial $\monomial_\lambda \left( \calX_{[m]}\right)$ is sent to an element of our basis for $\Sym^k \left(\calX_{[n]} \right)$, unless $l(\lambda) \geq n$. This observation entails that $\rho^k_{m,n}$ is surjective -- and even bijective in case $k \leq n$. We take the inverse limit \label{symbol_ring_of_symmetric_functions}
$$\Sym^k = \lim_{\leftarrow} \Sym^k \left( \calX_{[n]}\right)$$
relative to the linear maps $\rho^k_{m,n}$, and define the ring of symmetric functions as the direct sum
$$\Sym = \bigoplus_{k \geq 0} \Sym^k.$$ In order to specify particular symmetric functions, consider the projection $$\rho^k_n: \Sym^k \to \Sym^k \left(  \calX_{[n]} \right),$$
which is an isomorphism whenever $k \leq n$, owing to the observation made above. If $\lambda$ is a partition of size $k$, then the monomial/elementary/complete/power sum symmetric function associated to $\lambda$ is given by the condition that for all $n \geq k$, \label{symbol_symmetric_function}
$$\rho^k_n(f_\lambda) = f_\lambda \left( \calX_{[n]} \right)$$
where the symbol $f$ stand for $\monomial$/$\elementary$/$\complete$/$\power$. For example, the $r$-th elementary symmetric function is given by
$$\elementary_r = \sum_{1 \leq i_1 < \dots < i_r} x_{i_1} \cdots x_{i_r}.$$
Now, the fact that $\rho^k_n$ is bijective whenever $k \leq n$ allows us to give an alternative definition for the ring of symmetric functions:
\begin{defn} [ring of symmetric functions] The ring of symmetric functions ($\Sym$) is the complex vector space spanned by the monomial symmetric functions $\monomial_\lambda$ where $\lambda$ runs over all partitions.
\end{defn}
Owing to the fact that the product of two symmetric polynomials is again symmetric, $\Sym$ is endowed with a natural ring structure. It is worth noting that there exist formal power series in $\calX = (x_1, x_2, \dots)$ over $\C$ that are invariant under permutations of $\calX$ but that do not belong  to the ring of symmetric functions, such as the infinite product $\prod_{i \geq 1} (1 + x_i)$. Indeed, it cannot be written as a \emph{finite} linear combination of monomial symmetric polynomials.

The following classic theorem states that the monomial symmetric polynomials are not the only natural basis for the ring of the symmetric functions. A proof can be found in \cite[p.~154]{Sagan}.
\begin{thm} The following are bases for $\Sym^k$:
\begin{enumerate}
\item $\{\elementary_\lambda: \lambda \text{ is a partition of size $k$}\}$;
\item $\{\complete_\lambda: \lambda \text{ is a partition of size $k$}\}$;
\item $\{\power_\lambda: \lambda \text{ is a partition of size $k$}\}$.
\end{enumerate}
\end{thm}

\subsection{Schur functions: a combinatorial and a determinantal definition}
We are finally ready to define Schur functions. There are a variety of different definitions for Schur functions, each emphasizing one facet of this versatile symmetric function. In this section, we present a combinatorial definition (based on \cite{Sagan}) as well as a determinantal definition (based on \cite{mac}).

Consider an infinite set of variables $\calX = (x_1, x_2, \dots)$ and a generalized Young tableau $T$. The weight of $T$ in $\calX$ is given by
\begin{align*}
\calX^T = \calX^\mu = x_1^{\mu_1} x_2^{\mu_2} \cdots
\end{align*}
where $\mu$ is the content of $T$. By definition, the sequence $\mu$ only contains finitely many non-zero elements, ensuring that $\calX^T$ is a monomial in spite of its definition as a (seemingly) infinite product.

\begin{defn} [Schur functions, combinatorial definition] \label{1_defn_cominatorial_schur_function} Given a partition $\lambda$ and an infinite set of variables $\calX = (x_1, x_2, \dots)$, we define the associated Schur function 
\begin{align*}
\schur_\lambda = \sum_T \calX^T
\end{align*}
where the sum runs over all semistandard $\lambda$-tableaux $T$.
\end{defn}

\begin{ex*} Let us illustrate this definition on the example of $\lambda = \left\langle 1^r \right\rangle$ for some integer $r \geq 1$. Given that the columns of semistandard tableaux are strictly increasing, the content of any semistandard tableau of shape $\left\langle 1^r \right\rangle$ consists of exactly $r$ $1$'s and zeros. Hence,
$$\schur_{\left\langle 1^r \right\rangle} = \sum_{1 \leq i_1 < \dots < i_r} x_{i_1} \cdots x_{i_r} = \elementary_r.$$
\end{ex*}

It is not obvious from the combinatorial definition given above that Schur functions are symmetric. As symmetry is such a crucial property of Schur functions, we include a proof of this fact.

\begin{lem} The function $\schur_\lambda$ is symmetric.
\end{lem}

\begin{proof} This proof is due to Knuth \cite{BenderKnuth}. It is sufficient to show that $$\schur_\lambda(x_1, \dots, x_{i - 1}, x_{i + 1}, x_i, \dots) = \schur_\lambda(x_1, \dots, x_{i - 1}, x_i, x_{i + 1}, \dots)$$ for all $i \geq 1$. To this end, we construct an involution $\varphi$ on semistandard $\lambda$-tableaux with the following property: if $\varphi(\mu)$ denotes the content of the tableau $\varphi(T)$, then $\varphi(\mu)_i =\mu_{i + 1}$, $\varphi(\mu)_{i + 1} = \mu_i$ and $\varphi(\mu)_j = \mu_j$ for all $j \neq i, i + 1$. 

Let us consider a column of a semistandard tableau $T$: as its entries are in strictly increasing order, it either contains exactly one pair $i$, $i + 1$; exactly one of $i$ or $i + 1$; or no occurrence of either. If only $i$ (resp.\ $i + 1$) appears in one columns, call this occurrence of $i$ (resp.\ $i + 1$) free; call all other occurrences of $i$ and $i + 1$ (that appear in vertical pairs) fixed. For any semistandard $\lambda$-tableau $T$, each row of $\varphi(T)$ is obtained by switching the number of free $i$'s and free $(i + 1)$'s that occur in this row so that the affected entries are still in increasing order, without altering any other entry of the row in question. Let us illustrate the map $\varphi$ on a concrete example: if $i = 2$, then
\begin{center}
\begin{tikzpicture}
\draw[step=0.5cm, thin] (0, 0) grid (1, 0.5);
\draw[step=0.5cm, thin] (0, 0.5) grid (3, 1);
\draw[step=0.5cm, thin] (0, 0.999) grid (3.5, 1.5);
\node at (0.25, 0.25) {\small{3}};
\node at (0.75, 0.25) {\small{4}};
\node at (0.25, 0.75) {\small{2}};
\node at (0.75, 0.75) {\small{\textbf{2}}};
\node at (1.25, 0.75) {\small{\textbf{2}}};
\node at (1.75, 0.75) {\small{3}};
\node at (2.25, 0.75) {\small{4}};
\node at (2.75, 0.75) {\small{4}};
\node at (0.25, 1.25) {\small{1}};
\node at (0.75, 1.25) {\small{1}};
\node at (1.25, 1.25) {\small{1}};
\node at (1.75, 1.25) {\small{2}};
\node at (2.25, 1.25) {\small{\textbf{2}}};
\node at (2.75, 1.25) {\small{\textbf{3}}};
\node at (3.25, 1.25) {\small{\textbf{3}}};
\end{tikzpicture}
\begin{tikzpicture}
\node at (-0.5, 0.78) {$\xmapsto{\varphi}$};
\draw[step=0.5cm, thin] (0, 0) grid (1, 0.5);
\draw[step=0.5cm, thin] (0, 0.5) grid (3, 1);
\draw[step=0.5cm, thin] (0, 0.999) grid (3.5, 1.5);
\node at (0.25, 0.25) {\small{3}};
\node at (0.75, 0.25) {\small{4}};
\node at (0.25, 0.75) {\small{2}};
\node at (0.75, 0.75) {\small{\textbf{3}}};
\node at (1.25, 0.75) {\small{\textbf{3}}};
\node at (1.75, 0.75) {\small{3}};
\node at (2.25, 0.75) {\small{4}};
\node at (2.75, 0.75) {\small{4}};
\node at (0.25, 1.25) {\small{1}};
\node at (0.75, 1.25) {\small{1}};
\node at (1.25, 1.25) {\small{1}};
\node at (1.75, 1.25) {\small{2}};
\node at (2.25, 1.25) {\small{\textbf{2}}};
\node at (2.75, 1.25) {\small{\textbf{2}}};
\node at (3.25, 1.25) {\small{\textbf{3}}};
\end{tikzpicture}
\end{center}
We have marked the free occurrences of $i = 2$ and $i + 1 = 3$ in bold for the convenience of the reader. 

By construction, $\varphi(T)$ is a semistandard $\lambda$-tableau. Indeed, $\varphi$ only affects the columns that contain either $i$ or $i + 1$, leaving the strictly increasing order intact. Moreover, $\varphi$ is an involution, given that it does not alter the location of free (resp.\ fixed) occurrences. Clearly, $\varphi(\mu)_j = \mu_j$ for all $j \neq i, i + 1$. Since the fixed occurrences come in pairs $i$, $i + 1$, it also holds that  $\varphi(\mu)_i =\mu_{i + 1}$ and $\varphi(\mu)_{i + 1} = \mu_i$, which completes the proof.
\end{proof}

This lemma establishes that $\schur_\lambda$ is an element of the ring of symmetric functions. In fact, the set $\schur_\lambda$ so that $\lambda$ is a partition of size $k$ forms a basis for $\Sym^k$. Hence, the projection $\rho^k_n$ allows us to consider the symmetric polynomials $\schur_\lambda(x_1, \dots, x_n)$, which are homogeneous of degree $|\lambda|$. As the Schur function associated to the partition $\lambda$ is defined as a sum over tableaux with strictly increasing columns, $\schur_\lambda(x_1, \dots, x_n)$ is the zero polynomial whenever the length of $\lambda$ exceeds $n$. We also note for later reference that Schur functions possess the property of restriction, \textit{i.e.}\ $$\schur_\lambda(x_1, \dots, x_n, 0) = \rho^{|\lambda|}_n \left(\schur_\lambda \right) = \schur_\lambda(x_1, \dots, x_n).$$ Another application of the combinatorial definition for Schur functions is the following corollary to the Pieri rule (\textit{i.e.}\ Theorem~\ref{1_thm_pieri_rule}), which can also be derived by applying the Pieri rule to the equality in \eqref{1_eq_other_identity_for_LR_coeffs}: let $\lambda$ be a partition and $\calX = (x_1, \dots, x_n)$ a set of variables. If we set $\calX' = (x_1, \dots, x_{n - 1})$, then
\begin{align}
\label{1_eq_pieri_rule_schur_eq_add_variable}
\schur_\lambda \left( \calX \right) ={} & \sum_{\substack{\kappa \subset \lambda: \\ \lambda \setminus \kappa \text{ is a horizontal strip}}} \schur_\kappa \left( \calX' \right) x_n^{|\lambda| - |\kappa|}
.
\end{align}
Indeed, the left-hand side of the equation in \eqref{1_eq_pieri_rule_schur_eq_add_variable} is equal to
\begin{align*} 
\schur_\lambda(\calX) ={} & \sum_T \calX^T = \sum_{T} x_1^{\mu_1} \cdots x_{n - 1}^{\mu_{n - 1}} \cdot x_n^{\mu_n} 
\intertext{where $T$ runs through the semistandard tableaux of shape $\lambda$, and $\mu$ denotes the content of $T$. In fact, only the generalized tableaux $T$ that are filled with numbers form $1$ to $n$ contribute to the sum. As discussed on page \pageref{1_page_link_horizontal_strip_semistandard_tableau}, any semistandard $\lambda$-tableau $T$ filled with numbers up to $n$ can be divided into a semistandard $\kappa$-tableau $T'$ filled with numbers up to $n - 1$ and a horizontal strip $\lambda \setminus \kappa$ (that corresponds to the boxes of $T$ filled with the number $n$). If the shape of $T'$ is equal to the partition $\kappa$, then the content of $T$ is equal to $(\mu'_1, \dots, \mu'_{n - 1}, |\lambda| - |\kappa|)$ where $\mu'$ is the content of $T'$. Therefore,}
\schur_\lambda(\calX) ={} & \sum_{\substack{\kappa \subset \lambda: \\ \lambda \setminus \kappa \text{ is a horizontal strip}}} \left( \sum_{T'} x_1^{\mu'_1} \cdots x_{n - 1}^{\mu'_{n - 1}} \right) x_n^{|\lambda| - |\kappa|}
\end{align*}
where $T'$ runs through the semistandard $\kappa$-tableaux. The combinatorial definition for Schur functions allows us to conclude that $\schur_\lambda(\calX)$ is indeed equal to the right-hand side in \eqref{1_eq_pieri_rule_schur_eq_add_variable}.

\bigskip

We switch gears and turn to the determinantal definition for Schur functions, which does not rely on the notion of tableaux. It does rely, however, on the Vandermonde determinant: for $\calX = (x_1, \dots, x_n)$,
\begin{align*} 
\det \left( x^{n - j} \right)_{x \in \calX, 1 \leq j \leq n} = \prod_{1 \leq i < j \leq n} (x_i - x_j)
.
\end{align*}
For ease of notation, we use \label{1_page_first_Delta_function} \label{symbol_first_Delta_function} $\Delta(\calX)$ as a shorthand for the Vandermonde determinant.

\begin{defn} [Schur functions, determinantal definition] \label{1_defn_determinantal_schur_function} Let $\lambda$ be a partition and $\calX = (x_1, \dots, x_n)$ a set of variables. If $l(\lambda) > 0$, $\schur_\lambda(\calX) = 0$; otherwise, \label{symbol_schur_function} 
\begin{align*}
\schur_\lambda(\calX) = \frac{\det \left( x^{\lambda_j + n - j} \right)_{x \in \calX, 1 \leq j \leq n}}{\Delta(\calX)}
.
\end{align*} 
\end{defn}

This is the definition that Schur originally used \cite{schur}. In fact, this definition was introduced by Jacobi \cite{jacobi}, but Schur was the one to discover the connection with irreducible characters, which we will briefly touch upon in the section on orthogonality of Schur functions.

A priori it is not clear that $\schur_\lambda(\calX)$ is well defined because the $\Delta$-function in the denominator vanishes whenever $x_i = x_j$ for some $i \neq j$. However, setting $x_i = x_j$ in the matrix in the numerator results in two identical rows, meaning that its determinant is equal to zero. Hence, the numerator is divisible by $\Delta(\calX)$, which makes $\schur_\lambda(\calX)$ a polynomial. In practice, the determinantal definition tends not to be suitable for sets of variables with repetitions. Unlike the combinatorial definition, the determinantal definition makes it is easy to see that $\schur_\lambda(\calX)$ is a symmetric polynomial -- given that it is the ratio of two skew-symmetric polynomials. Technically, we have thus defined Schur \emph{polynomials}. We follow \cite{mac} in calling both $\schur_\lambda(\calX)$ and the corresponding symmetric function $\schur_\lambda$ Schur functions.

Owing to the multilinearity of determinants, the following identity (which will prove useful in Chapters~\ref{3_cha_overlap_ids} and \ref{4_cha_mixed_ratios}) is an immediate consequence of the determinantal formula given above: if $\calX$ consists of exactly $n$ variables,
\begin{align} \label{1_eq_schur_times_product_of_all_variables}
\schur_{\lambda + \langle m^n\rangle} (\calX) = \schur_\lambda(\calX) \left( \prod_{x \in \calX} x^m \right).
\end{align}

It is beyond the scope of this brief introduction to Schur functions to prove that the combinatorial and the determinantal definitions discussed here actually define the same object. Sagan provides an elegant proof based on the Jacobi-Trudi identities \cite[p.~165]{Sagan}, which is another classic definition for Schur functions. 

\subsection{Orthogonality}
One of the reasons why Schur functions are considered the most natural basis for the ring of symmetric functions is their orthogonality with respect to the so-called Hall inner product. Another reason is that Schur functions are intimately connected with irreducible representations of the symmetric group as well as a number of matrix groups, which in turn gives a different perspective on the fact that they are orthogonal. This section is divided into two parts: in the first part, we discuss the Hall inner product. In the second part, we briefly touch upon the link between Schur functions and representation theory of the unitary group. The latter will play a crucial role in the derivation of combinatorial formulas for mixed ratios of characteristic polynomials from the unitary group, which is the subject of Chapter~\ref{4_cha_mixed_ratios}. 

We follow Chapter~I.4 of \cite{mac} in our discussion of the Hall inner product. We start by giving three series expansions for the product
\begin{align*}
\prod_{\substack{x \in \calX \\ y \in \calY}} (1 - xy)^{-1}
\end{align*}
where $\calX$ and $\calY$ are sets of variables with the property that $|xy| < 1$ for all $x \in \calX$ and $y \in \calY$. The condition on the absolute value of $xy$ ensures that the series on the right-hand side of the following equalities convergence. It holds that
\begin{align} 
\label{1_eq_defining_for_Hall_inner_prod}
\prod_{\substack{x \in \calX \\ y \in \calY}} (1 - xy)^{-1} ={} & \sum_\lambda \complete_\lambda(\calX) \monomial_\lambda(\calY)
\intertext{and}
\label{1_eq_cauchy_id_power_sum_version}
\prod_{\substack{x \in \calX \\ y \in \calY}} (1 - xy)^{-1} ={} & \sum_\lambda z_\lambda^{-1} \power_\lambda(\calX) \power_\lambda(\calY)
\intertext{where $z_\lambda = \prod_{i \geq 1} i^{m_i(\lambda)} m_i(\lambda)!$. The third identity is the so-called Cauchy identity:}
\label{1_eq_cauchy_id}
\prod_{\substack{x \in \calX \\ y \in \calY}} (1 - xy)^{-1} ={} & \sum_\lambda \schur_\lambda(\calX) \schur_\lambda(\calY)
.
\end{align}
There exists a unique sesquilinear form $\langle \cdot, \cdot \rangle$ on the complex vector space $\Sym$ that satisfies the following condition: if $u_\lambda$ and $v_\lambda$ are bases of $\Sym^k$, indexed by the partitions of size $k$ (as $k$ runs through the non-negative integers), then $\langle u_\lambda, v_\kappa \rangle = \delta_{\lambda \kappa}$ (where $\delta_{\lambda \kappa}$ is the Kronecker delta) for all partitions $\lambda$, $\mu$ if and only if $$\prod_{\substack{x \in \calX \\ y \in \calY}} (1 - xy)^{-1} = \sum_\lambda u_\lambda(\calX) v_\lambda(\calY)$$
for all sets of variables $\calX$, $\calY$ so that $|xy| < 1$. This form is called the Hall inner product on the ring of symmetric functions. By definition, the Cauchy identity, given in \eqref{1_eq_cauchy_id}, entails that the Schur functions form an orthonormal basis of $\Sym$, from which we deduce that the sesquilinear form defined above is symmetric and positive definite, \textit{i.e.}\ the Hall inner product is indeed an inner product on $\Sym$. In addition, the identities in \eqref{1_eq_defining_for_Hall_inner_prod} and \eqref{1_eq_cauchy_id_power_sum_version} allow us to infer that $\langle \complete_\lambda, \monomial_\kappa \rangle = \delta_{\lambda \kappa}$ and $\langle \power_\lambda, \power_\kappa \rangle = z_\lambda \delta_{\lambda \kappa}$.

Let us define the following ring homomorphism on the ring of symmetric functions, which will turn out to be an isometry with respect to the Hall inner product: \label{symbol_involution_on_Sym}
$$\omega: \Sym \to \Sym \text{ given by } \omega(\elementary_r) = \complete_r$$
for all $r \geq 0$. Considering the generating functions of the symmetric functions in question (which we have not introduced in this brief overview on Schur functions), one sees that
\begin{itemize}
\item $\omega(\complete_r) = \elementary_r$, \textit{i.e.}\ $\omega$ is an involution, and thus an automorphism on $\Sym$,
\item $\omega(\power_r) = (-1)^{r - 1} \power_r$,
\item $\omega(\schur_\lambda) = \schur_{\lambda'}$.
\end{itemize}
The last property makes it obvious that $\omega$ is an isometry. Indeed,
\begin{align*}
\langle \omega(\schur_\lambda), \omega(\schur_\kappa) \rangle = \langle \schur_{\lambda'}, \schur_{\kappa'} \rangle = \delta_{\lambda' \kappa'} = \delta_{\lambda \kappa} = \langle \schur_\lambda, \schur_\kappa \rangle
.
\end{align*}
Furthermore, these properties of $\omega$ allow us to derive the so-called dual Cauchy identity from the Cauchy identity. In fact, the dual Cauchy identity is the result of applying $\omega$ in the variables $\calY$ to both sides of the equality in \eqref{1_eq_cauchy_id}. In order to apply $\omega$ to the left-hand side in \eqref{1_eq_cauchy_id}, we reformulate it as
\begin{align*}
\prod_{\substack{x \in \calX \\ y \in \calY}} (1 - xy)^{-1} ={} & \prod_{\substack{x \in \calX \\ y \in \calY}} \left[ \sum_{k \geq 0} (xy)^k \right] =  \prod_{x \in \calX} \left[ \sum_{k \geq 0} \sum_{\substack{\alpha_1, \dots, \alpha_m \geq 0: \\ \alpha_1 + \dots + \alpha_m = k}} x^k y_1^{\alpha_1} \cdots y_m^{\alpha_m} \right]
\intertext{where $\calY = (y_1, \dots, y_m)$. Hence,}
\prod_{\substack{x \in \calX \\ y \in \calY}} (1 - xy)^{-1} ={} & \prod_{x \in \calX} \left[ \sum_{k \geq 0} x^k \complete_k(\calY) \right]
,
\end{align*}
allowing us to see that applying $\omega$ in the variables $\calY$ to the left-hand side in \eqref{1_eq_cauchy_id} gives
\begin{multline*}
\prod_{x \in \calX} \left[ \sum_{k \geq 0} x^k \omega \left( \complete_k(\calY) \right) \right] = \prod_{x \in \calX} \left[ \sum_{k \geq 0} x^k \elementary_k(\calY) \right] \\ = \prod_{x \in \calX} \left[ \sum_{k \geq 0} \sum_{1 \leq i_1 < \dots < i_k} x^k y_{i_1} \cdots y_{i_k} \right] = \prod_{\substack{x \in \calX \\ y \in \calY}} (1 + xy)
.
\end{multline*}
We thus conclude that
\begin{align} \label{1_eq_dual_cauchy_id}
\prod_{\substack{x \in \calX \\ y \in \calY}} (1 + xy) = \sum_\lambda \schur_\lambda(\calX) \omega( \schur_\lambda(\calY) ) = \sum_\lambda \schur_\lambda(\calX) \schur_{\lambda'}(\calY)
,
\end{align}
which is called the dual Cauchy identity. We remark that if $\calX$ and $\calY$ contain $n$ and $m$ variables, respectively, then the sum on the right-hand side is actually finite, given that the $\schur_\lambda(\calX) = 0$ whenever $l(\lambda) > n$ and $\schur_{\lambda'}(\calY) = 0$ whenever $\lambda_1 = l(\lambda') > m$. In consequence, we need no longer concern ourselves with questions of convergence, which means that the dual Cauchy identity also holds for finite sets of variables whose pairwise products are not necessarily less than 1 in absolute value.

\bigskip
In this second part we give the essential facts that are required to view Schur orthogonality from a representation theoretic point of view. This quick introduction to representation theory of the unitary group is based on \cite{BumpLieGroups}. If $G$ is a compact topological group, there exists a unique regular Borel measure $\mu$ that is invariant under translation so that the group $G$ itself has volume $\mu(G) = 1$. More concretely, Borel measure means that $\mu$ is defined on all open sets of $G$ and translation invariance means that it satisfies $\mu(X) = \mu(gX) = \mu(Xg)$ for all measurable sets $X \subset G$ and all elements $g \in G$. This measure is called the \label{1_page_Haar_measure} Haar measure on $G$. Making use of the fact that the Haar measure is unique, we denote the corresponding integral by
\begin{align*}
\int_G f(g) dg
\end{align*}
for any Haar integrable function $f$ on $G$. It is worth noting that the uniqueness of the Haar measure entails that
\begin{align*}
\int_G f \left(g^{-1} \right) dg = \int_G f(g) dg
.
\end{align*}

We now introduce some basic notions from representation theory of compact groups. This will allow us to give a family of (square-integrable) functions on any compact group $G$ that are orthonormal with respect to the inner product defined by
\begin{align*}
\langle f_1, f_2 \rangle = \int_G f_1(g) \overline{f_2(g)} dg
.
\end{align*}
Let $G$ be some compact group. If $V$ is a finite-dimensional complex vector space and $\pi: G \to GL(V)$ a continuous homomorphism, then the pair $(\pi, V)$ is called a representation of $G$. A representation $(\pi, V)$ is irreducible if $V$ has no proper nonzero invariant subspaces. (A subspace $W$ of $V$ is called invariant if $\pi(g) w \in W$ for all $g \in G$ and $w \in W$.) In fact, we will only consider irreducible representations given that each representation is a direct sum of irreducible representations. The character of a representation $(\pi, V)$ is the function $\chi: G \to \C; g \mapsto \text{tr}(\pi(g))$ where $\text{tr}$ denotes the trace.
\begin{thm} [Schur orthogonality] Let $(\pi_1, V_1)$ and $(\pi_2, V_2)$ be irreducible representations of a compact group $G$ with characters $\chi_1$ and $\chi_2$, then
\begin{align*}
\int_G \chi_1(g) \overline{\chi_2(g)} dg = \begin{dcases} 1 &\text{if } (\pi_1, V_1) \cong (\pi_2, V_2); \\ 0 &\text{otherwise.} \end{dcases}
\end{align*}
\end{thm}
\noindent Here, $(\pi_1, V_1) \cong (\pi_2, V_2)$ symbolizes that the two representations are isomorphic. The reader who is not familiar with representation theory may think of them as equal. This simplification should not cause any confusion, given that we will now restrict our attention to the unitary group of degree $N$, whose irreducible representations are in a bijective correspondence with highest weights, \textit{i.e.}\ non-increasing sequences $(\lambda_1, \dots, \lambda_N)$ of (potentially negative) integers.

We recall that a unitary matrix of size $N$ is a complex $N \times N$ matrix whose conjugate transpose is also its inverse. The unitary group of degree $N$, denoted $U(N)$, is the group of all unitary matrices of size $N$. That fact that the unitary group is a compact connected Lie group entails that the characters of the irreducible representations of $U(N)$ are given by the Weyl character formula: the character of the irreducible representation with highest weight $\lambda$ is equal to
\begin{align*}
\chi(\lambda)(g) ={} & \frac{\sum_{\sigma \in S_N} \varepsilon(\sigma) e^{\imaginary \theta_1 (\lambda_{\sigma(1)} + N - \sigma(1))} e^{\imaginary \theta_2 (\lambda_{\sigma(2)} + N - \sigma(2))} \cdots e^{\imaginary \theta_N (\lambda_{\sigma(N)} + N - \sigma(N))}}{\sum_{\sigma \in S_N} \varepsilon(\sigma) e^{\imaginary \theta_1 (N - \sigma(1))} e^{\imaginary \theta_2 (N - \sigma(2))} \cdots e^{\imaginary \theta_N (N - \sigma(N))}},
\intertext{where $\theta_1, \dots, \theta_N$ are the eigenangles of $g \in U(N)$. According to the Leibniz formula for determinants, this is equal to the following ratio of determinants:}
\chi(\lambda)(g) ={} & \frac{\det \left( e^{\imaginary \theta_i(\lambda_j + N - j)}\right)_{1 \leq i,j \leq N}}{\det \left( e^{\imaginary \theta_i(N - j)}\right)_{1 \leq i,j \leq N}}.
\end{align*}
The determinantal definition for Schur functions thus allows us to conclude that if $\lambda$ is a partition; or equivalently, if $\lambda_N \geq 0$, then the character of the irreducible representation with highest weight $\lambda$ is given by 
$$\chi(\lambda): U(N) \to \C; g \mapsto \schur_\lambda \left(e^{\imaginary \theta_1}, \dots, e^{\imaginary \theta_N} \right)$$
where $e^{\imaginary \theta_1}, \dots, e^{\imaginary \theta_N}$ are the eigenvalues of $g$. For completeness we remark that dropping the condition that $\lambda_N \geq 0$, the equality in \eqref{1_eq_schur_times_product_of_all_variables} allows us to compute the characters of all irreducible representations of $U(N)$: if $\calR(g)$ denotes the multiset of eigenvalues of $g$, then the irreducible characters of $U(N)$ are given by
$$\chi(\lambda): U(N) \to \C; g \mapsto \schur_{\lambda - \left\langle \left( \lambda_N \right)^N \right\rangle}(\calR(g)) \times \det(g)^{\lambda_N}$$
as $\lambda$ runs through the highest weights. We conclude by stating Schur orthogonality in the case of the unitary group, which will become relevant in Chapter~\ref{4_cha_mixed_ratios}.
\begin{cor} [Schur orthogonality] 
Let $\mu$ and $\nu$ be partitions. If for each matrix $g \in U(N)$ we write $\calR(g)$ for the multiset of its eigenvalues, then
\begin{align*}
\int_{U(N)} \schur_\mu(\calR(g)) \overline{\schur_\nu(\calR(g))}dg = \begin{cases} 1 &\text{if $\mu = \nu$ and $l(\mu) \leq N$;} \\ 0 &\text{otherwise.}\end{cases}
\end{align*}
\end{cor}

\section{Littlewood-Schur functions} \label{1_sec_LS_functions}
In this section we present a generalization of Schur functions, the so-called Littlewood-Schur functions. The overlap identities, which form the core of Chapter~\ref{3_cha_overlap_ids}, are identities for Littlewood-Schur functions. Moreover, they are essential for the derivation of the formulas for mixed ratios of characteristic polynomials discussed in Chapter~\ref{4_cha_mixed_ratios}. In fact, this section only covers the classic combinatorial definition for Littlewood-Schur functions. A determinantal formula for Littlewood-Schur functions, which was discovered by Moens and Van der Jeugt \cite{vanderjeugt}, will be the main focus of Chapter~\ref{2_cha_det_defn_LS}.

\subsection{Littlewood-Richardson coefficients}
In order to state the combinatorial definition for Littlewood-Schur functions, we need to introduce Littlewood-Richardson coefficients. As the Schur functions form a (linear) basis of the ring of symmetric functions, the product of any two Schur functions can be uniquely written as a linear combination of Schur functions. Its coefficients are the Littlewood-Richardson coefficients. 

\begin{defn} [Littlewood-Richardson coefficients] \label{1_defn_Littlewood-Richardson_coefficients} \label{symbol_LR_coeff} Let $\mu$ and $\nu$ be partitions. The Littlewood-Richardson coefficients are defined by the property that
\begin{align} \label{1_eq_defn_Littlewood-Richardson_coefficients} 
\schur_\mu \schur_\nu = \sum_\lambda c_{\mu \nu}^\lambda \schur_\lambda
.
\end{align}
\end{defn}

The Littlewood-Richardson coefficients also appear in another identity for Schur functions: let $\calX$ and $\calY$ be sets of variables. If $\calX \cup \calY$ denotes the union of the two sets of variables, then
\begin{align} \label{1_eq_other_identity_for_LR_coeffs}
\schur_\lambda(\calX \cup \calY) = \sum_{\mu, \nu} c_{\mu \nu}^\lambda \schur_\mu(\calX) \schur_\nu(\calY)
.
\end{align}
A justification of the equality, which is based on the Cauchy identity stated in \eqref{1_eq_cauchy_id}, is given in \cite[p.~71]{mac}.

Before going on to famous combinatorial theorems about the Littlewood-Richardson coefficients, we discuss a few basic properties that follow directly from their definition:
\begin{enumerate} 
\item \label{1_page_basic_property_of_LS_1} The fact that the ring of symmetric functions is commutative allows us to infer that $c_{\mu \nu}^\lambda = c_{\nu \mu}^\lambda$ for all partitions $\mu$, $\nu$ and $\lambda$. 
\item Letting the involution $\omega$ act on both sides of the defining equation in \eqref{1_eq_defn_Littlewood-Richardson_coefficients}, we see that $c^\lambda_{\mu \nu} = c^{\lambda'}_{\mu' \nu'}$.
\item Given that for any set of variables $\calX$, the Schur function $\schur_\lambda(\calX)$ is a homogeneous polynomial of degree $|\lambda|$, the Littlewood-Richardson coefficient $c_{\mu \nu}^\lambda$ vanishes unless $|\lambda| = |\mu| + |\nu|$. We claim that it also vanishes if the Ferrers diagram of $\mu$ is not a subset of the diagram of $\lambda$. Indeed, if $\calX$ and $\calY$ contain $n$ and $m$ variables, respectively, the combinatorial definition for Schur functions allows us to view the polynomial $\schur_\lambda(\calX \cup \calY)$ as a sum over semistandard $\lambda$-tableaux $T$ filled with integers from the sequence $(1, \dots, n, n + 1, \dots, n + m)$ so that the first $n$ numbers are weighted by $\calX$ and the last $m$ numbers by $\calY$. In any such tableau $T$, the boxes filled with numbers from 1 to $n$ form a semistandard tableau $T'$ of shape $\mu$ for some partition $\mu \subset \lambda$. In addition, the ways of completing a semistandard tableau $T'$ of shape $\mu \subset \lambda$ to a semistandard $\lambda$-tableau $T$ by filling the boxes in the skew diagram $\lambda \setminus \mu$ with numbers from $n + 1$ to $n + m$ only depend on the shape $\mu$. Therefore,
\begin{align*}
\schur_\lambda(\calX \cup \calY) ={} & \sum_{\mu \subset \lambda} \left( \sum_{T'} \calX^{T'} \right) P_{\lambda, \mu} (\calY) = \sum_{\mu \subset \lambda} \schur_\mu(\calX) P_{\lambda, \mu} (\calY)
\end{align*}
for some polynomial $P_{\lambda, \mu} (\calY)$ in the variables $\calY$ that depends on the partitions $\lambda$ and $\mu$. We conclude that the coefficients associated to $\schur_\mu(\calX)$ in the expansion of $\schur_\lambda(\calX \cup \calY)$ vanish unless $\mu$ is a subset of $\lambda$, which implies that the Littlewood-Richardson coefficient $c^\lambda_{\mu \nu}$ vanishes whenever $\mu \not\subset \lambda$, according to the equality stated in \eqref{1_eq_other_identity_for_LR_coeffs}.
\item \label{1_page_basic_property_of_LS_4} The Schur function associated to the empty partition is the constant polynomial equal to 1, which implies that $c_{\mu \emptyset}^\lambda = \delta_{\lambda \mu}$ (where $\delta_{\lambda \mu}$ is the Kronecker delta).
\end{enumerate}
The Pieri rule describes the values of the Littlewood-Richardson coefficients whose subscripts contain a partition that consists of exactly one part. A proof can be found in \cite[p.~72]{mac}.

\begin{thm} [Pieri rule] \label{1_thm_pieri_rule} Let $r$ be a positive integer and $\lambda$, $\kappa$ partitions, then
\begin{align*}
c^\lambda_{\kappa \langle r \rangle} = \begin{dcases} 1 &\text{if $\lambda \setminus \kappa$ is a horizontal $r$-strip,} \\ 0 &\text{otherwise.} \end{dcases}
\end{align*}
\end{thm}

The following Corollary gives two equalities that are both equivalent to the Pieri rule. 

\begin{cor} \label{1_cor_pieri_rule_schur} Let $r$ be a positive integer and $\kappa$ a partition, then
\begin{align} \label{1_cor_pieri_rule_schur_eq_complete}
\schur_\kappa \complete_r ={} & \sum_{\substack{\lambda: \\ \lambda \setminus \kappa \text{ is a horizontal $r$-strip}}} \schur_\lambda
\intertext{and} \label{1_cor_pieri_rule_schur_eq_elementary}
\schur_\kappa \elementary_r ={} & \sum_{\substack{\lambda: \\ \lambda \setminus \kappa \text{ is a vertical $r$-strip}}} \schur_\lambda
.
\end{align}
\end{cor}

\begin{proof} One easily deduces that $\schur_{\langle r \rangle} = \complete_r$ from the combinatorial definition for Schur functions. Hence, the identity in \eqref{1_cor_pieri_rule_schur_eq_complete} is just the result of applying the Pieri rule to the equality that defines the Littlewood-Richardson coefficients $c^\lambda_{\kappa \langle r \rangle}$ for all partitions $\lambda$. The equality in \eqref{1_cor_pieri_rule_schur_eq_elementary} is obtained by letting the involution $\omega$ act on the identity in \eqref{1_cor_pieri_rule_schur_eq_complete}, given that $\lambda' \setminus \kappa'$ is a horizontal $r$-strip if and only if $\lambda \setminus \kappa$ is a vertical $r$-strip.
\end{proof}

We conclude this section by stating the Littlewood-Richardson rule, which will play a minor role in Chapter~\ref{4_cha_mixed_ratios}. In theory, it gives a complete combinatorial description for the values of the Littlewood-Richardson coefficients. In practice, the Littlewood-Richardson rule is hard to use because of its complexity. Although the statement is named after Littlewood and Richardson (who conjectured it in \cite{LRcoeffs}), the first complete proof is due to Thomas \cite{thomas}.

Three additional definitions are needed to describe the tableaux that the Littlewood-Richardson coefficients count.
\begin{defn} [generalized tableau of skew shape] Let $\kappa$ and $\lambda$ be partitions. Recall that if $\kappa$ is a subset of $\lambda$, then the corresponding skew diagram is the set of boxes $\lambda \setminus \kappa$ that are contained in $\lambda$ but not in $\kappa$. A generalized tableau of skew shape $\lambda \setminus \kappa$ is obtained by filling the boxes of the skew diagram $\lambda \setminus \kappa$ with positive integers, allowing repetitions. We remark that the conditions defining semistandard tableaux still make sense in this generalized setting. 
\end{defn}

\begin{defn} [lattice word] A lattice word is a sequence of positive integers, say $(i_1, i_2, \dots, i_n)$, such that in every prefix, say $(i_1, i_2, \dots, i_k)$ for some $1 \leq k \leq n$, any number $i$ occurs at least as often as the number $i + 1$.
\end{defn}

\begin{defn} [row word] The row word of a tableau $T$ is the sequence of integers $R_1 R_2 \dots$ where $R_i$ is the $i$-th row of $T$ read from right to left.
\end{defn}

For example, the row word of the tableau on the left-hand side in \eqref{1_ex_semistandard_tableaux} is given by $(4, 3, 3, 1, 1, 7, 4, 4, 4, 3, 5, 5)$.

\begin{thm} [Littlewood-Richardson rule, \cite{LRcoeffs}] \label{1_thm_Littlewood_Richardson rule} The value of the coefficient $c_{\mu \nu}^\lambda$ is equal to the number of semistandard $\lambda \setminus \nu$-tableaux $T$ with content $\mu$ with the property that the row word of $T$ is a lattice word.
\end{thm}

\subsection{Littlewood-Schur functions}
In this section we give the classic combinatorial definition for Littlewood-Schur functions. Littlewood-Schur functions are a generalization of Schur functions, whose combinatorial definition appeared for the first time in the work of Littlewood \cite{littlewood}. These functions were studied under a variety of different names: they are called hook Schur functions by Berele and Regev \cite{berele_regev}, supersymmetric polynomials by Nicoletti, Metropolis and Rota \cite{metropolis}, super-Schur functions by Brenti \cite{brenti}, and Macdonald denotes them $s_\lambda(x/y)$ \cite[p.~58ff]{mac}. 
We follow Bump and Gamburd in calling them Littlewood-Schur functions and denoting them $LS_\lambda(\calX, \calY)$ \cite{bump06}.

\begin{defn} [Littlewood-Schur functions] \label{1_defn_LS_functions} 
Given a partition $\lambda$ and two sets of variables $\calX$, $\calY$, we define the associated Littlewood-Schur function \label{symbol_LS_function}
\begin{align*}
LS_\lambda(\calX; \calY) = \sum_{\mu, \nu} c_{\mu \nu}^\lambda \schur_\mu (\calX) \schur_{\nu'}(\calY)
.
\end{align*}
\end{defn}

Going back to the identity involving Littlewood-Richardson coefficients stated in \eqref{1_eq_other_identity_for_LR_coeffs}, we see that the only difference between the Schur function $\schur_\lambda(\calX \cup \calY)$ and the Littlewood-Schur function $LS_\lambda(\calX; \calY)$ lies in the partitions associated to the Schur functions in the variables $\calY$: they are conjugate. The consequences of this seemingly innocuous alteration has deeper reaching consequences than one might expect. Probably the most obvious effect is that the Littlewood-Schur function $LS_\lambda(\calX; \calY)$ is \emph{not} symmetric in the variables $\calX \cup \calY$. Instead it is what we call doubly-symmetric, \textit{i.e.}\ it is symmetric in both sets of variables separately.

As we have mentioned before, Littlewood-Schur functions generalize Schur functions. More concretely, the fact that the Schur function $\schur_\nu(\emptyset)$ is equal to 1 if $\nu$ is the empty partition, and equal to 0 otherwise allows us to infer that for any partition $\lambda$,
\begin{align*}
LS_\lambda(\calX; \emptyset) = \sum_{\mu, \nu} c_{\mu \nu}^\lambda \schur_\mu (\calX) \schur_{\nu'}(\emptyset) = \sum_\mu c_{\mu \emptyset}^\lambda \schur_\mu (\calX) = \schur_\lambda(\calX)
\end{align*}
by the fourth basic property of Littlewood-Richardson coefficients listed on page \pageref{1_page_basic_property_of_LS_4}. Other properties of Littlewood-Richardson coefficients also translate into properties of Littlewood-Schur functions. For instance, the first and the second properties listed on page \pageref{1_page_basic_property_of_LS_1} entail that
\begin{multline*}
LS_{\lambda'}(\calX; \calY) = \sum_{\mu, \nu} c_{\mu \nu}^{\lambda'} \schur_\mu (\calX) \schur_{\nu'}(\calY) = \sum_{\mu, \nu} c_{\mu' \nu'}^\lambda \schur_\mu (\calX) \schur_{\nu'}(\calY) \\ = \sum_{\mu, \nu} c_{\nu \mu}^\lambda \schur_{\mu'} (\calX) \schur_\nu(\calY) = LS_\lambda(\calY; \calX)
.
\end{multline*}
Moreover, the second property of Littlewood-Richardson coefficients translates into the fact that the Littlewood-Schur function $LS_\lambda(\calX; \calY)$ is a homogeneous polynomial in the variables $\calX \cup \calY$ of degree $|\lambda|$. We will derive further properties of Littlewood-Schur functions from their combinatorial definition in Section~\ref{2_sec_5_characterizing properties}. 

Here, we state and prove a corollary to the Pieri rule, which we have taken from \cite[p.~241-242]{bump06}. It will allow us to show a sufficient condition for the vanishing of the function $LS_\lambda(\calX; \calY)$, which is a generalization of the fact that $\schur_\lambda(\calX)$ vanishes whenever the length of $\lambda$ exceeds the number of variables in $\calX$.
\begin{cor} [\cite{bump06}] \label{1_cor_Pieri_rule_for_LS} Let $\lambda$ be a partition. If $\calX = \calX' \cup (x_n)$, then
\begin{align} \label{1_cor_Pieri_rule_for_LS_hor_x_eq}
LS_\lambda(\calX; \calY) ={} & \sum_{\substack{\kappa \subset \lambda: \\ \lambda \setminus \kappa \text{ is a horizontal strip}}} LS_{\kappa} \left( \calX'; \calY \right) x_n^{|\lambda| - |\kappa|}
.
\intertext{If $\calY = \calY' \cup (y_m)$, then} \label{1_cor_Pieri_rule_for_LS_vert_y_eq}
LS_\lambda(\calX; \calY) ={} & \sum_{\substack{\kappa \subset \lambda: \\ \lambda \setminus \kappa \text{ is a vertical strip}}} LS_{\kappa} \left( \calX; \calY' \right) y_m^{|\lambda| - |\kappa|}
.
\end{align}
\end{cor}

\begin{proof} We prove a statement that implies the equality in \eqref{1_cor_Pieri_rule_for_LS_hor_x_eq}, which in its turn implies the equality in \eqref{1_cor_Pieri_rule_for_LS_vert_y_eq}: we claim that for any set of variables $\calX = \calS \cup \calT$,
\begin{align} \label{1_in_proof_cor_Pieri_rule_for_LS}
LS_\lambda(\calX; \calY) = \sum_{\mu, \nu} c^\lambda_{\mu \nu} LS_\mu(\calS; \calY) \schur_\nu(\calT)
.
\end{align}
Indeed, first applying the definition of Littlewood-Schur functions and then the identity on Littlewood-Richardson coefficients in \eqref{1_eq_other_identity_for_LR_coeffs} yields
\begin{align*}
LS_\lambda(\calX; \calY) ={} & \sum_{\mu, \nu} c^\lambda_{\mu \nu} \schur_\mu(\calX) \schur_{\nu'}(\calY) \displaybreak[2] \\
={} & \sum_{\mu, \nu, \phi, \psi} c^\lambda_{\mu \nu} c^\mu_{\phi \psi} \schur_\phi (\calS) \schur_\psi (\calT) \schur_{\nu'}(\calY) 
.
\intertext{Notice that the fact that the ring of symmetric functions is commutative implies that
$$\sum_\mu c^\lambda_{\mu \nu} c^\mu_{\phi \psi} = \langle \schur_\lambda, \schur_\nu \schur_\phi \schur_\psi \rangle = \langle \schur_\lambda, \schur_\psi \schur_\phi \schur_\nu \rangle = \sum_\theta c^{\lambda}_{\theta \psi} c^{\theta}_{\phi \nu}.$$ Hence,}
LS_\lambda(\calX; \calY) ={} & \sum_{\theta, \nu, \phi, \psi} c^{\lambda}_{\theta \psi} c^{\theta}_{\phi \nu} \schur_\phi (\calS) \schur_\psi (\calT) \schur_{\nu'}(\calY)
.
\intertext{By the definition of Littlewood-Schur functions, we thus have}
LS_\lambda(\calX; \calY) ={} & \sum_{\theta, \psi} c^{\lambda}_{\theta \psi} LS_\theta(\calS; \calY) \schur_\psi(\calT)
\end{align*}
as required. It remains to show that the equality in \eqref{1_in_proof_cor_Pieri_rule_for_LS} does indeed imply the first equality stated in Corollary~\ref{1_cor_Pieri_rule_for_LS}. Setting $\calS = \calX'$ and $\calT = x_n$, we see that the identity in \eqref{1_cor_Pieri_rule_for_LS_hor_x_eq} is a direct consequence of the Pieri rule stated in Theorem~\ref{1_thm_pieri_rule}. In its turn, the equality in \eqref{1_cor_Pieri_rule_for_LS_vert_y_eq} follows from \eqref{1_cor_Pieri_rule_for_LS_hor_x_eq}, given that $LS_\lambda(\calX; \calY) = LS_{\lambda'}(\calY; \calX)$ and that transposing a horizontal strip results in a vertical strip of the same size.
\end{proof}

\begin{cor} \label{1_cor_LS=0} Let $\calX$ and $\calY$ consist of $n$ and $m$ variables, respectively. If a partition $\lambda$ contains the box with coordinates (m + 1, n + 1), then $LS_\lambda(\calX; \calY) = 0$. 
\end{cor}

\begin{proof} This proof is an induction on $m$, the number of elements in $\calY$. Recalling that $LS_\lambda(\calX; \emptyset) = \schur_\lambda(\calX) = 0$ whenever $l(\lambda) > n$ takes care of the base case $m = 0$. For the induction step, we consider a set of variables $\calY = \calY' \cup (y_m)$ consisting of $m \geq 1$ elements. By the equality in \eqref{1_cor_Pieri_rule_for_LS_vert_y_eq},
\begin{align*}
LS_\lambda(\calX; \calY) ={} & \sum_{\substack{\kappa \subset \lambda: \\ \lambda \setminus \kappa \text{ is a vertical strip}}} LS_{\kappa} \left( \calX; \calY' \right) y_m^{|\lambda| - |\kappa|}
\end{align*}
for any partition $\lambda$. If $(m + 1, n + 1) \in \lambda$, then $(m, n + 1) \in \kappa$ whenever $\lambda \setminus \kappa$ is a vertical strip. Thus, the induction hypothesis allows us to conclude that $LS_\lambda(\calX; \calY) = 0$.
\end{proof}

\chapter{A Determinantal Definition for Littlewood-Schur Functions} \label{2_cha_det_defn_LS}
\section{Introduction} 
This chapter is dedicated to a determinantal formula for Littlewood-Schur functions, which was discovered by Moens and Van der Jeugt \cite{vanderjeugt}. Let us state it without 
giving the definition of index (an omission which we rectify on page \pageref{2_defn_index}): if $\calX$ and $\calY$ are sets containing $n$ and $m$ variables, respectively, and $\lambda$ is a partition with non-negative $(m,n)$-index $k$, then 
\begin{align*} 
LS_\lambda(-\calX; \calY) ={} & \varepsilon(\lambda) \frac{\Delta(\calY; \calX)}{\Delta(\calX) \Delta(\calY)} \det \begin{pmatrix} \left( (x - y)^{-1} \right)_{\substack{x \in \calX \\ y \in \calY}} & \left( x^{\lambda_j + n - m - j} \right)_{\substack{x \in \calX \\ 1 \leq j \leq n - k}} \\ \left( y^{\lambda'_i + m - n - i} \right)_{\substack{1 \leq i \leq m - k \\ y \in \calY}} & 0\end{pmatrix}
\end{align*}
where $\varepsilon(\lambda)$ is some sign. By setting $\calY$ equal to the empty set, we see that this formula generalizes the determinantal definition for Schur functions -- given that the $(0,n)$-index of any partition $\lambda$ of length at most $n$ is equal to 0 and that $$LS_\lambda(-\calX; \emptyset) = \schur_\lambda(-\calX) = \pm \schur_\lambda(\calX).$$
The identities for Littlewood-Schur functions that are presented in Chapter~\ref{3_cha_overlap_ids} are based on this determinantal formula. In this chapter, we reproduce a proof of Moens and Van der Jeugt's result, which relies on a list of five characterizing properties of Littlewood-Schur functions provided by \cite{mac}.

After having established the determinantal formula for Littlewood-Schur functions, we give an example of how it might be used: we will extend the classic determinantal proof of the Murnaghan-Nakayama rule for Schur functions to a determinantal proof that works for Littlewood-Schur functions. While the resulting generalization of the Murnaghan-Nakayama rule to Littlewood-Schur functions is known, the elementary proof presented in this chapter seems to be new.

Throughout this chapter, we assume that the reader is familiar with the concepts presented in Sections~\ref{1_sec_Schur_functions} and \ref{1_sec_LS_functions}.

\subsection{Structure of this chapter}
In Section~\ref{2_sec_5_characterizing properties}, we discuss the five properties that characterize Littlewood-Schur functions. In Section~\ref{2_sec_proof_of_det_formula_for_LS}, we state and prove the determinantal formula for Littlewood-Schur functions. In Section~\ref{2_sec_det_proof_of_MN_for_LS}, we first define the notion of ribbons and then present a determinantal proof of the Murnaghan-Nakayama rule for Littlewood-Schur functions.

\section[Five characterizing properties of Littlewood-Schur functions]{Five characterizing properties of Littlewood-Schur \\ functions} \label{2_sec_5_characterizing properties}
In this section we discuss the five characterizing properties of Littlewood-Schur functions that are listed in Example~23 of Chapter~I.3 in \cite{mac}. In fact, Macdonald studies functions that he denotes $s_\lambda(\calX / \calY)$, which are equal to $LS_\lambda(\calX; -\calY)$ in our notation. Moreover, he only states four properties but they are trivially equivalent to the five properties given here. Before stating the characterizing properties, we define the second $\Delta$-function -- the first having been introduced on page \pageref{1_page_first_Delta_function}.
\begin{defn} [$\Delta$-function] Let $\calX$ and $\calY$ be two sets of variables. We define \label{symbol_second_Delta_function}
\begin{align*}
\Delta(\calX; \calY) ={} & \prod_{\substack{x \in \calX \\ y \in \calY}} (x - y)
.
\end{align*}
\end{defn}

\begin{thm} [\cite{mac}] \label{2_thm_5_properties_of_LS} Suppose that for each partition $\lambda$, each pair of parameters $m$,$n$ and each pair of sets of variables $\calX = (x_1, \dots, x_n)$, $\calY = (y_1, \dots, y_m)$, we are given a function $LS_{\lambda}^{\ast}(-\calX; \calY)$. If these functions satisfy all of the following five conditions, then $LS_{\lambda}^{\ast}(-\calX; \calY) = LS_{\lambda}(-\calX; \calY)$.
\begin{enumerate}
\item (homogeneity) The function $LS_\lambda^{\ast}(-\calX; \calY)$ is a homogeneous polynomial in the variables $\calX \cup \calY$ of degree $|\lambda|$.
\item (double-symmetry) The function $LS_\lambda^{\ast}(-\calX; \calY)$ is symmetric in each set of variables separately.
\item (restriction) If $n \geq 1$, then setting $x_n = 0$ in $LS_\lambda^{\ast}(-\calX; \calY)$ results in $LS_\lambda^{\ast}\left(-\calX'; \calY\right)$ where $\calX' = (x_1, \dots, x_{n-1})$. 
If $m \geq 1$, then setting $y_m = 0$ in $LS_\lambda^{\ast}(-\calX; \calY)$ results in $LS_\lambda^{\ast}\left(-\calX; \calY'\right)$ where $\calY' = (y_1, \dots, y_{m - 1})$.
\item (cancellation) If $m$, $n \geq 1$, then setting $x_n = y_m$ in $LS_\lambda^{\ast}(-\calX; \calY)$ results in $LS_\lambda^{\ast}\left(-\calX'; \calY'\right)$ where $\calX' = (x_1, \dots, x_{n - 1})$ and $\calY' = (y_1, \dots, y_{m - 1})$.
\item (factorization) Suppose that $\calX$ and $\calY$ only contain non-zero variables. If the partition $\lambda$ satisfies $\lambda_n \geq m \geq \lambda_{n + 1}$, so that it can be written in the form $\left( \left\langle m^n \right\rangle + \alpha \right) \cup \beta'$ for some partitions $\alpha$ and $\beta$ of lengths at most $n$ and $m$, respectively, then 
$
LS_\lambda^{\ast}(-\calX; \calY) = \Delta(\calY; \calX) \schur_\alpha(-\calX) \schur_\beta(\calY)
.
$
\end{enumerate}
\end{thm}

We will not prove that these properties characterize Littlewood-Schur functions. We will merely verify that the Littlewood-Schur functions $LS_\lambda(-\calX; \calY)$ (defined on page \pageref{1_defn_LS_functions}) possess the five properties listed in Theorem~\ref {2_thm_5_properties_of_LS}. The first three properties follow immediately from the corresponding properties for Schur functions: firstly, for any partitions $\mu$ and $\nu$, the Schur functions $\schur_\mu(-\calX)$ and $\schur_{\nu'}(\calY)$ are homogeneous polynomials of degree $|\mu|$ and $|\nu|$, respectively. Hence, the fact that the Littlewood-Richardson coefficient $c^\lambda_{\mu \nu'}$ vanishes unless $|\lambda| = |\mu| + |\nu|$ implies that $$LS_\lambda(-\calX; \calY) = \sum_{\mu, \nu} c^\lambda_{\mu \nu} \schur_\mu(-\calX) \schur_{\nu'}(\calY)$$
is a homogeneous polynomial in $\calX \cup \calY$ of degree $|\lambda|$. Secondly, the Littlewood-Schur function $LS_\lambda(-\calX; \calY)$ is defined as a linear combination of products of symmetric polynomials in $\calX$ and in $\calY$, making double-symmetry obvious. Thirdly, the property of restriction is another consequence of the fact that Littlewood-Schur functions are linear combinations of products of Schur functions, which possess the property in question.

\begin{lem} [cancellation] If $m$, $n \geq 1$, then setting $x_n = y_m$ in $LS_\lambda(-\calX; \calY)$ results in $LS_\lambda\left(-\calX'; \calY'\right)$ where $\calX' = (x_1, \dots, x_{n - 1})$ and $\calY' = (y_1, \dots, y_{m - 1})$.
\end{lem}

\begin{proof} Let us set $x_n = t = y_m$. Applying the equalities stated in  \eqref{1_cor_Pieri_rule_for_LS_hor_x_eq} and \eqref{1_cor_Pieri_rule_for_LS_vert_y_eq} yields the following expression for the Littlewood-Schur function $LS_\lambda(-\calX; \calY)$. Let $S(\lambda, \mu) = \{\kappa: \mu \subset \kappa \subset \lambda, \; \lambda \setminus \kappa \text{ is a horizontal strip}, \; \kappa \setminus \mu \text{ is a vertical strip}\},$
then
\begin{align*}
LS_\lambda(-\calX; \calY) ={} & \sum_{r \geq 0} t^r \sum_{\substack{\mu \subset \lambda: \\ |\lambda| - |\mu| = r}} LS_\mu \left( -\calX'; \calY' \right) \sum_{\kappa \in S(\lambda, \mu)} \varepsilon(\kappa)
\end{align*}
where the sign is given by $\varepsilon(\kappa) = (-1)^{|\lambda| - |\kappa|}$.

The summand corresponding to $r = 0$ is equal to $LS_\lambda \left( -\calX'; \calY' \right)$. In order to show that the sum over $r \geq 1$ vanishes, fix a partition $\mu \subset \lambda$ so that $|\lambda| - |\mu| = r$ for some $r \geq 1$. We show cancellation by constructing an involution $\varphi$ on $S(\lambda, \mu)$ with the property that $\varepsilon(\varphi(\kappa)) = - \varepsilon(\kappa)$ for all $\kappa \in S(\lambda, \mu)$: if the bottom-left-most box of the difference $\lambda \setminus \mu$ belongs to $\kappa$, $\varphi$ removes it, and if the bottom-left-most box of $\lambda \setminus \mu$ is not contained in $\kappa$, $\varphi$ adds it. In either case the sign $\varepsilon(\kappa) = (-1)^{|\lambda| - |\kappa|}$ is multiplied by $-1$. In addition, it follows immediately from the definition that $\varphi$ is an involution.

It remains to show that $\varphi(\kappa) \in S(\lambda, \mu)$ for all $\kappa \in S(\lambda, \mu)$. If we color the horizontal strip $\lambda \setminus \kappa$ in light gray and the vertical strip $\kappa \setminus \mu$ in dark gray, then the bottom-left-most boxes of $\lambda \setminus \mu$ come in the following two types:
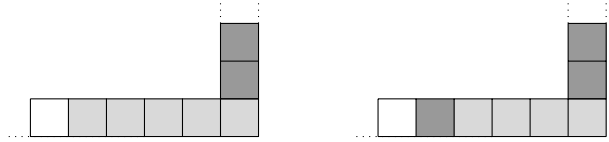
\begin{center}
\begin{tikzpicture}
\fill[black!15!white] (-0.5, 0) rectangle (1.5, 0.5);
\fill[black!40!white] (1, 0.5) rectangle (1.5, 1.5);
\fill[black!15!white] (-1, 0) rectangle (-0.5, 0.5);
\draw[step=0.5cm, thin] (-1.5, 0) grid (1.5, 0.5);
\draw[step=0.5cm, thin] (0.999, 0.5) grid (1.5, 1.5);
\draw[dotted] (0, 0) -- (-1.8, 0);
\draw[dotted] (1, 1.5) -- (1, 1.8);
\draw[dotted] (1.5, 1.5) -- (1.5, 1.8);
\end{tikzpicture}
\hspace{1cm}
\begin{tikzpicture}
\fill[black!15!white] (-0.5, 0) rectangle (1.5, 0.5);
\fill[black!40!white] (1, 0.5) rectangle (1.5, 1.5);
\fill[black!40!white] (-1, 0) rectangle (-0.5, 0.5);
\draw[step=0.5cm, thin] (-1.5, 0) grid (1.5, 0.5);
\draw[step=0.5cm, thin] (0.999, 0.5) grid (1.5, 1.5);
\draw[dotted] (0, 0) -- (-1.8, 0);
\draw[dotted] (1, 1.5) -- (1, 1.8);
\draw[dotted] (1.5, 1.5) -- (1.5, 1.8);
\end{tikzpicture}
\end{center}
By definition, the map $\varphi$ changes the shade of the bottom-left-most gray box; in other words, $\varphi$ switches these two types. Hence, $\varphi$ is indeed a map from the set $S(\lambda, \mu)$ to itself.
\end{proof}

Our proof for the property of factorization is based on a variant of Pieri rule, which we show first. 

\begin{lem} \label{2_prop_variant_Pieri_rule} Let $\calX$ contain $n$ variables. For any partition $\lambda$,
\begin{align} \label{2_prop_variant_Pieri_rule_eq} \schur_\lambda(\calX) \elementary_{n - r}(\calX) = \sum_{\substack{\kappa: \\ \left( \lambda + \left\langle 1^n \right\rangle \right) \setminus \kappa \text{ is a vertical $r$-strip}}} \schur_\kappa(\calX)
.
\end{align}
\end{lem}

\begin{proof} If $\kappa$ is a partition of length at most $n$, then $\left( \lambda + \left\langle 1^n \right\rangle \right) \setminus \kappa$ is a vertical $r$-strip if and only if $\kappa \setminus \lambda$ is a vertical $(n - r)$-strip. Hence the right-hand side of the equality in \eqref{2_prop_variant_Pieri_rule_eq} may be reformulated as
\begin{align*}
\sum_{\substack{\kappa: \\ \left( \lambda + \left\langle 1^n \right\rangle \right) \setminus \kappa \text{ is a vertical $r$-strip}}} \schur_\kappa(\calX) ={} & \sum_{\substack{\kappa: \\ \kappa \setminus \lambda \text{ is a vertical $(n - r)$-strip} \\ l(\kappa) \leq n}} \schur_\kappa(\calX)
\end{align*}
Given that $\schur_\kappa(\calX) = 0$ for all $\kappa$ of length strictly greater than $n$, the result follows from the equality in \eqref{1_cor_pieri_rule_schur_eq_elementary}.
\end{proof}

\begin{lem} [factorization] Suppose that $\calX$ and $\calY$ only contain non-zero variables. If $\lambda$ satisfies $\lambda_n \geq m \geq \lambda_{n + 1}$, so that it can be written in the form $\left( \left\langle m^n \right\rangle + \alpha \right) \cup \beta'$ for some partitions $\alpha$ and $\beta$ of lengths at most $n$ and $m$, respectively, then 
$$
LS_\lambda(-\calX; \calY) = \Delta(\calY; \calX) \schur_\alpha(-\calX) \schur_\beta(\calY)
.
$$ 
\end{lem}

\begin{proof} This proof is an induction over $m$, the number of elements of $\calY$. Recalling that $LS_\lambda(-\calX; \emptyset) = \schur_\lambda(-\calX)$, the base case $m = 0$ is trivial. For the induction step, we consider a set of variables $\calY = \calY' \cup (y_m)$ consisting of $m \geq 1$ elements, and a partition $\lambda = \left( \left\langle m^n \right\rangle + \alpha \right) \cup \beta'$ with $l(\alpha) \leq n$ and $l(\beta) \leq m$. Corollary~\ref{1_cor_Pieri_rule_for_LS} states that
\begin{align} 
\begin{split} \label{2_in_proof_factorization_LS_Pieri}
LS_{\lambda}(-\calX; \calY) ={} & \sum_{\substack{\kappa \subset \lambda: \\ \lambda \setminus \kappa \text{ is a vertical strip}}} LS_\kappa \left(-\calX, \calY' \right) y_{m}^{|\lambda| - |\kappa|}
.
\end{split}
\intertext{We claim that the sum on the right-hand side in \eqref{2_in_proof_factorization_LS_Pieri} may be reformulated as}
\begin{split} \label{2_in_proof_factorization_LS_reformulate}
LS_{\lambda}(-\calX; \calY) ={} & \sum_{\substack{\mu \subset \left\langle 1^n \right\rangle + \alpha, \; \nu \subset \beta: \\ \left( \left\langle 1^n \right\rangle + \alpha \right) \setminus \mu \text{ is a vertical strip} \\ \beta \setminus \nu \text{ is a horizontal strip} \\ l(\nu) \leq m - 1}} LS_{\left( \left\langle (m - 1)^n \right\rangle + \mu \right) \cup \nu'} \left(-\calX, \calY' \right) y_m^{|\alpha| + n + |\beta| - |\mu| - |\nu|}
.
\end{split}
\end{align}
First, suppose that $l(\beta) \leq m - 1$, which entails that the Ferrers diagram of the partition $\lambda$ is of the following general shape:
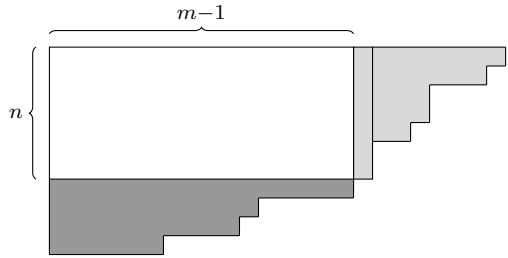
\begin{center}
\begin{tikzpicture}
\fill[black!40!white] (0, -1.75) rectangle (1.5, -2.75);
\fill[black!40!white] (1.5, -1.75) rectangle (2.5, -2.5);
\fill[black!40!white] (2.5, -1.75) rectangle (2.75, -2.25);
\fill[black!40!white] (2.75, -1.75) rectangle (4, -2);
\fill[black!15!white] (4, 0) rectangle (6, -0.25);
\fill[black!15!white] (4, -0.25) rectangle (5.75, -0.5);
\fill[black!15!white] (4, -0.5) rectangle (5, -1);
\fill[black!15!white] (4, -1) rectangle (4.75, -1.25);
\fill[black!15!white] (4, 0) rectangle (4.25, -1.75);
\draw (0,0) rectangle (4,-1.75);
\draw (4,0) rectangle (4.25, -1.75);
\draw (4.25, 0) -- (6, 0);
\draw (6,0) -- (6,-0.25);
\draw (6, -0.25) -- (5.75, -0.25);
\draw (5.75, -0.25) -- (5.75, -0.5);
\draw (5.75, -0.5) -- (5, -0.5);
\draw (5, -0.5) -- (5, -1);
\draw (5, -1) -- (4.75, -1);
\draw (4.75, -1) -- (4.75, -1.25);
\draw (4.75, -1.25) -- (4.25, -1.25);
\draw (4, -1.75) -- (4, -2);
\draw (4, -2) -- (2.75, -2);
\draw (2.75, -2) -- (2.75, -2.25);
\draw (2.75, -2.25) -- (2.5, -2.25);
\draw (2.5, -2.25) -- (2.5, -2.5);
\draw (2.5, -2.5) -- (1.5, -2.5);
\draw (1.5, -2.5) -- (1.5, -2.75);
\draw (1.5, -2.75) -- (0, -2.75);
\draw (0, -2.75) -- (0, -1.75);
\draw[decoration={brace, raise=5pt},decorate] (0, 0) -- node[above=6pt] {$\scriptstyle{m - 1}$} (4, 0);
\draw[decoration={brace, raise=5pt, mirror},decorate] (0,0) -- node[left=6pt] {$\scriptstyle{n}$} (0,-1.75);
\end{tikzpicture}
\end{center}
Here the diagrams of the partitions $\beta'$ and $\left\langle 1^n \right\rangle + \alpha$ are colored in dark and light gray, respectively. By assumption, $(m, n) \in \lambda$ and $(m, n + 1) \not\in \lambda$. Therefore, the map that decomposes any vertical strip $\lambda \setminus \kappa$ into a dark gray and a light gray part is a bijection between the set of all vertical strips $\lambda \setminus \kappa$ and the set of all pairs of vertical strips $\left( \left\langle 1^n \right\rangle + \alpha \right) \setminus \mu$, $\beta' \setminus \nu'$. Thus, our claim is an immediate consequence of the fact that transposing a vertical strip results in a horizontal strip. 

Second, suppose that $l(\beta) = m$. Then the Ferrers diagram of a typical partition $\lambda$ might look as follows:
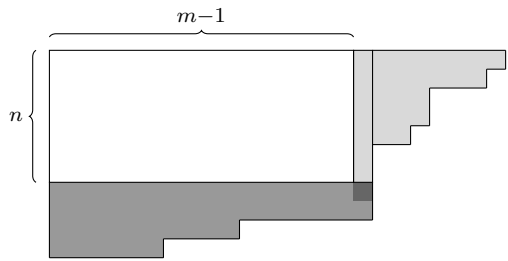
\begin{center}
\begin{tikzpicture}
\fill[black!40!white] (0, -1.75) rectangle (1.5, -2.75);
\fill[black!40!white] (1.5, -1.75) rectangle (2.5, -2.5);
\fill[black!40!white] (2.5, -1.75) rectangle (4.25, -2.25);
\fill[black!60!white] (4, -1.75) rectangle (4.25, -2);
\fill[black!15!white] (4, 0) rectangle (6, -0.25);
\fill[black!15!white] (4, -0.25) rectangle (5.75, -0.5);
\fill[black!15!white] (4, -0.5) rectangle (5, -1);
\fill[black!15!white] (4, -1) rectangle (4.75, -1.25);
\fill[black!15!white] (4, 0) rectangle (4.25, -1.75);
\draw (0,0) rectangle (4,-1.75);
\draw (4,0) rectangle (4.25, -1.75);
\draw (4.25, 0) -- (6, 0);
\draw (6,0) -- (6,-0.25);
\draw (6, -0.25) -- (5.75, -0.25);
\draw (5.75, -0.25) -- (5.75, -0.5);
\draw (5.75, -0.5) -- (5, -0.5);
\draw (5, -0.5) -- (5, -1);
\draw (5, -1) -- (4.75, -1);
\draw (4.75, -1) -- (4.75, -1.25);
\draw (4.75, -1.25) -- (4.25, -1.25);
\draw (4.25, -1.75) -- (4.25, -2.25);
\draw (4.25, -2.25) -- (2.5, -2.25);
\draw (2.5, -2.25) -- (2.5, -2.5);
\draw (2.5, -2.5) -- (1.5, -2.5);
\draw (1.5, -2.5) -- (1.5, -2.75);
\draw (1.5, -2.75) -- (0, -2.75);
\draw (0, -2.75) -- (0, -1.75);
\draw[decoration={brace, raise=5pt},decorate] (0, 0) -- node[above=6pt] {$\scriptstyle{m - 1}$} (4, 0);
\draw[decoration={brace, raise=5pt, mirror},decorate] (0,0) -- node[left=6pt] {$\scriptstyle{n}$} (0,-1.75);
\end{tikzpicture}
\end{center}
Here the diagrams of the partitions $\beta'$ and $\left\langle 1^n \right\rangle + \alpha$ are also colored in dark and light gray, respectively. We see that $\lambda$ contains the box with coordinates $(m, n + 1)$ (colored in even darker gray), which entails that decomposing any vertical strip $\lambda \setminus \kappa$ into two parts (of different shades of gray) no longer results in a bijection between the set of all vertical strips $\lambda \setminus \kappa$ and the set of all pairs of vertical strips $\left( \left\langle 1^n \right\rangle + \alpha \right) \setminus \mu$, $\beta' \setminus \nu'$. Indeed, any partition $\kappa$ that contains the box at $(m, n + 1)$ must also contain the box at $(m, n)$; in consequence, the map is not surjective. However, $LS_\kappa\left(-\calX; \calY' \right)$ vanishes whenever $\kappa$ contains the box at $(m, n + 1)$, according to Corollary~\ref{1_cor_LS=0}. We conclude that there is a bijection between the set of all partitions $\kappa$ that contribute to the sum in \eqref{2_in_proof_factorization_LS_Pieri} and the set of all pairs of partitions $\mu$, $\nu$ so that $\left( \left\langle 1^n \right\rangle + \alpha \right) \setminus \mu$, $\beta' \setminus \nu'$ are vertical strips \emph{and} and $\nu'_1 \leq m - 1$, which completes the proof of our claim that the sum in \eqref{2_in_proof_factorization_LS_Pieri} may be reformulated as in \eqref{2_in_proof_factorization_LS_reformulate}.

Applying the induction hypothesis to the Littlewood-Schur functions on the right-hand side in \eqref{2_in_proof_factorization_LS_reformulate} yields
\begin{align*} 
LS_{\lambda}(-\calX; \calY) ={} & \Delta \left( \calY'; \calX \right) \left[ \sum_{\substack{\mu \subset \left\langle 1^n \right\rangle + \alpha: \\ \left( \left\langle 1^n \right\rangle + \alpha \right) \setminus \mu \text{ is a vertical strip}}} \schur_\mu (-\calX) y_m^{|\alpha| + n - |\mu|} \right] \\
& \times \left[ \sum_{\substack{\nu \subset \beta: \\ \beta \setminus \nu \text{ is a horizontal strip} \\ l(\nu) \leq m - 1}} \schur_\nu \left( \calY' \right) y_m^{|\beta| - |\nu|} \right]
.
\intertext{Given that $\calY'$ contains $m - 1$ variables, the condition on the length of $\nu$ is superfluous. Hence, the equality in \eqref{1_eq_pieri_rule_schur_eq_add_variable} implies that} 
LS_{\lambda}(-\calX; \calY) ={} & \Delta \left( \calY'; \calX \right) \left[ \sum_{r = 0}^n y_m^r \sum_{\substack{\mu \subset \left\langle 1^n \right\rangle + \alpha: \\ \left( \left\langle 1^n \right\rangle + \alpha \right) \setminus \mu \text{ is a vertical $r$-strip}}} \schur_\mu (-\calX) \right] \schur_\beta(\calY)
.
\intertext{Lemma~\ref{2_prop_variant_Pieri_rule} allows us to conclude that}
LS_{\lambda}(-\calX; \calY) ={} & \Delta \left( \calY'; \calX \right) \left[ \sum_{r = 0}^n y_m^r \elementary_{n - r}(-\calX) \right] \schur_\alpha(-\calX) \schur_\beta(\calY) \displaybreak[2]
\\
={} & \Delta \left( \calY'; \calX \right) \prod_{x \in \calX} (y_m - x) \schur_\alpha(-\calX) \schur_\beta(\calY) = \Delta \left( \calY; \calX \right) \schur_\alpha(-\calX) \schur_\beta(\calY)
. \qedhere
\end{align*}
\end{proof}

\section{The determinantal formula} \label{2_sec_proof_of_det_formula_for_LS}
This section consists of the statement and proof of Moens and Van der Jeugt's determinantal formula for Littlewood-Schur functions \cite{vanderjeugt}. The proof relies on a few basic properties of determinants, such as Laplace expansion. Let us quickly recall this result from linear algebra. For an $n \times n$ matrix $A = (a_{ij})$, the symbol $A_{\bar{i} \bar{j}}$ denotes the $(n - 1) \times (n - 1)$ matrix $(a_{kl})_{k \neq i, l \neq j}$, \textit{i.e.}\ the matrix $A$ without the $i$-th row and the $j$-th column.

\begin{lem} [Laplace expansion] \label{2_lem_laplace_expansion} The determinant of an $n \times n$ matrix $A$ can be expanded in the following two ways:
\begin{enumerate}
\item for $1 \leq i \leq n$, $\displaystyle \det(A) = \sum_{j = 1}^n (-1)^{i + j} a_{ij} \det \left( A_{\bar{i} \bar{j}}\right)$;
\item for $1 \leq j \leq n$, $\displaystyle \det(A) = \sum_{i = 1}^n (-1)^{i + j} a_{ij} \det \left( A_{\bar{i} \bar{j}}\right)$.
\end{enumerate}
\end{lem}

The statement of the determinantal formula relies on the notion of the index of a partition.

\begin{defn} [index of a partition] \label{2_defn_index} The $(m,n)$-index of a partition $\lambda$ is the largest (possibly negative) integer $k$ with the properties that $(m + 1 - k, n + 1 - k) \not\in \lambda$ and $k \leq \min\{m,n\}$. Equivalently, the index $k$ is the smallest integer such that the rectangle $\left\langle (m - k)^{n - k} \right\rangle$ is contained in the Ferrers diagram of $\lambda$.
\end{defn}

If $(m,n) \not\in \lambda$, then $k$ is the side of the largest square with bottom-right corner $(m,n)$ that fits next to the diagram of the partition $\lambda$. If $(m,n) \in \lambda$, then $-k$ is the side of the largest square with top-left corner $(m,n)$ that fits inside the diagram of $\lambda$. Let us illustrate this by a sketch: the area colored in gray is the diagram of some partition $\lambda$; for various locations of the point $(m,n)$, we draw the square whose side is equal to the $(m,n)$-index of $\lambda$ in absolute value:
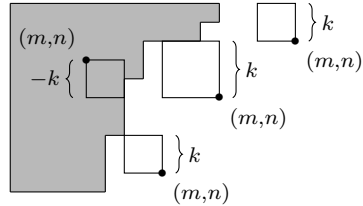
\begin{center}
\begin{tikzpicture}
\fill[black!27.5!white] (-0.25, 0) rectangle (1, 2.5);
\fill[black!27.5!white] (1, 0.75) rectangle (1.25, 2.5);
\fill[black!27.5!white] (1.25, 1.5) rectangle (1.5, 2.5);
\fill[black!27.5!white] (1.5, 2) rectangle (2.25, 2.5);
\fill[black!27.5!white] (2.25, 2.25) rectangle (2.5, 2.5);
\draw (-0.25, 0) -- (1, 0);
\draw (1, 0) -- (1, 0.75);
\draw (1, 0.75) -- (1.25, 0.75);
\draw (1.25, 0.75) -- (1.25, 1.5);
\draw (1.25, 1.5) -- (1.5, 1.5);
\draw (1.5, 1.5) -- (1.5, 2);
\draw (1.5, 2) -- (2.25, 2);
\draw (2.25, 2) -- (2.25, 2.25);
\draw (2.25, 2.25) -- (2.5, 2.25);
\draw (2.5, 2.25) -- (2.5, 2.5);
\draw (-0.25, 2.5) -- (2.5, 2.5);
\draw (-0.25, 2.5) -- (-0.25, 0);
\draw (1.25, 0.75) rectangle (1.75, 0.25);
\draw (1.75, 2) rectangle (2.5, 1.25);
\draw (1.25, 1.25) rectangle (0.75, 1.75);
\draw (3, 2.5) rectangle (3.5, 2);
\draw[decoration={brace, raise=5pt, mirror},decorate] (1.75, 0.25) -- node[right=6pt] {$\scriptstyle{k}$} (1.75, 0.75);
\draw[decoration={brace, raise=5pt, mirror},decorate] (2.5, 1.25) -- node[right=6pt] {$\scriptstyle{k}$} (2.5, 2);
\draw[decoration={brace, raise=5pt},decorate] (0.75, 1.25) -- node[left=6pt] {$\scriptstyle{-k}$} (0.75, 1.75);
\draw[decoration={brace, raise=5pt, mirror},decorate] (3.5, 2) -- node[right=6pt] {$\scriptstyle{k}$} (3.5, 2.5);
\node[anchor=north west] at (1.75, 0.25) {$\scriptstyle{(m,n)}$};
\node[anchor=north west] at (2.5, 1.25) {$\scriptstyle{(m,n)}$};
\node[anchor=south east] at (0.75, 1.75) {$\scriptstyle{(m,n)}$};
\node[anchor=north west] at (3.5, 2) {$\scriptstyle{(m,n)}$};
\node at (1.75, 0.25) {\tiny{\textbullet}};
\node at (2.5, 1.25) {\tiny{\textbullet}};
\node at (0.75, 1.75) {\tiny{\textbullet}};
\node at (3.5, 2) {\tiny{\textbullet}};
\node at (0, 2.6) {};
\end{tikzpicture}
\end{center}
We remark that the definition given above is not equivalent to the definition of index used in \cite{vanderjeugt}. The main advantage of our notion is that it is invariant under conjugation. Indeed, $(m + 1 - k, n + 1 - k) \not\in \lambda$ is equivalent to $(n + 1 - k, m + 1 - k) \not\in \lambda'$, which shows that the $(m,n)$-index of $\lambda$ is equal to the $(n,m)$-index of $\lambda'$. Another advantage is that the $(m,n)$-index of a partition $\lambda$ is 0 if and only if $\lambda$ satisfies the prerequisites for the property of factorization, thus giving a special role to the integer 0. In addition, the following theorem also makes a distinction between partitions with negative and non-negative index.

\begin{thm} [determinantal formula for Littlewood-Schur functions, adapted from \cite{vanderjeugt}] \label{2_thm_det_formula_for_Littlewood-Schur}
Let $\calX$ and $\calY$ be sets of variables with $n$ and $m$ elements, respectively, so that the elements of $\calX \cup \calY$ are pairwise distinct. Let $\lambda$ be a partition with $(m,n)$-index $k$. If $k$ is negative, then $LS_\lambda(-\calX; \calY) = 0$; otherwise,
\begin{align} \label{2_thm_det_formula_for_Littlewood-Schur_eq}
LS_\lambda(-\calX; \calY) ={} & \varepsilon(\lambda) \frac{\Delta(\calY; \calX)}{\Delta(\calX) \Delta(\calY)} \det \begin{pmatrix} \left( (x - y)^{-1} \right)_{\substack{x \in \calX \\ y \in \calY}} & \left( x^{\lambda_j + n - m - j} \right)_{\substack{x \in \calX \\ 1 \leq j \leq n - k}} \\ \left( y^{\lambda'_i + m - n - i} \right)_{\substack{1 \leq i \leq m - k \\ y \in \calY}} & 0\end{pmatrix}
\end{align}
where $\varepsilon(\lambda) = (-1)^{\left|\lambda_{[n - k]} \right|} (-1)^{mk} (-1)^{k(k - 1)/2}$.
\end{thm}
The sign $\varepsilon(\lambda)$ merits a few remarks. For any sequence $\lambda$, $\lambda_{[n]}$ is our notation for the sequence that consists of the first $n$ elements of $\lambda$; in symbols, $\lambda_{[n]} = (\lambda_1, \dots, \lambda_n)$. It is also worth mentioning that $\varepsilon(\lambda)$ not only depends on $\lambda$, but also on the parameters $m$ and $n$. If the parameters are not clear from the context, we add them as subscripts. Adding the proper parameters for this theorem would give $\varepsilon_{m,n}(\lambda)$.

We prove Theorem~\ref{2_thm_det_formula_for_Littlewood-Schur} by showing that the function defined by the right-hand side in \eqref{2_thm_det_formula_for_Littlewood-Schur_eq} satisfies the five characterizing properties listed in Theorem~\ref{2_thm_5_properties_of_LS}. Although Moens and Van der Jeugt originally discovered the determinantal formula for Littlewood-Schur functions through the study of representations of Lie superalgebras, they provide several independent justifications in \cite{vanderjeugt}, among which a proof based on the strategy used here.

For the remainder of this section we fix two parameters $m$, $n$, two sets of variables $\calX = (x_1, \dots, x_n)$, $\calY = (y_1, \dots, y_m)$ so that the elements in $\calX \cup \calY$ are pairwise distinct and a partition $\lambda$ with $(m,n)$-index $k$. Let us denote the right-hand side in \eqref{2_thm_det_formula_for_Littlewood-Schur_eq} by $LS_\lambda^{\ast}(-\calX; \calY)$ whenever $k \geq 0$; otherwise, we set $LS_\lambda^{\ast}(-\calX; \calY) = 0$.

\begin{lem} [homogeneity] \label{2_lem_homogeneity_with_det_formula} The function $LS_\lambda^{\ast}(-\calX; \calY)$ is a homogeneous polynomial in the variables $\calX \cup \calY$ of degree $|\lambda|$.
\end{lem}

\begin{proof} If $k$ is negative, then $LS_\lambda^{\ast}(-\calX; \calY)$ can be regarded as the zero polynomial. If $k$ is non-negative, we may use the factor $\Delta(\calX; \calY)$ to multiply each row of the matrix in \eqref{2_thm_det_formula_for_Littlewood-Schur_eq} that corresponds to some $x \in \calX$ by $\prod_{y \in \calY} (x - y)$, ensuring that the resulting determinant is a polynomial in $\calX \cup \calY$. Moreover, this polynomial is divisible by $\Delta(\calX) \Delta(\calY)$ because if any two variables in $\calX$ or in $\calY$ are equal then the two corresponding rows respectively columns of the matrix in \eqref{2_thm_det_formula_for_Littlewood-Schur_eq} are also identical and thus its determinant vanishes.

It remains to verify that $LS_\lambda^{\ast}(-\calX; \calY)$ is homogeneous of degree $|\lambda|$. For $a \in \C \setminus \{0\}$, the structure of the matrix entails that
\begin{align*}
LS_\lambda^{\ast}(-a\calX; a\calY) ={} & \frac{a^{mn} a^{-k}}{a^{n(n - 1)/2} a^{m(m - 1)/2}} \left[ \prod_{1 \leq j \leq n - k} a^{\lambda_j + n - m - j} \right] \!\! \left[ \prod_{1 \leq i \leq m  - k} a^{\lambda'_i + m - n - i} \right] \\
& \times LS_\lambda^{\ast}(-\calX; \calY) \displaybreak[2] \\
={} & a^{\left|\lambda_{[n - k]}\right| + \left| \lambda'_{[m - k]} \right|} a^{mn - k - k(k - 1) - m(n - k) - n(m - k)} LS_\lambda^{\ast}(-\calX; \calY) \displaybreak[2] \\
={} & a^{\left|\lambda_{[n - k]}\right| + \left| \lambda'_{[m - k]} \right| - (m - k)(n - k)} LS_\lambda^{\ast}(-\calX; \calY) = a^{|\lambda|} LS_\lambda^{\ast}(-\calX; \calY)
\end{align*}
because $(m - k, n - k) \in \lambda$ and $(m + 1 - k, n + 1 - k) \not\in \lambda$ by the definition of index.
\end{proof}

Although the following property is not listed in Theorem~\ref{2_thm_5_properties_of_LS}, we state and prove it here because it will shorten the proofs of some of the characterizing properties that are listed in Theorem~\ref{2_thm_5_properties_of_LS}.

\begin{cor} \label{2_cor_transposition_invariance} Transposing the indexing partition $\lambda$ results in
\begin{align} \label{2_cor_transposition_invariance_eq}
LS_{\lambda'}^{\ast}(-\calX; \calY) = LS_\lambda^{\ast}(\calY; -\calX)
.
\end{align}
\end{cor}

\begin{proof} Let $k'$ be the $(m,n)$-index of $\lambda'$. Given that $k'$ is also the $(n,m)$-index of $\lambda$, the equality in \eqref{2_cor_transposition_invariance_eq} holds whenever $k'$ is negative. If $k' \geq 0$, we use that transposing a matrix leaves its determinant invariant to see that
\begin{align*}
LS_{\lambda'}^{\ast}(-\calX; \calY) ={} & \varepsilon_{m,n}\left(\lambda'\right) \frac{\Delta(\calY; \calX)}{\Delta(\calX) \Delta(\calY)} \\
& \times \det \begin{pmatrix} \left( (x_j - y_i)^{-1} \right)_{\substack{1 \leq i \leq m \\ 1 \leq j \leq n}} & \left( y_i^{\lambda_j + m - n - j} \right)_{\substack{1 \leq i \leq m \\ 1 \leq j \leq m - k'}} \\ \left( x_j^{\lambda'_i + n - m - i} \right)_{\substack{1 \leq i \leq n - k' \\ 1 \leq j \leq n}} & 0\end{pmatrix}
.
\intertext{In order to obtain a top-left block of the required form, we multiply by $(-1)$ the first $m$ rows as well as the last $m - k'$ columns of the matrix, which results in}
LS_{\lambda'}^{\ast}(-\calX; \calY) ={} & \varepsilon_{m,n}\left(\lambda'\right) \frac{(-1)^{mn} \Delta(\calX; \calY)}{\Delta(\calX) \Delta(\calY)} (-1)^{m + m - k'} \\
& \times \det \begin{pmatrix} \left( (y_i - x_j)^{-1} \right)_{\substack{1 \leq i \leq m \\ 1 \leq j \leq n}} & \left( y_i^{\lambda_j + m - n - j} \right)_{\substack{1 \leq i \leq m \\ 1 \leq j \leq m - k'}} \\ \left( x_j^{\lambda'_i + n - m - i} \right)_{\substack{1 \leq i \leq n - k' \\ 1 \leq j \leq n}} &  0\end{pmatrix} \displaybreak[2] \\
={} & \varepsilon_{m,n}\left(\lambda'\right) (-1)^{mn} (-1)^{k'} \varepsilon_{n,m}(\lambda) LS_\lambda^{\ast}(-\calY; \calX)
\end{align*}
owing to the transposition invariance of the index. Recall that $$|\lambda| = \left| \lambda_{[m - k']} \right| + \left| \lambda'_{[n - k']} \right| - (m - k')(n - k').$$ Hence,
\begin{multline*} (-1)^{mn} (-1)^{k'} \varepsilon_{m,n}\left(\lambda'\right) \varepsilon_{n,m}(\lambda) \\ 
= (-1)^{mn + k'} \left[ (-1)^{\left|\lambda'_{[n - k']} \right|} (-1)^{mk'} (-1)^{k'(k' - 1)/2} \right] \left[ (-1)^{\left|\lambda_{[m - k']} \right|} (-1)^{nk'} (-1)^{k'(k' - 1)/2} \right] \\
= (-1)^{mn + k'} (-1)^{|\lambda| + (m - k')(n - k')} (-1)^{mk'} (-1)^{nk'} (-1)^{k'(k' - 1)}
= (-1)^{|\lambda|}
\end{multline*}
and the equality in \eqref{2_cor_transposition_invariance_eq} is an immediate consequence of the homogeneity shown in the preceding lemma.
\end{proof}

In order to prove the determinantal formula stated in Theorem~\ref{2_thm_det_formula_for_Littlewood-Schur}, we turn to the next properties listed in Theorem~\ref{2_thm_5_properties_of_LS}.

\begin{lem} [double-symmetry] The polynomial $LS_\lambda^{\ast}(-\calX; \calY)$ is symmetric in each set of variables separately.
\end{lem}

\begin{proof} If $k$ is negative, then $LS_\lambda^{\ast}(-\calX; \calY) = 0$. In particular, it is symmetric. Without loss of generality, we thus assume that $k \geq 0$. 

Let us first consider the determinantal formula on the right-hand side in \eqref{2_thm_det_formula_for_Littlewood-Schur_eq} from the point of view of permuting the variables in $\calX$. Given that $\Delta(\calY; \calX)$ is symmetric in $\calX$, while both $\Delta(\calX)$ and the determinant are skew-symmetric in $\calX$, the symmetry in $\calX$ follows immediately. The symmetry in $\calY$ can now be concluded from Corollary~\ref{2_cor_transposition_invariance}, or from an analogous argument.
\end{proof}

\begin{lem} [factorization] \label{2_lem_factorization} Suppose that $\calX$ and $\calY$ only contain nonzero variables. If $\lambda$ satisfies $\lambda_n \geq m \geq \lambda_{n + 1}$, so that it can be written in the form $\left( \left\langle m^n \right\rangle + \alpha \right) \cup \beta'$ for some partitions $\alpha$ and $\beta$ of lengths at most $n$ and $m$, respectively, then 
\begin{align*} 
LS_\lambda^{\ast}(-\calX; \calY) = \Delta(\calY; \calX) \schur_\alpha(-\calX) \schur_\beta(\calY)
.
\end{align*}
\end{lem}

\begin{proof} By definition, $\lambda_n \geq m \geq \lambda_{n + 1}$ if and only if $k = 0$. Therefore, the fact that $\alpha_j = \lambda_j - m$ and $\beta_i = \lambda'_i - n$ for all $1 \leq j \leq n$ and $1 \leq i \leq m$ implies that
\begin{align*}
LS_\lambda^{\ast}(-\calX; \calY) ={} & \varepsilon(\lambda) \frac{\Delta(\calY; \calX)}{\Delta(\calX) \Delta(\calY)} \det \begin{pmatrix} \left( (x - y)^{-1} \right)_{\substack{x \in \calX \\ y \in \calY}} & \left( x^{\alpha_j + n - j} \right)_{\substack{x \in \calX \\ 1 \leq j \leq n}} \\ \left( y^{\beta_i + m - i} \right)_{\substack{1 \leq i \leq m \\ y \in \calY}} & 0\end{pmatrix} \displaybreak[2] \\
={} & \varepsilon(\lambda) \frac{\Delta(\calY; \calX)}{\Delta(\calX) \Delta(\calY)} (-1)^{mn} \det \left( x^{\alpha_j + n - j} \right)_{\substack{x \in \calX \\ 1 \leq j \leq n}} \det \left( y^{\beta_i + m - i} \right)_{\substack{1 \leq i \leq m \\ y \in \calY}} \displaybreak[2] \\
={} & \varepsilon(\lambda) (-1)^{mn} \Delta(\calY; \calX) \schur_\alpha(\calX) \schur_\beta(\calY)
.
\end{align*}
Recalling that Schur functions are homogeneous, the result follows from the fact that $\varepsilon(\lambda) (-1)^{mn} = (-1)^{\left|\lambda_{[n]} \right|} (-1)^{mn} = (-1)^{|\alpha|}.$
\end{proof}

If $n \geq 1$ (resp.\ $m \geq 1$), define $\calX' = (x_1, \dots, x_{n - 1})$ (resp.\ $\calY' = (y_1, \dots, y_{m - 1})$).

\begin{lem} [cancellation] If $m$, $n \geq 1$, then 
\begin{align} \label{2_lem_cancellation_eq}
LS_\lambda^{\ast}(-\calX; \calY)_{{\big{\vert}}_{x_n = y_m}} = LS_\lambda^{\ast}\left(-\calX'; \calY'\right)
.
\end{align}
\end{lem}

\begin{proof} Technically it is not permissible to set $x_n = y_m$ since this violates the condition that the elements of $\calX \cup \calY$ be pairwise distinct. However, if we let $x_n$ and $y_n$ converge towards each other in $LS_\lambda^{\ast}(-\calX; \calY)$, the fact that it is a polynomial in $\calX \cup \calY$ implies that the limit is well defined.

We remark that since $k$ is the $(m,n)$-index of $\lambda$, $k - 1$ is its $(m - 1, n - 1)$-index. Hence, the equality in \eqref{2_lem_cancellation_eq} trivially holds if $k < 0$ and it is a consequence of the factorization stated in Lemma~\ref{2_lem_factorization} if $k = 0$. If $k > 0$, we exploit that the right-hand side in \eqref{2_thm_det_formula_for_Littlewood-Schur_eq} remains unaltered if we simultaneously divide $\Delta(\calY; \calX)$ by $(x_n - y_m)$ and multiply the $n$-th row of the matrix by $(x_n - y_m)$, in order to obtain an expression that may be specialized to $x_n = y_m$. Upon specializing, all entries of the $n$-th row of the matrix vanish except the entry in the $m$-th column, which remains 1. Hence, Laplace expansion implies that its determinant is equal to the determinant of the matrix without the $n$-th row and the $m$-th column (up to the sign $(-1)^{m + n}$). These linear algebra considerations allow us to conclude that
\begin{align*}
LS_\lambda^{\ast}(-\calX; \calY)_{{\big{\vert}}_{x_n = y_m}} ={} & \varepsilon_{m,n}(\lambda) \frac{(-1) \Delta \left(\calY'; \calX'\right) \Delta\left(y_m; \calX'\right) \Delta\left(\calY'; x_n\right)}{\Delta(\calX) \Delta(\calY)} (-1)^{m + n} \\
& \times \det \begin{pmatrix} \left( (x - y)^{-1} \right)_{\substack{x \in \calX' \\ y \in \calY'}} & \left( x^{\lambda_j + n - m - j} \right)_{\substack{x \in \calX' \\ 1 \leq j \leq n - k}} \\ \left( y^{\lambda'_i + m - n - i} \right)_{\substack{1 \leq i \leq m - k \\ y \in \calY'}} & 0\end{pmatrix} \displaybreak[2] \\
={} & \varepsilon_{m,n}(\lambda) \frac{(-1) \Delta \left(\calY'; \calX'\right) (-1)^{n - 1}}{\Delta\left(\calX'\right) \Delta\left(\calY'\right)} (-1)^{m + n} \\
& \times \det \begin{pmatrix} \left( (x - y)^{-1} \right)_{\substack{x \in \calX' \\ y \in \calY'}} & \left( x^{\lambda_j + n - m - j} \right)_{\substack{x \in \calX' \\ 1 \leq j \leq n - k}} \\ \left( y^{\lambda'_i + m - n - i} \right)_{\substack{1 \leq i \leq m - k \\ y \in \calY'}} & 0\end{pmatrix} \displaybreak[2] \\
={} & (-1)^m \varepsilon_{m,n}(\lambda) \varepsilon_{m-1,n-1}(\lambda) LS_\lambda^{\ast}\left(-\calX'; \calY'\right)
\end{align*}
where the subscripts indicate with respect to which parameters the sign of $\lambda$ is taken. The equality in \eqref{2_lem_cancellation_eq} is a consequence of the fact that the signs cancel.
\end{proof}

According to Theorem~\ref{2_thm_5_properties_of_LS}, the determinantal formula for Littlewood-Schur functions follows if $LS_\lambda^{\ast}(-\calX; \calY)$ also satisfies the property of restriction.

\begin{lem} [restriction] If $n \geq 1$, then
\begin{align} \label{2_lem_restriction_eq_x}
LS_\lambda^{\ast}(-\calX; \calY)_{{\big{\vert}}_{x_n = 0}} ={} & LS_\lambda^{\ast}\left(-\calX'; \calY\right)
.
\intertext{If $m \geq 1$, then}
LS_\lambda^{\ast}(-\calX; \calY)_{{\big{\vert}}_{y_m = 0}} ={} & LS_\lambda^{\ast}\left(-\calX; \calY'\right)
.
\end{align}
\end{lem}

\begin{proof} Owing to the equality in~\eqref{2_cor_transposition_invariance_eq}, it is sufficient to show one of the claims. We study the effect of specializing $x_n$ to 0.

The $(m, n - 1)$-index of $\lambda$ is either $k - 1$ or $k$. Hence, the equality in \eqref{2_lem_restriction_eq_x} holds if $k$ is negative. If $k = 0$, then the statement follows from the factorization stated in Lemma~\ref{2_lem_factorization}: given that Schur functions satisfy the property of restriction,
\begin{align*}
LS_\lambda^{\ast}(-\calX; \calY)_{{\big{\vert}}_{x_n = 0}} ={} & \left[ \Delta(\calY; \calX) \schur_{\lambda_{[n]} - \langle m^n \rangle}(-\calX) \schur_{\lambda'_{[m]} - \langle n^m \rangle
}(\calY) \right]_{{\big{\vert}}_{x_n = 0}} \displaybreak[2] \\
={} & \Delta\left(\calY; \calX' \right) \left( \prod_{y \in \calY} y \right) \schur_{\lambda_{[n]} - \langle m^n \rangle}\left(-\calX' \right) \schur_{\lambda'_{[m]} - \langle n^m \rangle 
}(\calY) \displaybreak[2] \\
={} & \Delta\left(\calY; \calX' \right) \schur_{\lambda_{[n]} - \langle m^n \rangle}\left(-\calX' \right) \schur_{\lambda'_{[m]} - \left\langle (n - 1)^m \right\rangle 
}(\calY)
.
\end{align*}
We remark that the $(m, n - 1)$-index of $\lambda$ is 0 if and only if $(m + 1, n) \not\in \lambda$ if and only if $\lambda_{[n]} - \langle m^n \rangle = \lambda_{[n - 1]} - \langle m^{n - 1} \rangle$, so that the equation in \eqref{2_lem_restriction_eq_x} holds in case the $(m, n - 1)$-index of $\lambda$ is 0. Otherwise, the $(m, n - 1)$-index of $\lambda$ is negative and $\schur_{\lambda_{[n]} - \langle m^n \rangle}\left(-\calX' \right) = 0$, which means that both sides of the equation in \eqref{2_lem_restriction_eq_x} vanish.

If $k > 0$, we distinguish between two cases, namely whether the $(m, n - 1)$-index of $\lambda$ is equal to $k$ or $k - 1$. We note for later reference that the index remains unchanged if and only if $(m - k + 1, n - k) \not\in \lambda$, while it decreases if and only if $(m - k + 1, n - k) \in \lambda$ or $k = n$. In both cases,
\begin{align}
\begin{split} \notag
LS_\lambda^{\ast}(-\calX; \calY)_{{\big{\vert}}_{x_n = 0}} ={} & \varepsilon_{m,n}(\lambda) \frac{\Delta\left(\calY; \calX'\right)}{\Delta \left( \calX' \right) \Delta(\calY)} \cdot \frac{\prod_{y \in \calY} y}{\prod_{x \in \calX'} x} \\
& \times \det \begin{pmatrix} \left( (x_i - y_j)^{-1} \right)_{\substack{1 \leq i \leq n-1 \\ 1 \leq j \leq m}} & \left( x_i^{\lambda_j + n - m - j} \right)_{\substack{1 \leq i \leq n - 1 \\ 1 \leq j \leq n - k}} \\
 \left( -y_j^{-1} \right)_{1 \leq j \leq m} & \left( 0^{\lambda_j + n - m - j} \right)_{1 \leq j \leq n - k} \\
\left( y_j^{\lambda'_i + m - n - i} \right)_{\substack{1 \leq i \leq m - k \\ 1 \leq j \leq m}} & 0 
\end{pmatrix}
.
\end{split}
\intertext{We use the additional factor in front of the determinant to multiply the $j$-th column by $y_j$ for each $1 \leq j \leq m$ and the $i$-th row by $x_i^{-1}$ for each $1 \leq i \leq n - 1$:}
\begin{split} \label{2_in_proof_restriction_k_positive}
LS_\lambda^{\ast}(-\calX; \calY)_{{\big{\vert}}_{x_n = 0}} ={} & \varepsilon_{m,n}(\lambda) \frac{\Delta\left(\calY; \calX'\right)}{\Delta \left( \calX' \right) \Delta(\calY)} \\
& \times \det \begin{pmatrix} \left( x_i^{-1}y_j(x_i - y_j)^{-1} \right)_{\substack{1 \leq i \leq n-1 \\ 1 \leq j \leq m}} & \left( x_i^{\lambda_j + (n - 1) - m - j} \right)_{\substack{1 \leq i \leq n - 1 \\ 1 \leq j \leq n - k}} \\
 \left( -1 \right)_{1 \leq j \leq m} & \left( 0^{\lambda_j + n - m - j} \right)_{1 \leq j \leq n - k} \\
\left( y_j^{\lambda'_i + m - (n - 1) - i} \right)_{\substack{1 \leq i \leq m - k \\ 1 \leq j \leq m}} & 0 
\end{pmatrix}
.
\end{split}
\end{align}
First suppose that the $(m, n - 1)$-index of $\lambda$ is $k > 0$. This entails that $k \leq n - 1$ and $(m - k + 1, n - k) \not\in \lambda$; however, $(m - k, n - k) \in \lambda$ by the definition of $k$ (unless $m = k$). Therefore,
\begin{align*}
\varepsilon_{m,n}(\lambda) = (-1)^{\left| \lambda_{[n - k]} \right|} (-1)^{mk} (-1)^{k(k - 1/2)} = (-1)^{\lambda_{n - k}} \varepsilon_{m,n-1}(\lambda) = (-1)^{m - k} \varepsilon_{m,n-1}(\lambda)
.
\end{align*}
In consequence, the result follows if the following determinant is of the required form: \label{2_in_proof_defn_D}
\begin{align*}
\text{D} \defeq {} & (-1)^{m - k} \det \begin{pmatrix} \left( x_i^{-1}y_j(x_i - y_j)^{-1} \right)_{\substack{1 \leq i \leq n-1 \\ 1 \leq j \leq m}} & \left( x_i^{\lambda_j + (n - 1) - m - j} \right)_{\substack{1 \leq i \leq n - 1 \\ 1 \leq j \leq n - k}} \\
 \left( -1 \right)_{1 \leq j \leq m} & \left( 0^{\lambda_j + n - m - j} \right)_{1 \leq j \leq n - k} \\
\left( y_j^{\lambda'_i + m - (n - 1) - i} \right)_{\substack{1 \leq i \leq m - k \\ 1 \leq j \leq m}} & 0 
\end{pmatrix}
.
\end{align*}
By definition, $\lambda_{n - k} + (n - 1) - m - (n - k) = m - k + (n - 1) - m - (n - k) = -1$. Hence, the last column of the matrix is equal to $\left(x_1^{-1}, \dots, x_{n - 1}^{-1}, 1, 0 \dots, 0 \right)$ transposed. Expanding the determinant along this last column with the help of Lemma~\ref{2_lem_laplace_expansion} thus gives
\begin{align*}
\text{D} ={} & (-1)^{m - k} \sum_{p = 1}^{n - 1} (-1)^{m + n - k + p} x_p^{-1} \\
& \hspace{15pt} \times \det \begin{pmatrix} \left( x_i^{-1}y_j(x_i - y_j)^{-1} \right)_{\substack{1 \leq i \leq n-1, i \neq p \\ 1 \leq j \leq m}} & \left( x_i^{\lambda_j + (n - 1) - m - j} \right)_{\substack{1 \leq i \leq n - 1, i \neq p \\ 1 \leq j \leq n - 1 - k}} \\
 \left( -1 \right)_{1 \leq j \leq m} & \left( 0 \right)_{1 \leq j \leq n - 1 - k} \\
\left( y_j^{\lambda'_i + m - (n - 1) - i} \right)_{\substack{1 \leq i \leq m - k \\ 1 \leq j \leq m}} & 0 
\end{pmatrix} \\
& + (-1)^{m - k} \sum_{p = 1}^{n - 1} \frac{1}{n - 1} (-1)^{m + n - k + n} \\
& \hspace{15pt} \times \det \begin{pmatrix} \left( x_i^{-1}y_j(x_i - y_j)^{-1} \right)_{\substack{1 \leq i \leq n-1 \\ 1 \leq j \leq m}} & \left( x_i^{\lambda_j + (n - 1) - m - j} \right)_{\substack{1 \leq i \leq n - 1 \\ 1 \leq j \leq n - 1 - k}} \\
\left( y_j^{\lambda'_i + m - (n - 1) - i} \right)_{\substack{1 \leq i \leq m - k \\ 1 \leq j \leq m}} & 0 
\end{pmatrix}
\intertext{where we have artificially added a sum over $p$ to the second term, which will be useful for the next step. We expand the determinants in the first sum along the $(n - 1)$-th row and the determinants in the second sum along the $p$-th row:}
\text{D} ={} & \sum_{p = 1}^{n - 1} (-1)^{n + p} x_p^{-1} \sum_{q = 1}^{m} (-1)^{1 + q + n - 1} \\
& \hspace{15pt} \times \det \begin{pmatrix} \left( x_i^{-1}y_j(x_i - y_j)^{-1} \right)_{\substack{1 \leq i \leq n-1, i \neq p \\ 1 \leq j \leq m, j \neq q}} & \left( x_i^{\lambda_j + (n - 1) - m - j} \right)_{\substack{1 \leq i \leq n - 1, i \neq p \\ 1 \leq j \leq n - 1 - k}} \\
\left( y_j^{\lambda'_i + m - (n - 1) - i} \right)_{\substack{1 \leq i \leq m - k \\ 1 \leq j \leq m, j \neq q}} & 0 
\end{pmatrix} \displaybreak[2] \\
& + \sum_{p = 1}^{n - 1} \frac{1}{n-1} \sum_{q = 1}^m (-1)^{q + p} x_p^{-1}y_q(x_p - y_q)^{-1} \\
& \hspace{15pt} \times \det \begin{pmatrix} \left( x_i^{-1}y_j(x_i - y_j)^{-1} \right)_{\substack{1 \leq i \leq n-1, i \neq p \\ 1 \leq j \leq m, j \neq q}} & \left( x_i^{\lambda_j + (n - 1) - m - j} \right)_{\substack{1 \leq i \leq n - 1, i \neq p \\ 1 \leq j \leq n - 1 - k}} \\
\left( y_j^{\lambda'_i + m - (n - 1) - i} \right)_{\substack{1 \leq i \leq m - k \\ 1 \leq j \leq m, j \neq q}} & 0
\end{pmatrix} \\
& + \sum_{p = 1}^{n - 1} \frac{1}{n - 1} \sum_{r = 1}^{n - 1- k} (-1)^{p + m + r} x_p^{\lambda_r + (n - 1) - m - r} \\
& \hspace{15pt} \times \det \begin{pmatrix} \left( x_i^{-1}y_j(x_i - y_j)^{-1} \right)_{\substack{1 \leq i \leq n-1, i \neq p \\ 1 \leq j \leq m}} & \left( x_i^{\lambda_j + (n - 1) - m - j} \right)_{\substack{1 \leq i \leq n - 1, i \neq p \\ 1 \leq j \leq n - 1 - k, j \neq r}} \\
\left( y_j^{\lambda'_i + m - (n - 1) - i} \right)_{\substack{1 \leq i \leq m - k \\ 1 \leq j \leq m}} & 0
\end{pmatrix}
.
\intertext{As that the matrices in question are identical, the sum of the first two terms simplifies to:}
\text{D} ={} & \sum_{p = 1}^{n - 1} \sum_{q = 1}^{m} \left[ (x_p - y_q)^{-1} - \frac{n - 2}{n - 1} x_p^{-1} y_q (x_p - y_q)^{-1} \right] (-1)^{p + q} \\
& \hspace{15pt} \times \det \begin{pmatrix} \left( x_i^{-1}y_j(x_i - y_j)^{-1} \right)_{\substack{1 \leq i \leq n-1, i \neq p \\ 1 \leq j \leq m, j \neq q}} & \left( x_i^{\lambda_j + (n - 1) - m - j} \right)_{\substack{1 \leq i \leq n - 1, i \neq p \\ 1 \leq j \leq n - 1 - k}} \\
\left( y_j^{\lambda'_i + m - (n - 1) - i} \right)_{\substack{1 \leq i \leq m - k \\ 1 \leq j \leq m, j \neq q}} & 0 
\end{pmatrix} 
\displaybreak[2] \\
& + \sum_{p = 1}^{n - 1} \frac{1}{n - 1} \sum_{r = 1}^{n - 1- k} (-1)^{p + m + r} x_p^{\lambda_r + (n - 1) - m - r} \\
& \hspace{15pt} \times \det \begin{pmatrix} \left( x_i^{-1}y_j(x_i - y_j)^{-1} \right)_{\substack{1 \leq i \leq n-1, i \neq p \\ 1 \leq j \leq m}} & \left( x_i^{\lambda_j + (n - 1) - m - j} \right)_{\substack{1 \leq i \leq n - 1, i \neq p \\ 1 \leq j \leq n - 1 - k, j \neq r}} \\
\left( y_j^{\lambda'_i + m - (n - 1) - i} \right)_{\substack{1 \leq i \leq m - k \\ 1 \leq j \leq m}} & 0 \end{pmatrix}
.
\intertext{Applying the Leibniz formula for determinants results in}
\text{D} ={} & \sum_{p = 1}^{n - 1} \sum_{q = 1}^m \sum_{\substack{\sigma \in S_{m + n - 1 - k}: \\ \forall i \geq n \; \sigma(i) \leq m \\ \sigma(p) = q}} \left[ \varepsilon(\sigma) (-1)^{p + \sigma(p)} \right] (x_p - y_{\sigma(p)})^{-1} (-1)^{p + \sigma(p)} \\
& \hspace{15pt} \times \prod_{\substack{i \leq n - 1: \\ \sigma(i) \leq m \\ i \neq p}} x_i^{-1} y_{\sigma(i)} (x_i - y_{\sigma(i)})^{-1}  \prod_{\substack{i \leq n - 1: \\ \sigma(i) > m}} x_i^{\lambda_{\sigma(i) - m} + (n - 1)- \sigma(i)} \prod_{i \geq n} y_{\sigma(i)}^{\lambda'_{i - n + 1} + m - i} \\
& + \sum_{p = 1}^{n - 1} \sum_{q = 1}^{m}  \sum_{\substack{\sigma \in S_{m + n - 1 - k}: \\ \forall i \geq n \; \sigma(i) \leq m \\ \sigma(p) = q}} \left[ \varepsilon(\sigma) (-1)^{p + \sigma(p)} \right] (-1)^{p + \sigma(p) + 1} \frac{n - 2}{n - 1} x_p^{-1} y_{\sigma(p)} (x_p - y_{\sigma(p)})^{-1} \\
& \hspace{15pt} \times \prod_{\substack{i \leq n - 1: \\ \sigma(i) \leq m \\ i \neq p}} x_i^{-1} y_{\sigma(i)} (x_i - y_{\sigma(i)})^{-1}  \prod_{\substack{i \leq n - 1: \\ \sigma(i) > m}} x_i^{\lambda_{\sigma(i) - m} + (n - 1)- \sigma(i)} \prod_{i \geq n} y_{\sigma(i)}^{\lambda'_{i - n + 1} + m - i}
\displaybreak[2] \\
& + \sum_{p = 1}^{n - 1} \sum_{r = 1}^{n - 1 - k} \sum_{\makebox[48pt]{$\substack{\sigma \in S_{m + n - 1 - k}: \\ \forall i \geq n \; \sigma(i) \leq m \\ \sigma(p) = m + r}$}} \left[ \varepsilon(\sigma) (-1)^{p + m + \sigma(p)} \right] \frac{1}{n - 1} (-1)^{p + m + \sigma(p)} x_p^{\lambda_{\sigma(p) - m} + (n - 1) - \sigma(p)} \\
& \hspace{15pt} \times \prod_{\substack{i \leq n - 1: \\ \sigma(i) \leq m}} x_i^{-1} y_{\sigma(i)} (x_i - y_{\sigma(i)})^{-1}  \prod_{\substack{i \leq n - 1: \\ \sigma(i) > m \\ i \neq p}} x_i^{\lambda_{\sigma(i) - m} + (n - 1) - \sigma(i)} \prod_{i \geq n} y_{\sigma(i)}^{\lambda'_{i - n + 1} + m - i} 
.
\intertext{We reformulate each of the three terms by means of elementary algebraic manipulations:}
\text{D} ={} & \sum_{\substack{\sigma \in S_{m + n - 1 - k}: \\ \forall i \geq n \; \sigma(i) \leq m}} \varepsilon(\sigma)  \left( \sum_{\substack{1 \leq p \leq n - 1: \\ \sigma(p) \leq m}} x_p y_{\sigma(p)} \right) \\
& \hspace{15pt} \times \prod_{\substack{i \leq n - 1: \\ \sigma(i) \leq m}} x_i^{-1} y_{\sigma(i)} (x_i - y_{\sigma(i)})^{-1}  \prod_{\substack{i \leq n - 1: \\ \sigma(i) > m}} x_i^{\lambda_{\sigma(i) - m} + (n - 1)- \sigma(i)} \prod_{i \geq n} y_{\sigma(i)}^{\lambda'_{i - n + 1} + m - i} 
\displaybreak[2] \\
& - \sum_{\substack{\sigma \in S_{m + n - 1 - k}: \\ \forall i \geq n \; \sigma(i) \leq m}} \varepsilon(\sigma) \left( \sum_{\substack{1 \leq p \leq n - 1: \\ \sigma(p) \leq m}} 1 \right) \frac{n - 2}{n - 1} \\
& \hspace{15pt} \times \prod_{\substack{i \leq n - 1: \\ \sigma(i) \leq m}} x_i^{-1} y_{\sigma(i)} (x_i - y_{\sigma(i)})^{-1}  \prod_{\substack{i \leq n - 1: \\ \sigma(i) > m}} x_i^{\lambda_{\sigma(i) - m} + (n - 1)- \sigma(i)} \prod_{i \geq n} y_{\sigma(i)}^{\lambda'_{i - n + 1} + m - i}
\displaybreak[2] \\
& + \sum_{\substack{\sigma \in S_{m + n - 1 - k}: \\ \forall i \geq n \; \sigma(i) \leq m}} \varepsilon(\sigma) \left( \sum_{\substack{1 \leq p \leq n - 1: \\ \sigma(p) > m}} 1 \right)  \frac{1}{n - 1} \\
& \hspace{15pt} \times \prod_{\substack{i \leq n - 1: \\ \sigma(i) \leq m}} x_i^{-1} y_{\sigma(i)} (x_i - y_{\sigma(i)})^{-1}  \prod_{\substack{i \leq n - 1: \\ \sigma(i) > m}} x_i^{\lambda_{\sigma(i) - m} + (n - 1) - \sigma(i)} \prod_{i \geq n} y_{\sigma(i)}^{\lambda'_{i - n + 1} + m - i} 
.
\intertext{These manipulations make it obvious that summing the three terms results in}
\text{D} ={} & \sum_{\substack{\sigma \in S_{m + n - 1 - k}: \\ \forall i \geq n \; \sigma(i) \leq m}} \varepsilon(\sigma)  \left( \sum_{\substack{1 \leq p \leq n - 1: \\ \sigma(p) \leq m}} x_p y_{\sigma(p)} - \frac{k(n - 2)}{n - 1} + \frac{n - 1 - k}{n - 1}\right) \\
& \hspace{15pt} \times \prod_{\substack{i \leq n - 1: \\ \sigma(i) \leq m}} x_i^{-1} y_{\sigma(i)} (x_i - y_{\sigma(i)})^{-1}  \prod_{\substack{i \leq n - 1: \\ \sigma(i) > m}} x_i^{\lambda_{\sigma(i) - m} + (n - 1)- \sigma(i)} \prod_{i \geq n} y_{\sigma(i)}^{\lambda'_{i - n + 1} + m - i} 
.
\intertext{Another application of the Leibniz formula for determinants allows us to reformulate this expression as}
\text{D} ={} & \sum_{\substack{\sigma \in S_{m + n - 1 - k}: \\ \forall i \geq n \; \sigma(i) \leq m}} \varepsilon(\sigma)  \left( \sum_{\substack{1 \leq p \leq n - 1: \\ \sigma(p) \leq m}} x_p y_{\sigma(p)} - (k - 1) - \prod_{\substack{i \leq n - 1: \\ \sigma(i) \leq m}} x_i y_{\sigma(i)}^{-1} \right) \\
& \hspace{15pt} \times \prod_{\substack{i \leq n - 1: \\ \sigma(i) \leq m}} x_i^{-1} y_{\sigma(i)} (x_i - y_{\sigma(i)})^{-1}  \prod_{\substack{i \leq n - 1: \\ \sigma(i) > m}} x_i^{\lambda_{\sigma(i) - m} + (n - 1)- \sigma(i)} \prod_{i \geq n} y_{\sigma(i)}^{\lambda'_{i - n + 1} + m - i} \\
& + \det \begin{pmatrix} \left( (x_i - y_j)^{-1} \right)_{\substack{1 \leq i\leq n - 1 \\ 1 \leq j \leq m}} & \left( x_i^{\lambda_j + (n - 1) - m - j} \right)_{\substack{1 \leq i \leq n - 1 \\ 1 \leq j \leq n - 1 - k}} \\ \left( y_j^{\lambda'_i + m - (n - 1) - i} \right)_{\substack{1 \leq i \leq m - k \\ 1 \leq j \leq m}} & 0 \end{pmatrix}
.
\end{align*}
Going back to the definition of D on page \pageref{2_in_proof_defn_D}, we see that property of restriction holds (in case the $(m, n - 1)$ index of $\lambda$ is $k > 0$) if and only if the above sum over $\sigma$ vanishes. In fact, we will show that each partial sum over all partitions $\sigma$ so that $\sigma(i) = p_i$ if $i \geq n$ and $\sigma(i) = q_{\sigma(i)}$ if $\sigma(i) > m$ for some fixed sequences $(p_n, \dots, p_{n - 1 + m - k})$ and $(q_{m + 1}, \dots, q_{m + n - 1 - k})$ vanishes. This is equivalent to the claim that 
\begin{align*}
\sum_{\tau \in S_k} \varepsilon(\tau) \left( \sum_{1 \leq i \leq k} x_i y_{\tau(i)}^{-1} - (k - 1) - \prod_{1 \leq i \leq k} x_i y_i^{-1} \right) \left( \prod_{1 \leq i \leq k} x_i y_i^{-1} (x_i - y_{\tau(i)}) \right)^{-1} = 0
.
\end{align*}
By reducing the fractions to a common denominator, we infer that this equality holds if and only if the following homogeneous polynomial of degree $k(k - 1) + k = k^2$ vanishes:
\begin{align*}
\text{P}_k \defeq{} & \sum_{\tau \in S_k} \varepsilon(\tau) \left[ \prod_{1 \leq i \leq k} \prod_{\substack{1 \leq j \leq k: \\ j \neq \tau(i)}} (x_i - y_j) \right] \\
& \times \left[ \left( \sum_{1 \leq i \leq k} x_i \prod_{\substack{1 \leq j \leq k:\\ j \neq \tau(i)}} y_j \right) - (k - 1) \left( \prod_{1 \leq j \leq k} y_j \right)  - \left( \prod_{1 \leq i \leq k} x_i \right) \right]
.
\end{align*}
We show that $P_k = 0$ for all $k \geq 2$ by induction on $k$. For the base case $k = 2$, the formula reads
\begin{align*}
\text{P}_2 ={} & \sum_{\tau \in S_2} \varepsilon(\tau) \left[ \prod_{1 \leq i \leq 2} \prod_{\substack{1 \leq j \leq 2: \\ j \neq \tau(i)}} (x_i - y_j) \right] \\
& \times \left[ \left( \sum_{1 \leq i \leq 2} x_i \prod_{\substack{1 \leq j \leq 2:\\ j \neq \tau(i)}} y_j \right) - \left( \prod_{1 \leq j \leq 2} y_j \right)  - \left( \prod_{1 \leq i \leq 2} x_i \right) \right]
\displaybreak[2] \\
={} & (x_1 - y_2)(x_2 - y_1) \left[ x_1 y_2 + x_2 y_1 - y_1 y_2 - x_1 x_2 \right] \\
& - (x_1 - y_1) (x_2 - y_2) \left[ x_1 y_1 + x_2 y_2 - y_1 y_2 - x_1 x_2 \right]
\displaybreak[2] \\
={} & - (x_1 - y_2)(x_2 - y_1)(x_1 - y_1)(x_2 - y_2) + (x_1 - y_1) (x_2 - y_2)(x_1 - y_2)(x_2 - y_1) \\
={} & 0
.
\end{align*}
For the induction step, we first set $x_p = t = y_q$ for some $1 \leq p, q \leq k + 1$ in P$_{k + 1}$. Given that the factor 
$$\prod_{1 \leq i \leq k + 1} \prod_{\substack{1 \leq j \leq k + 1: \\ j \neq \tau(i)}} (x_i - y_j)$$
vanishes unless $\tau(p) = q$, the formula simplifies to
\begin{align*}
{\text{P}_{k + 1}}_{{\big{\vert}}_{x_p = y_q}} ={} & \prod_{\substack{1 \leq i \leq k + 1: \\ i \neq p}} (x_i - t) \prod_{\substack{1 \leq j \leq k + 1: \\ j \neq q}} (t - y_j) \sum_{\substack{\sigma \in S_{k + 1}: \\ \sigma(p) = q}} \varepsilon(\sigma) \left[ \prod_{\substack{1 \leq i \leq k + 1: \\ i \neq p}} \prod_{\substack{1 \leq j \leq k + 1: \\ j \neq \sigma(i) \\ j \neq q}} (x_i - y_j) \right] \\
\times & \Bigg[ \left( t \prod_{\substack{1 \leq j \leq k + 1: \\ j \neq q}} y_j + \sum_{\substack{1 \leq i \leq k + 1: \\ i \neq p}} x_i t \prod_{\substack{1 \leq j \leq k + 1: \\ j \neq \sigma(i) \\ j \neq q}} y_j \right) \\
& \hspace{20pt} - k \left( t \prod_{\substack{1 \leq j \leq k + 1: \\ j \neq q}} y_j \right) - \left( t \prod_{\substack{1 \leq i \leq k + 1: \\ i \neq p}} x_i \right) \Bigg]
\displaybreak[2] \\
={} & t \prod_{\makebox[36pt]{$\substack{1 \leq i \leq k + 1: \\ i \neq p}$}} (x_i - t) \prod_{\substack{1 \leq j \leq k + 1: \\ j \neq q}} (t - y_j) \sum_{\substack{\sigma \in S_{k + 1}: \\ \sigma(p) = q}} \varepsilon(\sigma) \left[ \prod_{\substack{1 \leq i \leq k + 1: \\ i \neq p}} \prod_{\substack{1 \leq j \leq k + 1: \\ j \neq \sigma(i) \\ j \neq q}} (x_i - y_j) \right] \\
\times & \left[ \left(\sum_{\substack{1 \leq i \leq k + 1: \\ i \neq p}} x_i \prod_{\substack{1 \leq j \leq k + 1: \\ j \neq \sigma(i) \\ j \neq q}} y_j \right) -  (k - 1) \left( \prod_{\substack{1 \leq j \leq k + 1: \\ j \neq q}} y_j \right) - \left( \prod_{\substack{1 \leq i \leq k + 1: \\ i \neq p}} x_i \right) \right]
\end{align*}
Up to a sign we may view the sum over all permutations $\sigma \in S_{k + 1}$ with $\sigma(p) = q$ as a sum over all permutations in $S_k$. By the induction hypothesis, this sum vanishes. We conclude that for all $1 \leq p,q \leq k + 1$, setting $x_p = y_q$ in P$_{k + 1}$ results in $0$. Therefore, the fact that P$_{k + 1}$ is a polynomial of degree $(k + 1)^2$, implies that 
$$\text{P}_{k + 1} = c \times \prod_{1 \leq p,q \leq k + 1} (x_p - y_q)$$
for some constant $c$. Let us compute this constant by specializing the polynomial P$_{k + 1}$ to $x_1 = \dots = x_{k + 1} = 1$ and $y_1 = \dots = y_{k + 1} = 0$:
\begin{align*}
c ={} & {\text{P}_{k + 1}}_{{\big{\vert}}_{\substack{x_1 = \dots = x_{k + 1} = 1 \\ y_1 = \dots = y_{k + 1} = 0}}} = \sum_{\tau \in S_{k + 1}} - \varepsilon(\tau) = 0
.
\end{align*}
This proves that the the property of restriction holds for $LS^{\ast}_\lambda(-\calX; \calY)$ indexed by partitions $\lambda$ whose $(m, n - 1)$-index is $k > 0$.

Second suppose that the $(m,n-1)$-index of $\lambda$ is $k - 1 \geq 0$, which is equivalent to $\lambda'_{m - k + 1} = n - k$ and $\lambda_{n - k} > m - k$. In this case, $\lambda'_{m - (k - 1)} + m - (n - 1) - (m - (k - 1)) = 0$ and $\lambda_{n - k} +  n - m - (n - k) > 0$. Therefore, multiplying the $n$-th row of the matrix in \eqref{2_in_proof_restriction_k_positive} by $(-1)$ and then moving it to the last row results in
\begin{align}
\begin{split} \notag
LS_\lambda^{\ast}(-\calX; \calY)_{{\big{\vert}}_{x_n = 0}} ={} & \varepsilon_{m,n}(\lambda) (-1)^{1 + m - k}  \frac{\Delta\left(\calY; \calX'\right)}{\Delta \left( \calX' \right) \Delta(\calY)} \\
& \times \det \begin{pmatrix} \left( x_i^{-1}y_j(x_i - y_j)^{-1} \right)_{\substack{1 \leq i \leq n-1 \\ 1 \leq j \leq m}} \!\! & \left( x_i^{\lambda_j + (n - 1) - m - j} \right)_{\substack{1 \leq i \leq n - 1 \\ 1 \leq j \leq n - k}} \\
\left( y_j^{\lambda'_i + m - (n - 1) - i} \right)_{\substack{1 \leq i \leq m - (k - 1) \\ 1 \leq j \leq m}} \!\! & 0 
\end{pmatrix} 
\!\!.
\end{split}
\end{align}
One easily checks that $\varepsilon_{m,n}(\lambda) (-1)^{1 + m - k} = \varepsilon_{m, n - 1}(\lambda)$. Moreover, the determinant is of the required form owing to Lemma~\ref{2_lem_alternative_matrix_for_det_formula_LS}, which we state and prove right afterwards. Indeed, the matrix is equal to $A(1)$ (defined in Lemma~\ref{2_lem_alternative_matrix_for_det_formula_LS}) and $$m - (k - 1) - \lambda_{n - k + 1} \geq m - (k - 1) - (m - k) = 1$$
since the $(m,n)$-index of $\lambda$ is $k$. This completes the proof that $LS^{\ast}_\lambda(-\calX; \calY)$ possesses the property of restriction condition to Lemma~\ref{2_lem_alternative_matrix_for_det_formula_LS}.
\end{proof}

\begin{lem} \label{2_lem_alternative_matrix_for_det_formula_LS} Let $\calX$ and $\calY$ consist of $n$ and $m$ non-zero variables, respectively, such that $\Delta(\calX; \calY) \neq 0$. If the $(m,n)$-index $k$ of a partition $\lambda$ is non-negative, then the determinant in \eqref{2_thm_det_formula_for_Littlewood-Schur_eq} is equal to the determinant of the following matrix $A(r)$ for all integers $r$ with $0 \leq r \leq \max\{m - k - \lambda_{n - k + 1}, n - k - \lambda'_{m - k + 1} \}$:
\begin{align*}
A(r) = \begin{pmatrix} \left( x_i^{-r} y_j^r (x_i - y_j)^{-1} \right)_{\substack{1 \leq i \leq n \\ 1 \leq j \leq m}} & \left( x_i^{\lambda_j + n - m - j} \right)_{\substack{1 \leq i \leq n \\ 1 \leq j \leq n - k}} \\ \left( y_j^{\lambda'_i + m - n - i} \right)_{\substack{1 \leq i \leq m - k \\ 1 \leq j \leq m}} & 0\end{pmatrix}
.
\end{align*}
\end{lem}

\begin{proof} First suppose that $r \leq m - k - \lambda_{n - k + 1}$. For each index $1 \leq i \leq n$, perform the following row operation, which will leave the determinant invariant, and call the resulting matrix $B(r)$:
$$A(r)_i \mapsto A(r)_i + \sum_{l = m - k - r + 1}^{m - k} x_i^{l - m + k - 1} A(r)_{l + n}
.
$$
By construction, the matrices $A(r)$ and $B(r)$ only differ in the top-left block. In fact, for $1 \leq i \leq n$ and $1 \leq j \leq m$, 
\begin{align*}
B(r)_{ij} ={} & \frac{x_i^{-r} y_j^r}{x_i - y_j} + \sum_{l = m - k - r + 1}^{m - k} x_i^{l - m + k - 1} y_j^{\lambda'_{l} + m - n - l}
.
\intertext{Given that $\lambda_{n - k + 1} \leq m - k - r < l \leq m - k$ and $\lambda_{n - k} \geq m - k$ by the definition of index, the exponent of $y_j$ simplifies to $\lambda'_l + m - n - l = n - k + m - n - l = m - k - l$:}
B(r)_{ij} ={} & \frac{x_i^{-r} y_j^r}{x_i - y_j} + \sum_{l = m - k - r + 1}^{m - k} x_i^{l - m + k - 1} y_j^{m - k - l} \displaybreak[2 ]\\
={} & \frac{x_i^{-r} y_j^r}{x_i - y_j} + x_i^{-r} \sum_{p = 0}^{r - 1} x_i^p y_j^{r - 1 - p} \displaybreak[2] \\
={} & \frac{x_i^{-r} y_j^r}{x_i - y_j} + \frac{x_i^{-r} \left( x_i^r - y_j^r \right)}{x_i - y_j} = \frac{1}{x_i - y_j}
,
\end{align*}
completing the proof in case $r \leq m - k - \lambda_{n - k + 1}$. If $r \leq n - k - \lambda'_{m - k + 1}$, the result follows by symmetry.
\end{proof}

\begin{proof}[Proof of Theorem~\ref{2_thm_det_formula_for_Littlewood-Schur}] Owing to the five characterizing properties of Littlewood-Schur functions listen in Theorem~\ref{2_thm_5_properties_of_LS}, these five Lemmas allow us to conclude that the determinantal formula on the right-hand side in \eqref{2_thm_det_formula_for_Littlewood-Schur_eq} is indeed equal to the Littlewood-Schur function on the left-hand side, under the assumption that the variables in $\calX \cup \calY$ are pairwise distinct. 
\end{proof}
Moreover, the fact that both sides of the equality are polynomials in $\calX \cup \calY$ even allows us to drop this condition. In practice, however, the determinantal formula might not be the best definition to use for sets of variables that contain repetitions.

\section{Example: A determinantal proof of the Murnaghan-Nakayama rule for Littlewood-Schur functions} \label{2_sec_det_proof_of_MN_for_LS}
As the title indicates, the Murnaghan-Nakayama rule for Littlewood-Schur functions is \emph{not} the primary focus of this section. The goal of this section is to present a proof that only requires elementary linear algebra -- and the determinantal formula for Littlewood-Schur functions. In particular, our proof of the Murnaghan-Nakayama rule does not rely on Littlewood-Richardson coefficients. The point of this example is to illustrate that if there exists a determinantal proof for some statement about Schur functions, then there should also exist an analogous determinantal proof for the generalization of the statement to Littlewood-Schur functions. This idea yields elementary  -- if not necessarily very elegant -- proofs for some properties of Littlewood-Schur functions. As always, using more elaborate machinery leads to shorter and more elegant proofs: we will present the theory of specializations in Chapter~\ref{4_cha_mixed_ratios}, which will allow us to derive the Murnaghan-Nakayama rule for Littlewood-Schur functions from the classic Murnaghan-Nakayama rule for Schur functions in one line.

Given that the theory of specializations makes it obvious, the generalization of the Murnaghan-Nakayama rule to Littlewood-Schur functions is certainly not new (although the elementary proof given in this section seems to be). However, the triviality of the generalization also makes it difficult to find in the literature. The Murnaghan-Nakayama rule for skew-Schur functions can be found in \cite[p.~523]{MNrule}.

\subsection{Ribbons}
\begin{defn} [ribbon] Let $\lambda$ and $\mu$ be partitions so that $\mu$ is a subset of $\lambda$. The skew diagram $\lambda \setminus \mu$ is the set-theoretic difference between the Ferrers diagrams of $\lambda$ and $\mu$, \textit{i.e.}\ it is the set of boxes that are contained in $\lambda$ but not in $\mu$. A ribbon is a skew diagram that is edgewise connected and contains no $2 \times 2$ subset of boxes. The size of a ribbon is the number of its boxes. We sometimes call a ribbon of size $r$ an $r$-ribbon. The height (\label{symbol_height_of_a ribbon} $\height$) of a ribbon is one less than the number of its rows. 
\end{defn} 
Let us illustrate this definition by some examples: only the left-most diagram is a ribbon. Indeed, the skew-diagram in the middle is not edgewise connected and the skew-diagram on the right-hand side contains a $2 \times 2$ subset of boxes.
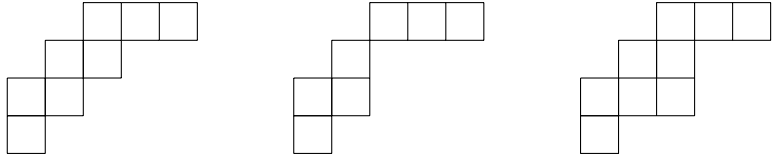
\begin{center}
\begin{tikzpicture} 
\draw[step=0.5cm, thin] (0, 0) grid (0.5, 0.5);
\draw[step=0.5cm, thin] (0, 0.5) grid (0.5, 1);
\draw[step=0.5cm, thin] (0.5, 0.5) grid (1, 1);
\draw[step=0.5cm, thin] (0.5, 1) grid (1, 1.5);
\draw[step=0.5cm, thin] (0.999, 0.999) grid (1.5, 1.5);
\draw[step=0.5cm, thin] (0.999, 1.5) grid (1.5, 2);
\draw[step=0.5cm, thin] (1.5, 1.499) grid (2, 2);
\draw[step=0.5cm, thin] (2, 1.499) grid (2.5, 2);
\end{tikzpicture}
\hspace{1cm}
\begin{tikzpicture} 
\draw[step=0.5cm, thin] (0, 0) grid (0.5, 0.5);
\draw[step=0.5cm, thin] (0, 0.5) grid (0.5, 1);
\draw[step=0.5cm, thin] (0.5, 0.5) grid (1, 1);
\draw[step=0.5cm, thin] (0.5, 1) grid (1, 1.5);
\draw[step=0.5cm, thin] (0.999, 1.499) grid (1.5, 2);
\draw[step=0.5cm, thin] (1.5, 1.499) grid (2, 2);
\draw[step=0.5cm, thin] (2, 1.499) grid (2.5, 2);
\end{tikzpicture}
\hspace{1cm}
\begin{tikzpicture} 
\draw[step=0.5cm, thin] (0, 0) grid (0.5, 0.5);
\draw[step=0.5cm, thin] (0, 0.5) grid (0.5, 1);
\draw[step=0.5cm, thin] (0.5, 0.5) grid (1, 1);
\draw[step=0.5cm, thin] (1, 0.5) grid (1.5, 1);
\draw[step=0.5cm, thin] (0.999, 1) grid (1.5, 1.5);
\draw[step=0.5cm, thin] (0.999, 1.5) grid (1.5, 2);
\draw[step=0.5cm, thin] (1.5, 1.499) grid (2, 2);
\draw[step=0.5cm, thin] (2, 1.499) grid (2.5, 2);
\draw[step=0.5cm, thin] (0.5, 1) grid (1, 1.5);
\end{tikzpicture}
\end{center}
This visual definition makes it easy to see that $\lambda \setminus \mu$ is an $r$-ribbon if and only if $\lambda' \setminus \mu'$ is. In this case,
\begin{align} \label{2_eq_height_of_transpose}
\height\left( \lambda' \setminus \mu' \right) = r - 1 - \height(\lambda \setminus \mu)
.
\end{align}
What we call ribbon is also known as skew or rim hook \cite[p.~180]{Sagan}, and as border strip \cite[p.~5]{mac}.

Given that partitions can be viewed as diagrams or sequences, it is natural to ask what characterizes ribbons from the point of view of sequences. To answer this question, we first introduce the relevant notation for sequences: if two sequences, say $\alpha$ and $\beta$, are equal up to reordering, we write $\alpha \sorteq \beta$. For $n \geq 1$, we define the partition $
\rho_n = (n - 1, \dots, 1, 0)$. The following lemma is taken from \cite[p.~48]{mac}.

\begin{lem} [\cite{mac}] \label{2_lem_sequence_defn_for_ribbons} Let $n$, $r$ be positive integers and $\mu$ a partition of length at most $n$. If we define
\begin{align} \label{2_eq_defn_mu[q]}
\mu[q] = (\mu_1, \dots, \mu_{q-1}, \mu_q + r, \mu_{q + 1}, \dots, \mu_n)
\end{align} 
for all $1 \leq q \leq n$, then the following two sets are equal:
\begin{enumerate}
\item $\{\lambda: \lambda \setminus \mu \text{ is an $r$-ribbon and } l(\lambda) \leq n\}$
\item $\{\lambda: \lambda + \rho_n \sorteq \mu[q] + \rho_n \text{ for some } 1 \leq q \leq n\}$
\end{enumerate}
In addition, the sign of the sorting permutation corresponding to $\lambda + \rho_n \sorteq \mu[q] + \rho_n$ is equal to $(-1)^{\height(\lambda \setminus \mu)}$.
\end{lem}

\begin{proof} Fix an integer $q$ that lies between $1$ and $n$. On the one hand, let $\lambda^{(1)}(q)$ denote the unique partition with the following properties if it exists: $\lambda^{(1)}(q) \setminus \mu$ is an $r$-ribbon, $l\left(\lambda^{(1)}(q) \right) \leq n$, $\lambda^{(1)}(q)_q \neq \mu_q$ and $\left(\lambda^{(1)}(q)_{q + 1}, \dots, \lambda^{(1)}(q)_n \right) = (\mu_{q + 1}, \dots, \mu_n)$. Assuming the existence of $\lambda^{(1)}(q)$, its Ferrers diagram is of the following generic form:
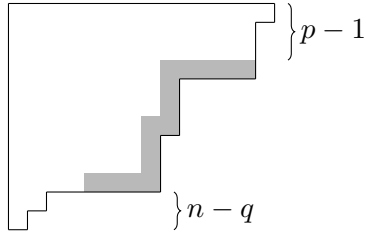
\begin{center}
\begin{tikzpicture}
\fill[black!27.5!white] (0, 0) rectangle (1, 0.25);
\fill[black!27.5!white] (0.75, 0) rectangle (1, 1);
\fill[black!27.5!white] (1, 0.75) rectangle (1.25, 1);
\fill[black!27.5!white] (1, 1) rectangle (1.25, 1.75);
\fill[black!27.5!white] (1.25, 1.5) rectangle (2.25, 1.75);
\draw (-1, -0.5) -- (-0.75, -0.5);
\draw (-0.75, -0.5) -- (-0.75, -0.25);
\draw (-0.75, -0.25) -- (-0.5, -0.25);
\draw (-0.5, -0.25) -- (-0.5, 0);
\draw (-0.5, 0) -- (1, 0);
\draw (1, 0) -- (1, 0.75);
\draw (1, 0.75) -- (1.25, 0.75);
\draw (1.25, 0.75) -- (1.25, 1.5);
\draw (1.25, 1.5) -- (2.25, 1.5);
\draw (2.25, 1.5) -- (2.25, 2.25);
\draw (2.25, 2.25) -- (2.5, 2.25);
\draw (2.5, 2.25) -- (2.5, 2.5);
\draw (-1, 2.5) -- (2.5, 2.5);
\draw (-1, 2.5) -- (-1, -0.5);
\draw[decoration={brace, raise=5pt, mirror},decorate] (1, -0.5) -- node[right=6pt] {$n - q$} (1, 0);
\draw[decoration={brace, raise=5pt, mirror},decorate] (2.5, 1.75) -- node[right=6pt] {$p - 1$} (2.5, 2.5);
\end{tikzpicture}
\end{center}
In this sketch, the diagram of $\mu$ is in white and the $r$-ribbon $\lambda^{(1)}(q) \setminus \mu$ is colored in gray. Still assuming that $\lambda^{(1)}(q)$ exists, we see that there is an integer $1 \leq p \leq q$ so that
\begin{align} \label{2_in_proof_sequence_defn_for_ribbons}
\lambda^{(1)}(q) = \left(\mu_1, \dots, \mu_{p - 1}, x, \mu_p + 1, \dots, \mu_{q-1} + 1, \mu_{q + 1}, \dots, \mu_n \right)
\end{align}
for some integer $x$. The fact that $\left| \lambda^{(1)}(q) \right| - |\mu| = r$ allows us to compute $x$. Indeed, their difference in size is also equal to
\begin{align*}
\left| \lambda^{(1)}(q) \right| - |\mu| = (x - \mu_q) + (\mu_p + 1 - \mu_p) + \dots + (\mu_{q-1} + 1 - \mu_q) = x - \mu_q + q - p
,
\end{align*}
which entails that $x = \mu_q - q + p + r$. Moreover, this shows that the existence of $\lambda^{(1)}(q)$ is equivalent to the condition that
\begin{align} \label{2_in_proof_sequence_defn_for_ribbons_cond_1}
\mu_{p - 1} \geq \mu_q - q + p + r \geq \mu_p + 1
\end{align}
where we use the convention that $\mu_0$ is infinitely large, while $\mu_{n + 1} = 0$.

On the other hand, let $\lambda^{(2)}(q)$ denote the unique partition with the property that $\lambda^{(2)}(q) + \rho_n \sorteq \mu[q] + \rho_n$ if it exists. By definition, $\lambda^{(2)}(q)$ exists if and only if the elements of $\mu[q] + \rho_n$ are pairwise distinct, which is equivalent to the existence of $1 \leq p \leq q$ such that
\begin{align} \label{2_in_proof_sequence_defn_for_ribbons_cond_2}
\mu_{p - 1} + n - (p - 1) > \mu_q + r + n - q > \mu_p + n - p
.
\end{align}
One easily checks that the conditions given in \eqref{2_in_proof_sequence_defn_for_ribbons_cond_1} and \eqref{2_in_proof_sequence_defn_for_ribbons_cond_2} are equivalent. Thus, the result is consequence of the following equality. If the conditions that guarantee existence are satisfied for some integer $p$, then
\begin{align*}
\lambda^{(2)}(q) ={} & ( \mu_1 + n - 1, \dots, \mu_{p - 1} + n - (p - 1), \mu_q + r + n - q, \mu_p + n - p, \dots, \\ 
& \hspace{15pt} \mu_{q - 1} + n - (q - 1), \mu_{q + 1} + n - (q + 1), \dots, \mu_n ) - \rho_n \displaybreak[2] \\
={} & (\mu_1, \dots, \mu_{p - 1}, \mu_q + r + p - q, \mu_p + 1, \dots, \mu_{q - 1} + 1, \mu_{q + 1}, \dots, \mu_n) \displaybreak[2] \\
={} & \lambda^{(1)}(q)
\end{align*}
by \eqref{2_in_proof_sequence_defn_for_ribbons}.

In order to justify the statement on the sign, we still assume that the conditions for existence are satisfied for some integer $p$. We observe that sorting $\mu[q] + \rho_n$ into a strictly decreasing sequence requires moving the element $\mu_q + r - n - q$ (at position $q$) $q - p$ position to the left such that it is relocated at position $p$. Hence, the sign of the sorting is $(-1)^{q - p}$. Moreover, the explicit sequence given in \eqref{2_in_proof_sequence_defn_for_ribbons} shows that the height of $\lambda^{(1)}(q) \setminus \mu$ is $q - p$, as required.
\end{proof}

\subsection{The classic determinantal proof for Schur functions} \label{2_sec_det_proof_of_MN_for_Schur}
Due to Lemma~\ref{2_lem_sequence_defn_for_ribbons}, it is an easy exercise to derive the classic Murnaghan-Nakayama rule from the determinantal definition for Schur functions (\textit{i.e.}\ Definition~\ref{1_defn_determinantal_schur_function}), which can also be found in Example~11 of Chapter~I.3 in \cite{mac}.

\begin{thm} [Murnaghan-Nakayama rule, \cite{Murnaghan, Nakayama}] \label{2_thm_MN_rule} Let $\mu$ be a partition. For any strictly positive integer $r$,
\begin{align} \label{2_thm_MN_rule_eq}
\power_r \schur_\mu = \sum_{\substack{\lambda: \\ \lambda \setminus \mu \text{ is an $r$-ribbon}}} (-1)^{\height(\lambda \setminus \mu)} \schur_\lambda.
\end{align}
\end{thm}

\begin{proof} The statement follows if the equation in \eqref{2_thm_MN_rule_eq} holds for all sets of variables $\calX = (x_1, \dots, x_n)$ where $n$ ranges over the non-negative integers. Whenever $n < l(\mu)$, both sides of the equation vanish. In case $n \geq l(\mu)$, we consider
\begin{align*}
\Delta(\calX) \schur_\mu (\calX) \power_r (\calX)
={} & \det \left( x_i^{\mu_j + n - j} \right)_{1 \leq i,j \leq n} \left( \sum_{p = 1}^n x_p^r \right)
.
\intertext{According to the Leibniz formula for determinants,}
\Delta(\calX) \schur_\mu (\calX) \power_r (\calX)
={} & \left( \sum_{\sigma \in S_n} \varepsilon(\sigma) \prod_{i = 1}^n x_i^{\mu_{\sigma(i)} + n - \sigma(i)} \right) \left( \sum_{p = 1}^n x_p^r \right) \displaybreak[2]
\\
= {}& \sum_{p = 1}^n \sum_{\sigma \in S_n} \varepsilon(\sigma) x_p^{\mu_{\sigma(p)} + n - \sigma(p) + r} \prod_{\substack{i = 1 \\ i \neq p}}^n x_i^{\mu_{\sigma(i)} + n - \sigma(i)} \displaybreak[2] \\
={} & \sum_{p = 1}^n \sum_{q = 1}^n \sum_{\substack{\sigma \in S_n: \\ \sigma(p) = q}} \varepsilon(\sigma) x_p^{\mu_{q} + n - q + r} \prod_{\substack{i = 1 \\ i \neq p}}^n x_i^{\mu_{\sigma(i)} + n - \sigma(i)} 
.
\intertext{Recalling the definition of $\mu[q]$ given in \eqref{2_eq_defn_mu[q]}, this can be reformulated as}
\Delta(\calX) \schur_\mu (\calX) \power_r (\calX)
={} & \sum_{q = 1}^n \sum_{p = 1}^n \sum_{\substack{\sigma \in S_n \\ \sigma(p) = q}} \varepsilon(\sigma) \prod_{i = 1}^n x_i^{(\mu[q] + \rho_n)_{\sigma(i)}} \displaybreak[2] \\
={} & \sum_{q = 1}^n \sum_{\sigma \in S_n} \varepsilon(\sigma) \prod_{i = 1}^n x_i^{(\mu[q] + \rho_n)_{\sigma(i)}}
.
\intertext{Another application of the Leibniz formula for determinants results in}
\Delta(\calX) \schur_\mu (\calX) \power_r (\calX)
={} & \sum_{q = 1}^n \det \left( x_i^{(\mu[q] + \rho_n)_j} \right)_{1 \leq i,j \leq n}
.
\intertext{If two elements of $\mu[q] + \rho_n$ are identical, the determinant vanishes; otherwise, there exists a partition $\lambda$ such that $\lambda + \rho_n \sorteq \mu[q] + \rho_n$. Hence, Lemma~\ref{2_lem_sequence_defn_for_ribbons} allows us to view the sum over $q$ as a \emph{signed} sum over partitions $\lambda$ with the properties that $\lambda \setminus \mu$ is an $r$-ribbon and $l(\lambda) \leq n$. As determinants are antisymmetric, the sign corresponding to $\lambda$ with $\lambda + \rho_n \sorteq \mu[q] + \rho_n$ is given by the sign of the sorting permutation, which is equal to $(-1)^{\height(\lambda \setminus \mu)}$. We conclude that}
\Delta(\calX) \schur_\mu (\calX) \power_r (\calX)
={} & \sum_{\substack{\lambda: \\ \lambda \setminus \mu \text{ is an $r$-ribbon} \\ l(\lambda) \leq n}} (-1)^{\height(\lambda \setminus \mu)} \det \left( x_i^{\lambda_j + n - j} \right)_{1 \leq i,j \leq n}
.
\end{align*}
Taking the factor $\Delta(\calX)$ to the other side of the equation completes the proof, given that by definition $\schur_\lambda(\calX) = 0$ whenever $l(\lambda) > n$.
\end{proof}

\subsection{A determinantal proof for Littlewood-Schur functions}
Since Littlewood-Schur functions do not have the property that they vanish whenever their indexing partition exceeds a given length (which depends on the number of variables), we need the following auxiliary lemma in order to extend the determinantal proof for Schur functions to Littlewood-Schur functions.

\begin{lem} \label{2_lem_ribbon_change_index} Let $m$, $n$ and $r$ be strictly positive integers. Let $\mu$ be a partition with the property that $(m, n) \not\in \mu$. If we define
$$\alpha(l) = (\mu_1 - m, \dots, \mu_{n - 1} - m, l) \text{ and } \beta(l) = (\mu'_1 - n, \dots, \mu'_{m - 1} - n, r - 1 - l)$$
for all $0 \leq l \leq r - 1$, then there is a 1-to-1 correspondence between the following two sets:
\begin{enumerate}
\item $\{\lambda: \lambda \setminus \mu \text{ is an $r$-ribbon and } (m, n) \in \lambda \}$
\item $\{(\delta, \gamma): \delta, \gamma \text{ are partitions, } \delta + \rho_n \sorteq \alpha(l) + \rho_n \text{ and } \gamma + \rho_m \sorteq \beta(l) + \rho_m \text{ for} \text{ some } \\ 0 \leq l \leq r - 1\}$
\end{enumerate}
The correspondence is given by $$(\delta, \gamma) \mapsto \left( \left\langle m^n \right\rangle + \delta \right) \cup \gamma'.$$
In addition, if $\lambda$ is the partition associated to some $l$, the product of the signs of the permutations that sort $\alpha(l) + \rho_n$ and $\beta(l) + \rho_m$ into strictly decreasing sequences is equal to $(-1)^{r - 1 - l} (-1)^{\height(\lambda \setminus \mu)}$. 
\end{lem}

\begin{proof} Both sets are empty in case $m,n > 1$ and $(m - 1, n - 1) \not\in \mu$. Indeed, for any partition $\lambda$ with $(m,n) \in \lambda$ the difference $\lambda \setminus \mu$ contains the boxes with coordinates $(m - 1, n - 1)$ and $(m,n)$, and thus a $2 \times 2$ block. To see that the second set is empty, consider
$$\alpha(l)_{n - 1} + \left( \rho_n \right)_{n - 1} = \mu_{n - 1} - m + 1 \leq m - 2 - m + 1 < 0,$$
which implies that there cannot exist a partition $\delta$ with the required property.

We thus assume that $m = 1$, $n = 1$ or $(m - 1, n - 1) \in \mu$. In order to fix our notation, we draw a sketch of an annotated Ferrers diagram of a partition $\lambda$ that belongs to the first set: 
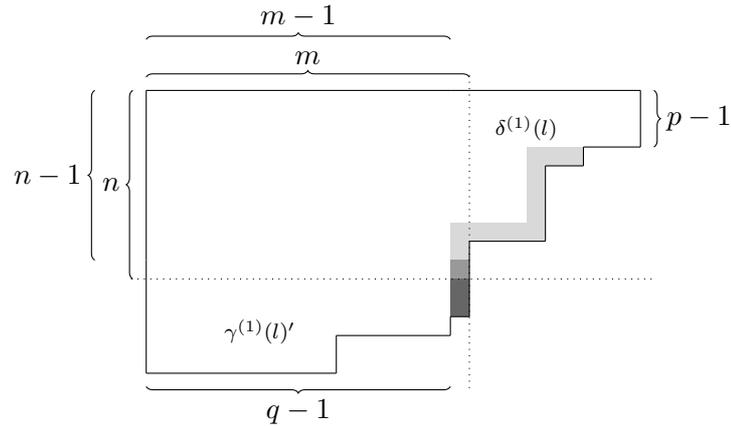
\label{2_page_visual_defn_of_l}
\begin{center}
\begin{tikzpicture}
\fill[black!15!white] (5.75,-0.75) rectangle (5,-1);
\fill[black!15!white] (5.25,-1) rectangle (5,-2);
\fill[black!15!white] (5.25,-1.75) rectangle (4,-2);
\fill[black!15!white] (4.25,-2.25) rectangle (4,-2);
\fill[black!40!white] (4.25,-2.5) rectangle (4,-2.25);
\fill[black!60!white] (4.25,-3) rectangle (4,-2.5);
\draw (0,0) -- (4,0);
\draw (0,0) -- (0,-2.25);
\draw (4,0) -- (6.5,0);
\draw (6.5,0) -- (6.5,-0.75);
\draw (6.5,-0.75) -- (5.75,-0.75);
\draw (5.75,-0.75) -- (5.75,-1);
\draw (5.75,-1) -- (5.25, -1);
\draw (5.25,-1) -- (5.25,-2);
\draw (5.25,-2) -- (4.25, -2);
\draw (4.25,-2) -- (4.25,-2.5);
\draw (4.25,-2.5) -- (4.25,-3);
\draw (4.25,-3) -- (4,-3);
\draw (4,-3) -- (4,-3.25);
\draw (4,-3.25) -- (2.5,-3.25);
\draw (2.5,-3.25) -- (2.5,-3.75);
\draw (2.5,-3.75) -- (0,-3.75);
\draw (0,-3.75) -- (0,-2.25);
\draw[dotted] (4.25,0.2) -- (4.25,-3.95);
\draw[dotted] (-0.2,-2.5) -- (6.7,-2.5);
\draw[decoration={brace, raise=5pt},decorate] (0,0) -- node[above=6pt] {$m$} (4.25, 0);
\draw[decoration={brace, raise=5pt},decorate] (0,0.5) -- node[above=6pt] {$m - 1$} (4, 0.5);
\draw[decoration={brace, raise=5pt, mirror},decorate] (0,0) -- node[left=6pt] {$n$} (0, -2.5);
\draw[decoration={brace, raise=5pt, mirror},decorate] (-0.5,0) -- node[left=6pt] {$n - 1$} (-0.5, -2.25);
\draw[decoration={brace, raise=5pt, mirror},decorate] (0,-3.75) -- node[below=6pt] {$q - 1$} (4, -3.75);
\draw[decoration={brace, raise=5pt},decorate] (6.5,0) -- node[right=6pt] {$p - 1$} (6.5, -0.75);
\node at (5,-0.5) {$\scriptstyle{\delta^{(1)}(l)}$};
\node at (1.5,-3.2) {$\scriptstyle{\gamma^{(1)}(l)'}$};
\end{tikzpicture}
\end{center}
The Ferrers diagram of $\mu$ is colored in white and the ribbon $\lambda \setminus \mu$ in (different shades of) gray. We see that the box with coordinates $(m,n)$ is contained in $\lambda$ but not in $\mu$, as required. The quantity $p - 1$ (resp.\ $q - 1$) counts the number of rows above (resp.\ columns to the left of) the ribbon $\lambda \setminus \mu$. We define $l$ as the number of boxes that are colored in light gray; more formally, $l$ is the number of boxes in the $r$-ribbon $\lambda \setminus \mu$ that lie to the right and above the box $(m,n)$ without counting the box $(m,n)$ itself. Thus, there are $r - 1 - l$ dark gray boxes. We say that the partition $\lambda$ is associated to the parameter $l$. Notice that there is at most one partition $\lambda$ associated to each parameter $0 \leq l \leq r - 1$. Finally, the boxes to the right of the vertical dotted line and below the horizontal dotted line define two partitions, which we denote $\delta^{(1)}(l)$ and $\gamma^{(1)}(l)'$, respectively. By definition, $\lambda = \left( \left\langle m^n \right\rangle + \delta^{(1)}(l) \right) \cup \gamma^{(1)}(l)'$. We conclude that it is sufficient to show that the set $$\left\{ \left( \delta^{(1)}(l), \gamma^{(1)}(l) \right): 0 \leq l \leq r - 1 \text{ and there exists a partition $\lambda$ associated to $l$}\right\}$$ is equal to the second set stated in the lemma. 

Fix an integer $0 \leq l \leq r - 1$. On the one hand, if we assume that there exists a partition $\lambda$ associated to $l$, then 
\begin{align} \label{2_in_proof_lem_ribbon_change_index_delta_1}
\delta^{(1)}(l) ={} & (\mu_1 - m, \dots \mu_{p - 1} - m, x_\delta, \mu_p - m + 1, \dots, \mu_{n - 1} - m + 1)
\intertext{and} \label{2_in_proof_lem_ribbon_change_index_gamma_1}
\gamma^{(1)}(l) ={} & (\mu'_1 - n, \dots, \mu'_{q - 1} - n, x_\gamma, \mu'_q - n + 1, \dots, \mu'_{m - 1} - n + 1)
\end{align}
for some integers $x_\delta$, $x_\gamma$. The above diagram shows that
\begin{align*}
\left| \delta^{(1)}(l)\right| ={} & \mu_1 + \dots + \mu_{n - 1} + (m - 1) + (l + 1) - mn
\intertext{and}
\left| \gamma^{(1)}(l)\right| ={} & \mu'_1 + \dots + \mu'_{m - 1} + (n - 1) + (r - l) - mn
,
\end{align*}
which entails that $x_\delta = l - n + p$ and $x_\gamma = r - 1 - l - m + q$. In addition, we infer that the existence of a partition $\lambda$ associated to $l$ is equivalent to the existence of a pair of integers $1 \leq p \leq n$ and $1 \leq q \leq m$ so that
\begin{align} \label{2_in_proof_lem_ribbon_change_index_cond_1}
\mu_{p - 1} - m \geq l - n + p \geq \mu_p - m + 1 \text{ and } \mu'_{q - 1} - n \geq r - 1 - l - m + q \geq \mu'_q - n + 1
\end{align}
where we use the convention that $\mu_0$ and $\mu'_0$ are infinitely large.

On the other hand, there exist partitions $\delta^{(2)}(l)$ and $\gamma^{(2)}(l)$ with the property that $\delta^{(2)}(l) + \rho_n \sorteq \alpha(l) + \rho_n$ and $\gamma^{(2)}(l) + \rho_m \sorteq \beta(l) + \rho_m $ if and only if the elements of both $\alpha(l) + \rho_n$ and $\beta(l) + \rho_m $ are pairwise distinct. By definition, this is equivalent to the following two inequalities:
\begin{align} 
\begin{split} \label{2_in_proof_lem_ribbon_change_index_cond_2}
\mu_{p - 1} - m + n - p + 1 >{} & l > \mu_p - m + n - p 
\\
\mu'_{q - 1} - n + m - q + 1 > r - {} & 1 - l > \mu'_q - n + m - q
\end{split}
\end{align}
for some $1 \leq p \leq n$ and $1 \leq q \leq m$. We observe that the conditions in \eqref{2_in_proof_lem_ribbon_change_index_cond_1} and \eqref{2_in_proof_lem_ribbon_change_index_cond_2} are equivalent. Furthermore, if the conditions are satisfied for some parameter $p$, then
\begin{align*}
\delta^{(2)}(l) ={} & (\mu_1 - m + n - 1, \dots, \mu_{p - 1} - m - n - p + 1, l, \\
& \hspace{40pt} \mu_p - m + n - p , \dots, \mu_{n - 1} - m + 1) - \rho_n \displaybreak[2] \\
={} & (\mu_1 - m, \dots, \mu_{p - 1} - m, l - n + p, \mu_p - m + 1, \dots, \mu_{n - 1} - m + 1) = \delta^{(1)}(l)
\end{align*}
by \eqref{2_in_proof_lem_ribbon_change_index_delta_1}. An analogous argument shows that $\gamma^{(1)}(l)$ and $\gamma^{(2)}(l)$ are identical, completing the proof that there is correspondence between the two sets stated in the lemma.

It remains to show the statement on the signs. Let us assume that the conditions for existence are satisfied for some integers $p$ and $q$. Given that the sorting that defines the second set moves the $n$-th (resp.\ $m$-th) element of a sequence to position $p$ (resp.\ $q$), its sign is $(n - p)$ (resp.\ $(m - q)$). For the first set, recall the visual definitions of the parameters $p$, $q$ and $l$, to see that 
\begin{align*}
\height(\lambda \setminus \mu) ={} & \height\left( \left(\delta + \langle m^n \rangle \right) \setminus \left( \mu_{[n - 1]} \cup (m - 1) \right) \right) \\
& + \height\left( \left(\gamma + \langle n^m \rangle \right)' \setminus \left( \mu'_{[m - 1]} \cup (n - 1) \right)' \right) \displaybreak[2] \\
={} & n - p + [(r - l - 1) - (m - q)]
\end{align*}
owing two the equality in \eqref{2_eq_height_of_transpose}. 
\end{proof}

\begin{thm} [Murnaghan-Nakayama rule for Littlewood-Schur functions] \label{2_thm_MN_for_LS} Let $\mu$ be a partition. For any strictly positive integer $r$,
\begin{align} \label{2_thm_MN_for_LS_eq}
LS_\mu(-\calX; \calY) \left[ \power_r(-\calX) + (-1)^{r - 1} \power_r(\calY) \right] = \sum_{\substack{\lambda: \\ \lambda \setminus \mu \text{ is an $r$-ribbon}}} (-1)^{\height(\lambda \setminus \mu)} LS_\lambda(-\calX; \calY).
\end{align} 
\end{thm}

\begin{proof} Let $\calX = (x_1, \dots, x_n)$ and $\calY = (y_1, \dots, y_m)$ for some $n$ and $m$. Let $k$ denote the $(m,n)$-index of the partition $\mu$. By definition, the $(m,n)$-index of any partition $\lambda$ so that $\lambda \setminus \mu$ is a ribbon is equal to $k$ or $k -1$. In fact, the index of $\lambda$ is $k - 1$ if and only if $(m + 1 - k, n + 1 - k) \in \lambda$. In particular, both sides of the equation in \eqref{2_thm_MN_for_LS_eq} vanish if $k < 0$. If $k \geq 0$, then Theorem~\ref{2_thm_det_formula_for_Littlewood-Schur} states that
\begin{align}
\begin{split} \notag
\text{D} \defeq {} & \varepsilon(\mu) \frac{\Delta(\calX) \Delta(\calY)}{\Delta(\calY; \calX)} LS_\mu(-\calX; \calY) \left[ \power_r(-\calX) + (-1)^{r - 1} \power_r(\calY) \right] 
\end{split} \displaybreak[2] \\
\begin{split} \notag
={} & \det \begin{pmatrix} \left( (x_i - y_j)^{-1} \right)_{\substack{1 \leq i \leq n \\ 1 \leq j \leq m}} & \left( x_i^{\mu_j + n - m - j} \right)_{\substack{1 \leq i \leq n \\ 1 \leq j \leq n - k}} \\ \left( y_j^{\mu'_i + m - n - i} \right)_{\substack{1 \leq i \leq m - k \\ 1 \leq j \leq m}} & 0\end{pmatrix} \\
& \hspace{25pt} \times \left[ \sum_{p = 1}^n (-x_p)^r + (-1)^{r - 1} \sum_{s = 1}^m y_s^r \right]
.
\end{split}
\intertext{According to the Leibniz formula for determinants,}
\begin{split} \notag
\text{D} = {} & \left[ \sum_{\substack{\sigma \in S_{n + m - k}: \\ \forall i > n \: \sigma(i) \leq m}} \!\! \varepsilon(\sigma) \prod_{\substack{1 \leq i \leq n: \\ \sigma(i) \leq m}} (x_i - y_{\sigma(i)})^{-1} \prod_{\substack{1 \leq i \leq n: \\ \sigma(i) > m}} x_i^{\mu_{\sigma(i) - m} + n - \sigma(i)} \prod_{\makebox[48pt]{$\substack{n < i \leq n + m - k}$}} y_{\sigma(i)}^{\mu'_{i - n} + m - i}  \right] \\
& \hspace{25pt} \times \left[ (-1)^r \sum_{p = 1}^n x_p^r + (-1)^{r - 1} \sum_{s = 1}^m y_s^r \right]
.
\end{split}
\intertext{We remark that the restriction on the permutations $\sigma$ that appear in the sum is due to the block that is equal to a $(m - k) \times (n - k)$ zero matrix. As in the proof for Schur functions, we multiply out. However, unlike in the proof for Schur functions, it is not clear whether to interpret $x_p^r$ (resp.\ $y_s^r$) as part of the first or second (resp.\ first or third) product inside the sum over $\sigma$. In fact, which interpretation is the most natural depends on whether $q = \sigma(p) \in [m]$ or not (resp.\ $t = \sigma^{-1}(s) \in [n]$ or not). Therefore multiplying out results in the following four sums:}
\begin{split} \label{2_in_proof_MN_for_LS_like_Schur_x}
\text{D} = {} & (-1)^r \sum_{p = 1}^n \sum_{q = m + 1}^{m + n - k} \sum_{\substack{\sigma \in S_{n + m - k}: \\ \forall i > n \: \sigma(i) \leq m \\ \sigma(p) = q}} \varepsilon(\sigma) \left[ \prod_{\substack{1 \leq i \leq n: \\ \sigma(i) \leq m}} (x_i - y_{\sigma(i)})^{-1} \right] \\ 
& \hspace{25pt} \times \left[ x_p^{\mu_{q - m} + n - q + r} \prod_{\substack{1 \leq i \leq n: \\ \sigma(i) > m \\ i \neq p}} x_i^{\mu_{\sigma(i) - m} + n - \sigma(i)} \right] \left[ \prod_{\substack{n < i \leq n + m - k}} y_{\sigma(i)}^{\mu'_{i - n} + m - i} \right] \end{split} \displaybreak[2] \\
\begin{split} \label{2_in_proof_MN_for_LS_like_Schur_y}
& + (-1)^{r - 1} \sum_{s = 1}^m \sum_{t = n + 1}^{n + m - k} \sum_{\substack{\sigma \in S_{n + m - k}: \\ \forall i > n \: \sigma(i) \leq m \\ \sigma(t) = s}} \varepsilon(\sigma) \left[ \prod_{\substack{1 \leq i \leq n: \\ \sigma(i) \leq m}} (x_i - y_{\sigma(i)})^{-1} \right] \\
& \hspace{25pt} \times \left[ \prod_{\substack{1 \leq i \leq n: \\ \sigma(i) > m}} x_i^{\mu_{\sigma(i) - m} + n - \sigma(i)} \right] \left[ y_s^{\mu'_{t - n} + m - t + r} \prod_{\substack{n < i \leq n + m - k \\ i \neq t}} y_{\sigma(i)}^{\mu'_{i - n} + m - i} \right]  \end{split} \displaybreak[2] \\
\begin{split} \label{2_in_proof_MN_for_LS_unlike_Schur}
& + (-1)^r \sum_{p = 1}^n \sum_{q = 1}^m \sum_{\substack{\sigma \in S_{n + m - k}: \\ \forall i > n \: \sigma(i) \leq m \\ \sigma(p) = q}} \varepsilon(\sigma) x_p^r (x_p - y_q)^{-1} \left[ \prod_{\substack{1 \leq i \leq n: \\ \sigma(i) \leq m \\ i \neq p}} (x_i - y_{\sigma(i)})^{-1} \right] \\ 
& \hspace{25pt} \times \left[ \prod_{\substack{1 \leq i \leq n: \\ \sigma(i) > m}} x_i^{\mu_{\sigma(i) - m} + n - \sigma(i)} \right] \left[ \prod_{\substack{n < i \leq n + m - k}} y_{\sigma(i)}^{\mu'_{i - n} + m - i} \right] \\
& + (-1)^{r - 1} \sum_{s = 1}^m \sum_{t = 1}^{n} \sum_{\substack{\sigma \in S_{n + m - k}: \\ \forall i > n \: \sigma(i) \leq m \\ \sigma(t) = s}} \varepsilon(\sigma) \left[ y_s^r (x_t - y_s)^{-1} \prod_{\substack{1 \leq i \leq n: \\ \sigma(i) \leq m \\ i \neq t}} (x_i - y_{\sigma(i)})^{-1} \right] \\
& \hspace{25pt} \times \left[ \prod_{\substack{1 \leq i \leq n: \\ \sigma(i) > m}} x_i^{\mu_{\sigma(i) - m} + n - \sigma(i)} \right] \left[ \prod_{\substack{n < i \leq n + m - k}} y_{\sigma(i)}^{\mu'_{i - n} + m - i} \right]
.
\end{split}
\end{align}
Let us focus on the first sum (out of the four). Notice that the expression in \eqref{2_in_proof_MN_for_LS_like_Schur_x} is the empty sum if and only if $k = n$. For now we assume that $k < n$. Recalling the definition of $\mu[q]$ given in \eqref{2_eq_defn_mu[q]}, the two lines in \eqref{2_in_proof_MN_for_LS_like_Schur_x} can be reformulated as
\begin{align}
\begin{split} \notag
\text{first sum} \defeq{} & (-1)^r \sum_{p = 1}^n \sum_{q = m + 1}^{m + n - k} \sum_{\substack{\sigma \in S_{n + m - k}: \\ \forall i > n \: \sigma(i) \leq m \\ \sigma(p) = q}} \varepsilon(\sigma) \left[ \prod_{\substack{1 \leq i \leq n: \\ \sigma(i) \leq m}} (x_i - y_{\sigma(i)})^{-1} \right] \\ 
& \times \left[ \prod_{\substack{1 \leq i \leq n: \\ \sigma(i) > m}} x_i^{\mu[q - m]_{\sigma(i) - m} + n - \sigma(i)} \right] \left[ \prod_{\substack{n < i \leq n + m - k}} y_{\sigma(i)}^{\mu'_{i - n} + m - i} \right] 
\end{split} \displaybreak[2] \\
\begin{split} \notag
={} & (-1)^r \sum_{q = m + 1}^{m + n - k} \sum_{\substack{\sigma \in S_{n + m - k}: \\ \forall i > n \: \sigma(i) \leq m}} \varepsilon(\sigma) \left[ \prod_{\substack{1 \leq i \leq n: \\ \sigma(i) \leq m}} (x_i - y_{\sigma(i)})^{-1} \right] \\ 
& \times \left[ \prod_{\substack{1 \leq i \leq n: \\ \sigma(i) > m}} x_i^{\mu[q - m]_{\sigma(i) - m} + n - \sigma(i)} \right] \left[ \prod_{\substack{n < i \leq n + m - k}} y_{\sigma(i)}^{\mu'_{i - n} + m - i} \right]
. \end{split} 
\intertext{Another application of the Leibniz formula for determinants results in}
\begin{split} \notag
\text{first sum} ={} & (-1)^r \sum_{q = 1}^{n - k} \det \begin{pmatrix} \left( (x_i - y_j)^{-1} \right)_{\substack{1 \leq i \leq n \\ 1 \leq j \leq m}} & \left( x_i^{\mu[q]_j + n - m - j} \right)_{\substack{1 \leq i \leq n \\ 1 \leq j \leq n - k}} \\ \left( y_j^{\mu'_i + m - n - i} \right)_{\substack{1 \leq i \leq m - k \\ 1 \leq j \leq m}} & 0 \end{pmatrix}
.
\end{split}
\intertext{We truncate the sequence $\mu[q]$ to $\mu[q]_{[n - k]}$. If two elements of $\mu[q]_{[n - k]} + \rho_{n - k}$ are identical, the determinant vanishes; otherwise, there exists a partition $\kappa$ such that $\kappa + \rho_{n - k} \sorteq \mu[q]_{[n - k]} + \rho_{n - k}$. Hence, Lemma~\ref{2_lem_sequence_defn_for_ribbons} allows us to view the sum over $q$ as a signed sum over partitions $\kappa$ with the properties that $\kappa \setminus \mu_{[n - k]}$ is an $r$-ribbon and $l(\kappa) \leq n - k$. The sign corresponding to $\kappa$ with $\kappa + \rho_{n - k} \sorteq \mu[q]_{[n - k]} + \rho_{n - k}$ is given by the sign of the sorting permutation, which is equal to $(-1)^{\height(\kappa \setminus \mu_{[n - k]})}$. Therefore,}
\begin{split} \notag
\text{first sum} ={} & (-1)^r \sum_{\substack{\kappa: \\ \kappa \setminus \mu_{[n - k]} \text{ is an $r$-ribbon} \\ l(\kappa) \leq n - k}} (-1)^{\height \left( \kappa \setminus \mu_{[n - k]} \right)} \\
& \times\det \begin{pmatrix} \left( (x_i - y_j)^{-1} \right)_{\substack{1 \leq i \leq n \\ 1 \leq j \leq m}} & \left( x_i^{\kappa_j + n - m - j} \right)_{\substack{1 \leq i \leq n \\ 1 \leq j \leq n - k}} \\ \left( y_j^{\mu'_i + m - n - i} \right)_{\substack{1 \leq i \leq m - k \\ 1 \leq j \leq m}} & 0 \end{pmatrix}
.
\end{split}
\intertext{Consider the partition $\lambda = (\kappa_1, \dots, \kappa_{n - k}, \mu_{n - k + 1}, \dots)$. By the definition of the $(m,n)$-index, $(m - k, n - k) \in \mu$ and thus $(m - k, n - k) \in \kappa$, which entails that $\mu'_i = \lambda'_i$ for all $1 \leq i \leq m - k$. Moreover, the $(m,n)$-index of $\lambda$ is also $k$ because $(m + 1 - k, n + 1 - k) \not\in \mu$ is equivalent to $(m + 1 - k, n + 1 - k) \not\in \lambda$. We apply Theorem~\ref{2_thm_det_formula_for_Littlewood-Schur} to obtain}
\begin{split} \notag
\text{first sum} ={} & (-1)^r \sum_{\substack{\lambda: \\ \lambda \setminus \mu \text{ is an $r$-ribbon} \\ \text{for all } i > n - k \text{, } \lambda_i = \mu_i}} (-1)^{\height \left( \lambda \setminus \mu \right)} \varepsilon(\lambda) \frac{\Delta(\calX) \Delta(\calY)}{\Delta(\calY; \calX)} LS_\lambda (-\calX; \calY)
.
\end{split}
\intertext{As any partition $\lambda$ that appears in the sum has $(m,n)$-index $k$, $\left|\lambda_{[n - k]} \right| - \left|\mu_{[n - k]} \right| = r$ implies that $(-1)^r \varepsilon(\mu) = \varepsilon(\lambda)$. Hence,}
\begin{split} \label{2_in_proof_MN_for_LS_first_sum}
\text{first sum} ={} & \varepsilon(\mu) \frac{\Delta(\calX) \Delta(\calY)}{\Delta(\calY; \calX)} \sum_{\substack{\lambda: \\ \lambda \setminus \mu \text{ is an $r$-ribbon} \\ \text{for all } i > n - k \text{, } \lambda_i = \mu_i}} (-1)^{\height \left( \lambda \setminus \mu \right)} LS_\lambda (-\calX; \calY)
.
\end{split}
\end{align}
We remark that this equality also holds when $k = n$. Indeed, in this case the sums in \eqref{2_in_proof_MN_for_LS_like_Schur_x} and \eqref{2_in_proof_MN_for_LS_first_sum} are both empty. An analogous computation shows that the sum in \eqref{2_in_proof_MN_for_LS_like_Schur_y} is equal to
\begin{align}
\begin{split} \notag
\text{second sum} \defeq{} & (-1)^{r - 1} \sum_{s = 1}^m \sum_{t = n + 1}^{n + m - k} \sum_{\substack{\sigma \in S_{n + m - k}: \\ \forall i > n \: \sigma(i) \leq m \\ \sigma(t) = s}} \varepsilon(\sigma) \left[ \prod_{\substack{1 \leq i \leq n: \\ \sigma(i) \leq m}} (x_i - y_{\sigma(i)})^{-1} \right] \\
& \hspace{25pt} \times \left[ \prod_{\substack{1 \leq i \leq n: \\ \sigma(i) > m}} x_i^{\mu_{\sigma(i) - m} + n - \sigma(i)} \right] \left[ y_s^{\mu'_{t - n} + m - t + r} \prod_{\substack{n < i \leq n + m - k \\ i \neq t}} y_{\sigma(i)}^{\mu'_{i - n} + m - i} \right]  \end{split} \displaybreak[2] \\
\begin{split} \label{2_in_proof_MN_for_LS_second_sum}
={} & \varepsilon(\mu) \frac{\Delta(\calX) \Delta(\calY)}{\Delta(\calY; \calX)} \sum_{\substack{\lambda: \\ \lambda \setminus \mu \text{ is an $r$-ribbon} \\ \text{for all } j > m - k \text{, } \lambda'_j = \mu'_j}} (-1)^{\height \left( \lambda \setminus \mu \right)} LS_\lambda (-\calX; \calY)
.
\end{split}
\end{align}
Summing up the equalities in \eqref{2_in_proof_MN_for_LS_first_sum} and \eqref{2_in_proof_MN_for_LS_second_sum},
\begin{align*}
\text{first sum} + \text{second sum} = \varepsilon(\mu) \frac{\Delta(\calX) \Delta(\calY)}{\Delta(\calY; \calX)} \sum_{\substack{\lambda: \\ \lambda \setminus \mu \text{ is an $r$-ribbon} \\ (m + 1 - k, n + 1 - k) \not\in \lambda}} (-1)^{\height \left( \lambda \setminus \mu \right)} LS_\lambda (-\calX; \calY)
.
\end{align*}

It therefore remains to show that the two sums in \eqref{2_in_proof_MN_for_LS_unlike_Schur} correspond to the case when the $r$-ribbon $\lambda \setminus \mu$ contains the box with coordinates $(m + 1 - k, n + 1 - k)$. First observe that if $k = 0$, any partition $\lambda$ that contains the box $(m + 1 - k, n + 1 - k)$ has a negative $(m,n)$-index, meaning that $LS_\lambda(-\calX; \calY) = 0$. Hence, under the assumption that $k = 0$,
\begin{align*}
\text{first sum} + \text{second sum} = \varepsilon(\mu) \frac{\Delta(\calX) \Delta(\calY)}{\Delta(\calY; \calX)} \sum_{\substack{\lambda: \\ \lambda \setminus \mu \text{ is an $r$-ribbon}}} (-1)^{\height \left( \lambda \setminus \mu \right)} LS_\lambda (-\calX; \calY)
.
\end{align*}
In addition, the condition on the permutations $\sigma$ that appear in the sums entails that both sums in \eqref{2_in_proof_MN_for_LS_unlike_Schur} are empty whenever $k = 0$. This completes the proof for indexing partitions $\mu$ whose $(m,n)$-index is 0. 

Let us suppose that $k > 0$. Replacing the summation indices $s$ and $t$ by $q$ and $p$, respectively, we see that the expression in \eqref{2_in_proof_MN_for_LS_unlike_Schur} is equal to
\begin{align*}
\text{last sums} \defeq {} & (-1)^r \sum_{p = 1}^n \sum_{q = 1}^m \sum_{\substack{\sigma \in S_{n + m - k}: \\ \forall i > n \: \sigma(i) \leq m \\ \sigma(p) = q}} \varepsilon(\sigma) \frac{x_p^r - y_q^r}{x_p - y_q} \left[ \prod_{\substack{1 \leq i \leq n: \\ \sigma(i) \leq m \\ i \neq p}} (x_i - y_{\sigma(i)})^{-1} \right] \\ 
& \times \left[ \prod_{\substack{1 \leq i \leq n: \\ \sigma(i) > m}} x_i^{\mu_{\sigma(i) - m} + n - \sigma(i)} \right] \left[ \prod_{\substack{n < i \leq n + m - k}} y_{\sigma(i)}^{\mu'_{i - n} + m - i} \right] \displaybreak[2] \\
= & (-1)^r \sum_{l = 0}^{r - 1} \sum_{p = 1}^n \sum_{q = 1}^m \sum_{\substack{\sigma \in S_{n + m - k}: \\ \forall i > n \: \sigma(i) \leq m \\ \sigma(p) = q}} \varepsilon(\sigma) \left[ \prod_{\substack{1 \leq i \leq n: \\ \sigma(i) \leq m \\ i \neq p}} (x_i - y_{\sigma(i)})^{-1} \right] \\ 
& \times \left[ x_p^l \prod_{\substack{1 \leq i \leq n: \\ \sigma(i) > m}} x_i^{\mu_{\sigma(i) - m} + n - \sigma(i)} \right] \left[ y_q^{r - 1 - l} \prod_{\substack{n < i \leq n + m - k}} y_{\sigma(i)}^{\mu'_{i - n} + m - i} \right]
\end{align*}
For every triple $p$, $q$, $\tau$, we define a permutation $\tau \in S_{n + m - k + 1}$ by $\tau(p) = n + m - k + 1$, $\tau(n + m - k + 1) = q$ and $\tau(i) = \sigma(i)$ for all $i \neq p, n + m - k + 1$. We note for later reference that $\varepsilon(\tau) = -\varepsilon(\sigma)$. In addition, we recall the definitions of the following sequences: $$\alpha(l) = (\mu_1 - m - 1 + k, \dots, \mu_{n - k} - m - 1 + k, l)$$ and $$\beta(l) = (\mu'_1 - n - 1 + k, \dots, \mu'_{m - k} - n - 1 + k, r - 1 - l).$$ These definitions simplify the expression to
\begin{align*}
\text{last sums} ={} & (-1)^r \sum_{l = 0}^{r - 1} \sum_{p = 1}^n \sum_{q = 1}^m \sum_{\substack{\tau \in S_{n + m - k + 1}: \\ \forall i > n \: \tau(i) \leq m \\ \tau(p) = n + m - k + 1 \\ \tau(n + m - k + 1) = q}} -\varepsilon(\tau) \left[ \prod_{\substack{1 \leq i \leq n: \\ \tau(i) \leq m}} (x_i - y_{\tau(i)})^{-1} \right] \\ 
& \times \left[ \prod_{\substack{1 \leq i \leq n: \\ \tau(i) > m}} x_i^{\alpha(l)_{\tau(i) - m} + n - k + 1 - (\tau(i) - m)} \right] \left[ \prod_{\substack{n < i \\ i \leq n + m - k + 1}} y_{\tau(i)}^{\beta(l)_{i - n} + m - k + 1 - (i - n)} \right] \displaybreak[2] \\
={} & (-1)^r \sum_{l = 0}^{r - 1} \sum_{\substack{\tau \in S_{n + m - k + 1}: \\ \forall i > n \: \tau(i) \leq m}} -\varepsilon(\tau) \left[ \prod_{\substack{1 \leq i \leq n: \\ \tau(i) \leq m}} (x_i - y_{\tau(i)})^{-1} \right] \\ 
& \times \left[ \prod_{\substack{1 \leq i \leq n: \\ \tau(i) > m}} x_i^{\alpha(l)_{\tau(i) - m} + n - k + 1 - (\tau(i) - m)} \right] \left[ \prod_{\substack{n < i \\ i \leq n + m - k + 1}} y_{\tau(i)}^{\beta(l)_{i - n} + m - k + 1 - (i - n)} \right]
.
\intertext{Another application of the Leibniz formula for determinants results in}
\text{last sums} ={} & (-1)^{r + 1} \sum_{l = 0}^{r - 1} \\
& \times \det \begin{pmatrix} \left((x_i - y_j)^{-1}\right)_{\substack{1 \leq i \leq n \\ 1 \leq j \leq m}} & \left( x_i^{\alpha(l)_j + n - k + 1 - j} \right)_{\substack{1 \leq i \leq n \\ 1 \leq j \leq n - k + 1}} \\ \left( y_j^{\beta(l)_i + m - k + 1 - i} \right)_{\substack{1 \leq i \leq m - k + 1 \\ 1 \leq j \leq m}} & 0 \end{pmatrix}
.
\end{align*}
Lemma \ref{2_lem_ribbon_change_index} allows us to view the sum over $l$ as a signed sum over partitions $\lambda$ with the properties that $\lambda \setminus \mu$ is an $r$-ribbon and $(m - k + 1, n - k + 1) \in \lambda$. To keep track of which partition $\lambda$ is associated to the parameter $l$, we denote the partition $$\lambda = \left( \left\langle (m - k + 1)^{n - k + 1} \right\rangle + \delta \right) \cup \gamma'$$ where $\delta + \rho_{n - k + 1} \sorteq \alpha(l) + \rho_{n - k + 1}$ and $\gamma + \rho_{m - k + 1} \sorteq \beta(l) + \rho_{m - k + 1}$ by $\lambda(l)$ \emph{if it exists}; otherwise, the determinant corresponding to the summation index $l$ vanishes, allowing us to ignore this case:
\begin{align*}
\text{last sums} ={} & \sum_{\substack{0 \leq l \leq r - 1: \\ \lambda(l) \text{ exists}}} (-1)^{r + 1} (-1)^{r - 1 - l} (-1)^{\height(\lambda(l) \setminus \mu)} \\
& \times \det \begin{pmatrix} \left((x_i - y_j)^{-1}\right)_{\substack{1 \leq i \leq n \\ 1 \leq j \leq m}} & \left( x_i^{\lambda(l)_j + n - m - j} \right)_{\substack{1 \leq i \leq n \\ 1 \leq j \leq n - k + 1}} \\ \left( y_j^{\lambda(l)'_i + m - n - i} \right)_{\substack{1 \leq i \leq m - k + 1 \\ 1 \leq j \leq m}} & 0 \end{pmatrix}
.
\intertext{By definition, the $(m,n)$-index of any partition $\lambda(l)$ that appears in the sum is $k - 1$. Therefore, Theorem~\ref{2_thm_det_formula_for_Littlewood-Schur} states that}
\text{last sums} ={} & \sum_{\substack{0 \leq l \leq r - 1: \\ \lambda(l) \text{ exists}}} (-1)^{\height(\lambda(l) \setminus \mu)} (-1)^l \varepsilon(\lambda(l)) \frac{\Delta(\calX) \Delta(\calY)}{\Delta(\calY; \calX)} LS_{\lambda(l)} (-\calX; \calY)
.
\end{align*}
Recalling the visual definition of the parameter $l$ given on page \pageref{2_page_visual_defn_of_l}, we see that
\begin{multline*}
(-1)^l \varepsilon(\lambda(l)) = (-1)^l (-1)^{\left| \lambda_{[n - k + 1]} \right|} (-1)^{m(k - 1)} (-1)^{(k -1)(k -2)/2} \\ = (-1)^l (-1)^{\left| \mu_{[n - k]} \right| + l + (m - k + 1)} (-1)^{m(k - 1)} (-1)^{(k -1)(k - 2)/2} = \varepsilon(\mu)
.
\end{multline*}
Hence,
\begin{align*}
\text{last sums} ={} & \varepsilon(\mu) \frac{\Delta(\calX) \Delta(\calY)}{\Delta(\calY; \calX)} \sum_{\substack{\lambda: \\ \lambda \setminus \mu \text{ is an $r$-ribbon} \\ (m - k + 1, n - k + 1) \in \lambda}} (-1)^{\height(\lambda \setminus \mu)} LS_\lambda(-\calX; \calY)
,
\end{align*}
which completes the proof.
\end{proof}

\chapter{Overlap Identities for Littlewood-Schur Functions} \label{3_cha_overlap_ids}
\section{Introduction}
In this chapter we introduce an operation on partitions, which we call \emph{overlap}. We formally define the overlap of two partitions on page~\pageref{3_defn_overlap} but for now let it suffice to merely introduce the notation that will permit us to state the first and the second overlap identities for Schur functions: if a partition $\lambda$ is the $(m,n)$-overlap of $\mu$ and $\nu$, we write $\lambda = \mu \star_{m,n} \nu$. In addition, we denote its sign by $\varepsilon(\mu, \nu)$.

The first overlap identity states that for any set $\calX$ consisting of $m + n$ pairwise distinct variables,
\begin{align*}
\schur_{\mu \star_{m,n} \nu} (\calX) = \sum_{\substack{\calS, \calT}} \frac{\varepsilon(\mu, \nu) \schur_\mu(\calS) \schur_\nu(\calT)}{\Delta(\calS; \calT)}
\end{align*}
where the sum is over all disjoint subsets $\calS$ and $\calT$ of $\calX$ with $m$ and $n$ elements, respectively, so that their union is $\calX$. The symbol $\Delta(\calS; \calT)$ denotes the product of all pairwise differences between $s \in \calS$ and $t \in \calT$. Although the concept of overlap seems to be new, this result is not. Recalling that Schur functions can be defined as fractions of determinants, Dehaye derives this identity by expanding the determinant in the numerator with the help of Laplace expansion \cite{dehaye12}. Taking the idea of applying Laplace expansion to the determinantal definition as a starting point, we show a second overlap identity: let $\calS$ and $\calT$ be sets consisting of $m$ and $n$ variables, respectively, and set $\calX = \calS \cup \calT$. If $\Delta(\calS; \calT) \neq 0$, then
\begin{align*}
\schur_\lambda(\calX) ={} & \sum_{\substack{\mu, \nu: \\ \mu \star_{m,n} \nu = \lambda}} \frac{\varepsilon(\mu, \nu) \schur_\mu(\calS) \schur_\nu(\calT)}{\Delta(\calS; \calT)}
.
\end{align*}
In fact, the two identities for Schur functions presented here are mere corollaries to two of the main results of this chapter, namely the first and the second overlap identities for Littlewood-Schur functions. Littlewood-Schur functions are a generalization of Schur functions, whose combinatorial definition (which we have presented in Section~\ref{1_sec_LS_functions}) was first studied by Littlewood \cite{littlewood}: for any two sets of variables $\calX$ and $\calY$, the Littlewood-Schur function associated to a partition $\lambda$ is defined by
\begin{align*}
LS_\lambda(\calX; \calY) = \sum_{\mu,\nu} c^\lambda_{\mu \nu} \schur_\mu(\calX) \schur_{\nu'}(\calY)
\end{align*}
where $c^\lambda_{\mu \nu}$ are Littlewood-Richardson coefficients. As we have discussed in Chapter~\ref{2_cha_det_defn_LS}, Littlewood-Schur functions can also be described by a determinantal formula due to Moens and Van der Jeugt \cite{vanderjeugt}, which forms the basis for the overlap identities presented in this chapter.

Dropping a few technical conditions, the first and the second overlap identities for Littlewood-Schur functions have the following prerequisites: let $\calX$ and $\calY$ be sets consisting of $n$ and $m$ variables, respectively, and let $\lambda$ be a partition with $(m,n)$-index $k$. (Turn to page~\pageref{3_defn_index} for the definition of index.) The first overlap identity states that if $\lambda_{[n - k]} = \mu \star_{l, n - k - l} \nu$, then 
\begin{align*}
LS_{\lambda} (-\calX; \calY) ={} & \sum_{\calS, \calT} \frac{\varepsilon(\mu, \nu) LS_{\mu + \left\langle k^l \right\rangle}(-\calS; \calY) LS_{\nu \cup \lambda_{(n + 1 - k, n + 2 - k, \dots)}}(- \calT; \calY)}{\Delta(\calT; \calS)}
\end{align*}
where the sum ranges over all disjoint subsets $\calS$ and $\calT$ of $\calX$ with $l$ and $n - l$ elements, respectively, so that their union is $\calX$. An explanation of the symbols that may be new to a non-chronological reader can be found in Sections~\ref{3_section_sequences} and \ref{3_section_partitions}. Partitioning $\calX$ into two disjoint subsets, say $\calS$ and $\calT$, consisting of $l$ and $n - l$ variables, respectively, the second overlap identity states that
\begin{align*}
LS_{\lambda} (-\calX; \calY) ={} & \sum_{p = 0}^{\min\{l,m\}} \sum_{\substack{\calU, \calV}} \hspace{10pt} \sum_{\substack{\mu, \nu: \\ \mu \star_{l - p, n - k - l + p} \nu = \lambda_{[n - k]}}} \frac{\Delta(\calV; \calS) \Delta(\calT; \calU)}{\Delta(\calV; \calU) \Delta(\calT; \calS)} \\
& \times \varepsilon(\mu, \nu) LS_{\mu - \left\langle (m - k)^{l - p} \right\rangle} (-\calS; \calU) LS_{\nu \cup \lambda_{(n + 1 - k, n + 2 - k, \dots)}} (-\calT; \calV)
\end{align*}
where $\calU$ and $\calV$ range over all disjoint subsets of $\calY$ with $p$ and $m - p$ elements, respectively, so that their union is equal to $\calY$.

We also provide two visual descriptions for the notion of overlap, which are more intuitive than its formal definition. Let us illustrate the first visualization (which is formally stated in Proposition~\ref{3_prop_visual_interpretation_overlap}) on an example, as it is the reason why this operation is called the $(m,n)$-overlap of two partitions. Fix two non-negative integers $m$ and $n$, say 3 and 5, as well as a partition of length at most $m + n = 8$, say $\lambda = (4, 2, 2, 2, 2, 1, 0, 0)$. Notice that we have appended zeros, ensuring that $\lambda$ is represented by a sequence of length exactly 8. Pairs of partitions whose $(m,n)$-overlap equals $\lambda$ may be viewed as staircase walks in an $n \times m$ rectangle whose steps are labeled by the parts of $\lambda$. Let us consider the following diagram of a labeled staircase walk $\pi$ in a $5 \times 3$ rectangle:
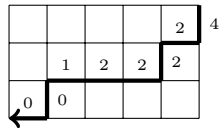
\begin{center}
\begin{tikzpicture} 
\draw[step=0.5cm, thin] (0, 0) grid (2.5, 1.5);
\draw[ultra thick, ->] (0.5, 0) -- (0, 0);
\draw[ultra thick] (0.5, 0) -- (0.5, 0.5);
\draw[ultra thick] (0.5, 0.5) -- (2, 0.5);
\draw[ultra thick] (2, 0.5) -- (2, 1);
\draw[ultra thick] (2, 1) -- (2.5, 1);
\draw[ultra thick] (2.5, 1) -- (2.5, 1.5);
\node[anchor=west] at (2.5, 1.25) {\tiny{4}};
\node[anchor=south] at (2.25, 1) {\tiny{2}};
\node[anchor=west] at (2, 0.75) {\tiny{2}};
\node[anchor=south] at (1.75, 0.5) {\tiny{2}};
\node[anchor=south] at (1.25, 0.5) {\tiny{2}};
\node[anchor=south] at (0.75, 0.5) {\tiny{1}};
\node[anchor=west] at (0.5, 0.25) {\tiny{0}};
\node[anchor=south] at (0.25, 0) {\tiny{0}};
\end{tikzpicture} 
\end{center}
For the duration of this chapter we always walk down stairs that connect the top-right and the bottom-left corners of a rectangle. By means of this diagram, we can now describe the pair of partitions $\mu$, $\nu$ that correspond to the staircase walk $\pi$, thus giving an example of overlap. In fact, we will visualize the two partitions with the help of Ferrers diagrams, which are defined on page \pageref{3_page_Ferrers_diagram}. The partition $\mu$ is specified by the diagram on the right-hand side:
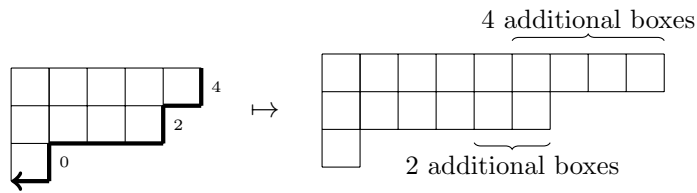
\begin{center}
\begin{tikzpicture} 
\draw[step=0.5cm, thin] (0, 0) grid (0.5, 1.5);
\draw[step=0.5cm, thin] (0.5, 0.5) grid (2, 1.5);
\draw[step=0.5cm, thin] (2, 1) grid (2.5, 1.5);
\draw[ultra thick, ->] (0.5, 0) -- (0, 0);
\draw[ultra thick] (0.5, 0) -- (0.5, 0.5);
\draw[ultra thick] (0.5, 0.5) -- (2, 0.5);
\draw[ultra thick] (2, 0.5) -- (2, 1);
\draw[ultra thick] (2, 1) -- (2.5, 1);
\draw[ultra thick] (2.5, 1) -- (2.5, 1.5);
\node[anchor=west] at (2.5, 1.25) {\tiny{4}};
\node[anchor=west] at (2, 0.75) {\tiny{2}};
\node[anchor=west] at (0.5, 0.25) {\tiny{0}};
\end{tikzpicture} 
\begin{tikzpicture} 
\node(arrow) at (-0.75, 0.74) {$\mapsto$};
\draw[step=0.5cm, thin] (0, 0) grid (0.5, 1.5);
\draw[step=0.5cm, thin] (0.5, 0.5) grid (3, 1.5);
\draw[step=0.5cm, thin] (3, 0.999) grid (4.5, 1.5);
\draw[decoration={brace, raise=5pt},decorate] (2.5, 1.5) -- node[above=6pt] {\small{4 additional boxes}} (4.5, 1.5);
\draw[decoration={brace, raise=5pt, mirror},decorate] (2, 0.5) -- node[below=6pt] {\small{2 additional boxes}} (3, 0.5);
\end{tikzpicture} 
\end{center}
In words, the labels attached to the vertical steps of $\pi$ indicate how many boxes must be added to each row of the Ferrers diagram consisting of the boxes that lie above $\pi$ in order to obtain the Ferrers diagram of the partition $\mu$. An analogous construction gives the Ferrers diagram of $\nu$:
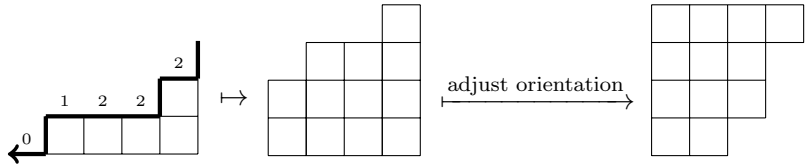
\begin{center}
\begin{tikzpicture} 
\draw[step=0.5cm, thin] (0.5, 0) grid (2, 0.5);
\draw[step=0.5cm, thin] (2, 0) grid (2.5, 1);
\draw[ultra thick, ->] (0.5, 0) -- (0, 0);
\draw[ultra thick] (0.5, 0) -- (0.5, 0.5);
\draw[ultra thick] (0.5, 0.5) -- (2, 0.5);
\draw[ultra thick] (2, 0.5) -- (2, 1);
\draw[ultra thick] (2, 1) -- (2.5, 1);
\draw[ultra thick] (2.5, 1) -- (2.5, 1.5);
\node[anchor=south] at (2.25, 1) {\tiny{2}};
\node[anchor=south] at (1.75, 0.5) {\tiny{2}};
\node[anchor=south] at (1.25, 0.5) {\tiny{2}};
\node[anchor=south] at (0.75, 0.5) {\tiny{1}};
\node[anchor=south] at (0.25, 0) {\tiny{0}};
\end{tikzpicture} 
\begin{tikzpicture} 
\node(arrow) at (-0.45, 0.74) {$\mapsto$};
\draw[step=0.5cm, thin] (0, 0) grid (0.5, 1);
\draw[step=0.5cm, thin] (0.5, 0) grid (1.5, 1.5);
\draw[step=0.5cm, thin] (1.499, 0) grid (2, 2);
\end{tikzpicture} 
\begin{tikzpicture} 
\node(arrow) at (-1.5, 0.85) {$\xmapsto{\text{adjust orientation}}$
};
\draw[step=0.5cm, thin] (0, 0) grid (1, 2);
\draw[step=0.5cm, thin] (1, 0.5) grid (1.5, 2);
\draw[step=0.5cm, thin] (1.5, 1.499) grid (2, 2);
\end{tikzpicture} 
\end{center}
The only difference is that the resulting collection of boxes needs to be rotated by 180 degrees and then transposed to adjust their orientation. We conclude that the $(3,5)$-overlap of the partitions $\mu = (9, 6, 1)$ and $\nu = (4, 3, 3, 2)$ is $\lambda = (4, 2, 2, 2, 2, 1)$; in symbols, $\mu \star_{m,n} \nu = \lambda$. In brief, the $(3,5)$-overlap of a pair of partitions encodes which boxes must be deleted in order to assemble the diagrams of both partitions into a $5 \times 3$ rectangle, \textit{i.e.}\ it keeps track of where the two partitions overlap. In addition, we define the sign of an overlap to be $(-1)$ to the power of the number of boxes below the staircase walk that corresponds to it. In our example the sign is thus given by $\varepsilon(\mu, \nu) = (-1)^5$.

One application of this visualization for two overlapping partitions is a new proof of the dual Cauchy identity. Another application of the overlap identities presented in this chapter is the formula for mixed ratios of characteristic polynomials, which we will derive in Chapter~\ref{4_cha_mixed_ratios}.

\subsection{Structure of this chapter}
In Section~\ref{3_section_background_and_notation}, we give the required background on partitions, Schur functions and Littlewood-Schur functions. In Section~\ref{3_section_lapalace_expansion_for_LS}, we introduce the formal definition of overlap and then use Laplace expansion to derive two overlap identities for Littlewood-Schur functions. Section~\ref{3_section_visualizations_of_overlap} contains two visual characterizations for the set of all pairs of partitions with the same overlap. Applying these visualizations to the second overlap identity results in more overlap identities, which are listed in Section~\ref{3_section_more_overlap_identities}.

\section{Background and notation} \label{3_section_background_and_notation}
We encourage the chronological reader, who has read Chapters~\ref{1_cha_algebraic_combinatorics} and \ref{2_cha_det_defn_LS}, to skip the sections on Schur and Littlewood-Schur functions -- with the exception of the notation introduced in \eqref{3_eq_e(X)_defn} on page \pageref{3_eq_e(X)_defn}.

\subsection{Sequences and sets of variables} \label{3_section_sequences}
Throughout this chapter a sequence will be a \emph{finite} enumeration of elements, such as \label{symbol_sequence} $\calX = \left(\calX_1, \dots, \calX_n \right)$. Its length is the number of its elements, denoted by \label{symbol_length_of_sequence} $l(\calX) = n$. A subsequence $\calY$ of $\calX$ is a sequence given by $\calY_k = \calX_{n_k}$ where $1 \leq n_1 < n_2 < \dots \leq n$ is an increasing sequence of indices; in other words, if $K$ is a subsequence of \label{symbol_n_in_square_brackets} $[n] = (1,2, \dots, n)$, then the $K$-subsequence of $\calX$ is given by \label{symbol_subsequence_indexed_by_K} $\calX_K = \left(\calX_{K_1}, \dots, \calX_{K_{l(K)}}\right)$. We denote the complement of a subsequence \label{symbol_subsequence} $\calY \subset \calX$ by \label{symbol_complement_of_subsequence} $\calX \setminus \calY \subset \calX$. The union of two sequences \label{symbol_union_of_sequences} $\calX \cup \calY$ is obtained by appending $\calY$ to $\calX$; we sometimes add subscripts to indicate the lengths of the two sequences in question. All other operations on sequences, such as \label{symbol_sum_of_sequences} addition, are understood to be element wise.

For sequences whose elements lie in a ring, we define the following two functions: \label{3_page_Delta}
\begin{align*}
\Delta(\calX) = \prod_{1 \leq i < j \leq n} (\calX_i - \calX_j) \:\text{ and }\: \Delta(\calX; \calY) = \prod_{x \in \calX, y \in \calY} (x - y)
.
\end{align*}
If two sequences $\calX$ and $\calY$ of the same length are equal up to reordering their elements, we write \label{symbol_sorteq} $\calX \sorteq \calY$. In that case $\Delta(\calX) = \varepsilon(\sigma)\Delta(\calY)$, where $\varepsilon(\sigma)$ is the sign of the sorting permutation $\sigma$ with the property that $\sigma(\calX) = \calY$. We implicitly view all sets of variables as sequences but for simplicity of notation we will not fix the order of the variables explicitly. It is important, however, to stick to one order throughout a computation or within a formula.

\subsection{Partitions} \label{3_section_partitions}
A partition is a non-increasing sequence $\lambda = (\lambda_1, \dots, \lambda_n)$ of non-negative integers, called parts. If two partitions only differ by a sequence of zeros, we regard them as equal. By an abuse of notation, we say that the length of a partition is the length of the subsequence that consists of its positive parts. The size of a partition $\lambda$ is the sum of its parts, denoted $|\lambda|$.

\label{3_page_Ferrers_diagram} The Ferrers diagram of a partition $\lambda$ is defined as the set of points $(i,j) \in \Z \times \Z$ such that $1 \leq i \leq \lambda_j$; it is often convenient to replace the points by square boxes. Turn to page \pageref{3_ferrers_diagram} for some examples of Ferrers diagrams. In this chapter, we use the convention that for any partition $\lambda$ and any non-negative integer $i$, $(i, 0) \in \lambda$ and $(0, i) \in \lambda$, although the Ferrers diagram of $\lambda$ only contains points with strictly positive coordinates. Similarly, we define the $0$-th part of any partition to be infinitely large. This counter-intuitive usage will allow us to avoid case analysis during later computations. Given two partitions $\kappa$ and $\lambda$, we say that $\kappa$ is a subset of $\lambda$ if their Ferrers diagrams satisfy that containment relation. Note that $\kappa \subset \lambda$ is our shorthand for both subset and subsequence. It will be clear from the context whether we view $\kappa$ and $\lambda$ as sequences or diagrams. For instance, we will study sub\emph{sets} of partitions whose Ferrers diagram are rectangular -- not subsequences of constant sequences. We denote by $\langle m^n \rangle$ the partition $(m, \dots, m)$ of length $n$, whose Ferrers diagram is a rectangle.

The conjugate partition $\lambda'$ of $\lambda$ is given by the condition that the Ferrers diagram of $\lambda'$ is the transpose of the Ferrers diagram of $\lambda$, \textit{e.g.}\ the conjugate of $(5,5,2)$ is $(3,3,2,2,2)$. We note for later reference that if the union of two partitions $\mu$ and $\nu$ is a partition, then $(\mu \cup \nu)' = \mu' + \nu'$. 

\subsection{Schur functions} \label{3_section_schur}
We briefly present symmetric polynomials and Schur functions, following the conventions of Macdonald \cite{mac}. For each non-negative integer $r$, the $r$-th elementary symmetric polynomial is defined by
\begin{align*}
\elementary_r(\calX) = \sum_{\substack{\calY \subset \calX: \\ l(\calY) = r}} \prod_{y \in \calY} y
.
\end{align*}
We observe that for any set of variables $\calX$, the $l(\calX)$-th elementary polynomial $\elementary_{l(\calX)}(\calX)$ is just the product of all elements of $\calX$. This observation motivates the following non-standard notation:
\begin{align} \label{3_eq_e(X)_defn}
\elementary(\calX) = \prod_{x \in \calX} x
.
\end{align} 
The elementary polynomials are called symmetric as they are invariant under permutations of the elements of $\calX$. In this context, Schur functions are another family of symmetric polynomials, indexed by partitions. 

\begin{defn} [Schur functions] Let $\lambda$ be a partition and $\calX$ a set of pairwise distinct variables of length $n$. If $l(\lambda) > n$, then $\schur_\lambda(\calX) = 0$; otherwise, 
\begin{align*}
\schur_\lambda(\calX) = \frac{\det \left( x^{\lambda_j + n - j} \right)_{x \in \calX, 1 \leq j \leq n}}{\Delta(\calX)}
.
\end{align*}
This definition can be extended to sets of variables that contain repetitions given that the polynomial $\Delta(\calX)$ is a divisor of the determinant in the numerator.
\end{defn}
Using that the determinant is multilinear, one quickly checks that the Schur function $\schur_\lambda(\calX)$ is homogeneous of degree $|\lambda|$. The multilinearity of determinants also entails that for any set $\calX$ consisting of exactly $n$ variables, 
\begin{align}
\schur_{\lambda + \langle m^n \rangle} (\calX) = \elementary(\calX)^m \schur_\lambda(\calX)
.
\end{align} 

\subsection{Littlewood-Schur functions}
\begin{defn} [Littlewood-Schur functions] \label{3_defn_comb_LS} Let $\calX$ and $\calY$ be two sets of variables. Define
$$LS_\lambda(\calX; \calY) = \sum_{\mu, \nu} c^\lambda_{\mu \nu} \schur_\mu(\calX) \schur_{\nu'}(\calY)$$
where the Littlewood-Richardson coefficients $c^\lambda_{\mu \nu}$ are given in Definition~\ref{1_defn_Littlewood-Richardson_coefficients}.
\end{defn}

In this chapter, we will not work with the combinatorial definition of Littlewood-Schur functions. Instead, the identities for Littlewood-Schur functions described in
Section~\ref{3_section_two_overlap_identities_for_LS_functions} are based on a determinantal formula for Littlewood-Schur functions discovered by Moens and Van der Jeugt [MdJ03]. In order to state their result, we need
to introduce the notion of an index of a partition.

\begin{defn} [index of a partition] \label{3_defn_index} The $(m,n)$-index of a partition $\lambda$ is the largest (possibly negative) integer $k$ which satisfies $(m + 1 - k, n + 1 - k) \not\in \lambda$ and $k \leq \min\{m,n\}$. Making use of the convention introduced in Section~\ref{3_section_partitions}, it is equivalent to define the $(m,n)$-index of $\lambda$ as the smallest integer $k$ so that $(m - k, n - k) \in \lambda$.
\end{defn}

It is worth noting that this definition is not equivalent to the definition used in \cite{vanderjeugt}. They work with a similar notion of index except that it is not invariant under conjugation. 

Given the Ferrers diagram of a partition $\lambda$, its $(m,n)$-index can be read off visually:
If $(m,n) \not\in \lambda$, then $k$ is the side of the largest square with bottom-right corner $(m,n)$ that fits next to the diagram of the partition $\lambda$. If $(m,n) \in \lambda$, then $-k$ is the side of the largest square with top-left corner $(m,n)$ that fits inside the diagram of $\lambda$. Let us illustrate this by a sketch: the area colored in gray is the diagram of the partition $\lambda = (7, 4, 2, 2)$.

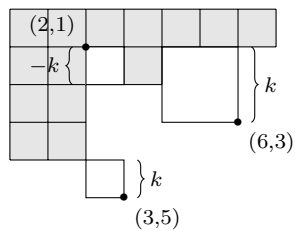
\begin{center}
\begin{tikzpicture}
\fill[black!10!white] (0, 0) rectangle (1, 2);
\fill[black!10!white] (1, 1) rectangle (2, 2);
\fill[black!10!white] (2, 1.5) rectangle (3.5, 2);
\draw[step=0.5cm, thin] (0, 0) grid (1, 2);
\draw[step=0.5cm, thin] (1, 0.999) grid (2, 2);
\draw[step=0.5cm, thin] (2, 1.499) grid (3.5, 2);
\draw (1, 0) rectangle (1.5, -0.5);
\draw (2, 1.5) rectangle (3, 0.5);
\fill[white] (1, 1.5) rectangle (1.5, 1);
\draw (1, 1.5) rectangle (1.5, 1);
\node[anchor=north west] at (1.5, -0.5) {$\scriptstyle{(3, 5)}$};
\node[anchor=north west] at (3, 0.5) {$\scriptstyle{(6, 3)}$};
\node[anchor=south east] at (1, 1.5) {$\scriptstyle{(2, 1)}$};
\node at (1.5, -0.5) {\tiny{\textbullet}};
\node at (3, 0.5) {\tiny{\textbullet}};
\node at (1, 1.5) {\tiny{\textbullet}};
\draw[decoration={brace, raise=5pt},decorate] (3, 1.5) -- node[right=6pt] {$\scriptstyle{k}$} (3, 0.5);
\draw[decoration={brace, raise=5pt, mirror},decorate] (1.5, -0.5) -- node[right=6pt] {$\scriptstyle{k}$} (1.5, 0);
\draw[decoration={brace, raise=5pt},decorate] (1, 1) -- node[left=6pt] {$\scriptstyle{-k}$} (1, 1.5);
\end{tikzpicture}
\end{center}
We see that the $(6,3)$-index of $\lambda$ is 2, its $(3, 5)$-index is 1, and its $(2,1)$-index is $-1$. 

\begin{thm} [determinantal formula for Littlewood-Schur functions, adapted from \cite{vanderjeugt}] \label{3_thm_det_formula_for_Littlewood-Schur}
Let $\calX$ and $\calY$ be sets of variables of length $n$ and $m$, respectively, so that the elements of $\calX \cup \calY$ are pairwise distinct. Let $\lambda$ be a partition with $(m,n)$-index $k$. If $k$ is negative, then $LS_\lambda(-\calX; \calY) = 0$; otherwise,
\begin{align*}
LS_\lambda(-\calX; \calY) ={} & \varepsilon(\lambda) \frac{\Delta(\calY; \calX)}{\Delta(\calX) \Delta(\calY)} \det \begin{pmatrix} \left( (x - y)^{-1} \right)_{\substack{x \in \calX \\ y \in \calY}} & \left( x^{\lambda_j + n - m - j} \right)_{\substack{x \in \calX \\ 1 \leq j \leq n - k}} \\ \left( y^{\lambda'_i + m - n - i} \right)_{\substack{1 \leq i \leq m - k \\ y \in \calY}} & 0\end{pmatrix}
\end{align*}
where $\varepsilon(\lambda) = (-1)^{\left|\lambda_{[n - k]} \right|} (-1)^{mk} (-1)^{k(k - 1)/2}$.
\end{thm}

Clearly, the sign $\varepsilon(\lambda)$ does not only depend on the partition $\lambda$, but also on the lengths of the sets of variables $\calX$ and $\calY$. However, the additional parameters $m$ and $n$ are omitted for simplicity of notation. Owing to this determinantal formula, it becomes a linear algebra exercise to check basic properties of Littlewood-Schur functions, such as that $LS_\lambda(-\calX; \calY)$ is a homogeneous polynomial of degree $|\lambda|$, which possesses the property that $LS_{\lambda'}(\calX; \calY) = LS_\lambda(\calY; \calX)$. In addition, it follows easily that Littlewood-Schur functions are symmetric in both sets of variables separately; more precisely, $LS_\lambda(-\calX; \emptyset) = \schur_\lambda(-\calX)$ and $LS_\lambda(\emptyset;\calY) = \schur_{\lambda'}(\calY)$. In fact, the solutions to these exercises can be found in Section~\ref{2_sec_proof_of_det_formula_for_LS}. Theorem~\ref{3_thm_det_formula_for_Littlewood-Schur} also allows us to give new elementary proofs for some old results, such a special case of Littlewood's formula for Littlewood-Schur functions whose partition is a square.

\begin{cor} [\cite{littlewood}] \label{3_cor_littlewood} Let $\calX$ and $\calY$ be sets of variables with $n$ and $m$ elements, respectively. For any integer $l \geq 0$,
\begin{align} \label{3_cor_littlewood_eq}
LS_{\left\langle (m + l)^n \right\rangle} (-\calX; \calY) = \elementary(-\calX)^l \Delta(\calY; \calX)
.
\end{align}
\end{cor}

In this proof we will omit obvious subscripts, such as $x \in \calX$, as they clutter up the block matrices unnecessarily. Throughout this chapter we take the liberty of omitting similarly intuitive subscripts during proofs whenever we feel that they are more hindrance than help.

\begin{proof} First suppose that the elements of $\calX \cup \calY$ are pairwise distinct. Given that the $(m,n)$-index of the partition $\left\langle (m + l)^n \right\rangle$ is 0,
\begin{align*}
LS_{\left\langle (m + l)^n \right\rangle}(-\calX; \calY) ={} & (-1)^{(m + l)n} \frac{\Delta(\calY; \calX)}{\Delta(\calX) \Delta(\calY)} \det \begin{pmatrix}  (x - y)^{-1}  & \left( x^{n + l - j} \right)_{\substack{1 \leq j \leq n}} \\ \left( y^{m - i} \right)_{\substack{1 \leq i \leq m}} & 0\end{pmatrix}
.
\intertext{The two off-diagonal blocks in the matrix are squares. Hence,}
LS_{\left\langle (m + l)^n \right\rangle}(-\calX; \calY) ={} & (-1)^{(m + l)n + mn} \frac{\Delta(\calY; \calX)}{\Delta(\calX) \Delta(\calY)} \det \left( x^{n + l - j} \right)_{\substack{1 \leq j \leq n}} \det \left( y^{m - i} \right)_{\substack{1 \leq i \leq m }}
.
\intertext{Using that the determinant is multilinear, we infer that both determinants are essentially Vandermonde determinants, which cancel with $\Delta(\calX)$ and $\Delta(\calY)$, respectively. This allows us to conclude that}
LS_{\left\langle (m + l)^n \right\rangle}(-\calX; \calY) ={} & (-1)^{ln} \Delta(\calY; \calX) \elementary(\calX)^l = \Delta(\calY; \calX) \elementary(-\calX)^l
.
\end{align*}
If the elements of $\calX \cup \calY$ are not pairwise distinct, the equality in \eqref{3_cor_littlewood_eq} is a direct consequence of the fact that both sides are polynomials in $\calX \cup \calY$, which agree on infinitely many points.
\end{proof}

\section{Laplace expansion for Littlewood-Schur functions} \label{3_section_lapalace_expansion_for_LS}
In this section we present two equalities on Littlewood-Schur functions which are based on the Laplace expansion of determinants. Let us quickly recall this classical result from linear algebra. For an $n \times n$ matrix $A = \left( a_{ij} \right)$ and two subsequences $I, J \subset [n]$, we need the following notation:
\begin{align*} A_{IJ} = \left(a_{ij}\right)_{\substack{i \in I \\ j \in J}} \text{ and its complement } A_{\bar{I}\bar{J}} = \left(a_{ij}\right)_{\substack{i \not\in I \\ j \not\in J}}
.
\end{align*}

\begin{lem} [Laplace expansion] \label{3_lem_laplace_expansion_matrix} Let $A$ be an $n \times n$ matrix. For a subsequence $K \subset [n]$, the determinant of $A$ can be expanded in the two following ways:
\begin{enumerate}
\item $\displaystyle \det(A) = \sum_{\substack{J \subset [n]: \\ l(J) = l(K)}} \varepsilon(\sigma(K, J)) \det \left(A_{KJ} \right) \det \left( A_{\bar{K}\bar{J}} \right)$
\item $\displaystyle \det(A) = \sum_{\substack{I \subset [n]: \\ l(I) = l(K)}} \varepsilon(\sigma(I, K)) \det \left(A_{IK} \right) \det \left( A_{\bar{I} \bar{K}} \right)$
\end{enumerate}
where $\varepsilon(\sigma(I,J))$ is the sign of the permutation $\sigma(I,J) \in S_n$ given by the conditions that $\sigma(I) = J$ (and thus $\sigma(\bar{I}) = \bar{J}$) and that $\sigma$ respects the relative order of the indices in $I$ and $J$ as well as in $\bar{I}$ and $\bar{J}$.
\end{lem}

\subsection{Two overlap identities for Littlewood-Schur functions} \label{3_section_two_overlap_identities_for_LS_functions}
Before stating and proving the first and the second overlap identity, we formally introduce the notion of overlapping two partitions.

\begin{defn} [overlap] \label{3_defn_overlap} For any positive integer $n$, we define the partition $\rho_n$ by \label{symbol_rho_n} $\rho_n = (n - 1, \dots, 1, 0)$; for $n = 0$, we use the convention that $\rho_0$ is the empty partition. Let $\mu$, $\nu$ be partitions of length at most $m$ and $n$, respectively. The $(m,n)$-overlap of $\mu$ and $\nu$, denoted \label{symbol_overlap} $\mu \star_{m, n} \nu$, is the partition that satisfies 
\begin{align} \label{3_condition_overlap}
\mu \star_{m,n} \nu + \rho_{m + n} \sorteq (\mu + \rho_m) \cup (\nu + \rho_n)
\end{align}
if it exists; otherwise, we set $\mu \star_{m, n} \nu = \infty$. Here, $\infty$ is just a symbol with the property that $LS_\infty(\calX; \calY) = 0$ for any sets of variables $\calX$ and $\calY$, \textit{i.e.}\ it symbolizes a partition that contains the rectangle $\left\langle (m + 1)^{n + 1} \right\rangle$ for any pair of non-negative integers $m$ and $n$.

If the condition in \eqref{3_condition_overlap} is satisfied, then the sign of the overlap, denoted $\varepsilon_{m,n}(\mu,\nu)$, is just the sign of the corresponding sorting permutation; otherwise, we set the sign equal to 1. This notion is well defined because, unless the $(m,n)$-overlap of $\mu$ and $\nu$ is infinity, the sequence on the left-hand side in \eqref{3_condition_overlap} is strictly decreasing. If the parameters $m, n$ are clear from the context, they are sometimes omitted. 
\end{defn}

\begin{thm} [first overlap identity] \label{3_thm_laplace_expansion_LS_cut_before_index} Let $\calX$ and $\calY$ be sets of variables with $n$ and $m$ elements, respectively, so that $\calX$ consists of pairwise distinct elements. Let $\lambda$ be a partition with $(m,n)$-index $k$. If $\lambda_{[n - k]}$ is the $(l, n - k - l)$-overlap of $\mu$ and $\nu$ for some integer $0 \leq l \leq \min\{n - k, n\}$ and some partitions $\mu$ and $\nu$, then
\begin{align} \label{3_thm_laplace_expansion_LS_cut_before_index_eq}
LS_{\lambda} (-\calX; \calY) ={} & \sum_{\makebox[48pt]{$\substack{\calS, \calT \subset \calX: \\ \calS \cup_{l, n - l} \calT \sorteq \calX}$}} \frac{\varepsilon_{l, n - k - l}(\mu, \nu) LS_{\mu + \left\langle k^l \right\rangle}(-\calS; \calY) LS_{\nu \cup \lambda_{(n + 1 - k, n + 2 - k, \dots)}}(- \calT; \calY)}{\Delta(\calT; \calS)} .
\end{align}
\end{thm}

\begin{proof} In a first step, suppose that the variables in $\calX \cup \calY$ are also pairwise distinct, making the determinantal formula for Littlewood-Schur functions applicable. The proof boils down to writing all Littlewood-Schur functions as determinants, and then using Laplace expansion to show that the left-hand side and the right-hand side in \eqref{3_thm_laplace_expansion_LS_cut_before_index_eq} are indeed equal.

Using the definitions of both overlap and index, we determine the relevant indices of the partitions that appear on the right-hand side in \eqref{3_thm_laplace_expansion_LS_cut_before_index_eq}: The $(m,l)$-index of the partition $\mu + \left\langle k^l \right\rangle$ is 0 since $\mu_l + k \geq \lambda_{n - k} + k \geq m - k + k = m$. Furthermore, the fact that
$$\nu_{n - l - k} \geq \lambda_{n - k} \geq m - k \:\text{ and }\: \left( \nu \cup \lambda_{(n + 1 - k, \dots)} \right)_{n - l - k  + 1} = \lambda_{n + 1 - k} \leq m - k$$
implies that the $(m, n - l)$-index of $\nu \cup \lambda_{(n + 1 - k, \dots)}$ is still $k$. In particular, both sides of the equation in \eqref{3_thm_laplace_expansion_LS_cut_before_index_eq} vanish whenever $k$ is negative. Otherwise, Theorem~\ref{3_thm_det_formula_for_Littlewood-Schur} states that the right-hand side in \eqref{3_thm_laplace_expansion_LS_cut_before_index_eq} equals
\begin{align}
\begin{split} 
\RHS ={} & \sum_{\makebox[53pt]{$\substack{\calS, \calT \subset \calX: \\ \calS \cup_{l, n - l} \calT \sorteq \calX}$}} \frac{ \varepsilon_{l, n - k - l}(\mu, \nu)  \varepsilon\left(\mu + \left\langle k^l \right\rangle \right) \varepsilon \left(\nu \cup \lambda_{(n + 1 - k, n + 2 - k, \dots)} \right) \Delta(\calY; \calS) \Delta(\calY; \calT)}{\Delta(\calT; \calS) \Delta(\calS) \Delta(\calT) \Delta(\calY)^2} 
\notag \end{split}\\
\begin{split} \label{3_in_proof_first_overlap_id_first_det}
& \times \det \begin{pmatrix} \left( (s - y)^{-1} \right) & \left( s^{\mu_j + k + l - m - j} \right)_{1 \leq j \leq l} \\ \left( y^{\left( \mu + \left\langle k^l \right\rangle \right)'_i + m - l - i} \right)_{1 \leq i \leq m} & 0 \end{pmatrix}
\end{split} \\
\begin{split} \label{3_in_proof_first_overlap_id_second_det}
& \times \det \begin{pmatrix} \left( (t - y)^{-1} \right) & \left( t^{\nu_j + n - l - m - j} \right)_{1 \leq j \leq n - l - k} \\ \left( y^{\left( \nu \cup \lambda_{(n + 1 - k, \dots)} \right)'_i + m - n + l - i} \right)_{1 \leq i \leq m - k} & 0 \end{pmatrix}
.
\end{split}
\end{align}
Focusing on the bottom-left block of the matrix in line \eqref{3_in_proof_first_overlap_id_second_det}, we observe that for $1 \leq i \leq m - k$, the exponent of $y$ in the $i$-th row is 
\begin{align*}
\text{exponent}_i(y) \defeq{} & \left( \nu \cup \lambda_{(n + 1 - k, n + 2 - k, \dots)} \right)'_i + l + m - n - i \\
={} & \nu'_i + l + \left(\lambda_{(n + 1 - k, n + 2 - k, \dots)} \right)'_i + m - n - i
.
\intertext{As $(m - k, n - l - k) \in \nu$ and $\nu'_1 = l(\nu) \leq n - l - k$, we infer that $\nu'_i = n - l - k$ for $1 \leq i \leq m - k$:}
\text{exponent}_i(y) ={} & n - k + \left(\lambda_{(n + 1 - k, n + 2 - k, \dots)} \right)'_i + m - n - i
.
\intertext{Finally, the fact that $(m - k, n - k) \in \lambda$ allows us to simplify this expression to}
\text{exponent}_i(y) ={} & \lambda'_i + m - n - i
.
\end{align*} 
Focusing on the determinant in line \eqref{3_in_proof_first_overlap_id_first_det}, we first remark that for all $1 \leq i \leq m$, $\left( \mu + \left\langle k^l \right\rangle \right)'_i = l$ because $(m,l) \in \mu + \left\langle k^l \right\rangle$ and $\mu'_1 = l(\mu) \leq l$. Hence,
\begin{align*}
\text{determinant} \defeq{} & \det \begin{pmatrix} \left( (s - y)^{-1} \right) & \left( s^{\mu_j + k + l - m - j} \right)_{1 \leq j \leq l} \\ \left( y^{\left( \mu + \left\langle k^l \right\rangle \right)'_i + m - l - i} \right)_{1 \leq i \leq m} & 0 \end{pmatrix} \\
={} & \det \begin{pmatrix} \left( (s - y)^{-1} \right) & \left( s^{\mu_j + k + l - m - j} \right)_{1 \leq j \leq l} \\ \left( y^{m - i} \right)_{1 \leq i \leq m} & 0 \end{pmatrix}
.
\intertext{As the off-diagonal blocks are squares, expanding the determinant yields}
\text{determinant} ={} & (-1)^{ml} \det \left( s^{\mu_j + k + l - m - j} \right)_{1 \leq j \leq l} \det \left( y^{m - i} \right)_{1 \leq i \leq m} \\
={} & (-1)^{ml} \det \left( s^{\mu_j + k + l - m - j} \right)_{1 \leq j \leq l}  \Delta(\calY)
.
\end{align*}
In sum, the right-hand side in \eqref{3_thm_laplace_expansion_LS_cut_before_index_eq} is equal to the following simplified determinantal expression:
\begin{align}
\begin{split}
\RHS ={} & (-1)^{ml} \varepsilon_{l, n - k - l}(\mu, \nu) \varepsilon \left(\mu + \left\langle k^l \right\rangle \right) \varepsilon \left( \nu \cup \lambda_{(n + 1 - k, \dots)} \right) \frac{\Delta(\calY; \calX)}{\Delta(\calY)} \\
& \times \sum_{\substack{\calS, \calT \subset \calX: \\ \calS \cup_{l, n - l} \calT \sorteq \calX}} \frac{1}{\Delta(\calT; \calS) \Delta(\calS) \Delta(\calT)} \det \left( s^{\mu_j + k + l - m - j} \right)_{1 \leq j \leq l} \\
& \times \det \begin{pmatrix} \left( (t - y)^{-1} \right) & \left( t^{\nu_j + n - l - m - j} \right)_{1 \leq j \leq n - l - k} \\ \left( y^{\lambda'_i + m - n - i} \right)_{1 \leq i \leq m - k} & 0 \end{pmatrix} 
.
\notag \end{split}
\intertext{Notice that $\Delta(\calT; \calS) \Delta(\calS) \Delta(\calT)$ is equal to $\Delta(\calX)$ up to a sign that measures the number of inversions with respect to the order on $\calX$. Hence, Lemma~\ref{3_lem_laplace_expansion_matrix} allows us to view this sum over $\calS$ and $\calT$ as a Laplace expansion of one determinant:}
\begin{split} \label{3_in_proof_first_overlap_id_for_cor}
\RHS ={} & (-1)^{ml + l(n + m - l)} \varepsilon_{l, n - k - l}(\mu, \nu) \varepsilon\left(\mu + \left\langle k^l \right\rangle \right) \varepsilon \left( \nu \cup \lambda_{(n + 1 - k, \dots)} \right) \\
& \times \frac{\Delta(\calY; \calX)}{\Delta(\calX) \Delta(\calY)} \\
& \times \det \begin{pmatrix} \left( (x - y)^{-1} \right) & \!\!\!\!\!\! \left( x^{\mu_j + k + l - m - j} \right)_{1 \leq j \leq l} & \!\! \left( x^{\nu_j + n - l - m - j} \right)_{1 \leq j \leq n - l - k} \\ \left( y^{\lambda'_i + m - n - i} \right)_{\substack{1 \leq i \leq m - k}} & \!\!\!\!\!\!0 & \!\!0 \end{pmatrix} \!
.
\end{split}
\intertext{The condition that $\lambda_{[n - k]} = \mu \star_{l, n - k - l} \nu$ entails that permuting the $n - k$ last columns results in}
\begin{split}
\RHS ={} & (-1)^{l(n - l)} \varepsilon\left(\mu + \left\langle k^l \right\rangle \right) \varepsilon \left( \nu \cup \lambda_{(n + 1 - k, \dots)} \right) \varepsilon(\lambda) LS_\lambda(-\calX; \calY)
\notag \end{split}
\end{align}
through another application of Theorem~\ref{3_thm_det_formula_for_Littlewood-Schur}. Under the additional assumption that the variables in $\calX \cup \calY$ are pairwise distinct, the equality in \eqref{3_thm_laplace_expansion_LS_cut_before_index_eq} is thus an immediate consequence of the fact that the signs cancel each other out. If we permit that $\calX \cup \calY$ contains repetitions, it suffices to remark that for a fixed set of variables $\calX$, both sides of the equation in \eqref{3_thm_laplace_expansion_LS_cut_before_index_eq} are polynomials in $\calY$.
\end{proof}

\begin{rem} \label{3_rem_vanderjeugt_counterexample} In case the sorting of the overlap is the identity, the equation in \eqref{3_thm_laplace_expansion_LS_cut_before_index_eq} reads
\begin{align} \label{3_vanderjeugt_counterexample}
LS_{\lambda} (-\calX; \calY) ={} & \sum_{\substack{\calS, \calT \subset \calX: \\ \calS \cup_{l, n - l} \calT \sorteq \calX}} \frac{LS_{\lambda_{[l]} + \left\langle (n - l)^l \right\rangle}(-\calS; \calY) LS_{\lambda_{(l + 1, l + 2, \dots)}}(- \calT; \calY)}{\Delta(\calT; \calS)}
\end{align}
for any integer $0 \leq l \leq \min\{n - k, n \}$.

This specialization of Theorem~\ref{3_thm_laplace_expansion_LS_cut_before_index} is a slight generalization of Proposition 8 in \cite{bump06}. In fact, they prove equality \eqref{3_vanderjeugt_counterexample} under the assumption that $\lambda_l \geq \lambda_{l + 1} + m$ (and $l \leq n$). Their assumption in stronger than ours: it entails that $l \leq n - k$ where $k$ stands for the $(m,n)$-index of $\lambda$. Indeed, if $k \geq 1$, then $\lambda_i \leq m - k < m$ for all $i > n - k$ but $\lambda_l \geq m$. Bump and Gamburd's proof is an induction over $m$ based on Pieri's formula.

Independently of \cite{bump06}, Lemma 5.4 in \cite{vanderjeugt} also states equality \eqref{3_vanderjeugt_counterexample} but without any assumptions on $\lambda$. Moreover, their proof is also an induction based on Pieri's formula. However, it is possible to construct counter-examples to their claim when $l > n - k$: let us fix $n = 2$ and $m = 3$, then $\lambda = (1,1,1)$ has $(m,n)$-index $k = 2$. Setting $l = 1$, one computes that
\begin{align*}
LS_{(1,1,1)}(-x_1, -x_2; y_1, y_2, y_3) ={} & \sum_{\makebox[36.5pt]{$\substack{\calS, \calT \subset (x_1, x_2): \\ \calS \cup_{1,1} \calT \sorteq (x_1, x_2)}$}} \frac{LS_{(1) + (1)}(-\calS; y_1, y_2, y_3) LS_{(1,1)} (-\calT; y_1, y_2, y_3)}{\Delta(\calT; \calS)}\\
& + y_1y_2y_3
\end{align*}
despite the fact that equation \eqref{3_vanderjeugt_counterexample} does not predict the additional term $y_1y_2y_3$. The main theorem in \cite{vanderjeugt} still holds because they provide several independent proofs.
\end{rem}

\begin{cor} \label{3_cor_laplace_expansion_LS_mu_has_maximal_index} Let $\mu, \nu$ be partitions and $\calX$, $\calY$ sets of variables with $n$ and $m$ elements, respectively, so that the elements of $\calX$ are pairwise distinct. Fix an integer $l(\mu) \leq l \leq n$ and let $k$ denote the $(m, n - l)$-index of $\nu$. If $l \leq n - k$ and the $(m,l)$-index of $\mu + \left\langle k^l \right\rangle$ is $0$, then
\begin{align}
\begin{split} \label{3_cor_laplace_expansion_LS_mu_has_maximal_index_eq}
& LS_{\left( \mu \star_{l, n - l - k} \nu_{[n - l - k]} \right) \cup \nu_{(n + 1 - l - k, \dots)}} (-\calX; \calY) = \\
& \sum_{\substack{\calS, \calT \subset \calX: \\ \calS \cup_{l, n - l} \calT \sorteq \calX}} \frac{\varepsilon \left(\mu, \nu_{[n - l - k]} \right) LS_{\mu + \left\langle k^l \right\rangle}(-\calS; \calY) LS_\nu(- \calT; \calY)}{\Delta(\calT; \calS)}
.
\end{split}
\end{align} 
\end{cor}

\begin{proof} First suppose that there exists a partition $\lambda$ with the property that $$\lambda = \left( \mu \star_{l, n - l - k} \nu_{[n - l - k]} \right) \cup \nu_{(n + 1 - l - k, \dots)}.$$ It is easy to check that the $(m, n)$-index of $\lambda$ is $k$. Hence, the equality in \eqref{3_cor_laplace_expansion_LS_mu_has_maximal_index_eq} is a direct consequence of Theorem~\ref{3_thm_laplace_expansion_LS_cut_before_index}. 

Second suppose that $\mu \star_{l, n - l - k} \nu_{[n - l - k]} = \infty$. On the one hand, the left-hand side in \eqref{3_cor_laplace_expansion_LS_mu_has_maximal_index_eq} vanishes by definition. On the other hand, condition \eqref{3_condition_overlap} implies that there exist $1 \leq p \leq l$ and $1 \leq q \leq n - l - k$ such that $\mu_p + l - p = \alpha = \nu_q + n - l - k - q$. If $k$ is negative, the right-hand side in \eqref{3_cor_laplace_expansion_LS_mu_has_maximal_index_eq} also vanishes. If $k \geq 0$, the right-hand side in \eqref{3_cor_laplace_expansion_LS_mu_has_maximal_index_eq} is equal to the following determinantal expression, owing to the arguments used to justify that the right-hand side in \eqref{3_thm_laplace_expansion_LS_cut_before_index_eq} is equal to \eqref{3_in_proof_first_overlap_id_for_cor}:
\begin{align*}
\RHS ={} & \pm \frac{\Delta(\calY; \calX)}{\Delta(\calX) \Delta(\calY)} \\
& \times \det \begin{pmatrix} \left( (x - y)^{-1} \right) & \left( x^{\mu_j + k + l - m - j} \right)_{\substack{1 \leq j \leq l}} & \left( x^{\nu_j + n - l - m - j} \right)_{\substack{1 \leq j \leq n - l - k}} \\ \ast & 0 & 0\end{pmatrix}
\end{align*}
where $\ast$ stands for some block that is not relevant here. 
Given that 
\begin{align*}
\mu_p + k + l - m - p = \alpha + k - m = \nu_q + n - l - m - q 
,\end{align*}
we conclude that the matrix contains two identical columns, which means that the right-hand side in \eqref{3_cor_laplace_expansion_LS_mu_has_maximal_index_eq} also vanishes.
\end{proof}

\begin{thm} [second overlap identity] \label{3_thm_lapalace_expansion_new_LS} Let $0 \leq l \leq \min\{n - k, n\}$. Let $\calS$, $\calT$ and $\calY$ be sets containing $l$, $n - l$ and $m$ variables, respectively, so that $\Delta(\calY) \neq 0$ and $\Delta(\calS; \calT) \neq 0$. Suppose that $k$ is the $(m,n)$-index of a partition $\lambda$, then
\begin{align}
\begin{split} \label{3_thm_lapalace_expansion_new_LS_eq}
LS_{\lambda} (-(\calS \cup \calT); \calY) ={} & \sum_{p = 0}^{\min\{l,m\}} \sum_{\substack{\calU, \calV \subset \calY: \\ \calU \cup_{p, m - p} \calV \sorteq \calY}} \hspace{10pt} \sum_{\substack{\mu, \nu: \\ \mu \star_{l - p, n - k - l + p} \nu = \lambda_{[n - k]}}} \frac{\Delta(\calV; \calS) \Delta(\calT; \calU)}{\Delta(\calV; \calU) \Delta(\calT; \calS)} \\
& \times \varepsilon(\mu, \nu) LS_{\mu - \left\langle (m - k)^{l - p} \right\rangle} (-\calS; \calU) LS_{\nu \cup \lambda_{(n + 1 - k, n + 2 - k, \dots)}} (-\calT; \calV)
.
\end{split}
\end{align}
\end{thm}

It is worth noting that for any partition $\nu$ that appears in the sum the union $\nu \cup \lambda_{(n + 1 - k, \dots)}$ is again a partition. Indeed, the fact that the $(l - p, n - k - l + p)$-overlap of $\mu$ and $\nu$ is $\lambda_{[n - k]}$ implies that there exists an index $1 \leq i \leq n - k$ so that $\nu_{n - k - l + p} = \lambda_i + n - k - i \geq \lambda_{n - k}.$

\begin{proof} We remark that if $\mu \star_{l - p, n - k - l + p} \nu = \lambda_{[n - k]}$, then the $(m - p, n - l)$-index of the partition $\nu \cup \lambda_{(n + 1 - k, \dots)}$ is $k - p$. Indeed, recalling that $k$ is the $(m,n)$-index of $\lambda$ allows us to infer that
\begin{align*} \left( \nu \cup \lambda_{(n + 1 - k, \dots)} \right)_{n - l - (k - p)} ={} & \nu_{n - k - l + p} \geq \lambda_{n - k} \geq m - k \intertext{ and } \left( \nu \cup \lambda_{(n + 1 - k, \dots)} \right)_{n - l - (k - p) + 1} ={} & \lambda_{n - k + 1} \leq m - k
.
\end{align*}
Therefore, both sides of the equation in \eqref{3_thm_lapalace_expansion_new_LS_eq} vanish whenever $k$ is negative. To prove equality for non-negative $k$, we first suppose that the variables in $\calS \cup \calT \cup \calY$ are pairwise distinct. According to Theorem~\ref{3_thm_det_formula_for_Littlewood-Schur}, the left-hand side in \eqref{3_thm_lapalace_expansion_new_LS_eq} can be written as
\begin{align*}
\LHS ={} & \varepsilon(\lambda) \frac{\Delta(\calY; \calS \cup \calT)}{\Delta(\calS \cup \calT) \Delta(\calY)} \det \begin{pmatrix} \left( (s - y)^{-1} \right) & \left( s^{\lambda_j + n - m - j} \right)_{\substack{1 \leq j \leq n - k}} \\ \left( (t - y)^{-1} \right) & \left( t^{\lambda_j + n - m - j} \right)_{\substack{1 \leq j \leq n - k}} \\ \left( y^{\lambda'_i + m - n - i} \right)_{\substack{1 \leq i \leq m - k}} & 0\end{pmatrix}
.
\intertext{Let us expand the determinant along the first $l$ rows by applying Lemma~\ref{3_lem_laplace_expansion_matrix}. This results in a signed sum over all subsets of the $m$ first columns and the $n - k$ last columns, respectively, that contain exactly $l$ columns in total.  The sum over the $m$ first columns corresponds to dividing $\calY$ into two subsequences, while the sum over the $n - k$ last columns (essentially) corresponds to dividing $\lambda$ into two subpartitions:}
\LHS ={} & \varepsilon(\lambda) \frac{\Delta(\calY; \calS) \Delta(\calY; \calT)}{\Delta(\calS \cup \calT) \Delta(\calY)} \\
& \times \sum_{p = 0}^{\min\{l,m\}} \sum_{\substack{\calU, \calV \subset \calY: \\ \calU \cup_{p, m - p} \calV \sorteq \calY}} \sum_{\substack{\mu, \nu: \\ \mu \star_{l - p, n - k - l + p} \nu = \lambda_{[n - k]}}} \frac{\varepsilon(\mu, \nu) \Delta(\calY) (-1)^{(m - p)(l - p)}}{\Delta(\calU) \Delta(\calV) \Delta(\calU; \calV)} \\
& \times \det \begin{pmatrix} \left( (s - u)^{-1} \right) & \left( s^{(\mu_j - (m - k)) + l - p - j} \right)_{\substack{1 \leq j \leq l - p}} \end{pmatrix} \\
& \times \det \begin{pmatrix} \left( (t - v)^{-1} \right) & \left( t^{\nu_j + (n - l) - (m - p) - j} \right)_{\substack{1 \leq j \leq n - k - l + p}} \\ \left( v^{\lambda'_i + m - n - i} \right)_{\substack{1 \leq i \leq m - k}} & 0 \end{pmatrix}
.
\intertext{Notice that $\mu - \left\langle (m - k)^{l - p} \right\rangle$ is a partition since $\mu_{l - p} - (m - k) \geq \lambda_{n - k} - (m - k) \geq 0$. Moreover, the same inequality shows that the $(p,l)$-index of this partition is $p$. In addition, the fact that $(m - k, n - k) \in \lambda$ and $(m - k, n - l - k + p) \in \nu$ entails that for $1 \leq i \leq m - k$, $\lambda'_i = \left( \nu \cup \lambda_{(n - k + 1, \dots)} \right)'_i + l - p$. Hence, Theorem~\ref{3_thm_det_formula_for_Littlewood-Schur} states that}
\LHS ={} & \varepsilon(\lambda) \frac{\Delta(\calY; \calS) \Delta(\calY; \calT)}{\Delta(\calS \cup \calT)} \\
& \times \sum_{p = 0}^{\min\{l,m\}} \sum_{\substack{\calU, \calV \subset \calY: \\ \calU \cup_{p, m - p} \calV \sorteq \calY}} \sum_{\substack{\mu, \nu: \\ \mu \star_{l - p, n - k - l + p} \nu = \lambda_{[n - k]}}} \frac{\varepsilon(\mu, \nu) (-1)^{(m - p)(l - p)}}{\Delta(\calU) \Delta(\calV) \Delta(\calU; \calV)} \\
& \times \varepsilon\left(\mu - \left\langle (m - k)^{l - p} \right\rangle \right) \frac{\Delta(\calS) \Delta(\calU)}{\Delta(\calU; \calS)} LS_{\mu - \left\langle (m - k)^{l - p} \right\rangle}(-\calS; \calU) \\
& \times \varepsilon\left(\nu \cup \lambda_{(n - k + 1, \dots)} \right) \frac{\Delta(\calT) \Delta(\calV)}{\Delta(\calV; \calT)} LS_{\nu \cup \lambda_{(n - k + 1, \dots)}}(-\calT; \calV)
.
\end{align*}
Combining the different factors in front of the Littlewood-Schur functions gives the desired equality. If we weaken the assumption that $\calS \cup \calT \cup \calY$ consist of pairwise distinct elements to the condition that $\Delta(\calY) \neq 0$ and $\Delta(\calS; \calT) \neq 0$, the equality in \eqref{3_thm_lapalace_expansion_new_LS_eq} still holds as $\Delta(\calT; \calS) \times \LHS$ and $\Delta(\calT; \calS) \times \RHS$ are polynomials in $\calS \cup \calT$ for any fixed set of variables $\calY$.
\end{proof}

\subsection{Laplace expansion for Schur functions}
Recall that any Schur function may be viewed as a specialization of a Littlewood-Schur function, given that $\schur_\lambda(\calX) = LS_\lambda(\calX; \emptyset)$. The first and second overlap identities look much neater when specialized to Schur functions. The primary reason why these statements simplify so drastically is that the $(0,n)$-index of any partition with length less than $n$ is equal to $0$. Specializing Corollary~\ref{3_cor_laplace_expansion_LS_mu_has_maximal_index} to Schur functions gives the following identity for Schur functions. 

\begin{cor} [first overlap identity for Schur functions, \cite{dehaye12}] \label{3_cor_laplace_expansion_POD} Let the set $\calX$ consist of $m + n$ pairwise distinct variables. For any pair of partitions $\mu$ and $\nu$ of lengths at most $m$ and $n$, respectively, it holds that
\begin{align*}
\schur_{\mu \star_{m,n} \nu} (\calX) = \sum_{\substack{\calS, \calT \subset \calX: \\ \calS \cup_{m,n} \calT \sorteq \calX}} \frac{\varepsilon(\mu, \nu) \schur_\mu(\calS) \schur_\nu(\calT)}{\Delta(\calS; \calT)}
.
\end{align*}
\end{cor}

Corollary~\ref{3_cor_laplace_expansion_POD} is nothing more than a reformulation of Lemma 5 in \cite{dehaye12}. The only notable difference is that Dehaye does not introduce the notion of overlapping two partitions. In fact, Dehaye's lemma was the starting point for this entire chapter. Interestingly, the case $\mu_{m} \geq \nu_1 + n$ (i.e. when the sorting algorithm is the identity) appears independently in both \cite{bump06} and \cite{vanderjeugt} with essentially identical proofs.

The following corollary to Theorem~\ref{3_thm_lapalace_expansion_new_LS} is obtained by setting $\calY = \emptyset$.

\begin{cor} [second overlap identity for Schur functions] \label{3_cor_laplace_expansion_schur_new} Let $\lambda$ be a partition and let $\calS$ and $\calT$ be sets consisting of $m$ and $n$ variables, respectively. If $\Delta(\calS; \calT) \neq 0$, then
\begin{align*}
\schur_\lambda(\calS \cup \calT) ={} & \sum_{\substack{\mu, \nu: \\ \mu \star_{m,n} \nu = \lambda}} \frac{\varepsilon(\mu, \nu) \schur_\mu(\calS) \schur_\nu(\calT)}{\Delta(\calS; \calT)}
.
\end{align*}
\end{cor}

\section{Visualizing the overlap of two partitions} \label{3_section_visualizations_of_overlap}
In this section we present two visual interpretations for overlapping partitions. Both visualizations characterize the set of all pairs of partitions whose overlaps are equal by identifying their Ferrers diagrams with some part of the diagram of a so-called staircase walk.

\begin{defn} [staircase walks]  A staircase walk is a lattice walk that only uses west and south steps. Let \label{symbol_staircase_walks_in_rectangle} $\mathfrak{P}(n,m)$ be the set of all staircase walks going from the top-right to the bottom-left corner of an $n \times m$ rectangle. For $\pi \in \mathfrak{P}(n,m)$, \label{symbol_partition_associated_to_pi_mu} $\mu(\pi) \subset \langle n^m \rangle$ denotes the partition whose Ferrers diagram lies above $\pi$, while \label{symbol_partition_associated_to_pi_nu} $\nu(\pi) \subset \langle n^m \rangle$ denotes the partition whose Ferrers diagram (rotated by 180 degrees) lies below $\pi$. In addition, $V(\pi)$ and \label{symbol_sequence_associated_to_pi} $H(\pi)$ denote the sequences of all times of vertical and horizontal steps of $\pi$, respectively.
\end{defn}

\begin{ex*} The staircase walk $\pi \in \mathfrak{P}(6,3)$ cuts the $6 \times 3$ rectangle into the following two partitions: $\mu(\pi) = (5,5,2)$ and $\nu(\pi) = (4,1,1)$.
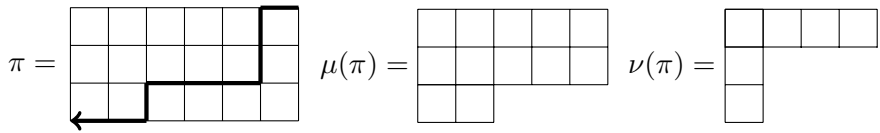
\begin{center}
\begin{tikzpicture} 
\node(pi) at (-0.5, 0.75) {$\pi =$};
\draw[step=0.5cm, thin] (0, 0) grid (3, 1.5);
\draw[ultra thick, ->] (1, 0) -- (0, 0);
\draw[ultra thick] (1, 0) -- (1, 0.5);
\draw[ultra thick] (1, 0.5) -- (2.5, 0.5);
\draw[ultra thick] (2.5, 0.5) -- (2.5, 1.5);
\draw[ultra thick] (2.5, 1.5) -- (3, 1.5);
\end{tikzpicture}
\begin{tikzpicture} \label{3_ferrers_diagram}
\node(mu) at (-0.7, 0.75) {$\mu(\pi) =$};
\draw[step=0.5cm, thin] (0, 0) grid (1, 0.5);
\draw[step=0.5cm, thin] (0, 0.5) grid (2.5, 1.5);
\end{tikzpicture}
\begin{tikzpicture}
\node(nu) at (-0.7, 0.75) {$\nu(\pi) =$};
\draw[step=0.5cm, thin] (0, 0) grid (0.5, 1.5);
\draw[step=0.5cm, thin] (0, 0.99) grid (2, 1.5);
\end{tikzpicture}
\end{center}
We further see that $V(\pi) = (2,3,7)$ and $H(\pi) = (1,4,5,6,8,9)$.
\end{ex*}

\subsection{Labeled staircase walks}
Our first visualization for the overlap of two partitions is based on the following lemma, which provides a non-visual description for the two partitions $\mu(\pi)$ and $\nu(\pi)$ associated to any staircase walk $\pi$.  

\begin{lem} \label{3_lem_macdonald_page_3} Let $\pi \in \mathfrak{P}(n,m)$, then $\mu(\pi)$ and $\nu(\pi)$ are the unique partitions that satisfy the following two equations:
$$\mu(\pi) + \rho_m = (\rho_{m + n})_{V(\pi)} \:\text{ and }\: \nu(\pi)' + \rho_n = (\rho_{m + n})_{H(\pi)}.$$
\end{lem}

\begin{proof} This proof reproduces arguments from \cite[p.~3]{mac}. Consider the following diagram of a staircase walk $\pi \in \mathfrak{P}(n,m)$ and the corresponding partition $\mu(\pi)$ (colored in gray):
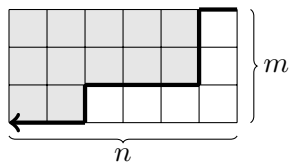
\begin{center}
\begin{tikzpicture} 
\fill[black!10!white] (0, 0) rectangle (1, 1.5);
\fill[black!10!white] (1, 0.5) rectangle (2.5, 1.5);
\draw[step=0.5cm, thin] (0, 0) grid (3, 1.5);
\draw[ultra thick, ->] (1, 0) -- (0, 0);
\draw[ultra thick] (1, 0) -- (1, 0.5);
\draw[ultra thick] (1, 0.5) -- (2.5, 0.5);
\draw[ultra thick] (2.5, 0.5) -- (2.5, 1.5);
\draw[ultra thick] (2.5, 1.5) -- (3, 1.5);
\draw[decoration={brace, raise=5pt},decorate] (3, 1.5) -- node[right=6pt] {$m$} (3, 0);
\draw[decoration={brace, raise=5pt, mirror},decorate] (0, 0) -- node[below=6pt] {$n$} (3, 0);
\node at (0, 1.7) {};
\end{tikzpicture}
\end{center}
We see that for $1 \leq i \leq m$, $V(\pi)_i = i + n - (\mu(\pi))_i$. Let us illustrate this observation for $i = 2$:
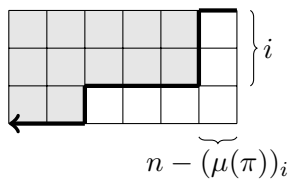
\begin{center}
\begin{tikzpicture} 
\fill[black!10!white] (0, 0) rectangle (1, 1.5);
\fill[black!10!white] (1, 0.5) rectangle (2.5, 1.5);
\draw[step=0.5cm, thin] (0, 0) grid (3, 1.5);
\draw[ultra thick, ->] (1, 0) -- (0, 0);
\draw[ultra thick] (1, 0) -- (1, 0.5);
\draw[ultra thick] (1, 0.5) -- (2.5, 0.5);
\draw[ultra thick] (2.5, 0.5) -- (2.5, 1.5);
\draw[ultra thick] (2.5, 1.5) -- (3, 1.5);
\draw[decoration={brace, raise=5pt},decorate] (3, 1.5) -- node[right=6pt] {$i$} (3, 0.5);
\draw[decoration={brace, raise=5pt, mirror},decorate] (2.5, 0) -- node[below=6pt] {$n - (\mu(\pi))_i$} (3, 0);
\end{tikzpicture}
\end{center}
In consequence,
$$m + n - V(\pi)_i = m + n - (i + n - (\mu(\pi))_i) = m - i + (\mu(\pi))_i,$$
which shows the first equality. By symmetry the analogue holds for $\nu(\pi)'$.
\end{proof}

\begin{prop} \label{3_prop_visual_interpretation_overlap} For a fixed partition $\lambda$ of length at most $m + n$, there is a 1-to-1 correspondence between $\mathfrak{P}(n,m)$ and $\{(\mu, \nu): \mu \star_{m,n} \nu = \lambda\}$
given by
\begin{align} \label{3_eq_visual_interpretation_overlap_map}
\pi \mapsto \left(\mu(\pi) + \lambda_{V(\pi)}, \nu(\pi)' + \lambda_{H(\pi)}\right).
\end{align}
Moreover, $\varepsilon_{m,n} \left(\mu(\pi) + \lambda_{V(\pi)}, \nu(\pi)' + \lambda_{H(\pi)}\right) = (-1)^{|\nu(\pi)|} = (-1)^{mn - |\mu(\pi)|}$.
\end{prop}

\begin{proof} By Lemma~\ref{3_lem_macdonald_page_3},
\begin{multline*}
\left( \mu(\pi) + \lambda_{V(\pi)} + \rho_m \right) \cup \left( \nu(\pi)' + \lambda_{H(\pi)} + \rho_n\right) \\ = \left( \rho_{m + n} + \lambda \right)_{V(\pi)} \cup \left( \rho_{m + n} + \lambda \right)_{H(\pi)} \sorteq \lambda + \rho_{m + n},
\end{multline*} 
which implies that the map given in \eqref{3_eq_visual_interpretation_overlap_map} is well defined. We further see that the sign of the sorting permutation, say $\sigma$, only depends on the timing of the vertical steps of $\pi$. In fact, the identity permutation corresponds to the first $m$ steps of $\pi$ being vertical, and thus $\nu(\pi)$ being the empty partition. Removing a box from the Ferrers diagram of $\mu(\pi)$ and adding it to the diagram of $\nu(\pi)$ corresponds to composing $\sigma$ with a transposition, \textit{i.e.}\ multiplying its sign by $(-1)$.

We show that the map in \eqref{3_eq_visual_interpretation_overlap_map} is bijective by giving its inverse. Let $\mu$, $\nu$ be a pair of partitions whose $(m,n)$-overlap is $\lambda$. By definition, there exists a pair of subsequences $V$, $H \subset [m + n]$ with $V \cup_{m, n} H = [m + n]$ such that
$$\mu + \rho_m = \left( \lambda + \rho_{m + n} \right)_V \:\text{ and }\: \nu + \rho_n = \left( \lambda + \rho_{m + n} \right)_H
.$$
If $\pi \in \mathfrak{P}(n,m)$ denotes the staircase walk determined by the condition that $V(\pi) = V$ and $H(\pi) = H$, then Lemma~\ref{3_lem_macdonald_page_3} implies that $\mu = \mu(\pi) + \lambda_{V(\pi)}$ and $\nu = \nu(\pi)' + \lambda_{H(\pi)}$, which allows us to conclude that $\pi$ is the preimage of the pair $\mu$, $\nu$. 
\end{proof}

\begin{ex} \label{3_ex_visual_interpretation_overlap} Let us fix $m = 3$, $n = 6$ and a partition $\lambda = (7, 4, 3, 3, 3, 1)$ of length less than $m + n$. Proposition~\ref{3_prop_visual_interpretation_overlap} tells us that any staircase walk $\pi \in \mathfrak{P}(n,m)$ corresponds to a pair of partitions whose $(m,n)$-overlap equals $\lambda$. In order to visualize this correspondence, consider the following diagram of a staircase walk $\pi \in \mathfrak{P}(6,3)$ labeled by the partition $\lambda$:
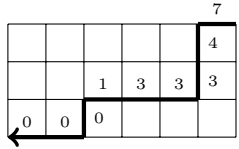
\vspace{-0.3cm}
\begin{center}
\begin{tikzpicture} 
\draw[step=0.5cm, thin] (0, 0) grid (3, 1.5);
\draw[ultra thick, ->] (1, 0) -- (0, 0);
\draw[ultra thick] (1, 0) -- (1, 0.5);
\draw[ultra thick] (1, 0.5) -- (2.5, 0.5);
\draw[ultra thick] (2.5, 0.5) -- (2.5, 1.5);
\draw[ultra thick] (2.5, 1.5) -- (3, 1.5);
\node[anchor=south] at (2.75, 1.5) {\tiny{7}};
\node[anchor=west] at (2.5, 1.25) {\tiny{4}};
\node[anchor=west] at (2.5, 0.75) {\tiny{3}};
\node[anchor=south] at (2.25, 0.5) {\tiny{3}};
\node[anchor=south] at (1.75, 0.5) {\tiny{3}};
\node[anchor=south] at (1.25, 0.5) {\tiny{1}};
\node[anchor=west] at (1, 0.25) {\tiny{0}};
\node[anchor=south] at (0.75, 0) {\tiny{0}};
\node[anchor=south] at (0.25, 0) {\tiny{0}};
\end{tikzpicture} 
\end{center}
Under the map defined in \eqref{3_eq_visual_interpretation_overlap_map}, $\pi$ is sent to the pair of partitions
\begin{multline*}
(\mu, \nu) = \left(\mu(\pi) + \lambda_{V(\pi)}, \nu(\pi)' + \lambda_{H(\pi)} \right) \\ = ((5, 5, 2) + (4, 3), (3, 1, 1, 1) + (7, 3, 3, 1)) = ((9, 8, 2), (10, 4, 4, 2)).
\end{multline*}
Recalling that the parts of $\mu(\pi)$ (or $\nu(\pi)'$) correspond to the rows of boxes above the staircase walk (or columns below the staircase walk), these numbers are easy to see in the diagram: for each part of $\mu(\pi)$ (or $\nu(\pi)'$), the label of the corresponding step of $\pi$ indicates the number of boxes that must be added to the row (or column) to obtain the corresponding part of $\mu$ (or $\nu$).
\end{ex}

The visualization of overlap explained in this example also provides a framework for visualizing pairs of partitions whose overlap is infinity. We recall that the $(m,n)$-overlap of two partitions, say $\mu$ and $\nu$, is infinity if and only if the sequence $\left( \mu + \rho_m \right) \cup \left( \nu + \rho_n \right)$ contains repetitions. This visualization makes use of the notion of quasi-partitions.

\begin{defn} [quasi-partition] Let $\pi \in \mathfrak{P}(n,m)$. We call a sequence $\alpha$ of length $m + n$ a quasi-partition associated to the staircase walk $\pi$ if it satisfies all of the following conditions:
\begin{enumerate}
\item the elements of $\alpha$ are possibly negative integers but $\alpha_{m + n} \geq 0$;
\item there is no index $i$ so that $\alpha_{i - 1} < \alpha_i < \alpha_{i + 1}$;
\item if $i, i + 1 \in V(\pi)$ (or $i, i + 1 \in H(\pi)$), then $\alpha_{i + 1} \leq \alpha_i$;
\item if $i \in V(\pi)$ and $i + 1 \in H(\pi)$ (or vice versa), then $\alpha_{i + 1} \leq \alpha_i + 1$.
\end{enumerate}
\end{defn}
We remark that a sequence of length $m + n$ is a partition if and only if it is a quasi-partition associated to \emph{all} staircase walks in $\mathfrak{P}(n,m)$ -- unless $m = 1 = n$. In case $m = 1 = n$, a partition of length at most $2$ is still a quasi-partition associated to all staircase walks in $\mathfrak{P}(n,m)$, but the converse does not hold.

\begin{prop} \label{3_prop_visualization_of_infinite_overlap} Let $m$ and $n$ be non-negative integers. The $(m,n)$-overlap of two partitions $\mu$ and $\nu$ is equal to infinity if and only if there exist a staircase walk $\pi \in \mathfrak{P}(n,m)$ and a quasi-partition $\alpha$ associated to $\pi$ with the properties that $\alpha$ is \emph{not} a partition, $\mu = \mu(\pi) + \alpha_{V(\pi)}$ and $\nu = \nu(\pi)' + \alpha_{H(\pi)}$.
\end{prop}

\begin{ex*} Let us illustrate this visualization for partitions whose overlap is infinity on a concrete example. The left-most diagram depicts a staircase walk $\pi \in \mathfrak{P}(6,3)$ labeled by a quasi-partition $\alpha$ associated to $\pi$. 
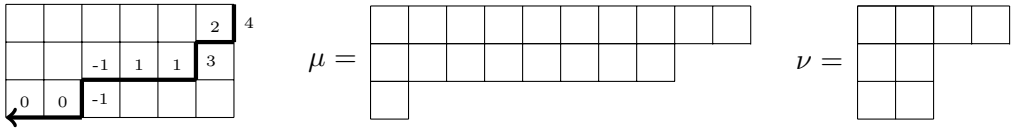
\begin{center}
\begin{tikzpicture} 
\draw[step=0.5cm, thin] (0, 0) grid (3, 1.5);
\draw[ultra thick, ->] (1, 0) -- (0, 0);
\draw[ultra thick] (1, 0) -- (1, 0.5);
\draw[ultra thick] (1, 0.5) -- (2.5, 0.5);
\draw[ultra thick] (2.5, 0.5) -- (2.5, 1);
\draw[ultra thick] (2.5, 1) -- (3, 1);
\draw[ultra thick] (3, 1) -- (3, 1.5);
\node[anchor=west] at (3, 1.25) {\tiny{4}};
\node[anchor=south] at (2.75, 1) {\tiny{2}};
\node[anchor=west] at (2.5, 0.75) {\tiny{3}};
\node[anchor=south] at (2.25, 0.5) {\tiny{1}};
\node[anchor=south] at (1.75, 0.5) {\tiny{1}};
\node[anchor=south] at (1.25, 0.5) {\tiny{-1}};
\node[anchor=west] at (1, 0.25) {\tiny{-1}};
\node[anchor=south] at (0.75, 0) {\tiny{0}};
\node[anchor=south] at (0.25, 0) {\tiny{0}};
\end{tikzpicture}
\hspace{0.2cm}
\begin{tikzpicture}
\node(mu) at (-0.5, 0.75) {$\mu =$};
\draw[step=0.5cm, thin] (0, 0) grid (0.5, 0.5);
\draw[step=0.5cm, thin] (0, 0.5) grid (4, 1);
\draw[step=0.5cm, thin] (0, 0.99) grid (5, 1.5);
\end{tikzpicture}
\hspace{0.2cm}
\begin{tikzpicture}
\node(nu) at (-0.5, 0.75) {$\nu =$};
\draw[step=0.5cm, thin] (0, 0) grid (1, 1.5);
\draw[step=0.5cm, thin] (0, 0.99) grid (2, 1.5);
\end{tikzpicture}
\end{center}
As in the preceding example, the label of each step of $\pi$ indicates the number of boxes that must be added to \emph{or removed from} the corresponding row of $\mu(\pi)$ (or column of $\nu(\pi)'$) to obtain the diagram of $\mu$ (or $\nu$). Using the formal definition of overlap, we compute that 
$$\mu \star_{3,6} \nu = (10, 8, 1) \star_{3,6} (4,2,2) = \infty$$
since $8 + 1 = 4 + 5$.
\end{ex*}

\begin{proof}[Proof of Proposition~\ref{3_prop_visualization_of_infinite_overlap}] Fix a staircase walk $\pi \in \mathfrak{P}(n,m)$ and a quasi-partition $\alpha$ associated to $\pi$ with the property that $\alpha$ is not a partition. First we show that the sequence $\mu = \mu(\pi) + \alpha_{V(\pi)}$ is a partition. For any index $1 \leq p \leq m - 1$, $\mu_{p + 1} \leq \mu_p$: if $V(\pi)_{p + 1} = V(\pi)_p + 1$, the third property listed in the definition of quasi-partition implies that 
\begin{multline*}
\mu_{p + 1} = \mu(\pi)_{p + 1} + \alpha_{V(\pi)_{p + 1}} = \mu(\pi)_{p + 1} + \alpha_{V(\pi)_p + 1} \\ \leq \mu(\pi)_{p + 1} + \alpha_{V(\pi)_p}= \mu(\pi)_p + \alpha_{V(\pi)_p} = \mu_p
;
\end{multline*}
if $V(\pi)_{p + 1} = V(\pi)_p + q + 1$ for some $q \geq 2$, the combining the third and the fourth property implies that
\begin{multline*}
\mu_{p + 1} = \mu(\pi)_{p + 1} + \alpha_{V(\pi)_{p + 1}} = \mu(\pi)_{p + 1} + \alpha_{V(\pi)_p + q + 1} \\ \leq \mu(\pi)_{p + 1} + \alpha_{V(\pi)_p} + 2 = \mu(\pi)_p - q + \alpha_{V(\pi)_p} + 2 \leq \mu_p
;
\end{multline*}
if $V(\pi)_{p + 1} = V(\pi)_p + 2$, combining the second and the fourth property allows us to infer that
\begin{multline*}
\mu_{p + 1} = \mu(\pi)_{p + 1} + \alpha_{V(\pi)_{p + 1}} = \mu(\pi)_{p + 1} + \alpha_{V(\pi)_p + 2} \\ \leq \mu(\pi)_{p + 1} + \alpha_{V(\pi)_p} + 1 = \mu(\pi)_p - 1 + \alpha_{V(\pi)_p} + 1 = \mu_p
.
\end{multline*}
It also follows that $\mu_m \geq 0$: if $V(\pi)_m = m + n$, the first property implies that $$\mu_m = \mu(\pi)_{m} + \alpha_{V(\pi)_m} \geq \mu(\pi)_{m} = 0;$$
if $V(\pi)_m < m + n$, the first, third and fourth properties entail that $\alpha(V(\pi))_m \geq -1$ and thus that
$$\mu_m = \mu(\pi)_{m} + \alpha_{V(\pi)_m} \geq \mu(\pi)_{m} - 1 \geq 0.$$
We conclude that $\mu$ is indeed a partition. An analogous argument shows that the sequence $\nu = \nu(\pi)' + \alpha_{H(\pi)}$ is also a partition.

Second we show that the $(m,n)$-overlap of $\mu$ and $\nu$ is infinity by constructing a pair of indices $1 \leq p \leq m$ and $1 \leq q \leq n$ so that $\mu_p + m - p = \nu_q + n - q$. By assumption, $\alpha$ is not a partition, which a priori means that the quasi-partition $\alpha$ contains a strictly negative element or a strict increase. However, the condition that the last element of $\alpha$ be non-negative allows us to infer that $\alpha$ must contain a strict increase. More precisely, there exists an index $1 \leq i \leq m + n - 1$ so that $\alpha_{i + 1} = \alpha_i + 1$. According to the third and the fourth property, $i \in V(\pi)$ and $i + 1 \in H(\pi)$ (or vice versa). As the two cases are exact analogues, we may assume the first case. Let us introduce $p$ and $q$ by means of an annotated diagram of a possible staircase walk $\pi$:
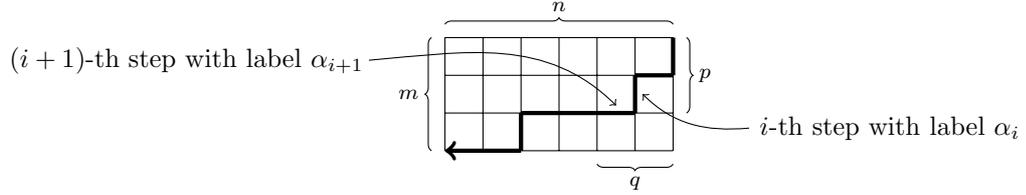
\begin{center}
\begin{tikzpicture} 
\draw[step=0.5cm, thin] (0, 0) grid (3, 1.5);
\draw[ultra thick, ->] (1, 0) -- (0, 0);
\draw[ultra thick] (1, 0) -- (1, 0.5);
\draw[ultra thick] (1, 0.5) -- (2.5, 0.5);
\draw[ultra thick] (2.5, 0.5) -- (2.5, 1);
\draw[ultra thick] (2.5, 1) -- (3, 1);
\draw[ultra thick] (3, 1) -- (3, 1.5);
\draw[decoration={brace, raise=5pt, mirror},decorate] (3, 0.5) -- node[right=6pt] {$\scriptstyle{p}$} (3, 1.5);
\draw[decoration={brace, raise=5pt, mirror},decorate] (0, 1.5) -- node[left=6pt] {$\scriptstyle{m}$} (0, 0);
\draw[decoration={brace, raise=5pt},decorate] (0, 1.5) -- node[above=6pt] {$\scriptstyle{n}$} (3, 1.5);
\draw[decoration={brace, raise=5pt},decorate] (3, 0) -- node[below=6pt] {$\scriptstyle{q}$} (2, 0);
\draw (4, 0.3) edge[out=185,in=-45,->] (2.6, 0.75);
\draw (-1, 1.2) edge[out=5,in=135,->] (2.25, 0.6);
\node at (-3.4, 1.2) {\small{$(i + 1)$-th step with label $\alpha_{i + 1}$}};
\node at (5.85, 0.3) {\small{$i$-th step with label $\alpha_i$}};
\end{tikzpicture}
\end{center}
\vspace{-0.3cm}
We see that
\begin{align*}
\mu_p + m - p ={} & \mu(\pi)_p + \alpha_{V(\pi)_p} + m - p = n - q + 1 + \alpha_i + m - p
\\
\nu_q + n - q ={} & \nu(\pi)'_q + \alpha_{H(\pi)_q} + n - q = m - p + \alpha_{i + 1} + n - q
.
\end{align*}
By construction, $\alpha_{i + 1} = \alpha_i + 1$, from which we conclude one direction of implication in the equivalence to be shown. 

In order to show the other direction, fix a pair of partitions, say $\mu$ and $\nu$, whose $(m,n)$-overlap is infinity. Let $\alpha$ be the sequence determined by the conditions that $\alpha + \rho_{m + n}$ be non-increasing and
\begin{align*}
\alpha + \rho_{m + n} \sorteq (\mu + \rho_m) \cup (\nu + \rho_n)
.
\end{align*}
Choose a pair of subsequences $V$, $H \subset [m + n]$ so that $\left( \alpha + \rho_{m + n} \right)_V = \mu + \rho_m$ and $\left( \alpha + \rho_{m + n} \right)_H = \nu + \rho_n$. It is worth noting that this choice is not unique since the sequences $\mu + \rho_m$ and $\nu + \rho_n$ have at least one element in common. Let $\pi \in \mathfrak{P}(n,m)$ be the staircase walk determined by $V(\pi) = V$ and $H(\pi) = H$. We claim that $\alpha$ is a quasi-partition associated to $\pi$ with the properties that $\alpha$ is \emph{not} a partition, $\mu = \mu(\pi) + \alpha_{V(\pi)}$ and $\nu = \nu(\pi)' + \alpha_{H(\pi)}$. The latter two follow immediately from Lemma~\ref{3_lem_macdonald_page_3}. In addition, the sequence $\alpha$ cannot be a partition since $\alpha + \rho_{m + n}$ is not strictly decreasing. The justification that $\alpha$ is indeed a quasi-partition is left to the reader. 
\end{proof}

The fact that the choice of $V$ and $H$ in the proof of Proposition~\ref{3_prop_visualization_of_infinite_overlap} is not unique entails that there is \emph{no} 1-to-1 correspondence between $\mathfrak{P}(n,m)$ and pairs of partitions whose $(m,n)$-overlap equals infinity.

\subsection{Complementary partitions and some of their properties}
In this section we study how taking the complement of a partition interacts with other operations on partitions, such as overlap. 

\begin{defn} [complement of a partition] The $(m,n)$-complement of a partition $\lambda$ contained in the rectangle $\left\langle m^n \right\rangle$ is given by \label{symbol_complement_of_partition}
\begin{align*}
\tilde \lambda = (m - \lambda_n, \dots, m - \lambda_1) \subset \langle m^n \rangle.
\end{align*}
When it is clear from the context with respect to which rectangle the complement is taken, we dispense with stating the parameters.
\end{defn}
The visual interpretation of this notion is that for any staircase walk $\pi \in \mathfrak{P}(m,n)$, the partitions $\mu(\pi)$ and $\nu(\pi)$ are $(m,n)$-complementary. Therefore, Proposition~\ref{3_prop_visual_interpretation_overlap} entails that two partitions are $(m,n)$-complementary if and only if they do not overlap, \textit{i.e.}\ if their $(n,m)$-overlap is the empty partition. In fact, this special case of Proposition~\ref{3_prop_visual_interpretation_overlap} is equivalent to Lemma~6 in \cite{dehaye12}.

In the following remark we collect a few properties of the complement. These observations are certainly not new given that they follow directly from the definition, but they will prove useful later.

\begin{rem} \label{3_rem_properties_of_complement} Taking the complement commutes with both conjugation and addition. More concretely, if $\lambda \subset \langle m^n \rangle$, then the $(n,m)$-complement of $\lambda'$ is conjugate to the $(m,n)$-complement of $\lambda$. For an additional partition $\kappa \subset \left\langle k^n \right\rangle$, the $(m + k,n)$-complement of $\lambda + \kappa$ is given by $\tilde\lambda + \tilde\kappa$. 
\end{rem}

Overlapping two partitions almost commutes with taking the complement. In fact, one could say that the two operations are skew-commutative.

\begin{lem} \label{3_lem_overlap_and_complement} If a partition $\lambda \subset \left\langle l^{m + n} \right\rangle$ is the $(m,n)$-overlap of the partitions $\mu$ and $\nu$, then $\tilde\lambda$ is the $(m,n)$-overlap of $\tilde\mu$ and $\tilde\nu$ where we view $\mu$ and $\nu$ as subsets of $\left\langle (n + l)^m \right\rangle$ and $\left\langle (m + l)^n \right\rangle$, respectively. Moreover, $\varepsilon_{m,n}(\tilde\mu, \tilde\nu) = (-1)^{mn} \varepsilon_{m,n}(\mu, \nu)$.
\end{lem}

\begin{proof} According to Proposition~\ref{3_prop_visual_interpretation_overlap}, whenever $\lambda = \mu \star_{m,n} \nu$ there exists a staircase walk $\pi \in \mathfrak{P}(n,m)$ such that $\mu = \mu(\pi) + \lambda_{V(\pi)}$ and $\nu = \nu(\pi)' + \lambda_{H(\pi)}$. Hence,
\begin{align*}
\tilde\mu = \nu(\pi) + \widetilde{\lambda_{V(\pi)}} \text{ and } \tilde\nu = \mu(\pi)' + \widetilde{\lambda_{H(\pi)}}
.
\end{align*}
If $\tau \in \mathfrak{P}(n,m)$ is the staircase walk obtained from $\pi$ by walking in the opposite direction (\textit{i.e.}\ up the stairs) and rotating the entire grid by 180 degrees, then  $\mu(\tau) = \nu(\pi)$ and $\nu(\tau) = \mu(\pi)$. Let us illustrate this relationship between the partitions associated to $\pi$ and $\tau$ by means of diagrams:
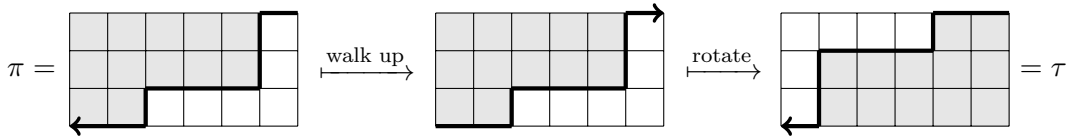
\begin{center}
\begin{tikzpicture} 
\node(pi) at (-0.5, 0.75) {$\pi =$};
\fill[black!10!white] (0, 0) rectangle (1, 1.5);
\fill[black!10!white] (1, 0.5) rectangle (2.5, 1.5);
\draw[step=0.5cm, thin] (0, 0) grid (3, 1.5);
\draw[ultra thick, ->] (1, 0) -- (0, 0);
\draw[ultra thick] (1, 0) -- (1, 0.5);
\draw[ultra thick] (1, 0.5) -- (2.5, 0.5);
\draw[ultra thick] (2.5, 0.5) -- (2.5, 1.5);
\draw[ultra thick] (2.5, 1.5) -- (3, 1.5);
\end{tikzpicture}
\begin{tikzpicture} 
\node(arrow) at (-0.9, 0.85) {$\xmapsto{\text{walk up}}$
};
\fill[black!10!white] (0, 0) rectangle (1, 1.5);
\fill[black!10!white] (1, 0.5) rectangle (2.5, 1.5);
\draw[step=0.5cm, thin] (0, 0) grid (3, 1.5);
\draw[ultra thick] (1, 0) -- (0, 0);
\draw[ultra thick] (1, 0) -- (1, 0.5);
\draw[ultra thick] (1, 0.5) -- (2.5, 0.5);
\draw[ultra thick] (2.5, 0.5) -- (2.5, 1.5);
\draw[ultra thick, ->] (2.5, 1.5) -- (3, 1.5);
\end{tikzpicture}
\begin{tikzpicture} 
\node(arrow) at (-0.75, 0.85) {$\xmapsto{\text{rotate}}$
};
\fill[black!10!white] (0.5, 0) rectangle (2, 1);
\fill[black!10!white] (2, 0) rectangle (3, 1.5);
\draw[step=0.5cm, thin] (0, 0) grid (3, 1.5);
\draw[ultra thick, ->] (0.5, 0) -- (0, 0);
\draw[ultra thick] (0.5, 0) -- (0.5, 1);
\draw[ultra thick] (0.5, 1) -- (2, 1);
\draw[ultra thick] (2, 1) -- (2, 1.5);
\draw[ultra thick] (2, 1.5) -- (3, 1.5);
\node(pi) at (3.45, 0.75) {$= \tau$};
\end{tikzpicture}
\end{center}
We see that the Ferrers diagrams of both $\mu(\pi)$ and $\nu(\tau)$ are determined by the boxes colored in gray, while the diagrams of both $\nu(\pi)$ and $\mu(\tau)$ correspond to the white boxes.
In addition, we infer from
\begin{align*}
V(\pi)_{m + 1 - i} = m + n + 1 - V(\tau)_i \:\text{ and }\: H(\pi)_{n + 1 - i} = m + n + 1 - H(\tau)_i
\end{align*}
that $\widetilde{\lambda_{V(\pi)}} = \tilde\lambda_{V(\tau)}$ and $\widetilde{\lambda_{H(\pi)}} = \tilde\lambda_{H(\tau)}$. Invoking again Proposition~\ref{3_prop_visual_interpretation_overlap}, we conclude that $\tilde{\lambda} = \tilde{\mu} \star_{m,n} \tilde{\nu}$.
The sign of this overlap is given by
\begin{align*}
\hspace{35.1pt} \varepsilon_{m,n}(\tilde\mu, \tilde\nu) = (-1)^{|\nu(\tau)|} = (-1)^{|\mu(\pi)|} = (-1)^{mn - |\nu(\pi)|} = (-1)^{mn} \varepsilon_{m,n}(\mu, \nu)
. \hspace{35.1pt} \qedhere 
\end{align*}
\end{proof}

\subsection{Marked staircase walks}
Our second visualization for the overlap of two partitions makes use of the notion of subpartitions, which are obtained by viewing overlap as a containment relation on partitions.

\begin{defn} [subpartition] Let $\lambda$ and $\mu$ be partitions. We call $\mu$ an $(m,n)$-subpartition of $\lambda$ if there exists a partition $\nu$ such that
$\mu \star_{m,n} \nu = \lambda$. Equivalently, $\mu$ is an $(m,n)$-subpartition of $\lambda$ if there exists a subsequence $M \subset [m + n]$ with $l(M) = m$ such that for all $1 \leq j \leq m$
\begin{align*}
\mu_j + m - j = \lambda_{M_j} + m + n - M_j.
\end{align*}
and $\mu_{m + 1} = 0 = \lambda_{m + n + 1}$. We denote the subpartition of $\lambda$ corresponding to the subsequence $M \subset [m + n]$ by \label{symbol_subpartition} $\sub_{m + n}(\lambda, M)$. If the parameter $m + n$ is clear from the context, it is sometimes omitted.
\end{defn}

\begin{ex} Let $\lambda$ be a partition of length at most $n$. The easiest example of a subpartition of $\lambda$ corresponds to removing one element from $[n]$: for $1 \leq j \leq n$, \begin{align} \label{3_ex_easiest_subpartition_eq}
\sub_n\left(\lambda, [n] \setminus (j) \right) = \lambda_{[n] \setminus (j)} + \left\langle 1^{j - 1} \right\rangle.
\end{align} 
This observation makes it possible to construct subpartitions $\sub(\lambda, M)$ for $M \subset [n]$ iteratively.
\end{ex}

In order to formally state our second visual interpretation for overlap, we also require the following technical definition.

\begin{defn} Let $n$ be a non-negative integer and $K \subset [n]$ a subsequence. We define \label{symbol_C_n(K)}
\begin{align*}
C_n(K) \sorteq (n - j + 1: j \not\in K)
\end{align*}
so that $C_n(K)$ is a subsequence of $[n]$.
\end{defn}
Let us make a quick numerical example: $C_6((1,2,4,5)) = (1,4)$.

\begin{lem} \label{3_lem_overlap_with_subpartition} Let $\lambda \subset \langle m^n \rangle$ be a partition and $K \subset [n]$ a subsequence. If $\kappa$ is the subpartition $\sub_n(\lambda, K)$, then
\begin{align} \label{3_lem_overlap_with_subpartition_eq}
\lambda' \star_{m, l(C_n(K))} \sub\left(\tilde\lambda, C_n(K) \right) = \kappa'
.
\end{align}
Furthermore, $\varepsilon_{m, l(C_n(K))}\left(\lambda', \sub_n\left(\tilde\lambda, C_n(K) \right) \right) = (-1)^{\left|\tilde\lambda_{C_n(K)}\right|}$.
\end{lem}

\begin{proof} We prove a slightly stronger statement by induction on the length of $K$. We claim that for each $K \subset [n]$, there exists a staircase walk $\pi \in \mathfrak{P}(l(C_n(K)), m)$ with the following properties:
\begin{enumerate}
\item $\lambda' = \mu(\pi) + \left( \sub(\lambda, K)' \right)_{V(\pi)}$;
\item $\sub\left(\tilde\lambda, C_n(K) \right) = \nu(\pi)' + \left( \sub(\lambda, K)' \right)_{H(\pi)}$;
\item for each element $i \in [n]$ with $i < \min\{K\}$, the $(l(C_n(K)) - i + 1)$-th \emph{horizontal} step of $\pi$ is the $(\lambda_i + l(C_n(K)) - i + 1)$-th step of $\pi$.
\end{enumerate}
Notice that the existence of $\pi$ with the first two properties is equivalent to the equality stated in \eqref{3_lem_overlap_with_subpartition_eq}, according to the correspondence given in Proposition~\ref{3_prop_visual_interpretation_overlap}. For the base case $l(K) = 0$, we put a visual interpretation on the fact that $\lambda'$ and $\tilde\lambda'$ are $(n,m)$-complementary to infer the existence of $\pi \in \mathfrak{P}(n,m)$ such that $\lambda' = \mu(\pi)$ and $\tilde\lambda = \nu(\pi)'$. By definition, $\pi$ satisfies the first two conditions, and the third can be read off the following annotated diagram of the staircase walk $\pi \in \mathfrak{P}(6,3)$ with $\mu(\pi) = \lambda' = (5, 5, 2)$ (colored in gray):
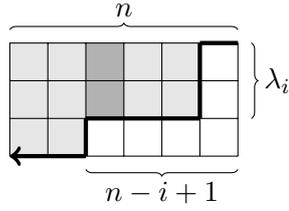
\begin{center}
\begin{tikzpicture} 
\fill[black!10!white] (0, 0) rectangle (1, 1.5);
\fill[black!10!white] (1, 0.5) rectangle (2.5, 1.5);
\fill[black!30!white] (1, 0.5) rectangle (1.5, 1.5);
\draw[step=0.5cm, thin] (0, 0) grid (3, 1.5);
\draw[ultra thick, ->] (1, 0) -- (0, 0);
\draw[ultra thick] (1, 0) -- (1, 0.5);
\draw[ultra thick] (1, 0.5) -- (2.5, 0.5);
\draw[ultra thick] (2.5, 0.5) -- (2.5, 1.5);
\draw[ultra thick] (2.5, 1.5) -- (3, 1.5);
\draw[decoration={brace, raise=5pt},decorate] (3, 1.5) -- node[right=6pt] {$\lambda_i$} (3, 0.5);
\draw[decoration={brace, raise=5pt, mirror},decorate] (1, 0) -- node[below=6pt] {$n - i + 1$} (3, 0);
\draw[decoration={brace, raise=5pt},decorate] (0, 1.5) -- node[above=6pt] {$n$} (3, 1.5);
\end{tikzpicture}
\end{center}
The annotations are based on the case $i = 3$.

For the induction step, consider a subsequence $(k) \cup K \subset [n]$ together with a staircase walk $\pi \in \mathfrak{P}(l(C_n(K)), m)$ which possesses the three properties stated above for the subsequence $K$. Construct a staircase walk $\tau \in \mathfrak{P}(l(C_n(K)) - 1, m)$ by removing the $(l(C_n(K)) - k + 1)$-th horizontal step from $\pi$; or equivalently, by removing the $(\lambda_k + l(C_n(K)) - k + 1)$-th step from $\pi$. By construction, the third property still holds for all elements $i \in [n]$ with $i < \min \{ (k) \cup K \} = k$. In order to justify that $\tau$ also satisfies the other two conditions, we first observe that 
\begin{multline} \label{3_in_proof_lem_overlap_with_subpartition}
\sub(\lambda, (k) \cup K)' = \left( (\lambda_k + l(C_n(K)) - k) \cup \sub(\lambda, K) \right)' \\ = \left\langle 1^{\lambda_k + l(C_n(K)) - k} \right\rangle + \sub(\lambda, K)'
.
\end{multline}
In particular, $\sub(\lambda, K)'$ has length at most $\lambda_k + l(C_n(K)) - k$. Hence, the fact that the first $\lambda_k + l(C_n(K)) - k$ steps of $\pi$ and $\tau$ are identical allows us to deduce that
\begin{align*}
\nu(\tau)' + \left(\sub(\lambda, (k) \cup K)'\right)_{H(\tau)} ={} & \nu(\tau)'+ \left(\sub(\lambda, K)'\right)_{H(\tau)} + \left\langle 1^{\lambda_k + l(C_n(K)) - k} \right\rangle_{H(\tau)} 
\\
={} & \nu(\tau)' + \left(\sub(\lambda, K)'\right)_{H(\pi)} + \left\langle 1^{l(C_n(K)) - k} \right\rangle
.
\intertext{By construction, $\nu(\tau)'$ is a subsequence of $\nu(\pi)'$:}
\nu(\tau)' + \left(\sub(\lambda, (k) \cup K)'\right)_{H(\tau)} ={} & \left( \nu(\pi)' \right)_{[l(C_n(K))] \setminus (l(C_n(K)) - k + 1)} \\ & + \left(\sub(\lambda, K)'\right)_{H(\pi)} + \left\langle 1^{l(C_n(K)) - k} \right\rangle
.
\intertext{Hence, the second property of $\pi$, together with the observation that $\left(\sub(\lambda, K)'\right)_{H(\pi)}$ has length at most $l(C_n(K)) - k$, gives}
\nu(\tau)' + \left(\sub(\lambda, (k) \cup K)'\right)_{H(\tau)} ={} & \sub\left(\tilde{\lambda}, C_n(K)\right)_{[l(C_n(K))] \setminus (l(C_n(K)) - k + 1)} \\ & + \left\langle 1^{l(C_n(K)) - k} \right\rangle
.
\intertext{Finally, the equality in \eqref{3_ex_easiest_subpartition_eq} on page \pageref{3_ex_easiest_subpartition_eq} states that}
\nu(\tau)' + \left(\sub(\lambda, (k) \cup K)'\right)_{H(\tau)} ={} & \sub\left(\tilde{\lambda}(C_n(K)), [l(C_n(K))] \setminus (l(C_n(K)) - k + 1)\right) 
\\
={} & \sub\left(\tilde{\lambda}, C_n(K)_{[l(C_n(K))] \setminus (l(C_n(K)) - k + 1)} \right) 
\\
={} & \sub \left( \tilde{\lambda}, C_n(K) \setminus (n - k + 1)\right)
= \sub \left( \tilde{\lambda}, C((k) \cup K) \right)
\!.
\end{align*}
Combining the first and the third property of $\pi$ allows us to infer the first property for $\tau$. Indeed, by the equality given in \eqref{3_in_proof_lem_overlap_with_subpartition},
\begin{align*}
\mu(\tau) + \left(\sub(\lambda, (k) \cup K)'\right)_{V(\tau)} ={} & \mu(\tau) + \left\langle 1^{\lambda_k + l(C_n(K)) - k} \right\rangle_{V(\tau)} + \left(\sub(\lambda, K)'\right)_{V(\tau)}
.
\intertext{Given that $\sub(\lambda, K)'$ is at most of length $\lambda_k + l(C_n(K)) - k$, the observation that the first $\lambda_k + l(C_n(K)) - k$ steps of $\pi$ and $\tau$ are identical allows us to infer that}
\mu(\tau) + \left(\sub(\lambda, (k) \cup K)'\right)_{V(\tau)} ={} & \mu(\tau) + \left\langle 1^{\lambda_k + l(C_n(K)) - k} \right\rangle_{V(\pi)} + \left(\sub(\lambda, K)'\right)_{V(\pi)}
\intertext{Moreover the third property implies that there are exactly $\lambda_k$ vertical steps among the first $\lambda_k + l(C_n(K)) - k$ steps of $\pi$:}
\mu(\tau) + \left(\sub(\lambda, (k) \cup K)'\right)_{V(\tau)} ={} & \mu(\tau) + \left\langle 1^{\lambda_k} \right\rangle + \left(\sub(\lambda, K)'\right)_{V(\pi)}
.
\intertext{By construction, $\mu(\tau) = \mu(\pi) - \left\langle 1^{\lambda_k} \right\rangle$, from which we conclude that}
\mu(\tau) + \left(\sub(\lambda, (k) \cup K)'\right)_{V(\tau)} ={} & \mu(\pi) + \left(\sub(\lambda, K)'\right)_{V(\pi)} = \lambda'
.
\end{align*} 
This justifies the claim, and thus the equality in \eqref{3_lem_overlap_with_subpartition_eq}. For the statement on the sign, recall that Proposition ~\ref{3_prop_visual_interpretation_overlap} entails that 
\begin{align*}
\varepsilon_{m, l(C_n(K))}\left(\lambda', \sub\left(\tilde\lambda, C_n(K)\right)\right) = (-1)^{|\nu(\pi)|} = (-1)^{\left|\nu(\pi)' \right|} = (-1)^{\left| {\tilde{\lambda}}_{C_n(K)} \right|}
\end{align*}
since $\nu(\pi)' = \tilde{\lambda}_{C_n(K)}$ by construction.
\end{proof}

\begin{ex} \label{3_ex_visualize_construction_in_lem_overlap_with_subpartition} To visualize this construction on a concrete example, fix $m = 4$, $n = 7$, a partition $\lambda = (4,4,2,2,1,1,1) \subset \left\langle 4^7 \right\rangle$ and a subsequence $K = (1, 4, 5, 7) \subset [7]$. Draw the diagram of the staircase walk $\pi \in \mathfrak{P}(7,4)$ that is determined by the condition that $\mu(\pi) = \lambda'$, and then mark/color the horizontal steps that lie in the subsequence $H(\pi)_{[7] \setminus C_7(K)}$. In our example $C_7(K) = (2,5,6)$. 
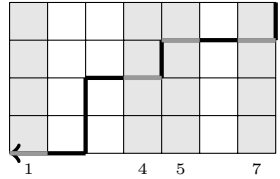
\begin{center}
\begin{tikzpicture}
\fill[black!10!white] (3, 0) rectangle (3.5, 2);
\fill[black!10!white] (2, 0) rectangle (2.5, 2);
\fill[black!10!white] (1.5, 0) rectangle (2, 2);
\fill[black!10!white] (0, 0) rectangle (0.5, 2);
\draw[step=0.5cm, thin] (0, 0) grid (3.5, 2);
\draw[ultra thick] (3.5, 2) -- (3.5, 1.5);
\draw[ultra thick] (3.5, 1.5) -- (2, 1.5);
\draw[ultra thick] (2, 1.5) -- (2, 1);
\draw[ultra thick] (2, 1) -- (1, 1);
\draw[ultra thick] (1, 1) -- (1, 0);
\draw[ultra thick, ->] (1, 0) -- (0, 0);
\draw[ultra thick, black!40!white] (3.5, 1.5) -- (3, 1.5);
\draw[ultra thick, black!40!white] (2.5, 1.5) -- (2, 1.5);
\draw[ultra thick, black!40!white] (2, 1) -- (1.5, 1);
\draw[ultra thick, black!40!white] (0.5, 0) -- (0, 0);
\node[anchor=north] at (3.25, 0) {\tiny{7}};
\node[anchor=north] at (2.25, 0) {\tiny{5}};
\node[anchor=north] at (1.75, 0) {\tiny{4}};
\node[anchor=north] at (0.25, 0) {\tiny{1}};
\end{tikzpicture}
\end{center}
The numbers along the bottom indicate which elements of $K$ the marked horizontal steps correspond to. Now, imagine that $\pi$ is labeled by the empty partition $(0, \dots, 0)$ as described in Example \ref{3_ex_visual_interpretation_overlap}. Throughout this construction the labeling of the staircase walk will keep track of $\sub(\lambda, L)'$ where $L$ consists of the elements of $K$ that have been removed from the diagram. At the moment, $\sub(\lambda, L)' = \sub(\lambda, \emptyset)' = \emptyset$, which matches the (imaginary) labels.

Following the construction outlined in the proof of Lemma~\ref{3_lem_overlap_with_subpartition}, remove the marked horizontal step corresponding to the largest element $k \in K$ and increase the label of each preceding step by 1:
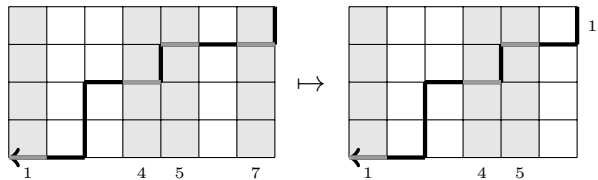
\begin{center}
\begin{tikzpicture}
\fill[black!10!white] (3, 0) rectangle (3.5, 2);
\fill[black!10!white] (2, 0) rectangle (2.5, 2);
\fill[black!10!white] (1.5, 0) rectangle (2, 2);
\fill[black!10!white] (0, 0) rectangle (0.5, 2);
\draw[step=0.5cm, thin] (0, 0) grid (3.5, 2);
\draw[ultra thick] (3.5, 2) -- (3.5, 1.5);
\draw[ultra thick] (3.5, 1.5) -- (2, 1.5);
\draw[ultra thick] (2, 1.5) -- (2, 1);
\draw[ultra thick] (2, 1) -- (1, 1);
\draw[ultra thick] (1, 1) -- (1, 0);
\draw[ultra thick, ->] (1, 0) -- (0, 0);
\draw[ultra thick, black!40!white] (3.5, 1.5) -- (3, 1.5);
\draw[ultra thick, black!40!white] (2.5, 1.5) -- (2, 1.5);
\draw[ultra thick, black!40!white] (2, 1) -- (1.5, 1);
\draw[ultra thick, black!40!white] (0.5, 0) -- (0, 0);
\node[anchor=north] at (3.25, 0) {\tiny{7}};
\node[anchor=north] at (2.25, 0) {\tiny{5}};
\node[anchor=north] at (1.75, 0) {\tiny{4}};
\node[anchor=north] at (0.25, 0) {\tiny{1}};
\end{tikzpicture}
\begin{tikzpicture}
\node(arrow) at (-0.5, 0.95) {$\mapsto$};
\fill[black!10!white] (2, 0) rectangle (2.5, 2);
\fill[black!10!white] (1.5, 0) rectangle (2, 2);
\fill[black!10!white] (0, 0) rectangle (0.5, 2);
\draw[step=0.5cm, thin] (0, 0) grid (3, 2);
\draw[ultra thick] (3, 2) -- (3, 1.5);
\draw[ultra thick] (3, 1.5) -- (2, 1.5);
\draw[ultra thick] (2, 1.5) -- (2, 1);
\draw[ultra thick] (2, 1) -- (1, 1);
\draw[ultra thick] (1, 1) -- (1, 0);
\draw[ultra thick, ->] (1, 0) -- (0, 0);
\draw[ultra thick, black!40!white] (2.5, 1.5) -- (2, 1.5);
\draw[ultra thick, black!40!white] (2, 1) -- (1.5, 1);
\draw[ultra thick, black!40!white] (0.5, 0) -- (0, 0);
\node[anchor=north] at (2.25, 0) {\tiny{5}};
\node[anchor=north] at (1.75, 0) {\tiny{4}};
\node[anchor=north] at (0.25, 0) {\tiny{1}};
\node[anchor=west] at (3, 1.75) {\tiny{1}};
\end{tikzpicture}
\end{center} 
According to the equality in \eqref{3_in_proof_lem_overlap_with_subpartition}, we have that
$$\sub(\lambda, L)' = \sub(\lambda, (7))' = \left\langle 1^{\lambda_7 + 7 - 7} \right\rangle + \emptyset = (1),$$
which thus matches the labels. In fact, the number of steps preceding the marked horizontal step corresponding to the largest remaining element $k \in K$ is always equal to $\lambda_k + l(C(L)) - k$, owing to the third property shown in the proof of the preceding Lemma. Therefore, the recursive equality in \eqref{3_in_proof_lem_overlap_with_subpartition} entails that increasing the labels of the preceding steps by 1 upon the removal of an element $k$ from the diagram ensures that the labels of the staircase walk will always match $\sub(\lambda, L)'$. Proceeding in this manner, you thus end up with a staircase walk labeled by $\sub(\lambda, K)'$:
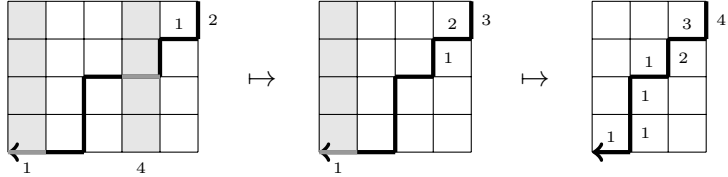
\begin{center}
\begin{tikzpicture}
\fill[black!10!white] (1.5, 0) rectangle (2, 2);
\fill[black!10!white] (0, 0) rectangle (0.5, 2);
\draw[step=0.5cm, thin] (0, 0) grid (2.5, 2);
\draw[ultra thick] (2.5, 2) -- (2.5, 1.5);
\draw[ultra thick] (2.5, 1.5) -- (2, 1.5);
\draw[ultra thick] (2, 1.5) -- (2, 1);
\draw[ultra thick] (2, 1) -- (1, 1);
\draw[ultra thick] (1, 1) -- (1, 0);
\draw[ultra thick, ->] (1, 0) -- (0, 0);
\draw[ultra thick, black!40!white] (2, 1) -- (1.5, 1);
\draw[ultra thick, black!40!white] (0.5, 0) -- (0, 0);
\node[anchor=west] at (2.5, 1.75) {\tiny{2}};
\node[anchor=south] at (2.25, 1.5) {\tiny{1}};
\node[anchor=north] at (1.75, 0) {\tiny{4}};
\node[anchor=north] at (0.25, 0) {\tiny{1}};
\end{tikzpicture}
\begin{tikzpicture}
\node(arrow) at (-0.75, 0.95) {$\mapsto$};
\fill[black!10!white] (0, 0) rectangle (0.5, 2);
\draw[step=0.5cm, thin] (0, 0) grid (2, 2);
\draw[ultra thick] (2, 2) -- (2, 1.5);
\draw[ultra thick] (2, 1.5) -- (1.5, 1.5);
\draw[ultra thick] (1.5, 1.5) -- (1.5, 1);
\draw[ultra thick] (1.5, 1) -- (1, 1);
\draw[ultra thick] (1, 1) -- (1, 0);
\draw[ultra thick, ->] (1, 0) -- (0, 0);
\draw[ultra thick, black!40!white] (0.5, 0) -- (0, 0);
\node[anchor=west] at (2, 1.75) {\tiny{3}};
\node[anchor=south] at (1.75, 1.5) {\tiny{2}};
\node[anchor=west] at (1.5, 1.25) {\tiny{1}};
\node[anchor=north] at (0.25, 0) {\tiny{1}};
\end{tikzpicture}
\begin{tikzpicture}
\node(arrow) at (-0.75, 0.95) {$\mapsto$};
\draw[step=0.5cm, thin] (0, 0) grid (1.5, 2);
\draw[ultra thick] (1.5, 2) -- (1.5, 1.5);
\draw[ultra thick] (1.5, 1.5) -- (1, 1.5);
\draw[ultra thick] (1, 1.5) -- (1, 1);
\draw[ultra thick] (1, 1) -- (0.5, 1);
\draw[ultra thick] (0.5, 1) -- (0.5, 0);
\draw[ultra thick, ->] (0.5, 0) -- (0, 0);
\node[anchor=south] at (0.25, 0) {\tiny{1}};
\node[anchor=west] at (0.5, 0.25) {\tiny{1}};
\node[anchor=west] at (0.5, 0.75) {\tiny{1}};
\node[anchor=south] at (0.75, 1) {\tiny{1}};
\node[anchor=west] at (1, 1.25) {\tiny{2}};
\node[anchor=south] at (1.25, 1.5) {\tiny{3}};
\node[anchor=west] at (1.5, 1.75) {\tiny{4}};
\node[anchor=north] at (0.25, 0) {\color{white}{\tiny{1}}};
\end{tikzpicture}
\end{center}
In fact, the last diagram is our visualization for the correspondence between staircase walks in a $3 \times 4$-rectangle and pairs of partitions whose $(4,3)$-overlap is equal to $\sub(\lambda, K)' = (4, 3, 2, 1, 1, 1, 1)$. According to Lemma~\ref{3_lem_overlap_with_subpartition}, the pair of partitions corresponding to this diagram is $\left( \lambda', \sub\left( \tilde{\lambda}, C(K)\right) \right)$. The following proposition states that this algorithmic procedure is invertible.
\end{ex}

\begin{prop} \label{3_prop_subpartition_interpretation_overlap} For a fixed partition $\kappa \subset \left\langle (m + n)^l \right\rangle$, there is a 1-to-1 correspondence between 
\begin{enumerate}
\item $\{(\mu, \nu): \mu \star_{m,n} \nu = \kappa'\}$
\item $\{(\lambda, K): \lambda \subset \left\langle m^{n + l} \right\rangle, K \subset [n + l], l(K) = l \text{ and } \sub_{n + l}(\lambda, K) = \kappa\}$
\end{enumerate}
given by the following mapping:
\begin{align} \label{3_prop_subpartition_interpretation_overlap_map}
(\lambda, K) \mapsto \left(\lambda', \sub_{n + l} \left(\tilde\lambda, C_{n + l}(K)\right)\right).
\end{align} 
\end{prop}

\begin{proof} By Lemma~\ref{3_lem_overlap_with_subpartition}, the map in \eqref{3_prop_subpartition_interpretation_overlap_map} is well defined. We observe that it is also injective. Indeed, if $(\lambda, K)$ and $(\eta, L)$ are both mapped to $(\mu, \nu)$, then $\lambda = \mu' = \eta$ and hence $$\left( \tilde\lambda + \rho_{n + l}\right)_{C_{n + l}(K)} = \nu + \rho_n = \left( \tilde\lambda + \rho_{n + l}\right)_{C_{n + l}(L)},$$ which entails that $K = L$ because the three sequences in question are strictly decreasing. Rather than showing directly that the map in \eqref{3_prop_subpartition_interpretation_overlap_map} is surjective, we will prove that both of the sets stated in the proposition are of cardinality $\binom{m + n}{m}$. For the first set, this is an immediate consequence of Proposition~\ref{3_prop_visual_interpretation_overlap}. 

For the second set, we prove the claim by induction on the length of the partition $\kappa$. Let us denote the second set by $\calS_2(\kappa)$. For the base case, we compute its cardinality under the assumption that $\kappa$ is the empty partition. In this case, any pair $(\lambda, K)$ that lies in $\calS_2(\kappa)$ satisfies $\rho_l = \left( \lambda + \rho_{n + l} \right)_K$. Given that the latter sequence is strictly decreasing, this equality implies that $K = (n + 1, \dots, n + l)$ and $\lambda_K = \emptyset$. Hence, $\lambda$ ranges over all partitions that are contained in the rectangle $\langle m^n \rangle$, of which there are exactly $\binom{m + n}{m}$. 

For the induction step, consider some partition $\kappa$ of length $0 < i \leq l$. We will construct a pair $(\lambda, K) \in \calS_2(\kappa)$ from each pair $(\eta, L) \in \calS_2((\kappa_1, \dots, \kappa_{i - 1}, 0))$. Let $j$ be the largest index so that $\eta_{j - 1} + n + l - (j - 1) > \kappa_i + l - i$ (where we use the convention that $\eta_0$ is infinitely large). By definition, $L_{i - 1} < j \leq L_i$. Indeed,
\begin{align*}
\eta_{L_{i - 1}} + n + l - L_{i - 1} ={} & \kappa_{i - 1} + l - (i - 1) > \kappa_i + l - i
\intertext{and}
\eta_{L_i} + n + l - L_i ={} & 0 + l - i \leq \kappa_i + l - i
.
\end{align*}
Hence, $K = (L_1, \dots, L_{i - 1}, j, L_{i + 1}, \dots, L_l)$ defines a subsequence of $[n + l]$. Moreover, the definition of $j$ also ensures that $\lambda = (\eta_1, \dots, \eta_{j - 1}, \kappa_i - i - n + j, \eta_j, \dots, \eta_{n + l - 1})$ is a partition. It is left to the reader to verify that the pair $(\lambda, K)$ is an element of $\calS_2(\kappa)$. The induction hypothesis thus allows us to conclude that the cardinality of the set $\calS_2(\kappa)$ is at least $\binom{m + n}{m}$. Recalling that the map in \eqref{3_prop_subpartition_interpretation_overlap_map} is an injection from $\calS_2(\kappa)$ to a set of cardinality $\binom{m + n}{m}$ completes the proof.
\end{proof}

Keeping in mind Example~\ref{3_ex_visualize_construction_in_lem_overlap_with_subpartition}, it is easy to give a visual description of the inverse of the map defined in \eqref{3_prop_subpartition_interpretation_overlap_map}: Given a pair of partitions $\mu$ and $\nu$ whose $(m,n)$-overlap equals $\kappa'$, use the visual interpretation described in Example~\ref{3_ex_visual_interpretation_overlap} to associate it to a staircase walk $\pi \in \mathfrak{P}(n,m)$ labeled by $\kappa'$. Iteratively construct a staircase walk $\tau \in \mathfrak{P}(n + l, m)$ by inserting a \emph{marked} horizontal step between any two adjacent steps of $\pi$ with distinct labels and simultaneously decreasing by 1 the labels of all steps to the right of the insertion. Here we use the convention that the ``label'' before the first step is $l$, while the ``label'' after the last step is 0. As illustrated in Example~\ref{3_ex_visualize_construction_in_lem_overlap_with_subpartition}, we can then map the pair $(\mu, \nu)$ to the partition $\lambda = \mu(\tau)'$ and the following subsequence $K \subset [n]$: $K \sorteq (n + l - i + 1: H(\pi)_i \text{ is marked})$.

\section{More overlap identities} \label{3_section_more_overlap_identities}
Applying the different ways of seeing the overlap of two partitions, which we discussed in the preceding section, to the second overlap identity allows us to derive more overlap identities. This will allow us to regard the dual Cauchy identity as an overlap identity.

\subsection{Variations on the second overlap identity}
\begin{cor} 
Let $0 \leq l \leq \min\{n - k, n\}$. Let $\calS$, $\calT$ and $\calY$ be sets containing $l$, $n - l$ and $m$ variables, respectively, so that $\Delta(\calY) \neq 0$ and $\Delta(\calS; \calT) \neq 0$. Suppose that $k$ is the $(m,n)$-index of a partition $\lambda$, then
\begin{align} \label{3_cor_first_visualization_of_overlap_LS_eq}
\begin{split} 
LS_{\lambda} (-(\calS \cup \calT); \calY) 
={} & \sum_{\pi \in \mathfrak{P}(m + n - k - l, l)} (-1)^{|\nu(\pi_1)|} \frac{\Delta\left(\calY_{H(\pi_2)}; \calS \right) \Delta\left(\calT; \calY_{V(\pi_2)}\right)}{\Delta \left(\calY_{H(\pi_2)}; \calY_{V(\pi_2)}\right) \Delta(\calT; \calS)} \\
& \times LS_{\mu(\pi_1) + \lambda_{V(\pi_1)} - \left\langle (m - k)^{l(V(\pi_1))} \right\rangle} \left(-\calS; \calY_{V(\pi_2)}\right) \\
& \times LS_{\nu(\pi_1)' + \lambda_{H(\pi_1)} \cup \lambda_{(n + 1 - k, n + 2 - k, \dots )}} \left(-\calT; \calY_{H(\pi_2)}\right)
\end{split}
\end{align}
where $\pi_1$ denotes the $n - k$ first steps of $\pi$, while $\pi_2$ denotes the $m$ last steps of $\pi$. We view $\pi_1$ and $\pi_2$ as staircase walks inside the appropriate rectangles.
\end{cor}

The definitions of the ``partial'' staircase walks $\pi_1$ and $\pi_2$ are best explained by means of a diagram: let $m = 5$, $n = 8$, $k = 4$ and $l = 3$, then the following staircase walk $\pi \in \mathfrak{P}(m + n - k - l, l)$ splits into $\pi_1 \in \mathfrak{P}(2,2)$ and $\pi_2 \in \mathfrak{P}(4,1)$ comprising of $n - k$ and $m$ steps, respectively.
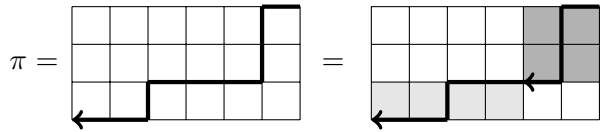
\begin{center}
\begin{tikzpicture} 
\node(pi) at (-0.5, 0.75) {$\pi =$};
\draw[step=0.5cm, thin] (0, 0) grid (3, 1.5);
\draw[ultra thick, ->] (1, 0) -- (0, 0);
\draw[ultra thick] (1, 0) -- (1, 0.5);
\draw[ultra thick] (1, 0.5) -- (2.5, 0.5);
\draw[ultra thick] (2.5, 0.5) -- (2.5, 1.5);
\draw[ultra thick] (2.5, 1.5) -- (3, 1.5);
\end{tikzpicture}
\begin{tikzpicture} 
\node(pi) at (-0.5, 0.75) {$=$};
\fill[black!10!white] (0, 0) rectangle (2, 0.5);
\fill[black!30!white] (2, 0.5) rectangle (3, 1.5);
\draw[step=0.5cm, thin] (0, 0) grid (3, 1.5);
\draw[ultra thick, ->] (1, 0) -- (0, 0);
\draw[ultra thick] (1, 0) -- (1, 0.5);
\draw[ultra thick] (1, 0.5) -- (2.5, 0.5);
\draw[ultra thick] (2.5, 0.5) -- (2.5, 1.5);
\draw[ultra thick] (2.5, 1.5) -- (3, 1.5);
\draw[ultra thick, ->] (2.5, 0.5) -- (2, 0.5);
\end{tikzpicture}
\end{center}
The diagrams of $\pi_1$ and $\pi_2$ are colored in different shades of gray.

\begin{proof} The right-hand side of the equation in \eqref{3_cor_first_visualization_of_overlap_LS_eq} is equal to
\begin{align*}
\RHS ={} & \sum_{p = 0}^{\min\{l,m\}} \sum_{\pi_2 \in \mathfrak{P}(m - p, p)} \hspace{10pt} \sum_{\substack{\pi_1 \in \mathfrak{P}(n - k - l + p, l - p)}} \frac{\Delta\left(\calY_{H(\pi_2)}; \calS \right) \Delta\left(\calT; \calY_{V(\pi_2)}\right)}{\Delta \left(\calY_{H(\pi_2)}; \calY_{V(\pi_2)}\right) \Delta(\calT; \calS)} \\
& \times (-1)^{|\nu(\pi_1)|} LS_{\mu(\pi_1) + \lambda_{V(\pi_1)} - \left\langle (m - k)^{l(V(\pi_1))} \right\rangle} \left(-\calS; \calY_{V(\pi_2)}\right) \\
& \times LS_{\nu(\pi_1)' + \lambda_{H(\pi_1)} \cup \lambda_{(n + 1 - k, n + 2 - k, \dots )}} \left(-\calT; \calY_{H(\pi_2)}\right)
.
\intertext{Setting $\calU = \calY_{V(\pi)}$ and $\calV = \calY_{H(\pi)}$, we may view the sum over $\pi_2$ as a sum over all subsequences $\calU$, $\calV \subset \calY$ with the property that $\calU \cup_{p, m - p} \calV \sorteq \calY$:}
\RHS ={} & \sum_{p = 0}^{\min\{l,m\}} \sum_{\substack{\calU, \calV \subset \calY: \\ \calU \cup_{p, m - p} \calV \sorteq \calY}} \hspace{10pt} \sum_{\substack{\pi_1 \in \mathfrak{P}(n - k - l + p, l - p)}} \frac{\Delta\left(\calV; \calS \right) \Delta\left(\calT; \calU \right)}{\Delta \left(\calV; \calU \right) \Delta(\calT; \calS)} \\
& \times (-1)^{|\nu(\pi_1)|} LS_{\mu(\pi_1) + \lambda_{V(\pi_1)} - \left\langle (m - k)^{l(V(\pi_1))} \right\rangle} \left(-\calS; \calU \right) \\
& \times LS_{\nu(\pi_1)' + \lambda_{H(\pi_1)} \cup \lambda_{(n + 1 - k, n + 2 - k, \dots )}} \left(-\calT; \calV \right)
.
\end{align*}
Proposition~\ref{3_prop_visual_interpretation_overlap} allows us to conclude that this expression is equal to the right-hand of the equality in \eqref{3_thm_lapalace_expansion_new_LS_eq}. Hence, the result follows directly from the second overlap identity (\textit{i.e.}\ Theorem~\ref{3_thm_lapalace_expansion_new_LS}).
\end{proof}

\begin{cor} \label{3_cor_second_overlap_id_labeled_walks}
Let $\lambda$ be a partition and let $\calS$ and $\calT$ be sets consisting of $m$ and $n$ variables, respectively. If $\Delta(\calS; \calT) \neq 0$, then
\begin{align} \label{3_cor_second_overlap_id_labeled_walks_eq}
\schur_\lambda(\calS \cup \calT) 
={} & \sum_{\substack{\pi \in \mathfrak{P}(n,m)}} \frac{(-1)^{|\nu(\pi)|} \schur_{\mu(\pi) + \lambda_{V(\pi)}} (\calS) \schur_{\nu(\pi)' + \lambda_{H(\pi)}} (\calT)}{\Delta(\calS; \calT)}
.
\end{align}
\end{cor}

\begin{proof} Owing to the correspondence given in Proposition~\ref{3_prop_visual_interpretation_overlap}, this identity is a direct consequence of Corollary~\ref{3_cor_laplace_expansion_schur_new}.
\end{proof}

\begin{cor} \label{3_cor_overlap_with_subpartitions_LS} Let $m$, $n$, $\tilde n$, $l$ and $q$ be non-negative integers such that $\tilde n \leq q$. Let $\calS$, $\calT$ and $\calY$ be sets of variables of length $m$, $n + \tilde n$ and $q$, respectively, so that $\Delta(\calY) \neq 0$ and $\Delta(\calS, \calT) \neq 0$. Suppose that a partition $\kappa \subset \left\langle (m + n)^l \right\rangle$ satisfies $(m + n, q - \tilde n) \in \kappa$, then
\begin{align*}
LS_{\kappa'}(-(\calS \cup \calT), \calY) ={} & \sum_{p = 0}^{\min\{m, q\}} \sum_{\substack{\calU, \calV \subset \calY: \\ \calU \cup_{p, q - p} \calV \sorteq \calY}} \sum_{\substack{\lambda \subset \left\langle (m - p)^{n + p + l} \right\rangle \\ K \subset [n + p + l] \text{ with } l(K) = l: \\ \sub(\lambda, K) = \kappa}} \frac{\Delta(\calV; \calS) \Delta(\calT; \calU)}{\Delta(\calV; \calU) \Delta(\calT; \calS)} \\
& \times (-1)^{\left| \tilde\lambda_{C_{n + p + l}(K)} \right|} LS_{\lambda' - \left\langle (q - \tilde n)^{m - p} \right\rangle} (-\calS; \calU) \\
& \times LS_{\sub\left(\tilde\lambda, C_{n + p + l}(K)\right)} (-\calT; \calV)
.
\end{align*}
\end{cor}

\begin{proof} Notice that the assumptions on $\kappa$ entail that the $(q, m + n + \tilde n)$-index of $\kappa'$ is $\tilde n$. Thus, the result follows by substituting the correspondence described in Proposition~\ref{3_prop_subpartition_interpretation_overlap} in the right-hand side of the equality in \eqref{3_thm_lapalace_expansion_new_LS_eq}.
\end{proof}

As usual the formula looks considerably nicer specialized to Schur functions, or equivalently to the case $\calY = \emptyset$.

\begin{cor} \label{3_cor_overlap_with_subpartitions_schur} Let $\calS$ and $\calT$ be two sets of variables of lengths $m$ and $n$, respectively, with the property that $\Delta(\calS; \calT) \neq 0$. For any partition $\kappa$ contained in the rectangle $\left\langle (m + n)^l \right\rangle$,
\begin{align*}
\schur_{\kappa'}(\calS \cup \calT) = \sum_{\substack{\lambda \subset \left\langle m^{n + l} \right\rangle \\ K \subset [n + l] \text{ with } l(K) = l: \\ \sub(\lambda, K) = \kappa}} \frac{(-1)^{\left|\tilde\lambda_{C_{n + l}(K)} \right|} \schur_{\lambda'}(\calS) \schur_{\sub\left(\tilde\lambda, C_{n + l}(K)\right)} (\calT)}{\Delta(\calS; \calT)}
.
\end{align*}
\end{cor}

\subsection{A first application}
In this section we present a small application of the second overlap identity. We have called it a first application because the recipe for mixed ratios, which is one of the main results of Chapter~\ref{4_cha_mixed_ratios}, can be viewed as an application of the first overlap identity for Littlewood-Schur functions.  

Our first application is to derive the dual Cauchy identity from the second overlap identity, or rather from Corollary~\ref{3_cor_second_overlap_id_labeled_walks}. Although the result is classical, this elegant proof seems to be new. The proof relies on the following relationship between the Schur functions indexed by a partition $\lambda$ and it complement $\tilde\lambda$, respectively.

\begin{lem} \label{3_lem_Schur_indexed_by_complement} Let $\calX$ contain $n$ non-zero variables. If a partition $\lambda$ is a subset of the rectangle $\langle m^n \rangle$, then
\begin{align} \label{3_lem_Schur_indexed_by_complement_eq}
\schur_{\tilde\lambda}(\calX) = \schur_\lambda \left( \calX^{-1} \right) \elementary(\calX)^m
.
\end{align}
\end{lem}

\begin{proof} This equality is a fairly immediate consequence of the determinantal definition for Schur functions:
\begin{align*}
\elementary(\calX)^m \schur_\lambda \left( \calX^{-1} \right) ={} & \frac{\elementary(\calX)^m \det \left( x^{-(\lambda_j + n - j)} \right)_{1 \leq j \leq n}}{\det \left( x^{-(n - j)} \right)_{1 \leq j \leq n}} \times \frac{\elementary(\calX)^{n - 1}}{\elementary(\calX)^{n - 1}}
\\
={} & \frac{\det \left( x^{m - \lambda_j + j - 1} \right)_{1 \leq j \leq n}}{\det \left( x^{j - 1} \right)_{1 \leq j \leq n}}
.
\intertext{Inverting the order of the columns in both the numerator and the denominator yields}
\elementary(\calX)^m \schur_\lambda \left( \calX^{-1} \right) ={} & \frac{\det \left( x^{\tilde\lambda_j + n - j} \right)_{1 \leq j \leq n}}{\det \left( x^{n - j} \right)_{1 \leq j \leq n}} = \schur_{\tilde\lambda}(\calX)
. \qedhere
\end{align*}
\end{proof}

\begin{cor} [dual Cauchy identity] Let $\calX$ and $\calY$ be two sets of variables. It holds that
\begin{align} \label{3_cor_dual_cauchy_id_eq}
\sum_\lambda \schur_\lambda(\calX) \schur_{\lambda'}(\calY) = \prod_{\substack{x \in \calX \\ y \in \calY}} (1 + xy).
\end{align}
\end{cor}

\begin{proof} Suppose that $\calX$ and $\calY$ have length $n$ and $m$, respectively. Observe that only partitions $\lambda$ contained in the rectangle $\langle m^n \rangle$ contribute to the sum. As we may assume without loss of generality that no variable in $\calX$ vanishes, we can reformulate the left-hand side in \eqref{3_cor_dual_cauchy_id_eq} as
\begin{align*}
\sum_\lambda \schur_\lambda(\calX) \schur_{\lambda'}(\calY) ={} & \sum_{\lambda \subset \langle m^n \rangle} \elementary \left(\calX^{-1} \right)^{-m} \schur_{\tilde\lambda} \left(\calX^{-1}\right) \schur_{\lambda'}(\calY)
.
\intertext{Using the visual interpretation for $(m,n)$-complementary partitions, this reads}
\sum_\lambda \schur_\lambda(\calX) \schur_{\lambda'}(\calY) ={} & \elementary (\calX)^m \sum_{\pi \in \mathfrak{P}(m,n)} (-1)^{|\nu(\pi)|} \schur_{\mu(\pi)} \left(\calX^{-1}\right) \schur_{\nu(\pi)'}(-\calY)
\end{align*}
where we have also exploited the homogeneity of Schur functions to obtain a signed sum. This sum is essentially equal to the right-hand side in \eqref{3_cor_second_overlap_id_labeled_walks_eq}, specialized to the case that $\lambda$ is the empty partition. Hence, Corollary~\ref{3_cor_second_overlap_id_labeled_walks} allows us to conclude that the above expression is equal to the right-hand side in \eqref{3_cor_dual_cauchy_id_eq}.
\end{proof}

Therefore, the dual Cauchy identity can be viewed as a special case of the second overlap identity: it corresponds to pairs of partitions whose overlap is empty.


\chapter{A Combinatorial Approach to Mixed Ratios of Characteristic Polynomials from the Unitary Group} \label{4_cha_mixed_ratios}

\section{Introduction}
In this chapter we present a combinatorial method to derive formulas for averages of products of ratios and/or logarithmic derivatives of characteristic polynomials. More concretely, we give combinatorial expressions for integrals of the following type:
\begin{multline*}
\text{mixed ratio}(\calA, \calB, \calC, \calD, \calE, \calF) \\ = \int_{U(N)} \frac{\prod_{\alpha \in \calA} \chi_g(\alpha) \prod_{\beta \in \calB} \chi_{g^{-1}}(\beta)}{\prod_{\delta \in \calD} \chi_g(\delta) \prod_{\gamma \in \calC} \chi_{g^{-1}} (\gamma)} \prod_{\varepsilon \in \calE} \frac{\chi'_g(\varepsilon)}{\chi_g(\varepsilon)} \prod_{\varphi \in \calF} \frac{\chi'_{g^{-1}}(\varphi)}{\chi_{g^{-1}}(\varphi)} dg
\end{multline*}
where $\chi_g$ denotes the characteristic polynomial of $g \in U(N)$. Our method produces results of the form
\begin{align*}
\text{mixed ratio } = \text{ combinatorial main term } + \text{ error term}.
\end{align*}
In fact, we are only able to provide a neat asymptotic bound for the error term of $\text{mixed ratio}(\calA, \calB, \calC, \calD, \calE, \calF)$ as $N \to \infty$ under the assumption that at least one of the sets of variables is empty. Setting some sets of variables equal to the empty set results in the four theorems discussed in Section~\ref{4_sec_results}: 
\begin{itemize}
\item Upon specializing our formula for mixed ratios to $\calE = \emptyset = \calF$, we essentially recover Bump and Gamburd's ratio theorem \cite{bump06}. 
\item If we prescribe $\calF = \emptyset$ or $\calA = \emptyset = \calD$, we obtain new expressions for the corresponding mixed ratios. 
\item Glossing over a few technical details, $\text{mixed ratio}(\emptyset, \emptyset, \emptyset, \emptyset, \calE, \calF)$ provides a compact combinatorial expression for the main term of the average of products of logarithmic derivatives. 
\end{itemize}
In \cite{CS}, Conrey and Snaith derive a different formula for averages of products of logarithmic derivatives, without using any combinatorial tools. By definition the two expressions for the main term are equal; however, ours is a sum of products of monomial symmetric functions, while theirs is a rather complicated ad-hoc expression. 

The following expression for products of logarithmic derivatives of \emph{completed} characteristic polynomials, which holds inside the unit circle, constitutes the principal application of our formula for products of logarithmic derivatives of classic characteristic polynomials. (Completed characteristic polynomials $\Lambda_g$ are introduced on page \pageref{1_defn_completed_char_pol}.) Let $\calE$ and $\calF$ be sets of non-zero variables having absolute value strictly less than $1$, then
\begin{align} \label{4_eq_in_intro} 
\begin{split}
& \int_{U(N)} \prod_{\varepsilon \in \calE} \varepsilon \frac{\Lambda'_g(\varepsilon)}{\Lambda_g(\varepsilon)} \prod_{\varphi \in \calF} \varphi \frac{\Lambda'_{g^{-1}}(\varphi)}{\Lambda_{g^{-1}}(\varphi)} dg = \sum_\lambda \left( -\frac N2 \right)^{l(\calE) + l(\calF) - 2l(\lambda)}
z_\lambda \monomial_\lambda(\calE) \monomial_\lambda(\calF) + \error
.
\end{split}
\end{align}
An asymptotic bound on the error term as $N \to \infty$ is given in Theorem~\ref{4_thm_completed_log_ders}. 

This equality allows us to derive an explicit formula for the eigenvalues of unitary matrices. More precisely, let the functions $h(z)$ and $z \mapsto f(z, z_2, \dots, z_n)$ ``behave well'' in a neighborhood of the unit circle and let $f(z_1, \dots, z_n)$ be symmetric, then Theorem~\ref{4_thm_explicit_formula} provides a combinatorial formula for the following expression: 
\begin{align*}
\int_{U(N)} \sum_{1 \leq j_1, \dots, j_n \leq N} h(\rho_{j_1}) \cdots h(\rho_{j_n}) f(\rho_{j_1}, \dots, \rho_{j_n}) dg
\end{align*}
where for every matrix $g$, $\calR(g) = \{\rho_1, \dots, \rho_N\}$ stands for the multiset of its eigenvalues. Our focus lies on the analogy to explicit formulas for the zeros of $L$-functions. As motivated in Chapter~\ref{1_cha_intro}, we view characteristic polynomials as a model for $L$-functions. In consequence, we regard eigenvalues, which by definition are the zeros of characteristic polynomials, as a model for zeros of $L$-functions. Not only is the question motivated by this analogy between characteristic polynomials and $L$-functions, but our \emph{proof} of Theorem~\ref{4_thm_explicit_formula} also mirrors the \emph{proof} of the explicit formula for $L$-functions reproduced in Theorem~\ref{1_thm_explicit_formula}. The principal difficulty is that the derivation of the explicit formula for $L$-functions is based on the fact that sufficiently far to the right of the critical line $L$-functions can be written as Euler products; however, there is no natural analogue for the Euler product in the characteristic polynomials for unitary matrices. In order to circumvent this obstacle, we need an alternative way to describe characteristic polynomials inside the unit circle. This is exactly what the equality in \eqref{4_eq_in_intro} provides. 

\bigskip
Our method for computing $\text{mixed ratio}(\calA, \calB, \calC, \calD, \calE, \calF)$ is a generalization of Bump and Gamburd's approach to a combinatorial formula for averages of ratios of the form $\text{mixed ratio}(\calA, \calB, \calC, \calD, \emptyset, \emptyset)$ \cite{bump06}. As such it is based on the observation that the integrand is symmetric in both $\calR(g)$ and $\overline{\calR(g)}$, which implies that it can be written as a linear combination of the form
\begin{align*}
\sum_{\mu, \nu} \schur_\mu(\calR(g)) \overline{\schur_\nu(\calR(g))}
.
\end{align*}
What makes this generalization possible is a combination of three identities, namely the first overlap identity (presented in Chapter~\ref{3_cha_overlap_ids}), a new variant of the Murnaghan-Nakayama rule and an equality that is inspired by the vertex operator formalism that we have touched upon in Section~\ref{1_sec_my_results}. Firstly, we only need the simplest case of the first overlap identity, which corresponds to the sorting of the overlap being the identity permutation. Secondly, our variant of the Murnaghan-Nakayama rule (stated in Proposition~\ref{4_prop_MN_negative_r}) provides an explicit expression for the following signed sum, under quite restrictive assumptions:
\begin{align*}  
\sum_{\substack{\lambda: \\ \mu \setminus \lambda \text{ is a $k$-ribbon}}} (-1)^{\height(\mu \setminus \lambda)} \schur_\lambda(\calX).
\end{align*}
Ribbons are defined at the very end of Section~\ref{4_sec_sequences_and_partitions}. Thirdly, the equality that is related to the vertex operator formalism (stated in Lemma~\ref{4_lem_reduction_operators}) describes the interaction between two ``power sum'' operators on the ring of symmetric functions.

\subsection{Structure of this chapter}

In Section~\ref{4_sec_background_and_notation} we collect combinatorial definitions and formulas that the results presented in this chapter are based on. Section~\ref{4_sec_MN_rule} contains some extensions and variations of the Murnaghan-Nakayama rule, which are used to prove formulas for the average of products of ratios and/or logarithmic derivatives of characteristic polynomials over the unitary matrices in Section~\ref{4_sec_ratios_and_log_der}. In Section~\ref{4_sec_explicit_formula} we first recall 
the notion of the completed characteristic polynomial of a unitary matrix, and then present an expression for the average of its logarithmic derivatives, which will allow us to derive an explicit formula for the eigenvalues of unitary matrices.

\section{Background and notation} \label{4_sec_background_and_notation}

The reader who has read Chapters~\ref{1_cha_algebraic_combinatorics} and \ref{3_cha_overlap_ids} may safely skip this section and only go back to specific statements while reading the rest of the chapter, except for the paragraphs on specializations of the ring of symmetric functions and power sum operators. In addition, ribbons are defined at the very end of Section~\ref{4_sec_sequences_and_partitions}.

\subsection{Sequences and partitions} \label{4_sec_sequences_and_partitions}

For us a sequence is a \emph{finite} enumeration of elements, such as $\calX = \left(\calX_1, \dots, \calX_n \right)$. Its length is the number of its elements, denoted by $l(\calX) = n$. A subsequence $\calY$ of $\calX$ is a sequence given by $\calY_k = \calX_{n_k}$ where $1 \leq n_1 < n_2 < \dots \leq n$ is an increasing sequence of indices. If $K$ is a subsequence of $[n] = (1, \dots, n)$, then $\calX_K$ is shorthand for the subsequence of $\calX$ that corresponds to the indices in $K$. In consequence, any sequence of length $n$ contains exactly $2^n$ subsequences regardless of the number of repeated elements. If two sequences $\calX$ and $\calY$ of the same length are equal up to reordering their elements, we write $\calX \sorteq \calY$. The union of two sequences $\calX \cup \calY$ is obtained by appending $\calY$ to $\calX$; we sometimes add subscripts to indicate the lengths of the two sequences in question. The complement of a subsequence $\calY \subset \calX$ is the subsequence $\calX \setminus \calY$ of $\calX$ that satisfies $\calY \cup \left( \calX \setminus \calY \right) \sorteq \calX$. All operations on sequences that have not been mentioned are understood to be element wise. For instance, \label{symbol_abs} $\abs(\calX)$ is the sequence of absolute values of the elements of $\calX$. Moreover, we will write $\abs(\calX) < 1$ to indicate that all elements of the sequence $\calX$ are strictly less than 1 in absolute value. We do not denote the sequence of absolute values by $|\calX|$ (which would be more consistent with our usage of applying operations on sequences element by element) because vertical bars traditionally denote the size of a sequence or partition.

For sequences whose elements can be subtracted and multiplied, we define the following two functions: 
\begin{align*}
\Delta(\calX) = \prod_{1 \leq i < j \leq n} \left(\calX_i - \calX_j\right) \:\text{ and }\: \Delta(\calX; \calY) = \prod_{x \in \calX, y \in \calY} (x - y).
\end{align*}
We implicitly view all sets of variables as sequences but for simplicity of notation we will not fix the order of the variables explicitly. It is important, however, to stick to one order throughout a computation or within a formula.

A partition is a non-increasing sequence $\lambda = (\lambda_1, \dots, \lambda_n)$ of non-negative integers, called parts. If two partitions only differ by a sequence of zeros, we regard them as equal. By an abuse of notation, we say that the length of a partition is the length of the subsequence that consists of its positive parts. The size of a partition $\lambda$ is the sum of its parts, denoted $|\lambda|$. For any positive integer $i$, $m_i(\lambda)$ is the number of parts of $\lambda$ that are equal to $i$. It is sometimes convenient to use a notation for partitions that makes multiplicities explicit:
$$\lambda = \left\langle 1^{m_1(\lambda)} 2^{m_2(\lambda)} \dots i^{m_i(\lambda)} \dots \right\rangle.$$
The following statistic on the multiplicities will appear in some of our results:
$$z_\lambda = \prod_{i \geq 1} i^{m_i(\lambda)} m_i(\lambda)!$$

We will frequently view partitions as Ferrers diagrams. The Ferrers diagram associated to a partition $\lambda$ is defined as the set of points $(i,j) \in \Z \times \Z$ such that $1 \leq i \leq \lambda_j$; it is often convenient to visualize the points as square boxes. For instance, the Ferrers diagrams associated to partitions of the type $\langle m^n \rangle$ are just rectangles. The conjugate partition $\lambda'$ of $\lambda$ is given by the condition that the Ferrers diagram of $\lambda'$ is the transpose of the Ferrers diagram of $\lambda$. We note for later reference that if the union of two partitions $\mu$ and $\nu$ happens to be a partition, then $(\mu \cup \nu)' = \mu' + \nu'$. Given two partitions $\kappa$ and $\lambda$, we say that $\kappa$ is a subset of $\lambda$ if their Ferrers diagrams satisfy that containment relation. Note that $\kappa \subset \lambda$ is our shorthand for both subset and subsequence. It will be clear from the context whether we view $\kappa$ and $\lambda$ as sequences or diagrams. For a partition $\lambda$ that is contained in the rectangle $\langle m^n \rangle$, we call the partition
$$\tilde\lambda = (m - \lambda_n, \dots, m - \lambda_1) \subset \langle m^n \rangle$$
its $(m,n)$-complement.

If $\mu$ is a subset of $\lambda$, then the corresponding skew diagram is the set of boxes $\lambda \setminus \mu$ that are contained in $\lambda$ but not in $\mu$. A ribbon is a skew diagram that is connected and contains no $2 \times 2$ subset of boxes. What we call ribbon is also known as skew or rim hook \cite[p.~180]{Sagan}, and as border strip \cite[p.~5]{mac}. Let us illustrate this definition by some examples. The left-most diagram is a ribbon, while the other two violate one of the conditions. 
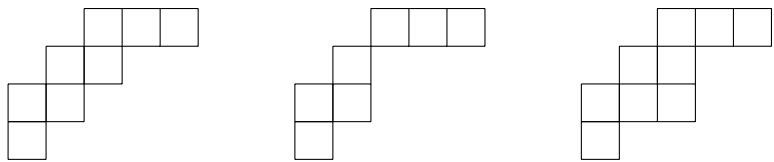
\begin{center}
\begin{tikzpicture} 
\draw[step=0.5cm, thin] (0, 0) grid (0.5, 0.5);
\draw[step=0.5cm, thin] (0, 0.5) grid (0.5, 1);
\draw[step=0.5cm, thin] (0.5, 0.5) grid (1, 1);
\draw[step=0.5cm, thin] (0.5, 1) grid (1, 1.5);
\draw[step=0.5cm, thin] (0.999, 0.999) grid (1.5, 1.5);
\draw[step=0.5cm, thin] (0.999, 1.5) grid (1.5, 2);
\draw[step=0.5cm, thin] (1.5, 1.499) grid (2, 2);
\draw[step=0.5cm, thin] (2, 1.499) grid (2.5, 2);
\end{tikzpicture}
\hspace{1cm}
\begin{tikzpicture} 
\draw[step=0.5cm, thin] (0, 0) grid (0.5, 0.5);
\draw[step=0.5cm, thin] (0, 0.5) grid (0.5, 1);
\draw[step=0.5cm, thin] (0.5, 0.5) grid (1, 1);
\draw[step=0.5cm, thin] (0.5, 1) grid (1, 1.5);
\draw[step=0.5cm, thin] (0.999, 1.499) grid (1.5, 2);
\draw[step=0.5cm, thin] (1.5, 1.499) grid (2, 2);
\draw[step=0.5cm, thin] (2, 1.499) grid (2.5, 2);
\end{tikzpicture}
\hspace{1cm}
\begin{tikzpicture} 
\draw[step=0.5cm, thin] (0, 0) grid (0.5, 0.5);
\draw[step=0.5cm, thin] (0, 0.5) grid (0.5, 1);
\draw[step=0.5cm, thin] (0.5, 0.5) grid (1, 1);
\draw[step=0.5cm, thin] (1, 0.5) grid (1.5, 1);
\draw[step=0.5cm, thin] (0.999, 1) grid (1.5, 1.5);
\draw[step=0.5cm, thin] (0.999, 1.5) grid (1.5, 2);
\draw[step=0.5cm, thin] (1.5, 1.499) grid (2, 2);
\draw[step=0.5cm, thin] (2, 1.499) grid (2.5, 2);
\draw[step=0.5cm, thin] (0.5, 1) grid (1, 1.5);
\end{tikzpicture}
\end{center}
The diagram in the middle illustrates that we only consider \emph{edgewise} connected skew diagrams connected.

The size of a ribbon is the number of its boxes. We sometimes call a ribbon of size $k$ a $k$-ribbon. The height ($\height$) of a ribbon is one less than the number of its rows. We use the following shorthand for the property that $\lambda \setminus \mu$ is a $k$-ribbon: \label{symbol_is_k_ribbon} $$\mu \overset{k}{\to} \lambda.$$ We note for later reference that $\lambda \setminus \mu$ is a $k$-ribbon if and only if $\lambda' \setminus \mu'$ is. In that case,
\begin{align} \label{4_eq_height_conjugate_ribbon}
\height \left(\lambda' \setminus \mu'\right) = k - 1 - \height \left( \lambda \setminus \mu \right).
\end{align}
For sequences $\lambda^{(0)} \overset{k_1}{\to} \lambda^{(1)} \overset{k_2}{\to} \dots \overset{k_n}{\to} \lambda^{(n)}$, the symbol $\height \left( \lambda^{(n)} \setminus \lambda^{(0)} \right)$ denotes the sum of the heights of the intermediate ribbons.

\subsection{The ring of symmetric functions}
In this section we introduce the most commonly used symmetric polynomials. In addition, we will briefly discuss the more abstract concept of symmetric functions, which is necessary in order to define specializations and operators.

\begin{defn} [monomial symmetric polynomials] Let $\calX = (x_1, \dots, x_n)$ be a set of variables and let $\lambda$ be a partition. If $l(\lambda) > n$, then the monomial symmetric polynomial $\monomial_\lambda(\calX)$ is identically zero; otherwise, 
$$\monomial_\lambda(\calX) = \sum_{\substack{(\alpha_1, \dots, \alpha_n): \\ (\alpha_1, \dots, \alpha_n) \sorteq \lambda}} x_1^{\alpha_1} \cdots x_n^{\alpha_n}
.$$
We remark that this definition makes use of the convention that any partition of length less than $n$ may be viewed as a sequence of length exactly $n$ by appending zeros. 
\end{defn}

These polynomials are called symmetric because they are invariant under permutations of the elements of $\calX$. The following definition lists three other commonly used families of symmetric polynomials.

\begin{defn} (power sums, elementary and complete symmetric polynomials) Let $k$ be a positive integer and let $\calX$ be a set of variables. 
\begin{enumerate}
\item The $k$-th elementary symmetric polynomial $\elementary_k(\calX)$ is given by $\monomial_{\left\langle1^k\right\rangle} (\calX)$, which is equal to the sum of all products of $k$ variables with distinct indices. We use the convention that $\elementary_0(\calX) = 1$.
\item The $k$-th complete symmetric polynomial $\complete_k(\calX)$ is equal to $\sum_{\lambda: |\lambda| = k} \monomial_\lambda(\calX)$. We use the convention that $\complete_0(\calX) = 1$.
\item The $k$-th power sum $\power_k(\calX)$ is defined by $\monomial_{(k)}(\calX) = \sum_{x \in \calX} x^k$.
\end{enumerate}
\end{defn}
Notice that for any set of variables $\calX$, the $l(\calX)$-th elementary polynomial $\elementary_{l(\calX)}(\calX)$ is simply the product of all variables. This observation motivates the following non-standard notation:
\begin{align*}
\elementary(\calX) = \prod_{x \in \calX} x
.
\end{align*} 

For theoretical considerations, it is often more convenient to work with symmetric functions instead of symmetric polynomials as they are not dependent on a set of variables. The monomial symmetric function corresponding to $\lambda$, which we denote by $\monomial_\lambda$, is determined by the condition that for any set of variables $\calX$, $\monomial_\lambda(\calX)$ is the monomial symmetric polynomial defined above. We will freely use this trick to get rid of the set of variables for all symmetric polynomials.

\begin{defn} [ring of symmetric functions] The ring of symmetric functions ($\Sym$) is the complex vector space spanned by the monomial symmetric functions $\monomial_\lambda$ where $\lambda$ runs over all partitions.
\end{defn}
Owing to the fact that the product of two symmetric polynomials is again symmetric, $\Sym$ is endowed with a natural ring structure. For a rigorous definition of the ring of symmetric functions turn to Section~\ref{1_sec_Schur_functions} or consult \cite[p.~17-19]{mac}. It turns out that the monomial symmetric functions are not the only natural basis for $\Sym$. If we use the convention that for any partition $\lambda$, $$\power_\lambda = \prod_{i \geq 1} \power_i^{m_i(\lambda)},$$ then the $\power_\lambda$ also form a basis of the ring of symmetric functions \cite[p.~24]{mac}. In fact, the same holds for the elementary and complete symmetric functions \cite[p.~20 and 22]{mac}.

\subsubsection{Schur functions}
Arguably the most natural basis for $\Sym$ is given by the Schur functions. We will see that they are orthonormal with respect to the Hall inner product. Moreover, they are the main link between the theory of symmetric functions and representation theory. We follow \cite{mac} in our presentation of Schur functions.

\begin{defn} [Schur functions] Let $\calX$ be a set of $n$ pairwise distinct variables and $\lambda$ a partition. If $l(\lambda) > n$, then $\schur_\lambda(\calX) = 0$; otherwise, 
\begin{align*}
\schur_\lambda(\calX) = \frac{\det \left( x^{\lambda_j + n - j} \right)_{x \in \calX, 1 \leq j \leq n}}{\Delta(\calX)}
.
\end{align*}
The fact that the polynomial $\Delta(\calX)$ is a divisor of the determinant in the numerator allows us to extend this definition to all sets of variables of length $n$.
\end{defn}

Technically, this defines a symmetric polynomial - not a symmetric function. For historical reasons, we call both $\schur_\lambda(\calX)$ and $\schur_\lambda$ the Schur function indexed by the partition $\lambda$. There are various definitions for Schur functions, each emphasizing a different aspect. In fact, the combinatorial definition (\textit{i.e.}\ Definition~\ref{1_defn_cominatorial_schur_function}) will also play a minor role in this chapter.

The Hall inner product on $\Sym$ is given by the condition that $\left\langle \complete_\lambda, \monomial_\mu \right\rangle = \delta_{\lambda \mu}$ for all partitions $\lambda, \mu$ where $\delta_{\lambda \mu}$ is the Kronecker delta. In order to state the main property of this inner product, we need to introduce the vector space $\Sym^k$, which is spanned by $\monomial_\lambda$ where $\lambda$ runs over all partitions $\lambda$ of size $k$. For each $k \geq 0$, let $u_\lambda$, $v_\lambda$ be bases of $\Sym^k$, indexed by partitions of size $k$. Then the following conditions are equivalent:
\begin{enumerate}
\item For all partitions $\lambda, \mu$, $\langle u_\lambda, v_\mu \rangle = \delta_{\lambda \mu}$.
\item For all sets of complex variables $\calX$, $\calY$ so that $|xy| < 1$ for all $x \in \calX$, $y \in \calY$, $\displaystyle \sum_\lambda u_\lambda(\calX) v_\lambda(\calY) = \prod_{\substack{x \in \calX \\ y \in \calY}} (1 - xy)^{-1}$.
\end{enumerate}

\begin{lem} [Cauchy identities] \label{4_lem_cauchy_identity_schur} Let $\calX$ and $\calY$ be two sets of variables with elements in $\C$ so that the product of any element in $\calX$ with any element in $\calY$ is strictly less than 1 in absolute value. The Cauchy identity states that
\begin{align} \label{4_eq_cauchy_id_schur}
\sum_{\lambda} \schur_\lambda(\calX) \schur_\lambda(\calY) ={} & \prod_{\substack{x \in \calX \\ y \in \calY}} (1 - xy)^{-1}
.
\intertext{Furthermore, what we call the power sum version of the Cauchy identity states that}
\label{4_eq_power_sum_version_cauchy_id}
\sum_{\lambda} \schur_\lambda (\calX) \schur_\lambda (\calY) ={} &\sum_\mu z_\mu^{-1} \power_\mu (\calX) \power_\mu (\calY)
\end{align}
where the three sums range over all partitions.
\end{lem}

In consequence, the Schur functions form an orthonormal basis for the ring of symmetric function, while the power sums satisfy $\left\langle \power_\lambda, \power_\mu \right\rangle = z_\lambda \delta_{\lambda \mu}$ for all partitions $\lambda, \mu$. The following Lemma gives another point of view on the orthonormality of Schur functions.

\begin{lem} [Schur orthogonality, \cite{BumpLieGroups}] \label{4_lem_Schur_ortho}
Let $U(N)$ denote the unitary group of degree $N$. As $U(N)$ is compact it possesses a unique Haar measure normalized so that the volume of the entire group is 1. Whenever we integrate over $U(N)$, we integrate with respect to this measure. If for each matrix $g \in U(N)$ we write $\calR(g)$ for the multiset of its eigenvalues, then
\begin{align*}
\int_{U(N)} \schur_\mu(\calR(g)) \overline{\schur_\nu(\calR(g))}dg = \begin{cases} 1 &\text{if $\mu = \nu$ and $l(\mu) \leq N$,} \\ 0 &\text{otherwise.}\end{cases}
\end{align*}
\end{lem}

\subsubsection{Specializations of the ring of symmetric functions} 
The definitions given in this paragraph are taken from \cite[p.~259]{Borodin2014}. A specialization $\rho$ of the ring of symmetric functions is an algebra homomorphism from $\Sym$ to $\C$. We denote the application of $\rho$ to a symmetric function $f$ as $f(\rho)$. For two specializations $\rho_1$ and $\rho_2$ we define their union $\rho = \rho_1 \cup \rho_2$ as the specialization defined on power sum symmetric functions via
$$\power_k (\rho_1 \cup \rho_2 ) = \power_k (\rho_1) + \power_k(\rho_2)$$
for all $k \geq 1$, and extended to all symmetric functions by the fact that the power sum symmetric functions form an algebraic basis of Sym. We note for later reference that 
\begin{align} \label{4_eq_union_of_specializations_power_sum}
z_\lambda^{-1} \power_\lambda(\rho_1 \cup \rho_2) & = \sum_{\substack{\mu, \nu: \\ \mu \cup \nu \sorteq \lambda}} z_\mu^{-1} \power_\mu(\rho_1) z_\nu^{-1} \power_\nu(\rho_2)
\intertext{and} \label{4_eq_union_of_specializations_schur}
\schur_\lambda(\rho_1 \cup \rho_2) & = \sum_{\mu, \nu} c_{\mu \nu}^\lambda \schur_\mu(\rho_1) \schur_{\nu}(\rho_2).
\end{align}
where $c^\lambda_{\mu \nu}$ are Littlewood-Richardson coefficients, described in Definition~\ref{1_defn_Littlewood-Richardson_coefficients}.

\begin{defn} Let $\omega$ be the involution on the ring of symmetric functions given by $\omega(\elementary_r) = \complete_r$. Recall that $\omega(\power_n) = (-1)^{n - 1} \power_n$ and $\omega(\schur_\lambda) = \schur_{\lambda'}$ \cite[p.~24 and 42]{mac}. We define the following two specializations:
\begin{align*} 
\rho^\alpha_\calX: \Sym \to \C; f \mapsto f(\calX) \;\text{ and }\; \rho^\beta_\calX: \Sym \to \C; f \mapsto \omega(f)(\calX)
.
\end{align*}
Borodin and Corwin call specializations of type $\rho^\alpha$ and \label{symbol_finite_length_specializations} $\rho^\beta$ finite length specializations and finite length dual specializations, respectively.
\end{defn}

\subsubsection{Power sum operators}
We define two types of power sum operators on the vector space $\Sym$. For $k \geq 1$, the $k$-th product operator, which we denote \label{symbol_k-th_product_operator} $\power_k$ by a slight abuse of notation, maps the symmetric function $f$ to the product $\power_k f$. In order to define the second type of operators, recall that any symmetric function $f$ can be uniquely written as a polynomial in the power sums $\power_1, \power_2, \dots$. For $k \geq 1$, the $k$-th derivation operator maps $f \in \Sym$ to the formal derivative of this polynomial with respect to $\power_k$; we denote it by \label{symbol_k-th_derivation_operator} $\displaystyle \tfrac{\partial}{\partial \power_k}$. In analogy to power sums, we use the convention that the $\lambda$-th product/derivation operator is given by the corresponding compositions of the respective operators: for a partition $\lambda$ of length $n$, \label{symbol_lambda-th_operator} $\power_\lambda = \power_{\lambda_1} \cdots \power_{\lambda_n}$ and $\tfrac{\partial}{\partial \power_\lambda} = \tfrac{\partial}{\partial \power_{\lambda_1}} \cdots \tfrac{\partial}{\partial \power_{\lambda_n}}$.

A definition of the two power sum operators as well as most of the properties described in the following two lemmas are given in \cite[p.~76]{mac}.

\begin{lem} \label{4_lem_properties_operators} Let $f$, $g$ be symmetric functions and $k$, $l$ strictly positive integers.
\begin{enumerate}
\item The two power sum operators satisfy the following commutation relations: 
$$ \frac{\partial}{\partial \power_k} \power_l f = \begin{dcases} \power_l \frac{\partial}{\partial \power_k} f & \text{if $l \neq k$,} \\ \power_k \frac{\partial}{\partial \power_k} f + f & \text{if $l = k$.} \end{dcases}
$$
\item The product and derivation operators are almost adjoint with respect to the Hall inner product; more precisely, $\displaystyle \left\langle k \frac{\partial}{\partial \power_k} f, g \right\rangle = \left\langle f, \power_k g \right\rangle$.
\end{enumerate}
\end{lem}

\begin{proof} The commutation relations are a direct consequence of the product rule for the derivative. To show the second property, it is enough to consider $f = \power_\mu$ and $g = \power_\nu$ since the $\power_\lambda$, where $\lambda$ ranges over all partitions, form a (linear) basis of $\Sym$. The fact that this basis is orthogonal implies that both sides of the equation vanish unless $\mu \sorteq \nu \cup (k)$. In this case,
\begin{align*}
\hspace{49.8pt} \left\langle k \frac{\partial}{\partial \power_k} \power_\mu, \power_\nu \right\rangle = \left\langle k m_k(\mu) \power_\nu, \power_\nu \right\rangle = k m_k(\mu) z_\nu =  z_\mu = \left\langle \power_\mu, \power_k \power_\nu \right\rangle
. \hspace{49.8pt} \qedhere
\end{align*}
\end{proof}

\begin{lem} \label{4_lem_reduction_operators} Let $\mu$ and $\nu$ be partitions. Then
\begin{align} \label{4_lem_reduction_operators_eq}
\frac{\partial}{\partial \power_\mu} \power_\nu = \begin{dcases} \prod_{i \geq 1} \frac{m_i(\nu)!}{m_i(\nu \setminus \mu)!} \power_{\nu \setminus \mu}  & \text{if $\mu \subset \nu$ as sequences} \\
0 & \text{otherwise}\end{dcases}
\end{align}
as elements of the ring of symmetric functions.
\end{lem}

\begin{proof} Given that the $l$-th product and the $k$-th derivation operator commute whenever $l \neq k$, we may write the left-hand side in \eqref{4_lem_reduction_operators_eq} as
\begin{align*}
\frac{\partial}{\partial \power_\mu} \power_\nu = \prod_{i \geq 1} \frac{\partial}{\partial \power_i^{m_i(\mu)}} \power_i^{m_i(\nu)}.
\end{align*}
Handling each factor separately gives the right-hand side in \eqref{4_lem_reduction_operators_eq}.
\end{proof}

Lemma~\ref{4_lem_reduction_operators} can be interpreted as moving the derivation operator to the right by means of the commutation relations in order to obtain a more concrete expression. We have taken this idea from \cite{Steeptilings}, in which it is used to simplify expressions involving another pair of operators that satisfy similar commutation relations.

\subsection{Littlewood-Schur functions}
\begin{defn} [Littlewood-Schur functions] \label{4_defn_comb_LS} Let $\calX$ and $\calY$ be two sets of variables. Define
$$LS_\lambda(\calX; \calY) = \sum_{\mu, \nu} c^\lambda_{\mu \nu} \schur_\mu(\calX) \schur_{\nu'}(\calY)$$
where the Littlewood-Richardson coefficients $c^\lambda_{\mu \nu}$ are given in Definition~\ref{1_defn_Littlewood-Richardson_coefficients}.
\end{defn}

The Littlewood-Schur function $LS_\lambda(\calX; \calY)$ is a polynomial in $\calX \cup \calY$. In contrast to the polynomials defined in the preceding section, Littlewood-Schur functions are not symmetric. However, this definition makes it apparent that $LS_\lambda(\calX; \calY)$ is symmetric in both sets of variables separately. This combinatorial approach can also be used to prove the following formula that generalizes the Cauchy identity (recalled in \eqref{4_eq_cauchy_id_schur}) as well as the dual Cauchy identity (stated in \eqref{1_eq_dual_cauchy_id}).

\begin{prop} [generalized Cauchy identity, \cite{berele_remmel}] \label{4_prop_gen_Cauchy}
Let $\calS$, $\calT$, $\calU$ and $\calV$ be sets of variables with elements in $\C$. Suppose that all numbers of the form $uv$ or $st$ with $s \in \calS$, $t \in \calT$, $u \in \calU$ and $v \in \calV$ are strictly less than 1 in absolute value. If the same holds for all numbers of the form $ut$ or for all numbers of the form $sv$, then
\begin{align*}
\sum_\lambda LS_\lambda(\calS; \calU) LS_{\lambda} (\calT; \calV) ={} & \prod_{\substack{s \in \calS \\ v \in \calV}} (1 + s v) \prod_{\substack{s \in \calS \\ t \in \calT}} (1 - s t)^{-1} \prod_{\substack{ u \in \calU \\ v \in \calV}} (1 - uv)^{-1} \prod_{\substack{u \in \calU \\ t \in \calT}} (1 + ut) .
\end{align*}
In particular, $\sum_\lambda \left| LS_\lambda(\calS; \calU) LS_{\lambda} (\calT; \calV) \right|$ possesses an upper bound that only depends on the absolute values of the elements in the four sets of variables in question.
\end{prop}

The last sentence is a consequence of the fact that Littlewood-Richardson coefficients are non-negative, which entails that 
\begin{align*}
|LS_\lambda(\calX; \calY)| \leq \sum_{\mu, \nu} c^\lambda_{\mu \nu} |\schur_\mu(\calX)| |\schur_{\nu'}(\calY)| \leq LS_\lambda(\abs(\calX); \abs(\calY)).
\end{align*}
We remark that the second inequality is due to the combinatorial definition for Schur functions (\textit{i.e.}\ Definition~\ref{1_defn_cominatorial_schur_function}).

\begin{rem} \label{4_rem_LS_as_specialization}
The theory of specializations provides an alternative expression for $LS_\lambda(\calX; \calY)$. Indeed,
\begin{align*}
LS_\lambda(\calX; \calY) = \sum_{\mu, \nu} c_{\mu \nu}^\lambda \schur_\mu(\calX) \schur_{\nu'}(\calY)
= \sum_{\mu, \nu} c_{\mu \nu}^\lambda \schur_\mu \left(\rho^\alpha_\calX\right) \schur_\nu\left(\rho^\beta_\calY\right)
= \schur_\lambda \left(\rho^\alpha_\calX \cup \rho^\beta_\calY\right).
\end{align*}
The last equality is due to \eqref{4_eq_union_of_specializations_schur}.
\end{rem}

This perspective allows us to consider Littlewood-Schur functions a special type of Schur functions, rendering the following equalities intuitive: $LS_\lambda(\calX; \calY) = LS_{\lambda'}(\calY; \calX)$ or $LS_\lambda(\calX; \emptyset) = \schur_\lambda(\calX)$.
Moens and Van der Jeugt's determinantal formula gives yet another way to view Littlewood-Schur functions. Their expression for $LS_\lambda(-\calX;\calY)$ depends on the index of the partition $\lambda$.

\begin{defn} [index of a partition] The $(m,n)$-index of a partition $\lambda$ is the largest (possibly negative) integer $k$ with the properties that $(m + 1 - k, n + 1 - k) \not\in \lambda$ and $k \leq \min\{m,n\}$.
\end{defn}

If $(m,n) \not\in \lambda$, then $k$ is the side of the largest square with bottom-right corner $(m,n)$ that fits next to the diagram of the partition $\lambda$. If $(m,n) \in \lambda$, then $-k$ is the side of the largest square with top-left corner $(m,n)$ that fits inside the diagram of $\lambda$. Let us illustrate this by a sketch: the hatched area is the diagram of some partition $\lambda$.
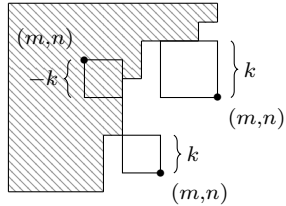
\begin{center}
\begin{tikzpicture}
\fill[pattern=north west lines, pattern color=black!40!white] (-0.25, 0) rectangle (1, 2.5);
\fill[pattern=north west lines, pattern color=black!40!white] (1, 0.75) rectangle (1.25, 2.5);
\fill[pattern=north west lines, pattern color=black!40!white] (1.25, 1.5) rectangle (1.5, 2.5);
\fill[pattern=north west lines, pattern color=black!40!white] (1.5, 2) rectangle (2.25, 2.5);
\fill[pattern=north west lines, pattern color=black!40!white] (2.25, 2.25) rectangle (2.5, 2.5);
\draw (-0.25, 0) -- (1, 0);
\draw (1, 0) -- (1, 0.75);
\draw (1, 0.75) -- (1.25, 0.75);
\draw (1.25, 0.75) -- (1.25, 1.5);
\draw (1.25, 1.5) -- (1.5, 1.5);
\draw (1.5, 1.5) -- (1.5, 2);
\draw (1.5, 2) -- (2.25, 2);
\draw (2.25, 2) -- (2.25, 2.25);
\draw (2.25, 2.25) -- (2.5, 2.25);
\draw (2.5, 2.25) -- (2.5, 2.5);
\draw (-0.25, 2.5) -- (2.5, 2.5);
\draw (-0.25, 2.5) -- (-0.25, 0);
\draw (1.25, 0.75) rectangle (1.75, 0.25);
\draw (1.75, 2) rectangle (2.5, 1.25);
\draw (1.25, 1.25) rectangle (0.75, 1.75);
\draw[decoration={brace, raise=5pt, mirror},decorate] (1.75, 0.25) -- node[right=6pt] {$\scriptstyle{k}$} (1.75, 0.75);
\draw[decoration={brace, raise=5pt, mirror},decorate] (2.5, 1.25) -- node[right=6pt] {$\scriptstyle{k}$} (2.5, 2);
\draw[decoration={brace, raise=5pt},decorate] (0.75, 1.25) -- node[left=6pt] {$\scriptstyle{-k}$} (0.75, 1.75);
\node[anchor=north west] at (1.75, 0.25) {$\scriptstyle{(m,n)}$};
\node[anchor=north west] at (2.5, 1.25) {$\scriptstyle{(m,n)}$};
\node[anchor=south east] at (0.75, 1.75) {$\scriptstyle{(m,n)}$};
\node at (1.75, 0.25) {\tiny{\textbullet}};
\node at (2.5, 1.25) {\tiny{\textbullet}};
\node at (0.75, 1.75) {\tiny{\textbullet}};
\node at (0, 2.6) {};
\end{tikzpicture}
\end{center}
We remark that the definition given above is not equivalent to the definition of index used in \cite{vanderjeugt}. Our notion has the advantage of being invariant under conjugation.

\begin{thm} [determinantal formula for Littlewood-Schur functions, adapted from \cite{vanderjeugt}] \label{4_thm_det_formula_for_Littlewood-Schur}
Let $\calX$ and $\calY$ be sets of variables of length $n$ and $m$, respectively, so that the elements of $\calX \cup \calY$ are pairwise distinct. Let $\lambda$ be a partition with $(m,n)$-index $k$. If $k$ is negative, then $LS_\lambda(-\calX; \calY) = 0$; otherwise,
\begin{align*}
LS_\lambda(-\calX; \calY) ={} & \varepsilon(\lambda) \frac{\Delta(\calY; \calX)}{\Delta(\calX) \Delta(\calY)} \det \begin{pmatrix} \left( (x - y)^{-1} \right)_{\substack{x \in \calX \\ y \in \calY}} & \left( x^{\lambda_j + n - m - j} \right)_{\substack{x \in \calX \\ 1 \leq j \leq n - k}} \\ \left( y^{\lambda'_i + m - n - i} \right)_{\substack{1 \leq i \leq m - k \\ y \in \calY}} & 0\end{pmatrix}
\end{align*}
where $\varepsilon(\lambda) = (-1)^{\left|\lambda_{[n - k]} \right|} (-1)^{mk} (-1)^{k(k - 1)/2}$.
\end{thm}

This theorem makes it easy to see that the Littlewood-Schur function $LS_\lambda(\calX: \calY)$ is a homogeneous polynomial in $\calX \cup \calY$ of degree $|\lambda|$. In fact, the necessary linear algebra manipulations are given in the proof of Lemma~\ref{2_lem_homogeneity_with_det_formula}. The following is another immediate consequence that will prove useful for the computations in Section~\ref{4_sec_ratios_and_log_der}. Corollary~\ref{4_cor_index=0} is a special case of Berele and Regev's factorization formula \cite{berele_regev}, which was originally derived without the help of the (more recent) determinantal formula. 

\begin{cor} [\cite{berele_regev}] \label{4_cor_index=0} Let $\calX$ and $\calY$ be sets of variables with $n$ and $m$ elements, respectively, and let $\lambda$ be a partition with $(m,n)$-index 0. If $l(\lambda) \leq n$, then
\begin{align} \label{4_cor_index=0_eq}
LS_\lambda(-\calX; \calY) = \Delta(\calY; \calX) \schur_{\lambda - \langle m^n \rangle} (-\calX)
.
\end{align}
\end{cor}

\begin{proof} First suppose that the elements in of $\calX \cup \calY$ are pairwise distinct. Then
\begin{align*}
LS_\lambda(-\calX; \calY) ={} & \varepsilon(\lambda) \frac{\Delta(\calY; \calX)}{\Delta(\calX) \Delta(\calY)} \det \begin{pmatrix} \left( (x - y)^{-1} \right)_{\substack{x \in \calX \\ y \in \calY}} & \left( x^{\lambda_j + n - m - j} \right)_{\substack{x \in \calX \\ 1 \leq j \leq n}} \\ \left( y^{\lambda'_i + m - n - i} \right)_{\substack{1 \leq i \leq m \\ y \in \calY}} & 0\end{pmatrix}
.
\intertext{The Schur blocks in the bottom-left and the top-right corner of the matrix are squares. Hence,}
LS_\lambda(-\calX; \calY) ={} & \varepsilon(\lambda) \frac{\Delta(\calY; \calX)}{\Delta(\calX) \Delta(\calY)} (-1)^{mn} \det \left( x^{\lambda_j + n - m - j} \right)_{\substack{x \in \calX \\ 1 \leq j \leq n}} \det \left( y^{\lambda'_i + m - n - i} \right)_{\substack{1 \leq i \leq m \\ y \in \calY}
.}
\intertext{On the one hand, the assumption that the length of $\lambda$ is less than $n$ implies that $\lambda'_i \leq n$. On the other hand, the assumption that the $(m,n)$-index of $\lambda$ is 0 implies that $\lambda'_i \geq n$ for $1 \leq i \leq m$. Therefore, the second determinant is actually a Vandermonde determinant, which cancels with $\Delta(\calY)$. This allows us to conclude that}
LS_\lambda(-\calX; \calY) ={} & \varepsilon(\lambda) \Delta(\calY; \calX) (-1)^{mn} \schur_{\lambda - \langle m^n \rangle} (\calX) = \Delta(\calY; \calX) \schur_{\lambda - \langle m^n \rangle} (-\calX)
\end{align*}
since $\varepsilon(\lambda) = (-1)^{|\lambda|}$ by the assumptions on $\lambda$. If the elements of $\calX \cup \calY$ are not pairwise distinct, the equality in \eqref{4_cor_index=0_eq} follows from the observation that both sides are polynomials in $\calX \cup \calY$, which agree on infinitely many points.
\end{proof}

A less immediate but equally useful consequence of the determinantal formula is the simplest case of the first overlap identity presented in Chapter~\ref{3_cha_overlap_ids}. We quickly recall it for the convenience of non-linear readers.

\begin{lem} \label{4_lem_first_overlap_id} Let $\calX$ and $\calY$ be sets of variables with $n$ and $m$ elements, respectively, so that the elements of $\calX$ are pairwise distinct. Let $\lambda$ be a partition with $(m,n)$-index $k$. If $0 \leq l \leq \min\{n - k, n\}$, then
\begin{align*} 
LS_{\lambda} (-\calX; \calY) ={} & \sum_{\substack{\calS, \calT \subset \calX: \\ \calS \cup_{l, n - l} \calT \sorteq \calX}} \frac{LS_{\lambda_{[l]} + \left\langle (n - l)^l \right\rangle}(-\calS; \calY) LS_{\lambda_{(l + 1, l + 2, \dots)}}(- \calT; \calY)}{\Delta(\calT; \calS)}
.
\end{align*}
\end{lem}

Recall that the subscripts in $\calS \cup_{l, n - l} \calT$ indicate that $l(S) = l$ and $l(\calT) = n - l$, respectively.

\section{On the Murnaghan-Nakayama rule} \label{4_sec_MN_rule}

This section is dedicated to the Murnaghan-Nakayama rule. After stating the rule in its original form we present a few generalizations and variations, some of which are new and some are already known.

\begin{thm} [Murnaghan-Nakayama rule, \cite{Murnaghan, Nakayama}] \label{4_thm_MN_rule} Let $\mu$ be a partition. For any strictly positive integer $k$,
\begin{align*}
\power_k \schur_\mu = \sum_{\substack{\lambda: \, \mu \overset{k}{\to} \lambda}} (-1)^{\height(\lambda \setminus \mu)} \schur_\lambda.
\end{align*}
\end{thm}

The following Corollary demonstrates what a powerful tool the theory of specializations of the symmetric group can be. It delivers the generalization of the Murnaghan-Nakayama rule to Littlewood-Schur almost for free. Neither the statement nor its proof is new but we did not manage to find an exact reference in the literature.

\begin{cor} [Murnaghan-Nakayama rule for Littlewood-Schur functions] \label{4_cor_MN_for_LS} Let $\mu$ be a partition. For any strictly positive integer $k$,
\begin{align} \label{4_cor_MN_for_LS_eq}
LS_\mu(\calX; \calY) \left[ \power_k(\calX) + (-1)^{k - 1} \power_k(\calY) \right] = \sum_{\substack{\lambda: \, \mu \overset{k}{\to} \lambda}} (-1)^{\height(\lambda \setminus \mu)} LS_\lambda(\calX; \calY).
\end{align} 
\end{cor}

\begin{proof} View the Littlewood-Schur functions $LS_\lambda(\calX; \calY)$ on the right-hand side in \eqref{4_cor_MN_for_LS_eq} as specializations of $\schur_\lambda$, following Remark~\ref{4_rem_LS_as_specialization}. Then use the Murnaghan-Nakayama rule to write the resulting sum as a specialization of $\power_k \schur_\lambda$, which equals the left-hand side in \eqref{4_cor_MN_for_LS_eq} when written out.
\end{proof}

On the one hand, the left-hand side of the Murnaghan-Nakayama rule can be viewed as a product of a power sum and a Schur function. On the other hand, it can be seen as applying the product operator to a Schur function. The second point of view immediately leads to the question whether there is a similar ``rule'' for the derivation operator.

\begin{cor} [dual Murnaghan-Nakayama rule] \label{4_cor_dual_MN_rule} Let $\lambda$ be a partition. For any strictly positive integer $k$,
\begin{align} \label{4_cor_dual_MN_rule_eq}
k \frac{\partial}{\partial \power_k} \schur_\lambda = \sum_{\substack{\mu: \, \mu \overset{k}{\to} \lambda}} (-1)^{\height(\lambda \setminus \mu)} \schur_\mu.
\end{align}
\end{cor}

The statement we have chosen to call the dual Murnaghan-Nakayama rule is standard but very elusive in the literature.

\begin{proof} Exploit that Schur functions form an orthonormal basis of $\Sym$ to write the left-hand side in \eqref{4_cor_dual_MN_rule_eq} as a linear combination of Schur functions:
\begin{align*}
k \frac{\partial}{\partial \power_k} \schur_\lambda ={} & \sum_{\mu} \left\langle k \frac{\partial}{\partial \power_k} \schur_\lambda, \schur_\mu \right\rangle \schur_\mu
.
\intertext{Using that derivation and product are almost adjoint (\textit{i.e.}\ Lemma~\ref{4_lem_properties_operators}), switch operators, and then apply the Murnaghan-Nakayama rule:}
k \frac{\partial}{\partial \power_k} \schur_\lambda ={} & \sum_{\mu} \left\langle \schur_\lambda, \power_k \schur_\mu \right\rangle \schur_\mu = \sum_{\mu} \sum_{\substack{\nu: \, \mu \overset{k}{\to} \nu}} (-1)^{\height(\nu \setminus \mu)} \left\langle \schur_\lambda, \schur_\nu \right\rangle \schur_\mu
.
\end{align*}  
The equality in \eqref{4_cor_dual_MN_rule_eq} now follows from the orthonormality of Schur functions.
\end{proof}

The derivation operator allows us to give a neat expression for the signed sum of Schur functions associated to $\mu$ where $\mu$ ranges over all partitions so that $\lambda \setminus \mu$ is a $k$-ribbon. However, this expression can be difficult to work with because for a general symmetric polynomial $f(\calX)$ it is hard to give an explicit expression for $\displaystyle \tfrac{\partial}{\partial \power_k} f \left( \rho^\alpha_\calX \right)$. The following proposition solves this problem under very specific assumptions, which will turn out to be sufficient for our purposes.

\begin{defn} Let $\calX$ be a set of non-zero variables. For any partition $\lambda$, we define the $-\lambda$-th power sum of $\calX$ by
$$\power_{-\lambda}(\calX) = \power_\lambda \left( \calX^{-1} \right)
.$$ 
We remark that $\power_{-\lambda}(\calX)$ is \emph{not} a symmetric polynomial in $\calX$, but a symmetric Laurent polynomial. 
\end{defn}

\begin{prop} \label{4_prop_MN_negative_r} Let $\calX$ consist of $n$ non-zero variables and let $\mu$ be a partition of length $n$. For any integer $k$ with $1 \leq k \leq \mu_n$,
\begin{align} \label{4_prop_MN_negative_r_eq}
\schur_\mu (\calX) \power_{-k} (\calX) = \sum_{\substack{\lambda: \, \lambda \overset{k}{\to} \mu}} (-1)^{\height(\mu \setminus \lambda)} \schur_\lambda(\calX).
\end{align}
\end{prop}

\begin{proof} Choose an integer $m$ such that $\mu$ is contained in the rectangle $\langle m^n \rangle$. Let $\tilde \mu$ denote the $(m,n)$-complement of $\mu$. We reformulate the left-hand side of the equation in \eqref{4_prop_MN_negative_r_eq} as a function in the variables $\calX^{-1}$:
\begin{align*}
\schur_\mu (\calX) \power_{-k} (\calX) ={} & \elementary\left(\calX^{-1}\right)^{-m} \schur_{\tilde \mu} \left(\calX^{-1}\right) \power_k \left(\calX^{-1}\right)
.
\intertext{This trick is an immediate consequence of the determinantal definition for Schur functions. A more detailed justification can be found in Lemma~\ref{3_lem_Schur_indexed_by_complement}. Applying the Murnaghan-Nakayama rule yields}
\schur_\mu (\calX) \power_{-k} (\calX) ={} & \elementary \left(\calX^{-1} \right)^{-m} \sum_{\substack{\nu: \, \tilde\mu \overset{k}{\to} \nu}} (-1)^{\height(\nu \setminus \tilde\mu)} \schur_\nu \left(\calX^{-1}\right)
.
\intertext{Notice that $\nu_1 \leq \tilde\mu_1 + k \leq \tilde\mu_1 + \mu_n = m$. Hence, all partitions $\nu$ that contribute to the sum are contained in $\langle m^n \rangle$. In consequence, the $(m,n)$-complement of $\nu$ is well defined. Replace $\lambda$ by $\tilde\nu$ in the summation index and reuse the trick to obtain}
\schur_\mu (\calX) \power_{-k} (\calX) ={} & \sum_{\substack{\lambda: \, \tilde\mu \overset{k}{\to} \tilde\lambda}} (-1)^{\height\left(\tilde\lambda \setminus \tilde\mu\right)} \schur_\lambda(\calX)
.
\end{align*}
It is easy to see that $\tilde\lambda \setminus \tilde\mu$ is a $k$-ribbon if and only if $\mu \setminus \lambda$ is. Together with the fact that the height remains unaltered this proves the claim.
\end{proof}

\begin{cor} \label{4_cor_MN_negative_lambda} Let $\calX$ consist of $n$ non-zero variables and let $\mu$ be a partition of length at most $n$. For any partition $\lambda$ with $|\lambda| \leq \mu_n$,
\begin{align} \label{4_cor_MN_negative_lambda_eq}
\left[ \prod_{i \geq 1} i^{m_i(\lambda)} \frac{\partial}{\partial \power_\lambda} \schur_\mu \right] \left(\rho^\alpha_\calX\right) = \schur_\mu(\calX) \power_{-\lambda}(\calX)
.
\end{align}
\end{cor}

\begin{proof} Set $l(\lambda) = l$. Repeated application of Corollary~\ref{4_cor_dual_MN_rule} to the left-hand side in \eqref{4_cor_MN_negative_lambda_eq} allows us to reformulate it as
\begin{align*}
\sum_{\substack{\mu^{(1)}, \dots, \mu^{(l)}: \\ \mu^{(l)} \overset{\lambda_l}{\to} \dots \overset{\lambda_2}{\to} \mu^{(1)} \overset{\lambda_1}{\to} \mu}} (-1)^{\height \left( \mu \setminus \mu^{(l)} \right)} \schur_{\mu^{(l)}} \left( \rho^\alpha_\calX \right)
.
\end{align*}
By definition, $\schur_{\mu^{(l)}} \left( \rho^\alpha_\calX \right) = \schur_{\mu^{(l)}} \left(\calX \right)$. Given that
\begin{align*}
\lambda_i \leq |\lambda| - \lambda_{i - 1} - \dots - \lambda_1 \leq \mu_n - \lambda_{i - 1} - \dots - \lambda_1 \leq \mu^{(i - 1)}_n
\end{align*}
for all $1 \leq i \leq l$, repeatedly applying Proposition~\ref{4_prop_MN_negative_r} results in the right-hand side of the equation in \eqref{4_cor_MN_negative_lambda_eq}.
\end{proof}

\section{Averages of ratios and logarithmic derivatives \\ of characteristic polynomials} \label{4_sec_ratios_and_log_der}

In this section we present a unified way to derive formulas for averages of products of ratios and/or logarithmic derivatives of characteristic polynomials over the group of unitary matrices $U(N)$. Most of these formulas are not exact but contain an error that decreases exponentially as $N$ goes to infinity.

\subsection{Tricks for bounding the error term}
As the heading suggests, this section is a collection of observations that will allow us to give asymptotic bounds for the various error terms. They are not particularly hard to prove or interesting in their own right.

\begin{lem} \label{4_lem_bound_number_of_ribbons_rectangle} Fix a positive integer $k$ and a partition $\lambda \subset \langle m^n \rangle$. Then there are at most $\min \{m,n\}$ partitions $\mu$ such that $\begin{cases} \text{$\lambda \setminus \mu$ is a $k$-ribbon.} \\ \text{$\mu \setminus \lambda$ is a $k$-ribbon and $\mu \subset \langle m^n \rangle$.} \end{cases}$
\end{lem}

\begin{proof} Let $\tilde \lambda$ denote the $(m,n)$-complement of the partition $\lambda$. For every partition $\mu \subset \langle m^n\rangle$, $\lambda \setminus \mu$ is a $k$-ribbon if and only if $\tilde \mu \setminus \tilde \lambda$ is a $k$-ribbon, where $\tilde \mu$ denotes the $(m,n)$-complement of $\mu$. Hence, it is sufficient to bound the number of partitions $\mu$ such that $\lambda \setminus \mu$ is a $k$-ribbon.

The condition that $\mu$ be a partition implies that the top-right box of any ribbon $\lambda \setminus \mu$ must not have any box to its right that is contained in $\lambda$. This gives at most $n$ possible positions for the top-right box, which entails that there are at most $n$ partitions $\mu$ such that $\lambda \setminus \mu$ is a $k$-ribbon. An analogous argument based on the bottom-left box of the ribbons bounds their number by $m$, thus concluding the proof.
\end{proof}

Before going on to the next trick we recall the big-$O$ notation -- primarily to fix notation. Given two functions $f$ and $g$ with domain $\calX$, we write \label{symbol_big_o_notation} $f = O_\calP(g)$ if there exists a real constant $c(\calP)$ that may depend on the set of parameters $\calP$ such that $|f(x)| \leq c(\calP)|g(x)|$ for all $x \in \calX$. In this setting, we call $c(\calP)$ the implicit constant. 

The following notation will also appear in the bounds for the error terms: the positive part of a real number $x$ is denoted by \label{symbol_positive_part} $x^+ = \max\{x,0\}$.

\begin{lem} \label{4_lem_bound_det} Fix a natural number $n$. For all square matrices $A$ whose size is less than $n$,
\begin{enumerate}
\item $\displaystyle \det A = O_n \left( \prod_{j = 1}^m \max_{1 \leq i \leq m} |a_{ij}| \right)$
\item $\displaystyle \det A = O_n \left( \prod_{i = 1}^m \max_{1 \leq j \leq m} |a_{ij}| \right)$
\end{enumerate}
where $m$ denotes the size of $A$.
\end{lem}

\begin{proof} Both statements follow directly from the Leibniz formula for determinants. We only give a justification for the first statement, as they are exact analogues. We have that
\begin{align*}
\hspace{33.4pt} |\det A| = \left| \sum_{\sigma \in S_m} \varepsilon(\sigma) \prod_{j = 1}^m a_{\sigma(j)j} \right| 
\leq \sum_{\sigma \in S_m} \prod_{j = 1}^m \max_{1 \leq i \leq m} \left| a_{ij} \right| \leq n! \prod_{j = 1}^m \max_{1 \leq i \leq m} \left| a_{ij} \right|
. \hspace{33.4pt}
\qedhere
\end{align*}
\end{proof}

We will solely use this lemma to infer asymptotic bounds for Schur and Littlewood-Schur functions based on their determinantal definitions.

\begin{lem} \label{4_lem_det_bound_Schur} Fix a positive number $r$ and a set $\calX$ of pairwise distinct variables.
\begin{enumerate} 
\item If $\abs(\calX) \leq r$, then $\displaystyle \schur_\lambda(\calX) = O_\calX \left( r^{|\lambda|} \right)$ as a function of $\lambda$.
\item If $\calY$ is the subsequence of $\calX$ that consists of the elements of absolute value greater than 1, $\displaystyle \schur_\lambda(\calX) = O_\calX \left( \elementary (\calY)^{\lambda_1} \right)$ as a function of $\lambda$.
\end{enumerate}
\end{lem}

\begin{proof} Set $l(\calX) = n$. To show the first bound, suppose that $\abs(\calX) \leq r$. By the determinantal definition for Schur functions,
\begin{align*}
\schur_\lambda(\calX) ={} & O_\calX \left( \det\left(x^{\lambda_j + n - j}\right)_{\substack{x \in \calX \\ 1 \leq j  \leq n}}\right) = O_\calX \left( \prod_{j = 1}^n \max_{x \in \calX} \left| x^{\lambda_j + n - j} \right| \right) = O_\calX \left( r^{|\lambda|}\right)
\end{align*}
where the second equality is due to the first statement of Lemma~\ref{4_lem_bound_det}. The second bound in this lemma is a consequence of the second statement of Lemma~\ref{4_lem_bound_det}:
\begin{align*}
\schur_\lambda(\calX) = O_\calX \left( \det\left(x^{\lambda_j + n - j}\right)_{\substack{x \in \calX \\ 1 \leq j  \leq n}}\right) ={} & O_\calX \left( \prod_{x \in \calX} \max_{1 \leq j \leq n} \left|x^{\lambda_j + n - j} \right| \right) \\
={} & O_\calX \left( \prod_{y \in \calY} \max_{1 \leq j \leq n} \left|y^{\lambda_j + n - j} \right| \right) = O_\calX \left( \prod_{y \in \calY} y^{\lambda_1} \right)
. 
\end{align*}
This concludes the proof since $\elementary(\calY)$ without index is our notation for the product.
\end{proof}

The first bound in Lemma \ref{4_lem_det_bound_Schur} can be viewed as a special case of Lemma~\ref{4_lem_det_bound_LS}, which gives an analogous statement for Littlewood-Schur functions.

\begin{lem} \label{4_lem_det_bound_LS} Fix a natural number $l$, a positive number $r$ and two sets of variables $\calX$ and $\calY$ such that $\abs(\calX) \leq r$ and the elements of $\calX \cup \calY$ are pairwise distinct. As a function of partitions $\lambda$ with $l(\lambda) \leq l$, 
\begin{align*}
LS_\lambda(-\calX; \calY) = O_\calP \left( r^{|\lambda|} \right)
\end{align*}
where the implicit constant depends on $\calP = \{\calX, \calY, l\}$.
\end{lem}

\begin{proof} Set $l(\calX) = n$, $l(\calY) = m$ and denote the $(m,n)$-index of $\lambda$ by $k$. We remark that $k$ depends on $\lambda$, while $m$ and $n$ are constants. The determinantal formula for Littlewood-Schur functions (\textit{i.e.}\ Theorem~\ref{4_thm_det_formula_for_Littlewood-Schur}) entails that if $k$ is non-negative 
\begin{align*}
LS_\lambda(-\calX, \calY) ={} & O_{\calX, \calY} \left( \det \begin{pmatrix} \left( (x - y)^{-1} \right)_{\substack{x \in \calX \\ y \in \calY}} & \left( x^{\lambda_j + n - m - j} \right)_{\substack{x \in \calX \\ 1 \leq j \leq n - k}} \\ \left( y^{\lambda'_i + m - n - i} \right)_{\substack{1 \leq i \leq m - k \\ y \in \calY}} & 0\end{pmatrix} \right)
;
\intertext{otherwise, the Littlewood-Schur function $LS_\lambda(-\calX; \calY)$ vanishes, allowing us to ignore the case $k < 0$. Let us call this matrix $A$. As the size of $A$ is $m + n - k \leq m + n$ for all partitions $\lambda$, Lemma~\ref{4_lem_bound_det} states that}
LS_\lambda (-\calX; \calY) ={} & O_{\calX, \calY} \left( \prod_{j = 1}^{m + n - k} \max_{1 \leq i \leq m + n - k} |a_{ij}| \right)
.
\intertext{The condition that $l(\lambda) \leq l$ is equivalent to $\lambda'_1 \leq l$, and thus implies that $\lambda'_i \leq l$ for all $i$. Hence, the $m$ first columns of $A$ make no asymptotically relevant contribution to the bound. Therefore,} 
LS_\lambda (-\calX; \calY) ={} & O_{\calX, \calY, l} \left( \prod_{j = 1}^{n - k} \max_{x \in \calX} \left| x^{\lambda_j + n - m - j} \right| \right) = O_{\calX, \calY, l} \left( r^{\left| \lambda_{[n - k]} \right|} \right)
.
\end{align*}
By the definition of index, $\lambda_i \leq m - k$ for all indices $i > n - k$. Combined with the condition that $l(\lambda) \leq l$, we infer that
$$\left| \lambda_{[n - k]} \right| \leq \left| \lambda \right| \leq \left|\lambda_{[n - k]}\right| + (l - (n - k))(m - k) \leq \left|\lambda_{[n - k]}\right| + lm
.$$
Therefore, if $\displaystyle \begin{cases} r \geq 1 \\ r \leq 1\end{cases}\!\!\!\!$, then $\displaystyle r^{\left|\lambda_{[n - k]}\right|} \leq \begin{dcases} r^{|\lambda|} \\ r^{|\lambda|  - lm}\end{dcases} = O_{l,m, r} \left( r^{|\lambda|} \right)$.
\end{proof}

If we drop the condition that the variables in $\calX \cup \calY$ be pairwise distinct, we can no longer use Lemma~\ref{4_lem_det_bound_LS} to obtain an asymptotic bound on $LS_\lambda(\calX; \calY)$, given that the implicit constant might grow arbitrarily large whenever elements of $\calX \cup \calY$ converge towards each other. The following lemmas, which are based on the combinatorial definitions for Schur and Littlewood-Schur functions, provide bounds that do not depend on $\calX$ and $\calY$. In particular, the variables need not to be pairwise distinct. However, the bounds stated in the combinatorial lemmas are not as good as the bounds based on the determinantal definition.

\begin{lem} \label{4_lem_comb_bound_schur} Fix a positive number $r$ and a set of variables $\calX$ such that $\abs(\calX) \leq r$. As a function of $\lambda$,
\begin{align*}
\schur_\lambda(\calX) = O \left( |\lambda|^{l(\calX)^2} r^{|\lambda|} \right)
.
\end{align*}
\end{lem}

\begin{proof} Owing to the combinatorial definition for Schur functions (\textit{i.e.}\ Definition~\ref{1_defn_cominatorial_schur_function}),  
\begin{align*}
\left| \schur_\lambda(\calX) \right| \leq \sum_T r^{|\lambda|}
\end{align*}
where the sum runs over all semistandard $\lambda$-tableaux $T$ whose entries do not exceed $l(\calX)$. Hence, it suffices to bound the number of tableaux that contribute to the sum. Given that the rows/columns of $T$ are weakly/strongly increasing, there are at most $\lambda_1 \cdots \lambda_i \leq |\lambda|^i$ possible choices for the boxes of $T$ that contain the positive integer $i$. Multiplying over all $1 \leq i \leq l(\calX)$ gives the desired bound.
\end{proof}

\begin{lem} \label{4_lem_comb_bound_LS} Fix a natural number $l$, a positive number $r$ and two sets of variables $\calX$ and $\calY$ such that $\abs(\calX) \leq r$. As a function of partitions $\lambda$ with $l(\lambda) \leq l$, 
\begin{align*}
LS_\lambda(\calX; \calY) = O_\calP \left( |\lambda|^{l(\calX)^2 + l^2}  r^{|\lambda|} \right)
\end{align*}
where the implicit constant depends on $\calP = \{l, r, l(\calY), \max(\abs(\calY))\}$.
\end{lem}

\begin{proof} Applying Lemma~\ref{4_lem_comb_bound_schur} to the combinatorial definition for Littlewood-Schur functions (\textit{i.e.}\ Definition~\ref{4_defn_comb_LS}) gives us
\begin{align*}
LS_\lambda(\calX; \calY) ={} & \sum_{\substack{\mu, \nu: \\ \nu_1 \leq l(\calY)}} c^\lambda_{\mu \nu} O \left( |\mu|^{l(\calX)^2} r^{|\mu|} |\nu|^{l(\calY)^2} R^{|\nu|} \right)
\intertext{where $R = \max(\abs(\calY))$. Since the Littlewood-Richardson coefficient $c^\lambda_{\mu \nu}$ vanishes unless $\nu$ is a subset of $\lambda$, only partitions $\nu$ contained in the rectangle $\left\langle l(\calY)^l \right\rangle$ appear in the sum. Hence, the fact that $|\nu| + |\mu| = |\lambda|$ for all partitions that contribute to the sum entails that $|\mu| \leq |\lambda| \leq |\mu| + l(\calY)l$, which allows us to replace $|\mu|$ by $|\lambda|$. Keeping track of the fact that $\mu \subset \lambda$ whenever $c^\lambda_{\mu \nu} \neq 0$, we thus have that}
LS_\lambda(\calX; \calY) ={} & O_{l, r, l(\calY), R} \left( |\lambda|^{l(\calX)^2} r^{|\lambda|} \right) \sum_{\substack{\mu, \nu: \\ \nu \subset \left\langle l(\calY)^l \right\rangle \\ \mu \subset \lambda \\ |\mu| + |\nu| = |\lambda| }} c^\lambda_{\mu \nu} 
.
\end{align*}
According to the Littlewood-Richardson rule (\textit{i.e.}\ Theorem~\ref{1_thm_Littlewood_Richardson rule}), $c^\lambda_{\mu \nu}$ can be bounded by the number of skew semistandard $\lambda \setminus \nu$-tableaux $T$ with weight $\mu$. For each positive integer $i$, there are at most $\lambda_1 \cdots \lambda_l \leq |\lambda|^l$ ways to choose the boxes of $T$ that contain $i$ (given that $l(\lambda) \leq l$). The condition that $l(\mu) \leq l$ thus implies that $c^\lambda_{\mu \nu} \leq |\lambda|^{l^2}$.

The bound stated above now follows from the observation that the number of pairs $\mu, \nu$ to sum over is less than $l(\calY)^l \times (l(\calY) l)^l$. Indeed, there are less than $l(\calY)^l$ partitions $\nu$ that are contained in the rectangle $\left\langle l(\calY)^l \right\rangle$. Fixing a partition $\nu$, the conditions that $\mu \subset \lambda$ and $|\mu| = |\lambda| - |\nu|$ allow us to infer that there are at most $|\nu| \leq l(\calY) l$ ways to choose a part $\mu_i$ for $1 \leq i \leq l$.
\end{proof}

\subsection{The recipe}
Before stating the recipe we quickly recall our notation for the characteristic polynomial.
\begin{defn} [characteristic polynomial] The characteristic polynomial of a unitary matrix $g \in U(N)$ is given by $\chi_g(z) = \det \left( I - zg^{-1}\right)$ where $I$ is the identity matrix.
\end{defn}

\begin{recipe*} [ratios and logarithmic derivatives] Let $\calA$, $\calB$, $\calC$, $\calD$, $\calE$ and $\calF$ be sets of non-zero variables so that the four latter only contain elements that are strictly less than 1 in absolute value. If $l(\calD) \leq l(\calA)$ and the elements of $\calA \cup \calB^{-1}$ are pairwise distinct, then 
\begin{align} \label{4_recipe_eq}
\begin{split}
& \hspace{-15pt} \int_{U(N)} \frac{\prod_{\alpha \in \calA} \chi_g(\alpha) \prod_{\beta \in \calB} \chi_{g^{-1}}(\beta)}{\prod_{\delta \in \calD} \chi_g(\delta) \prod_{\gamma \in \calC} \chi_{g^{-1}} (\gamma)} \prod_{\varepsilon \in \calE} \frac{\chi'_g(\varepsilon)}{\chi_g(\varepsilon)} \prod_{\varphi \in \calF} \frac{\chi'_{g^{-1}}(\varphi)}{\chi_{g^{-1}}(\varphi)} dg \\
={} & \elementary (-\calB)^N \sum_{\substack{\calS, \calT \subset \calA \cup \calB^{-1}: \\ \calS \cup_{l(\calB), l(\calA)} \calT \sorteq \calA \cup \calB^{-1}}} \elementary (-\calS)^{N + l(\calA) - l(\calD)} \frac{\Delta(\calD; \calS)}{\Delta(\calT; \calS)} \\
& \times (-1)^{l(\calE) + l(\calF)} \sum_{\substack{\calE', \calE'' \subset \calE: \\ \calE' \cup \calE'' \sorteq \calE}} \sum_{\substack{q,n \geq 0: \\ q + n \leq N - l(\calC)}} \left( \sum_{\substack{\chi: \\ l(\chi) = l\left(\calE''\right) \\ |\chi| = q}} \monomial_{\chi - \left\langle 1^{l\left(\calE''\right)} \right\rangle} \left(-\calE''\right) \power_{-\chi}(-\calS) \right) \\
& \times \sum_{\substack{\psi: \\ l(\psi) = l\left(\calE'\right)}} \monomial_{\psi - \left\langle 1^{l\left(\calE'\right)} \right\rangle} \left(\calE'\right) \sum_{\substack{\omega: \\ l(\omega) = l(\calF) \\ |\omega| = n}} \monomial_{\omega - \left\langle 1^{l(\calF)} \right\rangle} (\calF) \\
& \times \sum_{\substack{\lambda, \xi: \\ \omega \cup \xi \sorteq \psi \cup \lambda}} z_\lambda^{-1} \power_\lambda \left( \rho^\beta_{-\calT} \cup \rho^\alpha_{\calD} \right) \prod_{i \geq 1} \frac{i^{m_i(\omega)} m_i(\psi \cup \lambda)!}{m_i(\xi)!} \power_\xi (\calC)
\\
& + \error
.
\end{split}
\end{align}
An asymptotic bound for the error is given in \eqref{4_eq_error_bound_for_recipe} on page \pageref{4_eq_error_bound_for_recipe}.
\end{recipe*}

We call this statement a recipe rather than a theorem because we are not able to give a neat bound for the error term. In particular, the error term might be larger than the main term. However, when some of the sets of variables are empty the error term becomes more tractable, which will allow us to prove the results presented in Section~\ref{4_sec_results}. In this sense the recipe provides a unified way of showing formulas for products of ratios and/or logarithmic derivatives. 

On a more technical note, observe that for $g \in U(N)$ and $z \in \C \setminus \{0\}$,
\begin{align*}
\chi_g(z) = \det \left(I - z g^{-1}\right) = \det\left(-zg^{-1}\right) \det\left(-z^{-1}g + I\right) = (-z)^N \overline{\elementary(\calR(g))} \chi_{g^{-1}} \left(z^{-1}\right)
\end{align*}
where $\calR(g)$ is the multiset of eigenvalues of $g$. Considering the integrand on the left-hand side of \eqref{4_recipe_eq}, we see that this observation allows us to replace $\chi_g(\delta)$ by $\chi_{g^{-1}}(\gamma)$ with $\gamma = \delta^{-1}$ at the cost of a factor which is easy to handle. Hence, for any $r \in \R \setminus \{0\}$, the condition that $\abs(\calC)$, $\abs(\calD) \leq r$ is essentially equivalent to the condition that all elements of $\calC \cup \calD$ are less than $r$ or greater than $r^{-1}$ in absolute value. Moreover, the same holds for the sets of variables $\calA$ and $\calB$. In particular, prerequisites of the type $\abs(\calA)$, $\abs(\calB) \leq 1$ are essentially empty conditions. Of course, one has to be careful not to violate other conditions, such as $l(\calD) \leq l(\calA)$, when using this trick. 

\begin{proof} 
This proof is based on the observation that the integrand on the left-hand side is symmetric in the eigenvalues of the unitary matrix $g$, say $\calR(g)$, as well as in their complex conjugates $\overline{\calR(g)}$. It is thus (at least theoretically) possible to express the integrand as an infinite linear combination of products of Schur functions of the form 
$$\overline{\schur_\lambda(\calR(g))} \schur_\kappa(\calR(g)).$$
Once the coefficients of this linear combination are known, Schur orthogonality immediately gives an expression for the integral on the left-hand side. At the cost of an error (which is ultimately due to the fact that Schur functions are only \emph{essentially} orthonormal), we then simplify this expression by applying results presented in the preceding sections.

Given that $\abs(\calC)$, $\abs(\calD) < 1$ elementary linear algebra manipulations together with the generalized Cauchy identity (\textit{i.e.}\ Proposition~\ref{4_prop_gen_Cauchy}) give the following expression for the ratios on the left-hand side in \eqref{4_recipe_eq}:
\begin{align*}
& \frac{\prod_{\alpha \in \calA} \chi_g(\alpha) \prod_{\beta \in \calB} \chi_{g^{-1}}(\beta)}{\prod_{\delta \in \calD} \chi_g(\delta) \prod_{\gamma \in \calC} \chi_{g^{-1}} (\gamma)} \displaybreak[2]\\ 
={} & \prod_{\alpha \in \calA} \det \left( I - \alpha g^{-1} \right) \prod_{\beta \in \calB} \left[ \det(g) \det(-\beta I) \det \left( -\beta^{-1} g^{-1} + I \right) \right] \\
& \times \prod_{\delta \in \calD} \det \left( I - \delta g^{-1} \right)^{-1} \prod_{\gamma \in \calC} \det \left( I - \gamma g \right)^{-1} \displaybreak[2] \\
={} & \elementary (-\calB)^N \det(g)^{l(\calB)} \prod_{\substack{x \in \calA \cup \calB^{-1} \\ \rho \in \calR(g)}} (1 - x\overline{\rho}) \prod_{\substack{\delta \in \calD \\ \rho \in \calR(g)}} (1 - \delta \overline{\rho})^{-1} \prod_{\substack{\gamma \in \calC \\ \rho \in \calR(g)}} (1 - \gamma \rho)^{-1} \displaybreak[2]\\
={} & \elementary (-\calB)^N \elementary(\calR(g))^{l(\calB)} \left( \sum_{\lambda} LS_{\lambda'}\left(- \left(\calA \cup \calB^{-1}\right); \calD \right) \overline{\schur_\lambda(\calR(g))} \right) \left( \sum_{\kappa} \schur_\kappa (\calC) \schur_\kappa(\calR(g)) \right)
.
\end{align*}
Bump and Gamburd use the same algebraic manipulations and similar Cauchy identities to write ratios of characteristic polynomials in terms of Schur functions \cite[p.~245-246]{bump06}.
Furthermore, Dehaye remarks that for $\varepsilon \in \C$ with $|\varepsilon| < 1$ \cite{POD08}, 
\begin{align*}
\frac{\chi'_{g}(\varepsilon)}{\chi_{g}(\varepsilon)} ={} & \sum_{\rho \in \calR(g)} \frac{-\overline{\rho}}{1 - \varepsilon \overline{\rho}} = -\sum_{m = 1}^\infty \varepsilon^{m - 1} \overline{\power_m(\calR(g))}
.
\end{align*}
Setting $e = l(\calE)$, $f = l(\calF)$ and $\calE = (\varepsilon_1, \dots, \varepsilon_e)$, $\calF = (\varphi_1, \dots, \varphi_f)$, we may thus reformulate the integral on the left-hand side in \eqref{4_recipe_eq} as
\begin{align}
\begin{split}
\LHS ={} & \elementary (-\calB)^N \sum_{\lambda} LS_{\lambda'}\left(- \left(\calA \cup \calB^{-1}\right); \calD\right) \sum_{\kappa} \schur_\kappa (\calC) \\
& \times (-1)^{e + f} \sum_{m_1, \dots, m_e \geq 1} \left( \prod_{i = 1}^e \varepsilon_i^{m_i - 1} \right) \sum_{n_1, \dots, n_f \geq 1} \left( \prod_{j = 1}^f \varphi_j^{n_j - 1} \right) \\
& \times \int_{U(N)} \elementary(\calR(g))^{l(\calB)} \overline{\schur_\lambda(\calR(g))} \schur_\kappa(\calR(g)) \prod_{i = 1}^e \overline{\power_{m_i}(\calR(g))} \prod_{j = 1}^f \power_{n_j}(\calR(g)) dg
.
\notag \end{split}
\intertext{In order to write the integrand as a linear combination of products of Schur functions, we repeatedly apply the Murnaghan-Nakayama rule (\textit{i.e.}\ Theorem~\ref{4_thm_MN_rule}):}
\begin{split}
\LHS ={} & \elementary (-\calB)^N \sum_{\lambda} LS_{\lambda'}\left(- \left(\calA \cup \calB^{-1}\right); \calD\right) \sum_{\kappa} \schur_\kappa (\calC) \\
& \times (-1)^{e + f} \sum_{m_1, \dots, m_e \geq 1} \left( \prod_{i = 1}^e \varepsilon_i^{m_i - 1} \right) \sum_{n_1, \dots, n_f \geq 1} \left( \prod_{j = 1}^f \varphi_j^{n_j - 1} \right) \\
& \times \sum_{\substack{\lambda^{(1)}, \dots, \lambda^{(e)}: \\ \lambda \overset{m_1}{\to} \lambda^{(1)} \overset{m_2}{\to} \dots \overset{m_e}{\to} \lambda^{(e)}}} (-1)^{\height \left( \lambda^{(e)} \setminus \lambda \right)}  \sum_{\substack{\kappa^{(1)}, \dots, \kappa^{(f)}: \\ \kappa \overset{n_1}{\to} \kappa^{(1)} \overset{n_2}{\to} \dots \overset{n_f}{\to} \kappa^{(f)}}} (-1)^{\height \left( \kappa^{(f)} \setminus \kappa \right)} \\
& \times \int_{U(N)} \elementary(\calR(g))^{l(\calB)} \overline{\schur_{\lambda^{(e)}}(\calR(g))} \schur_{\kappa^{(f)}}(\calR(g)) dg
.
\notag \end{split}
\intertext{It is a straightforward linear algebra exercise to show that for sequences $\calX$ of length $N$, $\elementary(\calX)^M \schur_\kappa(\calX) = \schur_{\kappa + \left\langle M^N \right\rangle} (\calX)$. Hence, Schur orthogonality (\textit{i.e.}\ Lemma~\ref{4_lem_Schur_ortho}) allows us to compute the integral. In practice, we just introduce the dummy variable $\pi$ to ensure that $\kappa^{(f)} + \left\langle l(\calB)^N \right\rangle = \lambda^{(e)}$, and that the length of the partition does not exceed $N$:}
\begin{split} \label{4_in_proof_recipe_lhs_before_main+error}
\LHS ={} & \elementary (-\calB)^N \sum_{\lambda} LS_{\lambda'}\left(- \left(\calA \cup \calB^{-1}\right); \calD\right) \sum_{\kappa} \schur_\kappa (\calC) \\
& \times (-1)^{e + f} \sum_{m_1, \dots, m_e \geq 1} \left( \prod_{i = 1}^e \varepsilon_i^{m_i - 1} \right) \sum_{n_1, \dots, n_f \geq 1} \left( \prod_{j = 1}^f \varphi_j^{n_j - 1} \right) \sum_{\substack{\pi: \\ l(\pi) \leq N}} \\
& \times \left( \sum_{\substack{\lambda^{(1)}, \dots, \lambda^{(e)}: \\ \lambda \overset{m_1}{\to} \lambda^{(1)} \overset{m_2}{\to} \dots \overset{m_e}{\to} \lambda^{(e)} \\ \lambda^{(e)} = \pi + \left\langle l(\calB)^N \right\rangle}} (-1)^{\height \left( \lambda^{(e)} \setminus \lambda \right)} \right) \left(  \sum_{\substack{\kappa^{(1)}, \dots, \kappa^{(f)}: \\ \kappa \overset{n_1}{\to} \kappa^{(1)} \overset{n_2}{\to} \dots \overset{n_f}{\to} \kappa^{(f)} \\ \kappa^{(f)} = \pi}} (-1)^{\height \left( \kappa^{(f)} \setminus \kappa \right)} \right)
.
\end{split}
\end{align}
The remainder of the proof is dedicated to simplifying the expression above, which seems to come at the cost of introducing an error term. We will replace $\lambda^{(i)}$ by the following sum of partitions: $\lambda^{(i)} = \nu^{(i)} + \mu^{(i)}$ where $\nu^{(i)}$ is the intersection of $\left\langle l(\calB)^N \right\rangle$ and $\lambda^{(i)}$. Notice that every $m_i$-ribbon $\lambda^{(i)} \setminus \lambda^{(i - 1)}$ that appears in the expression above can be cut into two ribbons: a $q_i$-ribbon $\nu^{(i)} \setminus \nu^{(i - 1)}$ that is a subset of the rectangle $\left\langle l(\calB)^N \right\rangle$, and a $p_i$-ribbon $\mu^{(i)} \setminus \mu^{(i - 1)}$ whose boxes lie strictly to the right of the vertical line given by $x = l(\calB)$. 

For the main term, we restrict ourselves to ribbon sizes that satisfy
\begin{align} \label{4_in_proof_recipe_error_condition}
q_1 + \dots + q_e + n_1 + \dots + n_f \leq N - l(\calC)
.
\end{align}
This restriction leads to a number of simplifications: Given that only partitions $\kappa$ of length less than $l(\calC)$ contribute to the sum (since otherwise $\schur_\kappa(\calC)$ vanishes), the fact that $n_1 + \dots + n_f + l(\calC) \leq N$ \emph{entails} that $l(\pi) \leq N$. Moreover, the restriction implies that for every $m_i$-ribbon that appears in the main term, $p_i = 0$ or $q_i = 0$. This last simplification is probably best explained by means of a sketch. The following drawing depicts possible Ferrers diagrams of the partition $\left\langle l(\calB)^N \right\rangle + \pi$ (white) and its subset $\lambda$ (hatched). 
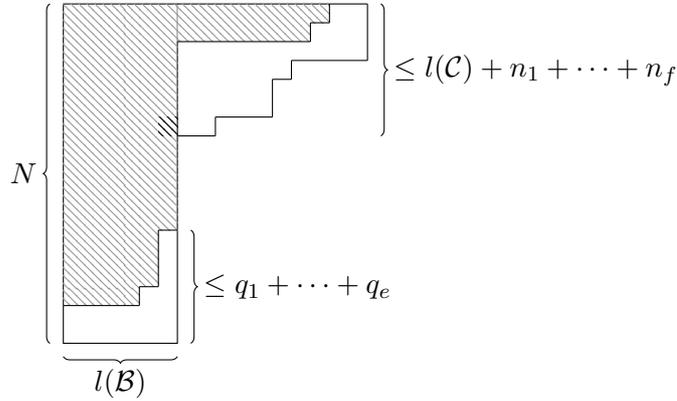
\begin{center}
\begin{tikzpicture}
\draw (0, 0) rectangle (1.5, 4.5); 
\draw[white] (1.5, 4.5) -- (1.5, 4); 
\draw (1.5, 4.5) -- (4, 4.5); 
\draw (4, 4.5) -- (4, 3.75);
\draw (4, 3.75) -- (3, 3.75);
\draw (3, 3.75) -- (3, 3.5);
\draw (3, 3.5) -- (2.75, 3.5);
\draw (2.75, 3.5) -- (2.75, 3);
\draw (2.75, 3) -- (2, 3);
\draw (2, 3) -- (2, 2.75);
\draw (2, 2.75) -- (1.5, 2.75);
\fill[pattern=north west lines, pattern color=black!40!white] (0, 0.5) rectangle (1, 4.5);
\fill[pattern=north west lines, pattern color=black!40!white] (1, 0.75) rectangle (1.25, 4.5);
\fill[pattern=north west lines, pattern color=black!40!white] (1.25, 1.5) rectangle (1.5, 4.5);
\fill[pattern=north west lines, pattern color=black!40!white] (1.5, 4) rectangle (3.25, 4.5);
\fill[pattern=north west lines, pattern color=black!40!white] (3.25, 4.25) rectangle (3.5, 4.5);
\fill[pattern=north west lines] (1.25, 2.75) rectangle (1.5, 3);
\draw (0, 0.5) -- (1, 0.5);
\draw (1, 0.5) -- (1, 0.75);
\draw (1, 0.75) -- (1.25, 0.75);
\draw (1.25, 0.75) -- (1.25, 1.5);
\draw (1.25, 1.5) -- (1.5, 1.5);
\draw (1.5, 4) -- (3.25, 4);
\draw (3.25, 4) -- (3.25, 4.25);
\draw (3.25, 4.25) -- (3.5, 4.25);
\draw (3.5, 4.25) -- (3.5, 4.5);
\draw[decoration={brace, raise=5pt},decorate] (0, 0) -- node[left=6pt] {$N$} (0, 4.5);
\draw[decoration={brace, raise=5pt, mirror},decorate] (0, 0) -- node[below=6pt] {$l(\calB)$} (1.5, 0);
\draw[decoration={brace, raise=5pt},decorate] (4, 4.5) -- node[right=6pt] {$\leq l(\calC) + n_1 + \dots + n_f$} (4, 2.75);
\draw[decoration={brace, raise=5pt, mirror},decorate] (1.5, 0) -- node[right=6pt] {$\leq q_1 + \dots + q_e$} (1.5, 1.5);
\end{tikzpicture}
\end{center}
By definition of the $q_i$, $\left| \lambda \cap \left\langle l(\calB)^N \right\rangle \right| = N l(\calB) - q_1 - \dots - q_e$, which implies that $N - \lambda'_{l(\calB)} \leq q_1 + \dots + q_e$, as indicated on the sketch. In addition, we have already seen that $l(\pi) \leq l(\calC) + n_1 + \dots + n_f$. Therefore, the condition given in \eqref{4_in_proof_recipe_error_condition} implies that the box with coordinates $(l(\calB), l(\pi))$ must be contained in $\lambda$. The box in question is marked by a slightly darker pattern. We thus conclude that every $m_i$-ribbon $\lambda^{(i)} \setminus \lambda^{(i - 1)}$ lies either to the left or strictly to the right of this box, which is the graphical way of saying that either $m_i = q_i$ or $m_i = p_i$. 

In sum, the main term is equal to
\begin{align*}
\main ={} & \elementary (-\calB)^N \sum_{\substack{\mu, \nu: \\ \nu' \cup \mu' \text{ is a partition}}} LS_{\nu' \cup \mu'}\left(- \left(\calA \cup \calB^{-1}\right); \calD\right) \sum_{\kappa} \schur_\kappa (\calC) \\
& \times (-1)^{e + f} \sum_{\substack{g,h \geq 0: \\ g + h = e}} \sum_{\substack{G,H \subset [e]: \\ G \cup_{g,h} H = [e]}} \sum_{p_1, \dots, p_g \geq 1} \left( \prod_{i = 1}^g \varepsilon_{G_i}^{p_i - 1} \right) \\
& \times \sum_{\substack{q,n \geq 0: \\ q + n \leq N - l(\calC)}} \sum_{\substack{q_1, \dots, q_h \geq 1: \\ q_1 + \dots + q_h = q}} \left( \prod_{i = 1}^h \varepsilon_{H_i}^{q_i - 1} \right) \sum_{\substack{n_1, \dots, n_f \geq 1: \\ n_1 + \dots + n_f = n}} \left( \prod_{j = 1}^f \varphi_j^{n_j - 1} \right) \sum_\pi \\
& \times \left( \sum_{\substack{\nu^{(1)}, \dots, \nu^{(h)}: \\ \nu \overset{q_1}{\to} \nu^{(1)} \overset{q_2}{\to} \dots \overset{q_h}{\to} \nu^{(h)} \\ \nu^{(h)} = \left\langle l(\calB)^N \right\rangle}} (-1)^{\height \left( \nu^{(h)} \setminus \nu \right)} \right)
\left( \sum_{\substack{\mu^{(1)}, \dots, \mu^{(g)}: \\ \mu \overset{p_1}{\to} \mu^{(1)} \overset{p_2}{\to} \dots \overset{p_g}{\to} \mu^{(g)} \\ \mu^{(g)} = \pi}} (-1)^{\height \left( \mu^{(g)} \setminus \mu \right)} \right) \\
& \times \left( \sum_{\substack{\kappa^{(1)}, \dots, \kappa^{(f)}: \\ \kappa \overset{n_1}{\to} \kappa^{(1)} \overset{n_2}{\to} \dots \overset{n_f}{\to} \kappa^{(f)} \\ \kappa^{(f)} = \pi}} (-1)^{\height \left( \kappa^{(f)} \setminus \kappa \right)} \right)
.
\end{align*}
First notice that under the assumption that the condition given in \eqref{4_in_proof_recipe_error_condition} is satisfied, the restriction to pairs of partitions $\mu$, $\nu$ so that $\nu' \cup \mu'$ is a partition is actually superfluous. Indeed,
$$\nu'_{l(\calB)} \geq N - q_1 - \dots - q_h \geq l(\calC) + n_1 + \dots + n_f \geq l(\pi) \geq l(\mu) = \mu'_1.$$
We use Lemma~\ref{4_lem_first_overlap_id} to write $LS_{\nu' \cup \mu'}\left(- \left( \calA \cup \calB^{-1}\right); \calD \right)$ as a sum of products of Littlewood-Schur functions that depend on $\nu$ or $\mu$ but not on both. This is permissible given that the elements of $\calA \cup \calB^{-1}$ are pairwise distinct and that $l(\calD) \leq l(\calA)$, which implies that the $(l(\calD), l(\calA) + l(\calB))$-index of any partition is less than $l(\calA)$. More concretely, we obtain
$$LS_{\nu' \cup \mu'}\left(- \left(\calA \cup \calB^{-1}\right); \calD\right) = \sum_{\substack{\calS, \calT \subset \calA \cup \calB^{-1}: \\ \calS \cup_{l(\calB), l(\calA)} \calT \sorteq \calA \cup \calB^{-1}}}  \frac{LS_{\nu' + \left\langle l(\calA)^{l(\calB)} \right\rangle}(- \calS; \calD) LS_{\mu'}(- \calT; \calD)}{\Delta(\calT; \calS)}.$$
Again due to the fact that $l(\calD) \leq l(\calA)$, Corollary~\ref{4_cor_index=0} states that
\begin{multline*}
LS_{\nu' + \left\langle l(\calA)^{l(\calB)} \right\rangle}(- \calS; \calD) \\ = \Delta(\calD; \calS)\schur_{\nu' + \left\langle (l(\calA) - l(\calD))^{l(\calB)} \right\rangle} (-\calS) = \Delta(\calD; \calS) \elementary (-\calS)^{l(\calA) - l(\calD)} \schur_{\nu'}(-\calS).
\end{multline*}
Hence, rearranging the various sums in the main term yields
\begin{align} \label{4_in_proof_recipe_main_after_overlap_id}
\begin{split}
\main ={} & \elementary (-\calB)^N \sum_{\substack{\calS, \calT \subset \calA \cup \calB^{-1}: \\ \calS \cup_{l(\calB), l(\calA)} \calT \sorteq \calA \cup \calB^{-1}}} \elementary (-\calS)^{l(\calA) - l(\calD)} \frac{\Delta(\calD; \calS)}{\Delta(\calT; \calS)} \\
& \times (-1)^{e + f} \sum_{\substack{g,h \geq 0: \\ g + h = e}} \sum_{\substack{G,H \subset [e]: \\ G \cup_{g,h} H = [e]}} \sum_{p_1, \dots, p_g \geq 1} \left( \prod_{i = 1}^g \varepsilon_{G_i}^{p_i - 1} \right) \\
& \times \sum_{\substack{q,n \geq 0: \\ q + n \leq N - l(\calC)}} \sum_{\substack{q_1, \dots, q_h \geq 1: \\ q_1 + \dots + q_h = q}} \left( \prod_{i = 1}^h \varepsilon_{H_i}^{q_i - 1} \right) \sum_{\substack{n_1, \dots, n_f \geq 1: \\ n_1 + \dots + n_f = n}} \left( \prod_{j = 1}^f \varphi_j^{n_j - 1} \right) \\
& \times \sum_\nu \schur_{\nu'}(-\calS) \sum_{\substack{\nu^{(1)}, \dots, \nu^{(h)}: \\ \nu \overset{q_1}{\to} \nu^{(1)} \overset{q_2}{\to} \dots \overset{q_h}{\to} \nu^{(h)} \\ \nu^{(h)} = \left\langle l(\calB)^N \right\rangle}} (-1)^{\height \left( \nu^{(h)} \setminus \nu \right)} \\
& \times \sum_\mu LS_{\mu'}(- \calT; \calD) \sum_\pi \sum_{\substack{\mu^{(1)}, \dots, \mu^{(g)}: \\ \mu \overset{p_1}{\to} \mu^{(1)} \overset{p_2}{\to} \dots \overset{p_g}{\to} \mu^{(g)} \\ \mu^{(g)} = \pi}} (-1)^{\height \left( \mu^{(g)} \setminus \mu \right)} \\
& \times \sum_\kappa \schur_\kappa(\calC) \sum_{\substack{\kappa^{(1)}, \dots, \kappa^{(f)}: \\ \kappa \overset{n_1}{\to} \kappa^{(1)} \overset{n_2}{\to} \dots \overset{n_f}{\to} \kappa^{(f)} \\ \kappa^{(f)} = \pi}} (-1)^{\height \left( \kappa^{(f)} \setminus \kappa \right)}
.
\end{split}
\end{align}

Let us now focus on the three sums over ribbons:
\begin{align*}
\ribbon(q_1, \dots, q_h) \defeq{} & \sum_\nu \schur_{\nu'}(-\calS) \sum_{\substack{\nu^{(1)}, \dots, \nu^{(h)}: \\ \nu \overset{q_1}{\to} \nu^{(1)} \overset{q_2}{\to} \dots \overset{q_h}{\to} \nu^{(h)} \\ \nu^{(h)} = \left\langle l(\calB)^N \right\rangle}} (-1)^{\height \left( \nu^{(h)} \setminus \nu \right)} 
\displaybreak[2] \\
={} & \sum_{\substack{\nu^{(0)}, \nu^{(1)}, \dots, \nu^{(h - 1)}: \\ \nu^{(0)} \overset{q_1}{\to} \dots \overset{q_{h - 1}}{\to} \nu^{(h - 1)} \overset{q_h}{\to} \left\langle l(\calB)^N \right\rangle}} (-1)^{\height \left(\left\langle l(\calB)^N \right\rangle \setminus \nu^{(0)} \right)} \schur_{{\nu^{(0)}}'}(-\calS)
.
\intertext{The equality in \eqref{4_eq_height_conjugate_ribbon} allows us to get rid of the conjugation in the index of the Schur function. Since $\left\langle l(\calB)^N \right\rangle' = \left\langle N^{l(\calB)} \right\rangle$,}
\ribbon(q_1, \dots, q_h) ={} & (-1)^{q - h} \sum_{\substack{\nu^{(0)}, \nu^{(1)}, \dots, \nu^{(h - 1)}: \\ \nu^{(0)} \overset{q_1}{\to} \dots \overset{q_{h - 1}}{\to} \nu^{(h - 1)} \overset{q_h}{\to} \left\langle N^{l(\calB)} \right\rangle}} (-1)^{\height \left(\left\langle N^{l(\calB)} \right\rangle \setminus \nu^{(0)} \right)} \schur_{\nu^{(0)}}(\rho^\alpha_{-\calS})
.
\intertext{Repeatedly applying Corollary \ref{4_cor_MN_negative_lambda} results in}
\ribbon(q_1, \dots, q_h) ={} & (-1)^{q - h} \left[ q_1 \frac{\partial}{\partial \power_{q_1}} \cdots q_h \frac{\partial}{\partial \power_{q_h}} \schur_{\left\langle N^{l(\calB)} \right\rangle} \right] \left( \rho^\alpha_{-\calS} \right)
.
\intertext{The theory of operators makes it apparent that $\ribbon(q_1, \dots, q_h)$ is independent of the order of the $q_i$. Without loss of generality, we may thus assume that $(q_1, \dots, q_h)$ is a partition of length $h$, say $\chi$. In this notation the preceding equality reads}
\ribbon(\chi) ={} & (-1)^{|\chi| - h} \left[ \prod_{i \geq 1} i^{m_i(\chi)} \frac{\partial}{\partial \power_\chi} \schur_{\left\langle N^{l(\calB)} \right\rangle} \right] \left( \rho^\alpha_{-\calS} \right).
\intertext{Given that $|\chi| = q \leq N$ and $l(\calS) = l(\calB)$, Corollary~\ref{4_cor_MN_negative_lambda} states that}
\ribbon(\chi) ={} & (-1)^{|\chi| - h} \schur_{\left\langle N^{l(\calB)} \right\rangle}(-\calS) \power_{-\chi} (-\calS) = (-1)^{|\chi| - h} \elementary(-\calS)^N \power_{-\chi} (-\calS)
.
\end{align*}
The remaining two sums over ribbons that appear in \eqref{4_in_proof_recipe_main_after_overlap_id} can in fact be viewed as one sum:
\begin{align*}
\ribbon(p_1, \dots, p_g; n_1, \dots, n_f) \defeq{} & \sum_\mu LS_{\mu'}(- \calT; \calD) \sum_\pi \sum_{\substack{\mu^{(1)}, \dots, \mu^{(g)}: \\ \mu \overset{p_1}{\to} \mu^{(1)} \overset{p_2}{\to} \dots \overset{p_g}{\to} \mu^{(g)} \\ \mu^{(g)} = \pi}} (-1)^{\height \left( \mu^{(g)} \setminus \mu \right)} \\
& \times \sum_\kappa \schur_\kappa(\calC) \sum_{\substack{\kappa^{(1)}, \dots, \kappa^{(f)}: \\ \kappa \overset{n_1}{\to} \kappa^{(1)} \overset{n_2}{\to} \dots \overset{n_f}{\to} \kappa^{(f)} \\ \kappa^{(f)} = \pi}} (-1)^{\height \left( \kappa^{(f)} \setminus \kappa \right)}
\\
={} & \sum_\mu LS_{\mu'}(- \calT; \calD) \\
& \times \sum_{\substack{\mu^{(1)}, \dots, \mu^{(g)}, \kappa^{(1)}, \dots, \kappa^{(f)}, \kappa: \\ \mu \overset{p_1}{\to} \mu^{(1)} \overset{p_2}{\to} \dots \overset{p_g}{\to} \mu^{(g)} = \kappa^{(f)} \overset{n_f}{\leftarrow} \dots \overset{n_2}{\leftarrow} \kappa^{(1)} \overset{n_1}{\leftarrow} \kappa \\ l\left(\mu^{(g)} \right) \leq N}} \\
& \times (-1)^{\height \left( \mu^{(g)} \setminus \mu \right)} (-1)^{\height \left( \kappa^{(f)} \setminus \kappa \right)} \schur_\kappa(\rho^\alpha_\calC)
.
\intertext{Repeatedly applying the Murnaghan-Nakayama rule and its dual (\textit{i.e.}\ Theorem~\ref{4_thm_MN_rule} and Corollary~\ref{4_cor_dual_MN_rule}) gives}
\ribbon(p_1, \dots, p_g; n_1, \dots, n_f) ={} & \sum_{\mu} \schur_\mu \left( \rho^\beta_{-\calT} \cup \rho^\alpha_\calD \right) \\
& \times \left[ n_1 \frac{\partial}{\partial \power_{n_1}} \cdots n_f \frac{\partial}{\partial \power_{n_f}} \power_{p_g} \cdots \power_{p_1} \schur_\mu \right] \left(\rho^\alpha_\calC\right)
\intertext{where we view the Littlewood-Schur function as a specialization of a Schur function, following Remark~\ref{4_rem_LS_as_specialization}. As above the theory of operators makes it obvious that $\ribbon(p_1, \dots, p_g; n_1, \dots, n_f)$ is symmetric in both $(p_1, \dots, p_g)$ and $(n_1, \dots, n_f)$, which we thus replace by the partitions $\psi$ and $\omega$ of lengths $g$ and $f$, respectively. This substitution yields}
\ribbon(\psi; \omega) ={} & \left[ \prod_{i \geq 1} i^{m_i(\omega)} \frac{\partial}{\partial \power_\omega} \power_\psi \sum_{\mu} \schur_\mu \left( \rho^\beta_{-\calT} \cup \rho^\alpha_\calD \right) \schur_\mu \right] \left(\rho^\alpha_\calC\right)
.
\intertext{Due to the power sum version of the Cauchy identity given in \eqref{4_eq_power_sum_version_cauchy_id}, this is equal to an expression that only involves power sums:}
\ribbon(\psi; \omega) ={} & \prod_{i \geq 1} i^{m_i(\omega)} \sum_\lambda z_\lambda^{-1} \power_\lambda\left( \rho^\beta_{-\calT} \cup \rho^\alpha_\calD \right) \left[ \frac{\partial}{\partial \power_\omega} \power_\psi \power_\lambda \right] \left(\rho^\alpha_\calC\right)
.
\intertext{According to Lemma~\ref{4_lem_reduction_operators}, this is equal to}
\ribbon(\psi; \omega) ={} & \prod_{i \geq 1} i^{m_i(\omega)} \\
& \times \sum_{\substack{\lambda, \xi: \\ \omega \cup \xi \sorteq \psi \cup \lambda}} z_\lambda^{-1} \power_\lambda\left( \rho^\beta_{-\calT} \cup \rho^\alpha_\calD \right) \prod_{i \geq 1} \frac{m_i(\psi \cup \lambda)!}{m_i(\xi)!} \power_\xi(\calC).
\end{align*}

Incorporating these simplifications into the expression for the main term given in \eqref{4_in_proof_recipe_main_after_overlap_id} on page \pageref{4_in_proof_recipe_main_after_overlap_id} results in
\begin{align*}
\main ={} & \elementary (-\calB)^N \sum_{\substack{\calS, \calT \subset \calA \cup \calB^{-1}: \\ \calS \cup_{l(\calB), l(\calA)} \calT \sorteq \calA \cup \calB^{-1}}} \elementary (-\calS)^{N + l(\calA) - l(\calD)} \frac{\Delta(\calD; \calS)}{\Delta(\calT; \calS)} \\
& \times (-1)^{e + f} \sum_{\substack{g,h \geq 0: \\ g + h = e}} \sum_{\substack{G,H \subset [e]: \\ G \cup_{g,h} H = [e]}} \sum_{\substack{q,n \geq 0: \\ q + n \leq N - l(\calC)}} \left( \sum_{\substack{\chi: \\ l(\chi) = h \\ |\chi| = q}} \monomial_{\chi - \left\langle 1^h \right\rangle}\left(-\calE_H\right) \power_{-\chi}(-\calS) \right) \\
& \times \sum_{\substack{\psi: \\ l(\psi) = g}} \monomial_{\psi - \left\langle 1^g \right\rangle} \left(\calE_G\right) \sum_{\substack{\omega: \\ l(\omega) = f \\ |\omega| = n}} \left( \prod_{i \geq 1} i^{m_i(\omega)} \right) \monomial_{\omega - \left\langle 1^f \right\rangle} (\calF) \\
& \times \sum_{\substack{\lambda, \xi: \\ \omega \cup \xi \sorteq \psi \cup \lambda}} z_\lambda^{-1} \power_\lambda \left( \rho^\beta_{-\calT} \cup \rho^\alpha_{\calD} \right) \prod_{i \geq 1} \frac{m_i(\psi \cup \lambda)!}{m_i(\xi)!} \power_\xi (\calC)
.
\end{align*}
This is the main term stated in the Recipe up to elementary algebraic manipulations. Going back to our expression in \eqref{4_in_proof_recipe_lhs_before_main+error} for the integral on the left-hand side, we obtain the error term by considering all ribbons that do not satisfy the condition given in \eqref{4_in_proof_recipe_error_condition}. Taking absolute values inside the sums results in
\begin{align}
\begin{split}
|\error| \leq & \left| \elementary (\calB)^N \right| \sum_{\lambda} \left| LS_{\lambda'}\left(- \left(\calA \cup \calB^{-1}\right); \calD\right) \right| \sum_{\kappa} \left| \schur_\kappa (\calC) \right| \notag \end{split} \displaybreak[2] \\
\begin{split}
& \times \sum_{\substack{q,n \geq 0: \\ q + n > N - l(\calC)}} \sum_{\substack{m_1, \dots, m_e \geq 1: \\ q_1 + \dots + q_e = q}} \left( \prod_{i = 1}^e \left| \varepsilon_i^{m_i - 1} \right| \right) \sum_{\substack{n_1, \dots, n_f \geq 1: \\ n_1 + \dots + n_f = n}} \left( \prod_{j = 1}^f \left| \varphi_j^{n_j - 1} \right| \right) \\
& \times \sum_{\substack{\pi: \\ l(\pi) \leq N}} \left( \sum_{\substack{\lambda^{(1)}, \dots, \lambda^{(e)}: \\ \lambda \overset{m_1}{\to} \lambda^{(1)} \overset{m_2}{\to} \dots \overset{m_e}{\to} \lambda^{(e)} \\ \lambda^{(e)} = \pi + \left\langle l(\calB)^N \right\rangle}} 1 \right)
\left( \sum_{\substack{\kappa^{(1)}, \dots, \kappa^{(f)}: \\ \kappa \overset{n_1}{\to} \kappa^{(1)} \overset{n_2}{\to} \dots \overset{n_f}{\to} \kappa^{(f)} \\ \kappa^{(f)} = \pi}} 1 \right)
.
\notag \end{split}
\intertext{where $q_i$ stands for the number of boxes of the $m_i$-ribbon that are contained in the rectangle $\left\langle l(\calB)^N \right\rangle$. As before, we view each $m_i$-ribbon as a pair of ribbons, namely a $q_i$- and a $p_i$-ribbon that are contained in $\left\langle l(\calB)^N \right\rangle$ and $\pi$, respectively. At the cost of counting too many ribbons, we forget that each pair of ribbons can be combined to form one ribbon:}
\begin{split} \label{4_in_proof_recipe_error_before_O}
|\error| \leq & \left| \elementary (\calB)^N \right| \sum_{\substack{\mu, \nu: \\ \nu' \cup \mu' \text{ is a partition}}} \left| LS_{\nu' \cup \mu'}\left(- \left(\calA \cup \calB^{-1}\right); \calD \right) \right| \sum_{\kappa} \left| \schur_\kappa (\calC) \right| \\
& \times \sum_{\substack{q,n \geq 0: \\ q + n > N - l(\calC)}} \sum_{\substack{q_1, \dots, q_e \geq 0: \\ q_1 + \dots + q_e = q}} \sum_{p_1, \dots, p_e \geq 0} \left( \prod_{i = 1}^e \left| \varepsilon_i^{q_i + p_i - 1} \right| \right) \sum_{\substack{n_1, \dots, n_f \geq 1: \\ n_1 + \dots + n_f = n}} \left( \prod_{j = 1}^f \left| \varphi_j^{n_j - 1} \right| \right) \\
& \times \sum_{\substack{\pi: \\ l(\pi) \leq N}} \left( \sum_{\substack{\nu^{(1)}, \dots, \nu^{(e)}: \\ \nu \overset{q_1}{\to} \nu^{(1)} \overset{q_2}{\to} \dots \overset{q_e}{\to} \nu^{(e)} \\ \nu^{(e)} = \left\langle l(\calB)^N \right\rangle}} 1 \right)
\left( \sum_{\substack{\mu^{(1)}, \dots, \mu^{(e)}: \\ \mu \overset{p_1}{\to} \mu^{(1)} \overset{p_2}{\to} \dots \overset{p_e}{\to} \mu^{(e)} \\ \mu^{(e)} = \pi}} 1 \right)
\left( \sum_{\substack{\kappa^{(1)}, \dots, \kappa^{(f)}: \\ \kappa \overset{n_1}{\to} \kappa^{(1)} \overset{n_2}{\to} \dots \overset{n_f}{\to} \kappa^{(f)} \\ \kappa^{(f)} = \pi}} 1 \right)
.
\end{split}
\intertext{The next step mirrors our derivation of the main term: We separate the partitions $\mu$ and $\nu$ in $LS_{\nu' \cup \mu'}\left(- \left(\calA \cup \calB^{-1}\right); \calD\right)$ by an application of Lemma~\ref{4_lem_first_overlap_id}, and then forget the condition that the union $\nu' \cup \mu'$ must still be a partition. In addition, we eliminate the dummy variable $\pi$ to combine the sequences of sums over the $p_i$- and $n_j$-ribbons:}
\begin{split} 
|\error| \leq & \left| \elementary (\calB)^N \right| \sum_{\substack{\calS, \calT \subset \calA \cup \calB^{-1}: \\ \calS \cup_{l(\calB), l(\calA)} \calT \sorteq \calA \cup \calB^{-1}}} \left| \elementary(\calS)^{l(\calA) - l(\calD)} \right| \frac{\left|\Delta(\calD; \calS)\right|}{\left|\Delta(\calT; \calS) \right|} \\
& \times \sum_{\kappa, \mu, \nu} \left| \schur_{\nu'} (\calS) \right| \left| LS_{\mu'}(- \calT; \calD) \right| \left| \schur_\kappa (\calC) \right|
\notag \end{split} \displaybreak[2] \\
\begin{split}
& \times \sum_{\substack{q,n \geq 0: \\ q + n > N - l(\calC)}} \sum_{\substack{q_1, \dots, q_e \geq 0: \\ q_1 + \dots + q_e = q}} \sum_{p_1, \dots, p_e \geq 0} \left( \prod_{i = 1}^e \left| \varepsilon_i^{q_i + p_i - 1} \right| \right) \sum_{\substack{n_1, \dots, n_f \geq 1: \\ n_1 + \dots + n_f = n}} \left( \prod_{j = 1}^f \left| \varphi_j^{n_j - 1} \right| \right) 
\notag \end{split} \displaybreak[2] \\
\begin{split}
& \times \left( \sum_{\substack{\nu^{(1)}, \dots, \nu^{(e)}: \\ \nu \overset{q_1}{\to} \nu^{(1)} \overset{q_2}{\to} \dots \overset{q_e}{\to} \nu^{(e)} \\ \nu^{(e)} = \left\langle l(\calB)^N \right\rangle}} 1 \right) \left( \sum_{\substack{\mu^{(1)}, \dots, \mu^{(e)}, \kappa^{(1)}, \dots, \kappa^{(f)}, \kappa: \\ \mu \overset{p_1}{\to} \mu^{(1)} \overset{p_2}{\to} \dots \overset{p_e}{\to} \mu^{(e)} = \kappa^{(f)} \overset{n_f}{\leftarrow} \dots \overset{n_2}{\leftarrow} \kappa^{(1)} \overset{n_1}{\leftarrow} \kappa \\ l \left( \mu^{(e)} \right) \leq N}} 1 \right)
.
\notag \end{split}
\end{align}
Lemma~\ref{4_lem_bound_number_of_ribbons_rectangle} provides an upper bound for the two ribbon-counting sequences of sums. Indeed, all partitions $\nu^{(\cdot)}$ are contained in a rectangle of width $l(\calB)$, while all $\mu^{(\cdot)}$ and $\kappa^{(\cdot)}$ are contained in a rectangle of height $N$. We conclude that
\begin{align} 
\begin{split} \label{4_eq_error_bound_for_recipe}
\error ={} & O_{l(\calA), l(\calB), l(\calD), l(\calE), \calA \cup \calB^{-1}} \left(N^{(l(\calE) + l(\calF) - 1)^+} \right) \left| \elementary (\calB)^N \right| \\
& \times \sum_{\substack{q,n \geq 0: \\ q + n > N - l(\calC)}} \left( \sum_{\substack{q_1, \dots, q_{l(\calE)} \geq 0: \\ q_1 + \dots + q_{l(\calE)} = q}} 1 \right) \left( \sum_{\substack{n_1, \dots, n_{l(\calF)} \geq 1: \\ n_1 + \dots + n_{l(\calF)} = n}} \prod_{j = 1}^{l(\calF)} \left| \calF_j^{n_j - 1} \right| \right) \\
& \times \left( \sum_{p_1, \dots, p_{l(\calE)} \geq 0} \prod_{i = 1}^{l(\calE)} \left| \calE_i^{q_i + p_i - 1} \right| \right) \\
& \times \sum_{\substack{\calS, \calT \subset \calA \cup \calB^{-1}: \\ \calS \cup_{l(\calB), l(\calA)} \calT \sorteq \calA \cup \calB^{-1}}} \left( \sum_{\substack{\nu: \\ \nu \subset \left\langle l(\calB)^N \right\rangle \\ |\nu| = l(\calB) N - q}} \left| \schur_{\nu'}(\calS) \right| \right) \left( \sum_{\kappa, \mu} \left| LS_\mu(\calD; -\calT) \right| \left| \schur_\kappa (\calC) \right| \right)
\end{split}
\end{align}
where the last sum is over pairs of partitions $\kappa, \mu$ so that there exists a partition which can be obtained by adding $n$ boxes to $\kappa$ or by adding $p = p_1 + \dots + p_e$ boxes to $\mu$.
\end{proof}

\subsection{The results} \label{4_sec_results}
In this section we present four theorems that can be viewed as special cases of the Recipe. In these instances we are able to give reasonable bounds for the error terms, unlike in the full generality of the Recipe. While the two formulas for averages of products of ratios \emph{and} logarithmic derivatives seem to be new, formulas for pure ratios and pure products of logarithmic derivatives can be found in the literature. In fact, our expression for the ratios is just a reformulation of Bump and Gamburd's ratio theorem \cite{bump06}. The logarithmic derivative theorem presented here gives a neater and more combinatorial expression for the leading term of Conrey and Snaith's expression for averages of logarithmic derivatives \cite{CS}. 

\begin{thm}[ratios] \label{4_thm_ratios} Let $\calA$, $\calB$, $\calC$ and $\calD$ be sets of non-zero variables so that the latter two only contain elements that are strictly less than 1 in absolute value. Let the elements of $\calA \cup \calB^{-1}$ be pairwise distinct. If $l(\calD) \leq N + l(\calA)$ and $l(\calC) \leq N$, then
\begin{align}
\begin{split} \label{4_thm_ratios_eq}
& \hspace{-15pt} \int_{U(N)} \frac{\prod_{\alpha \in \calA} \chi_g(\alpha) \prod_{\beta \in \calB} \chi_{g^{-1}}(\beta)}{\prod_{\delta \in \calD} \chi_g(\delta) \prod_{\gamma \in \calC} \chi_{g^{-1}} (\gamma)} dg \\
={} & \elementary (-\calB)^N \sum_{\substack{\calS, \calT \subset \calA \cup \calB^{-1}: \\ \calS \cup_{l(\calB), l(\calA)} \calT \sorteq \calA \cup \calB^{-1}}} \elementary (-\calS)^{N + l(\calA) - l(\calD)} \frac{\Delta(\calD; \calS)}{\Delta(\calT; \calS)} \prod_{\substack{\gamma \in \calC \\ \delta \in \calD}} (1 - \gamma \delta)^{-1} \prod_{\substack{t \in \calT \\ \gamma \in \calC}} (1 - t \gamma)
.
\end{split}
\end{align}
\end{thm}

\begin{rem*} Theorem~\ref{4_thm_ratios} is basically the ratio theorem presented in \cite{bump06}, except for the assumptions on the lengths of the sets of variables. This does not come as a surprise given that the proof of the Recipe is based on Bump and Gamburd's approach. Their theorem holds under the assumption that $l(\calC) + l(\calD) \leq N$. In fact, they only state $l(\calC)$, $l(\calD) \leq N$ as a requirement but their proof implicitly makes us of the stronger assumption: on page 246 of \cite{bump06} they apply Proposition 8 (a weaker version of Lemma~\ref{4_lem_first_overlap_id}), which is only permissible if $l(\calC) + l(\calD) \leq N$.
\end{rem*}

\begin{proof} In a first step, let us suppose that $l(\calD) \leq l(\calA)$, in which case the equality in \eqref{4_thm_ratios_eq} follows from the Recipe. We set $\calE = \emptyset = \calF$ in \eqref{4_rem_LS_as_specialization}. Under the assumption that $l(\calC) \leq N$, the error term vanishes. The main term simplifies to
\begin{align*}
\elementary (-\calB)^N \sum_{\substack{\calS, \calT \subset \calA \cup \calB^{-1}: \\ \calS \cup_{l(\calB), l(\calA)} \calT \sorteq \calA \cup \calB^{-1}}} \elementary (-\calS)^{N + l(\calA) - l(\calD)} \frac{\Delta(\calD; \calS)}{\Delta(\calT; \calS)} \sum_\lambda z_\lambda^{-1} \power_\lambda \left( \rho^\beta_{-\calT} \cup \rho^\alpha_{\calD} \right) \power_\lambda(\calC)
.
\end{align*}
We write the sum over $\lambda$ as a product: According to Lemma~\ref{4_lem_cauchy_identity_schur} and Remark~\ref{4_rem_LS_as_specialization},
\begin{align*}
\sum_\lambda z_\lambda^{-1} \power_\lambda \left( \rho^\beta_{-\calT} \cup \rho^\alpha_{\calD} \right) \power_\lambda(\calC) ={} & \sum_\mu \schur_\mu \left( \rho^\beta_{-\calT} \cup \rho^\alpha_{\calD} \right) \schur_\mu(\calC) \\
={} & \sum_\mu LS_\mu (\calD; -\calT) \schur_\mu(\calC) \\
={} & \prod_{\substack{\gamma \in \calC \\ \delta \in \calD}} (1 - \gamma \delta)^{-1} \prod_{\substack{t \in \calT \\ \gamma \in \calC}} (1 - t \gamma)
\end{align*}
where the last equality is a consequence of the generalized Cauchy identity (\textit{i.e.}\ Proposition~\ref{4_prop_gen_Cauchy}).

In order to justify the equality in \eqref{4_thm_ratios_eq} in case $l(\calD) \leq N + l(\calA)$, it suffices to note that in the proof of the Recipe the role of the assumption that $l(\calD) \leq l(\calA)$ is to ensure that the $(l(\calD), l(\calA) + l(\calB))$-index of of $\nu' \cup \mu'$ is less than $l(\calA)$, which makes it permissible to apply Lemma~\ref{4_lem_first_overlap_id} to $LS_{\nu' \cup \mu'}\left(- \left(\calA \cup \calB^{-1}\right); \calD\right)$. Given that $\calE = \emptyset$, the partition $\nu'$ is equal to the rectangle $\left\langle N^{l(\calB)} \right\rangle$. Therefore, the $(l(\calD), l(\calA) + l(\calB))$-index of $\nu' \cup \mu'$ is less than $l(\calA)$ whenever $l(\calD) \leq N + l(\calA)$.
\end{proof}

\begin{thm} Let $\calA$, $\calB$, $\calC$, $\calD$ and $\calE$ be sets of non-zero variables so that $l(\calD) \leq l(\calA)$ and the elements of $\calA \cup \calB^{-1}$ are pairwise distinct. If $\abs(\calA)$, $\abs(\calB) \leq 1$, $\abs(\calC)$, $\abs(\calD) < 1$ and there exists $r \in \R$ so that $\abs(\calE) \leq r < 1$, then
\begin{align*}
& \hspace{-15pt} \int_{U(N)} \frac{\prod_{\alpha \in \calA} \chi_g(\alpha) \prod_{\beta \in \calB} \chi_{g^{-1}}(\beta)}{\prod_{\delta \in \calD} \chi_g(\delta) \prod_{\gamma \in \calC} \chi_{g^{-1}} (\gamma)} \prod_{\varepsilon \in \calE} \frac{\chi'_g(\varepsilon)}{\chi_g(\varepsilon)} dg \\
={} & \elementary (-\calB)^N \sum_{\substack{\calS, \calT \subset \calA \cup \calB^{-1}: \\ \calS \cup_{l(\calB), l(\calA)} \calT \sorteq \calA \cup \calB^{-1}}} \elementary (-\calS)^{N + l(\calA) - l(\calD)} \frac{\Delta(\calD; \calS)}{\Delta(\calT; \calS)} \prod_{\substack{\gamma \in \calC \\ \delta \in \calD}} (1 - \gamma \delta)^{-1} \prod_{\substack{t \in \calT \\ \gamma \in \calC}} (1 - t \gamma) \\
& \times (-1)^{l(\calE)} \sum_{\substack{\calE', \calE'' \subset \calE: \\ \calE' \cup \calE'' \sorteq \calE}} \\
& \times \left( \sum_{\substack{\chi: \\ l(\chi) = l\left(\calE''\right) \\ |\chi| \leq N - l(\calC)}} \monomial_{\chi - \left\langle 1^{l\left(\calE''\right)} \right\rangle}\left(-\calE''\right) \power_{-\chi}(-\calS) \right) \left( \sum_{\substack{\psi: \\ l(\psi) = l\left(\calE'\right)}} \monomial_{\psi - \left\langle 1^{l\left(\calE'\right)} \right\rangle} \left(\calE'\right) \power_\psi(\calC) \right)
\\
& + O_{r, \calA, \calB, l(\calC), l(\calD), l(\calE), \max\{\abs(\calC), \abs(\calD)\}} \left( r^N N^{(l(\calB) - 1)^+ + 2(l(\calE) - 1)^+}\right)
.
\end{align*}
If, in addition to the conditions stated above, there exists $r_1 \in \R$ with $\abs(\calB) = r_1 \leq r$, then the bound on the error term can be improved by a factor of $r^N$.
\end{thm}

\begin{proof} We set $\calF = \emptyset$ in the statement of the Recipe. The main term of the expression on the right-hand side simplifies to
\begin{align*}
& \elementary (-\calB)^N \sum_{\substack{\calS, \calT \subset \calA \cup \calB^{-1}: \\ \calS \cup_{l(\calB), l(\calA)} \calT \sorteq \calA \cup \calB^{-1}}} \elementary (-\calS)^{N + l(\calA) - l(\calD)} \frac{\Delta(\calD; \calS)}{\Delta(\calT; \calS)} \\
& \times (-1)^{l(\calE)} \sum_{\substack{\calE', \calE'' \subset \calE: \\ \calE' \cup \calE'' \sorteq \calE}} \left( \sum_{\substack{\chi: \\ l(\chi) = l\left(\calE''\right) \\ |\chi| \leq N - l(\calC)}} \monomial_{\chi - \left\langle 1^{l\left(\calE''\right)} \right\rangle}\left(-\calE''\right) \power_{-\chi}(-\calS) \right) \\
& \times \left( \sum_{\substack{\psi: \\ l(\psi) = l\left(\calE'\right)}} \monomial_{\psi - \left\langle 1^{l\left(\calE'\right)} \right\rangle} \left(\calE'\right) \power_\psi(\calC) \right) 
\left( \sum_\lambda z_\lambda^{-1} \power_\lambda \left( \rho^\beta_{-\calT} \cup \rho^\alpha_{\calD} \right) \power_\lambda (\calC) \right)
.
\end{align*}
As in the proof of Theorem~\ref{4_thm_ratios}, the generalized Cauchy identity allows us to replace the sum over $\lambda$ by a product. This yields the main term of this theorem.

It remains to bound the error given in equation \eqref{4_eq_error_bound_for_recipe} on page \pageref{4_eq_error_bound_for_recipe}. We exploit that $\abs(\calE) \leq r$ and $n = 0$ (since $\calF = \emptyset$) to infer the following bound:
\begin{align*}
\error ={} & O_{r, l(\calA), l(\calB), l(\calD), l(\calE), \calA \cup \calB^{-1}} \left( N^{(l(\calE) - 1)^+} \right) \left| \elementary (\calB)^N \right| \\
& \times \sum_{\substack{q > N - l(\calC)}} r^q \left( \sum_{\substack{q_1, \dots, q_{l(\calE)} \geq 0: \\ q_1 + \dots + q_{l(\calE)} = q}} 1 \right) \sum_{p \geq 0} r^p \left( \sum_{\substack{p_1, \dots, p_{l(\calE)} \geq 0: \\ p_1 + \dots + p_{l(\calE)} = p}} 1 \right) \\
& \times \sum_{\substack{\calS, \calT \subset \calA \cup \calB^{-1}: \\ \calS \cup_{l(\calB), l(\calA)} \calT \sorteq \calA \cup \calB^{-1}}} \left( \sum_{\substack{\nu: \\ \nu \subset \left\langle l(\calB)^N \right\rangle \\ |\nu| = l(\calB) N - q}} \left| \schur_{\nu'}(\calS) \right| \right) \left( \sum_{\substack{\kappa, \mu: \\ \mu \subset \kappa \\ |\mu| + p = |\kappa|}} \left| LS_\mu(\calD; -\calT) \right| \left| \schur_\kappa (\calC) \right| \right)
.
\end{align*}
As $\schur_\kappa(\calC)$ vanishes whenever $l(\kappa) > l(\calC)$, we have that $l(\mu) \leq l(\kappa) \leq l(\calC)$ for all partitions that appear in the sum over $\kappa$ and $\mu$. Setting $R = \max\{\abs(\calC), \abs(\calD)\} < 1$, Lemmas~\ref{4_lem_comb_bound_schur} and \ref{4_lem_comb_bound_LS} thus entail that
\begin{align*}
\sum_{\substack{\kappa, \mu: \\ \mu \subset \kappa \\ |\mu| + p = |\kappa|}} \left| LS_\mu(\calD; -\calT) \right| \left| \schur_\kappa (\calC) \right| 
={} & \sum_{\makebox[34.6pt]{$\substack{\kappa, \mu: \\ l(\mu), l(\kappa) \leq l(\calC) \\ |\mu| + p = |\kappa|}$}} O_{l(\calC), R, l(\calA), \calA \cup \calB^{-1}} \left( |\mu|^{l(\calD)^2 + l(\calC)^2} R^{|\mu|} |\kappa|^{l(\calC)^2} R^{|\kappa|} \right) \\
={} & O_{l(\calC), l(\calD), R, l(\calA), \calA \cup \calB^{-1}} \left( p^{l(\calC)^2 + (l(\calC) - 1)^+} R^p\right) 
\end{align*}
where we have crudely bounded the number of partitions of length $n$ and size $m$ by $(m + 1)^{(n - 1)^+}$. Bounding the number of non-negative integers $p_1, \dots, p_{l(\calE)}$ whose sum equals $p$ by $(p + 1)^{(l(\calE) -1)^+}$, another argument based on geometric series thus allows us to conclude that the sum over $p$ is $O_{l(\calC), l(\calD), l(\calE), \max\{\abs(\calC), \abs(\calD)\}, l(\calA), \calA \cup \calB^{-1}} (1)$.

Two applications of Lemma~\ref{4_lem_det_bound_Schur} will allow us to bound 
\begin{align*}
S(N) \defeq{} &
\left| \elementary (\calB)^N \right| \sum_{\substack{\calS, \calT \subset \calA \cup \calB^{-1}: \\ \calS \cup_{l(\calB), l(\calA)} \calT \sorteq \calA \cup \calB^{-1}}} \sum_{\substack{\nu: \\ \nu \subset \left\langle l(\calB)^N \right\rangle \\ |\nu| = l(\calB) N - q}} \left| \schur_{\nu'}(\calS) \right|
.
\end{align*}
Suppose that $\abs(\calB) = r_1$ for some $r_1 \leq r$. As $\abs(\calA) \leq 1 < r_1^{-1}$, the first statement in Lemma~\ref{4_lem_det_bound_Schur} implies that
\begin{align*}
S(N) ={} & O_{\calA \cup \calB^{-1}} \left( r_1^{Nl(\calB)} \sum_{\substack{\nu: \\ \nu \subset \left\langle l(\calB)^N \right\rangle \\ |\nu| = l(\calB) N - q}} r_1^{-|\nu|} \right) = O_{\calA \cup \calB^{-1}} \left( r^q N^{(l(\calB) - 1)^+} \right)
.
\end{align*}
This last bound is due to the fact that for $l(\calB) \geq 1$, the number of partitions $\nu$ of some fixed size $Q$ that are contained in the rectangle $\left\langle l(\calB)^N \right\rangle$ is at most $(N + 1)^{l(\calB) - 1}$. Indeed, $\nu'_1$ to $\nu'_{l(\calB) - 1}$ are some integers between 0 and $N$, while $\nu'_{l(\calB)}$ is determined by the condition that $\nu'_1 + \dots + \nu'_{l(\calB)} = Q$.
If we only assume that $\abs(\calA)$, $\abs(\calB) \leq 1$, the second bound in Lemma~\ref{4_lem_det_bound_Schur} allows us to infer that
\begin{align*}
S(N) ={} & O_{\calA \cup \calB^{-1}} \left( \elementary (\calB)^N \sum_{\substack{\nu: \\ \nu \subset \left\langle l(\calB)^N \right\rangle \\ |\nu| = l(\calB) N - q}} \elementary \left( \calB^{-1} \right)^{\nu'_1} \right)
= O_{\calA \cup \calB^{-1}} \left( N^{(l(\calB) - 1)^+}\right)
,
\end{align*}
since $\nu'_1 \leq N$ for all partitions $\nu$ that appear in the sum. In conclusion, 
\begin{align*}
\error = O_\calP \left( N^{(l(\calE) - 1)^+ + (l(\calB) - 1)^+} \right) \sum_{q > N - l(\calC)} r^{q(1 + \delta(\abs(\calB) = r_1))} (q + 1)^{(l(\calE) - 1)^+}
\end{align*}
where $\delta(\abs(\calB) = r_1)$ indicates whether the additional condition on $\calB$ is satisfied. Here, the implicit constant depends on $\calP = \{r, l(\calC), l(\calD), l(\calE), \max\{\abs(\calC), \abs(\calD)\}, \calA, \calB \}$. The bound stated in the theorem follows from yet another argument based on geometric series.
\end{proof}

\begin{thm} \label{4_thm_log_ders_and_ratio} Let $r \in \R$ with $r < 1$. Let $\calB$, $\calC$, $\calE$ and $\calF$ be sets of non-zero variables so that $\abs(\calB) \leq 1$, $\abs(\calC) < 1$ and $\abs(\calE)$, $\abs(\calF) \leq r$. Then
\begin{align} \label{4_thm_log_ders_and_ratio_eq}
\begin{split}
& \hspace{-15pt} \int_{U(N)} \frac{\prod_{\beta \in \calB} \chi_{g^{-1}}(\beta)}{\prod_{\gamma \in \calC} \chi_{g^{-1}} (\gamma)} \prod_{\varepsilon \in \calE} \frac{\chi'_g(\varepsilon)}{\chi_g(\varepsilon)} \prod_{\varphi \in \calF} \frac{\chi'_{g^{-1}}(\varphi)}{\chi_{g^{-1}}(\varphi)}dg \\
={} & (-1)^{l(\calE) + l(\calF)} \sum_{\substack{\calE', \calE'' \subset \calE: \\ \calE' \cup \calE'' \sorteq \calE}} \Bigg[ \left( \sum_{\substack{\chi: \\ l(\chi) = l\left(\calE''\right)}} \monomial_{\chi - \left\langle 1^{l\left(\calE''\right)} \right\rangle}\left(-\calE''\right) \power_\chi(-\calB) \right) \\
& \times \sum_{\substack{\psi: \\ l(\psi) = l\left(\calE'\right)}} \monomial_{\psi - \left\langle 1^{l\left(\calE'\right)} \right\rangle} \left(\calE'\right) \sum_{\substack{\omega: \\ \omega \subset \psi \\ l(\omega) = l(\calF)}} \monomial_{\omega - \left\langle 1^{l(\calF)} \right\rangle} (\calF) \prod_{i \geq 1} \frac{i^{m_i(\omega)} m_i(\psi)!}{m_i(\psi \setminus \omega)!} \power_{\psi \setminus \omega}(\calC) \Bigg] \\
\\
& + \error
.
\end{split}
\end{align}
In particular, the main term vanishes unless $l(\calF) \leq l(\calE)$. To provide a bound for the error term we require one of the following additional conditions on the set of variables $\calB$. If there exists a real number $r_1 \leq r$ with $\abs(\calB) = r_1$, then 
\begin{align*}
\error ={} & O_\calP \left( r^{2N} N^{l(\calB)^2 + (l(\calB) - 1)^+ + (l(\calE) - 1)^+ + (l(\calF) - 1)^+ + (l(\calE) + l(\calF) - 1)^+ + 2} \right).
\intertext{where the implicit constant depends on $\calP = \{ r, l(\calB), l(\calC), l(\calE), l(\calF), \max(\abs(\calC)) \}$. If the elements of $\calB$ are pairwise distinct, then}
\error ={} & O_{r, \calB, l(\calC), l(\calE), l(\calF), \max(\abs(\calC))} \left( r^N N^{(l(\calB) - 1)^+ + (l(\calE) - 1)^+ + (l(\calF) - 1)^+ + (l(\calE) + l(\calF) - 1)^+ + 2} \right).
\end{align*}
\end{thm}

\begin{proof} The idea of the proof is to view this statement as a special case of the Recipe by setting $\calA = \emptyset = \calD$. Technically, this is not permissible since the elements of $\calB$ are not assumed to be pairwise distinct. However, for the main term it is enough to slightly perturb the elements of $\calB$ before applying the Recipe, and then make the perturbations vanish. For the error term, one quickly checks the proof of the Recipe to see that the implicit constant does in fact not depend on $\calB$ itself - but only on its length - in case $l(\calA) = 0 = l(\calD)$: the terms on the right-hand side in \eqref{4_in_proof_recipe_error_before_O} that depend on $\calB$ are
\begin{align*} 
\text{terms}(\calB) \defeq{} & \left| \elementary (\calB)^N \right| \sum_{\substack{\mu, \nu: \\ \nu' \cup \mu' \text{ is a partition} \\ \nu \subset \left\langle l(\calB)^N \right\rangle}} \left| LS_{\nu' \cup \mu'}\left(- \left(\calA \cup \calB^{-1}\right); \calD \right) \right|
.
\intertext{Under the assumption that $\calA = \emptyset = \calD$,}
\text{terms}(\calB) ={} & \left| \elementary (\calB)^N \right| \sum_{\substack{\mu, \nu: \\ \nu' \cup \mu' \text{ is a partition} \\ \nu' \subset \left\langle N^{l(\calB)} \right\rangle}} \left| \schur_{\nu' \cup \mu'} \left( -\calB^{-1} \right) \right|
.
\intertext{Moreover, the fact that any Schur function $\schur_\lambda(\calX)$ vanishes whenever $l(\lambda) > l(\calX)$ entails that only terms with $\mu = \emptyset$ contribute to the sum. Hence, we are left with} 
\text{terms}(\calB) ={} & \left| \elementary (\calB)^N \right| \sum_{\substack{\nu: \\ \nu' \subset \left\langle N^{l(\calB)} \right\rangle}} \left| \schur_{\nu'} (-\calB^{-1}) \right|
,
\end{align*}
which also appears in \eqref{4_eq_error_bound_for_recipe}, restricted to $\calA = \emptyset = \calD$.

Having resolved this technicality, we now set $\calA = \emptyset = \calD$ in the Recipe. It easily follows from the explicit expression for $z_\lambda^{-1} \power_\lambda\left(\rho^\beta_\emptyset \cup \rho^\alpha_\emptyset\right)$ given in \eqref{4_eq_union_of_specializations_power_sum} that this power sum vanishes unless $\lambda = \emptyset$. Thus, the main term on the right-hand side of the equality in \eqref{4_recipe_eq} simplifies to
\begin{align*}
\main ={} & (-1)^{l(\calE) + l(\calF)} \sum_{\substack{\calE', \calE'' \subset \calE: \\ \calE' \cup \calE'' \sorteq \calE}} \sum_{\substack{q,n \geq 0: \\ q + n \leq N - l(\calC)}} \left( \sum_{\substack{\chi: \\ l(\chi) = l\left(\calE''\right) \\ |\chi| = q}} \monomial_{\chi - \left\langle 1^{l\left(\calE''\right)} \right\rangle} \left(-\calE''\right) \power_\chi(-\calB) \right) \\
& \times \sum_{\substack{\psi: \\ l(\psi) = l\left(\calE'\right)}} \monomial_{\psi - \left\langle 1^{l\left(\calE'\right)} \right\rangle} \left(\calE'\right) \sum_{\substack{\omega: \\ \omega \subset \psi \\ l(\omega) = l(\calF) \\ |\omega| = n}} \monomial_{\omega - \left\langle 1^{l(\calF)} \right\rangle} (\calF) \prod_{i \geq 1} \frac{i^{m_i(\omega)} m_i(\psi)!}{m_i(\psi \setminus \omega)!} \power_{\psi \setminus \omega}(\calC)
.
\end{align*}
Owing to the condition that $\omega$ be a subsequence of $\psi$, this expression vanishes unless $l(\calF) \leq l(\calE)$. In addition, this condition allows us to eliminate the dependence on $N$ at the cost of incurring an error that is $O_{r, l(\calB), l(\calC), l(\calE), l(\calF)} \left( r^N N^{(l(\calE) - 1)^+ + (l(\calF) - 1)^+ + l(\calF) + 1}\right)$. Under the assumption that $\abs(\calB) \leq r$, the bound may even be multiplied by $r^N$. Indeed, the sum over $q,n \geq 0$ so that $q + n > N - l(\calC)$ is
\begin{align*}
& O \left( \sum_{\substack{q,n \geq 0: \\ q + n > N - l(\calC)}} r^{n + q(1 + \delta(\abs(\calB) \leq r))} n^{l(\calF)} \sum_{p \geq n} r^p \sum_{\substack{\calE', \calE'' \subset \calE: \\ \calE' \cup \calE'' \sorteq \calE}} \sum_{\substack{\chi: \\ l(\chi) = l\left(\calE''\right) \\ |\chi| = q}} \sum_{\substack{\psi: \\ l(\psi) = l\left(\calE'\right) \\ |\psi| = p}} \sum_{\substack{\omega: \\ l(\omega) = l(\calF) \\ |\omega| = n}} 1 \right)
\end{align*}
where $\delta(\abs(\calB) \leq r)$ indicates whether $\abs(\calB) \leq r$. Handling the sums counting partitions as in the preceding proof, and then employing an argument based on geometric series gives the desired bound.

It remains to show the bound on the error inherited from the Recipe. Given that $\calA = \emptyset = \calD$, the formula in \eqref{4_eq_error_bound_for_recipe} simplifies to
\begin{align*}
\error ={} & O_{r, l(\calB), l(\calE), l(\calF)} \left(N^{(l(\calE) + l(\calF) - 1)^+} \right) \left| \elementary (\calB)^N \right| \\
& \times \sum_{\substack{q,n \geq 0: \\ q + n > N - l(\calC)}} r^{q + n} \left(\sum_{\substack{q_1, \dots, q_{l(\calE)} \geq 0: \\ q_1 + \dots + q_{l(\calE)} = q}} 1 \right) \sum_{p \geq 0} r^p \left( \sum_{\substack{p_1, \dots, p_{l(\calE)} \geq 0: \\ p_1 + \dots + p_{l(\calE)} = p}} 1 \right) \! \left( \sum_{\substack{n_1, \dots, n_{l(\calF)} \geq 1: \\ n_1 + \dots + n_{l(\calF)} = n}} 1 \right) \\
& \times \left( \sum_{\substack{\nu: \\ \nu \subset \left\langle l(\calB)^N \right\rangle \\ |\nu| = l(\calB) N - q}} \left| \schur_{\nu'} \left(\calB^{-1}\right) \right| \right) \left( \sum_{\substack{\kappa: \\ n + |\kappa| = p}} \left| \schur_\kappa (\calC) \right| \right)
\end{align*}
where we have also used that $\abs(\calE)$, $\abs(\calF) \leq r$. First, consider the following function of $N$, which also depends on $\calB$:
\begin{align*}
S_q(N) \defeq{} & \left| \elementary (\calB)^N \right| \sum_{\substack{\nu: \\ \nu \subset \left\langle l(\calB)^N \right\rangle \\ |\nu| = l(\calB) N - q}} \left| \schur_{\nu'}(\calB^{-1}) \right|
.
\end{align*}
Under the assumption that $\abs(\calB) = r_1 \leq r$, Lemma~\ref{4_lem_comb_bound_schur} allows us to give an asymptotic bound for $S_q(N)$. More concretely,
\begin{align*}
S_q(N) ={} & O\left( r_1^{l(\calB) N} (l(\calB) N - q)^{l(\calB)^2} r_1^{-l(\calB) N + q} \right) \sum_{\makebox[45pt]{$\substack{\nu: \\ \nu \subset \left\langle l(\calB)^N \right\rangle \\ |\nu| = l(\calB) N - q}$}} 1 = O_{l(\calB)} \left( r^q N^{l(\calB)^2 + (l(\calB) - 1)^+} \right)
\end{align*}
since the number of partitions $\nu$ that appear in the sum is $O_{l(\calB)} \left( N^{(l(\calB) - 1)^+} \right)$.
On the other hand, if we suppose that the elements of $\calB$ are pairwise distinct, then the second statement of Lemma~\ref{4_lem_det_bound_Schur} provides the following bound for $S_q(N)$:
\begin{align*}
S_q(N) ={} & O_{\calB} \left( \elementary (\calB)^N \sum_{\substack{\nu: \\ \nu \subset \left\langle l(\calB)^N \right\rangle \\ |\nu| = l(\calB) N - q}} \elementary \left( \calB^{-1} \right)^{\nu'_1} \right) = O_\calB \left( N^{(l(\calB) - 1)^+} \right)
.
\end{align*}
Keeping these two bounds for $S_q(N)$ in mind, we proceed to bound the part of the error that is independent of $\calB$. As $\abs(\calC) < 1$, Lemma~\ref{4_lem_comb_bound_schur} entails that
$$ \sum_\kappa \left| \schur_\kappa(\calC) \right| = O_{l(\calC), \max(\abs(\calC))}(1)
.
$$
Hence,
\begin{align*}
\error ={} & O_{r, l(\calB), l(\calC), l(\calE), l(\calF), \max(\abs(\calC))} \left( N^{(l(\calE) + l(\calF) - 1)^+} \right) \\
& \times \sum_{\substack{q,n \geq 0: \\ q + n > N - l(\calC)}} S_q(N) q^{(l(\calE) - 1)^+} r^{q + n} n^{(l(\calF) - 1)^+} \sum_{p \geq n} r^p p^{(l(\calE) - 1 )^+}
\intertext{where we have used that the condition on $|\kappa|$ implies that $p \geq n$. An argument based on geometric series gives}
\error ={} & O_{r, l(\calB), l(\calC), l(\calE), l(\calF), \max(\abs(\calC))} \left( N^{(l(\calE) + l(\calF) - 1)^+} \right) \\
& \times \sum_{\substack{q,n \geq 0: \\ q + n > N - l(\calC)}} S_q(N) r^{q + 2n} (q + n)^{(l(\calE) - 1)^+ + (l(\calF) - 1)^+}
.
\end{align*}
Replacing $S_q(N)$ by the appropriate bound concludes the proof.
\end{proof}

\begin{thm} [logarithmic derivatives] \label{4_thm_log_ders} Let $r \in \R$, let $\calE$ and $\calF$ be sets of variables so that $0 < \abs(\calE), \abs(\calF) \leq r < 1$. Then
\begin{align*}
& \hspace{-15pt} \int_{U(N)} \prod_{\varepsilon \in \calE} \frac{\chi'_g(\varepsilon)}{\chi_g(\varepsilon)} \prod_{\varphi \in \calF} \frac{\chi'_{g^{-1}}(\varphi)}{\chi_{g^{-1}}(\varphi)} dg \\
={} & \begin{dcases} \sum_{\substack{\lambda: \\ l(\lambda) = l(\calE)}} z_\lambda \monomial_{\lambda - \left\langle 1^{l(\calE)} \right\rangle}(\calE) \monomial_{\lambda - \left\langle 1^{l(\calE)} \right\rangle}(\calF) &\text{if } l(\calE) = l(\calF) \\ 0 &\text{otherwise}\end{dcases} \\
& + O_{r, l(\calE), l(\calF)} \left( r^{2N} N^{(l(\calE) - 1)^+ + (l(\calF) - 1)^+ + (l(\calE) + l(\calF) - 1)^+ + 2} \right).
\end{align*}
\end{thm}

We have made no effort to optimize the exponent of $N$ in the bound for the error term.

\begin{proof} We set $\calB = \emptyset = \calC$ in Theorem~\ref{4_thm_log_ders_and_ratio}. As $\power_\lambda(\emptyset) = 0$ unless $\lambda = \emptyset$, the right-hand side of the equality in \eqref{4_thm_log_ders_and_ratio_eq} simplifies to
\begin{align*}
& (-1)^{l(\calE) + l(\calF)} \sum_{\substack{\psi: \\ l(\psi) = l(\calE)}} \monomial_{\psi - \left\langle 1^{l(\calE)} \right\rangle} (\calE) \sum_{\substack{\omega: \\ \omega = \psi \\ l(\omega) = l(\calF)}} \prod_{i \geq 1} i^{m_i(\omega)} m_i(\psi)! \monomial_{\omega - \left\langle 1^{l(\calF)} \right\rangle} (\calF) \\
& + O_{r, l(\calE), l(\calF)} \left( r^{2N} N^{(l(\calE) - 1)^+ + (l(\calF) - 1)^+ + (l(\calE) + l(\calF) - 1)^+ + 2} \right)
,
\end{align*}
which entails that the main term vanishes unless $l(\calE) = l(\calF)$. The expression stated in the theorem is obtained by substituting $\lambda$ for both $\psi$ and $\omega$.
\end{proof}

\begin{rem*} In \cite[p.~486]{CS}, Conrey and Snaith derive a formula for 
\begin{align*} 
\int_{U(N)} \prod_{\alpha \in \calA} \left(-e^{-\alpha} \right) \frac{\chi_g'\left(e^{-\alpha}\right)}{\chi_g\left(e^{-\alpha}\right)} \prod_{\beta \in \calB} \left(-e^{-\beta}\right) \frac{\chi_{g^{-1}}'\left(e^{-\beta}\right)}{\chi_{g^{-1}}\left(e^{-\beta}\right)} dg
\end{align*}
without employing any combinatorial methods. Compared to the logarithmic derivative theorem presented in this paper, their formula has the distinct advantage of providing an exact expression for the integral. Its principal disadvantage is that this expression is rather complicated, which makes it cumbersome to use. In Conrey and Snaith's theorem, it is not immediately obvious, for instance, that the leading term vanishes unless $l(\calA) = l(\calB)$. Hence, Theorem~\ref{4_thm_log_ders} is an improvement because it provides a simple expression in terms of one of the standard bases for the ring of symmetric functions.
\end{rem*}

\section{From logarithmic derivatives to an explicit formula} \label{4_sec_explicit_formula}
This section is dedicated to an application of the logarithmic derivative theorem, which is motivated by the analogy between $L$-functions and characteristic polynomials. We present an explicit formula for eigenvalues whose derivation mirrors the proof of the explicit formula for zeros of $L$-functions given in \cite{rudnick1996}. As Rudnick and Sarnak's proof is based on completed $L$-functions, which are more natural to work with than classic $L$-functions, we introduce the analogous notion of completed characteristic polynomials. In addition, we give a formula for products of logarithmic derivatives of completed characteristic polynomials.

\begin{defn} [completed characteristic polynomial] 
For unitary matrices $g$ that satisfy $\det(-g) \neq -1$, we define the completed characteristic polynomial as
\begin{align*}
\Lambda_g(z) ={} & \det(-g)^{1/2} z^{-N/2} \chi_g(z).
\end{align*}
Notice that while the characteristic polynomial $\chi_g$ is an entire function, $\Lambda_g$ might only be defined on $\C \setminus \R_-$.
\end{defn}

Why this is a natural definition for the random matrix analogue of completed $L$-functions is explored in Section~\ref{1_sec_NT_and_RMT}. Here, we just recall that the primary reason for considering the completed characteristic polynomial is the following symmetry with respect to the transformation given by $z \mapsto z^{-1}$. Its proof, which is a basic linear algebra exercise, can be found on page \pageref{1_proof_of_funct_eq_for_char_pol}.

\begin{lem} [functional equation] \label{4_lem_funct_eq_char_pol}
For $g \in U(N)$ with $\det(-g) \neq -1$ the following equalities hold.
\begin{enumerate}
\item For all $z \in \C \setminus \R_-$, $\Lambda_g(z) = \Lambda_{g^{-1}} \left( z^{-1}\right)$.
\item For all $z \in \C$ that are not eigenvalues of $g$, $\displaystyle z \frac{\Lambda_g'(z)}{\Lambda_g(z)} = - w \frac{\Lambda_{g^{-1}}' \left( w \right)}{\Lambda_{g^{-1}} \left( w \right)}$ where $w$ is equal to $z^{-1}$.
\end{enumerate}
\end{lem}

A formula for products of logarithmic derivatives of completed characteristic polynomials is easily deduced from the logarithmic derivative theorem for classic characteristic polynomials. 

\begin{thm} [completed logarithmic derivatives] \label{4_thm_completed_log_ders} Let $r \in \R$, let $\calE$ and $\calF$ be sets of non-zero variables so that $\abs(\calE), \abs(\calF) \leq r < 1$. Then
\begin{align} \label{4_thm_completed_log_ders_eq}
\begin{split}
& \hspace{-15pt} \int_{U(N)} \prod_{\varepsilon \in \calE} \varepsilon \frac{\Lambda'_g(\varepsilon)}{\Lambda_g(\varepsilon)} \prod_{\varphi \in \calF} \varphi \frac{\Lambda'_{g^{-1}}(\varphi)}{\Lambda_{g^{-1}}(\varphi)} dg \\
={} & \sum_\lambda \left( -\frac N2 \right)^{l(\calE) + l(\calF) - 2l(\lambda)}
z_\lambda \monomial_\lambda(\calE) \monomial_\lambda(\calF)
\\
& + O_{r, l(\calE), l(\calF)} \left(r^{2N} N^{l(\calE) + l(\calF) + (l(\calE) - 1)^+ + (l(\calF) - 1)^+ + (l(\calE) + l(\calF) - 1)^+ + 2} \right)
.
\end{split}
\end{align}
\end{thm}

The two logarithmic derivative theorems presented in this chapter are another reason why we consider it more natural to work with completed characteristic polynomials in the context of viewing random matrix theory as a model for number theory: the main term in Theorem~\ref{4_thm_completed_log_ders} is a sum that ranges over all partitions, while the main term in Theorem~\ref{4_thm_log_ders} is a sum that ranges over all partitions of a fixed length, which we consider an ``unnatural'' restriction.

\begin{proof} Notice that $\{g \in U(N): \det(-g) = -1\}$ is a null set with respect to Haar measure on $U(N)$. Hence, the fact that $\Lambda_g$ is not defined on this set is of no concern.

We reformulate the left-hand side in \eqref{4_thm_completed_log_ders_eq} such that we can apply Theorem~\ref{4_thm_log_ders}:
\begin{align*}
\LHS ={} & \int_{U(N)} \prod_{\varepsilon \in \calE} \left( -\frac{N}{2} + \varepsilon \frac{\chi'_g(\varepsilon)}{\chi_g(\varepsilon)} \right) \prod_{\varphi \in \calF} \left( -\frac{N}{2} + \varphi \frac{\chi'_{g^{-1}}(\varphi)}{\chi_{g^{-1}}(\varphi)} \right) dg \displaybreak[2] \\ 
={} & \sum_{\substack{\calE' \subset \calE \\ \calF' \subset \calF}} \left( -\frac N2 \right)^{l(\calE) - l\left(\calE'\right) + l(\calF) - l\left(\calF'\right)} \left( \prod_{\varepsilon \in \calE'} \varepsilon \right) \left( \prod_{\varphi \in \calF'} \varphi \right) \\
& \times \int_{U(N)} \prod_{\varepsilon \in \calE'} \frac{\chi_g'(\varepsilon)}{\chi_g(\varepsilon)} \prod_{\varphi \in \calF'} \frac{\chi_{g^{-1}}'(\varphi)}{\chi_{g^{-1}}(\varphi)} dg
.
\intertext{We remark that this equality holds thanks to the convention fixed in Section~\ref{4_sec_sequences_and_partitions} which ensures that every sequence of length $n$ has exactly $2^n$ subsequences. Theorem~\ref{4_thm_log_ders} allows us to compute the integral:}
\LHS ={} & \sum_{\substack{\calE' \subset \calE \\ \calF' \subset \calF \\ l\left(\calE'\right) = l\left(\calF'\right)}} \left( -\frac N2 \right)^{l(\calE) + l(\calF) - 2l\left(\calE'\right)} \left( \prod_{\varepsilon \in \calE'} \varepsilon \right) \left( \prod_{\varphi \in \calF'} \varphi \right) \\
& \times \sum_{\substack{\lambda: \\ l(\lambda) = l\left(\calE'\right)}} z_\lambda \monomial_{\lambda - \left\langle 1^{l\left(\calE'\right)} \right\rangle}\left(\calE'\right) \monomial_{\lambda - \left\langle 1^{l\left(\calE'\right)} \right\rangle}\left(\calF'\right)
\\
& + O_{r, l(\calE), l(\calF)} \left(r^{2N} N^{l(\calE) + l(\calF) + (l(\calE) - 1)^+ + (l(\calF) - 1)^+ + (l(\calE) + l(\calF) - 1)^+ + 2} \right) \displaybreak[2]\\
={} & \sum_\lambda \left( -\frac N2 \right)^{l(\calE) + l(\calF) - 2l(\lambda)} \sum_{\substack{\calE' \subset \calE \\ \calF' \subset \calF \\ l\left(\calE'\right) = l(\lambda) = l\left(\calF'\right)}} z_\lambda \monomial_\lambda\left(\calE'\right) \monomial_\lambda\left(\calF'\right)
\\
& + O_{r, l(\calE), l(\calF)} \left(r^{2N} N^{l(\calE) + l(\calF) + (l(\calE) - 1)^+ + (l(\calF) - 1)^+ + (l(\calE) + l(\calF) - 1)^+ + 2} \right)
.
\end{align*}
By the definition of the monomial symmetric polynomials, the main term simplifies to the desired expression.
\end{proof}

A formula for the average of products of logarithmic derivatives of completed characteristic polynomials over the unitary group \emph{which holds inside the unit circle} is the only tool we need to derive an explicit formula for eigenvalues of unitary matrices.

\begin{thm} [explicit formula] \label{4_thm_explicit_formula} Fix $r \in \R$ with $0 < r < 1$. Let $A(r)$ denote the closed annulus (about the origin) with inner radius $r$ and outer radius $r^{-1}$, and $D \left(r^{-1}\right)$ the closed disc (about the origin) of radius $r^{-1}$. Let $h$ be a meromorphic function on $D\left(r^{-1}\right)$ which is holomorphic on $A(r)$. Let $f$ be a symmetric function in $n$ variables such that $z \mapsto f(z, z_2, \dots, z_n)$ is meromorphic on $D\left(r^{-1}\right)$ and holomorphic on $A(r)$. If $\{\rho_1, \dots, \rho_N\}$ are the eigenvalues of $g \in U(N)$, then
\begin{align} \label{4_thm_explicit_formula_eq}
\begin{split}
& \hspace{-10pt} \int_{U(N)} \sum_{1 \leq j_1, \dots, j_n \leq N} h(\rho_{j_1}) \cdots h(\rho_{j_n}) f(\rho_{j_1}, \dots, \rho_{j_n}) dg \notag \end{split} \displaybreak[2] \\ \begin{split}
={} & \sum_\lambda \left( \frac N2 \right)^{n - 2l(\lambda)} \frac{z_\lambda}{(2\pi)^n} \sum_{k = 0}^n \binom{n}{k} \\
& \times \int_{[0,2\pi]^n} \Bigg( \left[ \prod_{j = 1}^k h \left(re^{-it_j}\right) \right] \! \left[ \prod_{j = k + 1}^n h\left(\frac{e^{it_j}}{r} \right) \right] \! f \! \left( \! re^{-it_1}, \dots, re^{-it_k}, \frac{e^{it_{k + 1}}}{r} , \dots, \frac{e^{it_n}}{r} \!\right) \\
& \hspace{10pt} \times \monomial_\lambda\left(re^{-it_1}, \dots, re^{-it_k}\right) \monomial_\lambda\left(re^{-it_{k + 1}}, \dots, re^{-it_n}\right) dt_1 \dots dt_n \Bigg)
\\
& + O_{r, n, h, f} \left(r^{2N} N^{3n + 2} \right)
.
\end{split}
\end{align}
In the context of this theorem we call a function $f$ symmetric if it is invariant under the permutation of its variables, which means that $f$ need not be an element of the ring of symmetric functions.
\end{thm}

\begin{proof}
Recall that the function $\Lambda_g'(z)/\Lambda_g(z)$ is meromorphic on the entire complex plane with simple poles at $\{0, \rho_1, \dots, \rho_N\}$; its residue at $\rho_i$ is the multiplicity of $\rho_i$. We consider the following path integral along the border of $A(r)$, \textit{i.e.}\ along $\delta = \delta(r) + \delta\left(r^{-1}\right)$ where $\delta(r): [0, 2\pi] \to \C; t \mapsto re^{-it}$ and $\delta\left(r^{-1}\right): [0, 2\pi] \to \C; t \mapsto r^{-1}e^{it}$:
\begin{align*} 
\text{Eig}(g) \defeq \frac{1}{(2\pi i)^n} \int_{\delta} \dots \int_{\delta} \prod_{i = 1}^n \frac{\Lambda_g'(z_i)}{\Lambda_g(z_i)} h(z_i) f(z_1, \dots, z_n)dz_1 \dots dz_n 
.
\end{align*}
Given that the interior of $\delta$ does not contain the origin, repeated application of the residue theorem allows us to infer that the above expression is equal to the integrand on the left-hand side in \eqref{4_thm_explicit_formula_eq}.

In a next step, we show that the integral of $\text{Eig}(g)$ over the unitary group is also equal to the right-hand side in \eqref{4_thm_explicit_formula_eq}. Recalling that each integral along the path $\delta$ is the sum of the integrals along $\delta(r)$ and $\delta \left( r^{-1}\right)$, we multiply out (exploiting the fact that $f$ is symmetric), and then apply the functional equation for the completed characteristic polynomial (\textit{i.e.}\ Lemma~\ref{4_lem_funct_eq_char_pol}) to the logarithmic derivatives that are integrated along $\delta\left(r^{-1}\right)$:
\begin{align*}
\text{Eig}(g) ={} & \sum_{\substack{\calE, \calF \subset [n]: \\ \calE \cup \calF \sorteq [n]}} \frac{1}{(2\pi i)^n} \int_{\delta(r)}^{(\calE)} \int_{\delta\left(r^{-1}\right)}^{(\calF)} \prod_{\varepsilon \in \calE} z_\varepsilon \frac{\Lambda_g'(z_\varepsilon)}{\Lambda_g(z_\varepsilon)} \frac{h(z_\varepsilon)}{z_\varepsilon} \prod_{\varphi \in \calF} \left( - z_\varphi^{-1} \right) \frac{\Lambda_{g^{-1}}' \left( z_\varphi^{-1} \right)}{\Lambda_{g^{-1}} \left( z_\varphi^{-1} \right)} \frac{h(z_\varphi)}{z_\varphi} \\
& \times f\left(z_\calE \cup z_\calF\right)dz_\calE dz_\calF
.
\end{align*}
Here the superscripts of the integrals indicate which variables are integrated along $\delta(r)$, and which along $\delta\left(r^{-1}\right)$.
Using Theorem~\ref{4_thm_completed_log_ders} to integrate this expression over $U(N)$ gives
\begin{align*}
\int_{U(N)} \text{Eig}(g) dg ={} & \sum_{\substack{\calE, \calF \subset [n]: \\ \calE \cup \calF \sorteq [n]}} \frac{(-1)^{l(\calF)}}{(2\pi i)^n} \int_{\delta(r)}^{(\calE)} \int_{\delta\left(r^{-1}\right)}^{(\calF)} \prod_{\varepsilon \in \calE} \frac{h(z_\varepsilon)}{z_\varepsilon} \prod_{\varphi \in \calF} \frac{h(z_\varphi)}{z_\varphi} f\left(z_\calE \cup z_\calF \right) \\
& \times \sum_\lambda \left( -\frac N2 \right)^{l(\calE) + l(\calF) - 2l(\lambda)}
z_\lambda \monomial_\lambda\left(z_\calE\right) \monomial_\lambda\left(z_\calF^{-1}\right) dz_\calE dz_\calF
\\
& + O_{r, n, h, f} \left(r^{2N} N^{3n + 2} \right) 
.
\intertext{Notice that we have exchanged the order of integration, which is permissible since we are only integrating continuous functions over compact spaces with respect to finite measures. Further notice that the terms only depend on $l(\calE)$, and not on the subsequence itself. Hence,}
\int_{U(N)} \text{Eig}(g) dg ={} & \sum_\lambda \left( -\frac N2 \right)^{n - 2l(\lambda)} \frac{z_\lambda}{(2\pi i)^n} \sum_{k = 0}^n (-1)^{n - k} \binom{n}{k} \\
& \times \int_{\delta(r)}^{(1, \dots, k)} \int_{\delta\left(r^{-1}\right)}^{(k + 1, \dots, n)} \left[ \prod_{j = 1}^n \frac{h(z_j)}{z_j} \right] f(z_1, \dots, z_n) \\
& \times \monomial_\lambda(z_1, \dots, z_k) \monomial_\lambda \left(z_{k + 1}^{-1}, \dots, z_n^{-1}\right) dz_1 \dots dz_n
\\
& + O_{r, n, h, f} \left(r^{2N} N^{3n + 2} \right)
.
\end{align*}
Writing out the path integrals gives the desired formula.
\end{proof}

\chapter{Ideas for Number Theoretic Applications} \label{5_cha_conclusions}

\section[Explicit formulae for zeros of $L$-functions and eigenvalues in comparison]{Explicit formulae for zeros of $\boldsymbol{L}$-functions and \\eigenvalues in comparison} \label{5_sec_explicit_formulae}
In our opinion, the main interest of our explicit formula for eigenvalues of a random unitary matrix lies in the fact that its derivation has the same basic structure as the derivation of the explicit formula for zeros of $L$-functions in \cite{rudnick1996}. This similarity in structure might give a deeper insight into the conjectured connection between $L$-functions and characteristic polynomials from the unitary group.

As shown in Section~\ref{1_sec_on_correlations}, Rudnick and Sarnak's explicit formula for zeros of $L$-functions (stated in Theorem~\ref{1_thm_explicit_formula}) is an application of the functional equation and the Euler product. Hence, our proof of the explicit formula for eigenvalues (stated in Theorem~\ref{4_thm_explicit_formula}) is based on two analogous properties of characteristic polynomials.
\begin{itemize}
\item The functional equation for $L$-functions used in \cite{rudnick1996} encodes a symmetry between the value of the completed $L$-function attached to some irreducible cuspidal automorphic representation $\pi$ of $GL_m$ over $\Q$ at the point $s$ and the value of the completed $L$-function associated to the contragredient of $\pi$ at the point $1 - s$. (For a formal statement of the functional equation in question turn to page \pageref{1_page_functional_equation_RS}.) As discussed in Section~\ref{1_sec_on_moments}, the transformation $s \mapsto 1 - s$ corresponds to the transformation $z \mapsto z^{-1}$. Hence, it is reasonable that the equality $$\Lambda_g(z) = \Lambda_{g^{-1}} \left(z^{-1} \right)$$ plays the role of the functional equation in our derivation of the explicit formula for eigenvalues, where $g \in U(N)$ and the inverse $g^{-1}$ is analogous to the contragredient $\tilde{\pi}$.
\item If we view the Euler product as a connector between $L$-functions and prime numbers, there is no hope of finding a random matrix theory analogue. However, if we view the Euler product as an explicit expression for the logarithmic derivative $\Lambda'(s, \pi)/\Lambda(s, \pi)$ that holds sufficiently far to the right of the critical line, then Theorem~\ref{4_thm_completed_log_ders} is a possible analogue. Indeed, it provides an explicit expression for (the main term of) the average of logarithmic derivatives of completed characteristic polynomials that holds inside the unit circle. 

As the unit circle is the ``critical line'' for the completed characteristic polynomial $\Lambda_g(z)$ \cite[p.~39]{CFKRS05}, the unit disc (\textit{i.e.}\ the inside of the unit circle) should correspond to either the half-plane to the left or the half-plane to the right of the critical line for completed $L$-functions. The substitute for the Euler product proposed above suggests that the unit disc is associated to the half-plane on the right-hand side. Another argument in support of this correspondence (which is also mentioned in \cite{CFKRS05}) is that under the assumption of the Riemann hypothesis, the zeros of $\zeta'(s)$ all lie to the right of the critical line (according to \cite{MontgomeryLevinson}), while the zeros of the derivative of any characteristic polynomial $\chi_g(z)$ with $g \in U(N)$ lie inside the unit circle (according to the Gauss-Lucas Theorem). 
\end{itemize}
The proofs of the explicit formulae (for zeros of $L$-functions and for eigenvalues) are both structured as follows: Consider the sum on the left-hand side of the equality to be proved (\textit{i.e.}\ the equality given in \eqref{1_thm_explicit_formula_eq} or \eqref{4_thm_explicit_formula_eq}), which we recall for the convenience of the reader: 
\begin{align*}
\sum_{\rho_\pi} h(\rho_\pi) - \delta(\pi) \left[ h(0) + h(1) \right]
\end{align*}
where $\rho_\pi$ is over the nontrivial zeros of $L(s, \pi)$ (and the second term vanishes unless $\pi$ corresponds to the $\zeta$-function), or
\begin{align*}
\int_{U(N)} \sum_{1 \leq j_1, \dots, j_n \leq N} h(\rho_{j_1}) \cdots h(\rho_{j_n}) f(\rho_{j_1}, \dots, \rho_{j_n}) dg
\end{align*}
where $\{\rho_1, \dots, \rho_N\}$ is the multiset of eigenvalues of $g \in U(N)$.
Use Cauchy's argument principle to express this sum over zeros as a contour integral. This results in an integral of the following abstract form:
\begin{align*}
\frac{1}{2\pi \imaginary} \int_{\gamma_1} \frac{\Lambda'(s)}{\Lambda(s)} h(s) ds - \frac{1}{2\pi \imaginary} \int_{\gamma_2} \frac{\Lambda'(s)}{\Lambda(s)} h(s) ds
\end{align*}
where $\Lambda$ stands for a completed $L$-function \emph{or} a random completed characteristic polynomial, and the contours $\gamma_1$ and $\gamma_2$ are vertical lines that are located to the right and to the left the critical line, respectively, \emph{or} the contours $\gamma_1$ and $\gamma_2$ are circles about the origin that are located inside and outside the unit circle, respectively. In a next step, apply the functional equation to the integrand corresponding to the contour $\gamma_2$, which allows us to situate both contours to the right of the critical line/inside the unit circle. Now, the explicit formula is a consequence of the Euler product/Theorem~\ref{4_thm_completed_log_ders}.

The underlying structure of these derivations of explicit formulae might be the same, but the resulting formulas look quite different. The principal reason for this difference is that we have substituted the Euler product by an equality that does not carry any arithmetic information. Another obvious difference is that our explicit formula for eigenvalues provides an asymptotic expression for the sums over all $n$-tuples of eigenvalues (for $n \geq 1$), whereas Rudnick and Sarnak's explicit formula for zeros of $L$-functions provides an exact expression for the sum over all $1$-tuples of zeros. It would be very interesting to investigate explicit formulae for sums of $n$-tuples of zeros of $L$-functions, whose proof follows the same structure. Such a proof would be based on an arithmetic expression for
\begin{align*}
\prod_{\varepsilon \in \calE} \varepsilon \frac{\Lambda'(\varepsilon, \pi)}{\Lambda (\varepsilon, \pi)} \prod_{\varphi \in \calF} \varphi \frac{\Lambda'(\varphi, \tilde{\pi})}{\Lambda(\varphi, \tilde{\pi})}
\end{align*}
where $\calE$ and $\calF$ are sets of complex numbers that lie sufficiently far to the right of the critical line. This arithmetic expression might even display the same combinatorial structure as our combinatorial formula for the average of products of logarithmic derivatives of completed characteristic polynomials (stated in Theorem~\ref{4_thm_completed_log_ders}).

\section{A comment for specialists of RMT predictions about number theory}
In the preceding section, we have argued that our proof of the explicit formula for eigenvalues mirrors Rudnick and Sarnak's proof of the explicit formula for zeros of $L$-functions. In doing so, we somewhat follow a combinatorial approach developed in \cite{dehaye12}, where he uncovers the combinatorics behind the recipe for conjecturing the moments of $L$-functions proposed in \cite{CFKRS05}. As discussed in Section~\ref{1_sec_POD12}, Dehaye shows that translating the CFKRS recipe into the language of symmetric functions results in the following conjecture for the $2k$-th shifted moment of the Riemann $\zeta$-function: for every pair of partitions $\mu$, $\nu$ there exists a coefficient $c_{\mu \nu}(t)$, which is a function in the variable $t$, so that
\begin{align*}
\frac{1}{T} \int_0^T Z \left( \frac 12 + \imaginary t, \calD \right) dt \sim \frac{1}{T} \int_0^T \left( \frac{t}{2 \pi} \right)^{\frac 12 \sum_{\delta \in \calD} \delta} \sum_{\mu, \nu}  c_{\mu \nu}(t) \sum_{A, B} \frac{\schur_\mu(\calA) \schur_\nu(\calB)}{\Delta(\calA; \calB)} dt
\end{align*}  
where the second sum is over subsets $\calA$, $\calB$ of the shifts $\calD$ such that $|\calA| = k = |\calB|$ and
$\calA \cup \calB = \calD$. Here $Z(1/2 + \imaginary t; \calD) = Z(1/2 + \imaginary t + \delta_1) \cdots Z(1/2 + \imaginary t + \delta_{2k})$ is a product of $2k$ shifted completed $\zeta$-functions. In a next step, Dehaye applies a generalized version of Lemma~\ref{1_lem_BG_laplace_Schur} to express the sum over $\calA, \calB$ as a Schur function $\schur_\lambda(\calA \cup \calB) = \schur_\lambda(\calD)$ indexed by a partition $\lambda$ that depends on $\mu$ and $\nu$. In a final step, he lets the shifts in $\calD$ go to 0 and thus obtains a conjectural expression for $I_k(\zeta, T)$, the $2k$-th moment of the Riemann $\zeta$-function.

On the random matrix theory side, one can easily reorganize Bump and Gamburd's combinatorial proof of Keating and Snaith's formula for $I_k(U(N))$ (\textit{i.e.}\ the $2k$-th moment of a random characteristic polynomial from the unitary group $U(N)$) so that it displays a similar structure: for any matrix $g \in U(N)$ (with $\det(-g) \neq -1$), let us define $Z_g(z) = \det(-g)^{1/2} z^{-N/2} \chi_g(z)$ in analogy to the $Z$-function given in \cite{CFKRS05}. Consider a sequence of pairwise distinct shifts $\calD = (\delta_1, \dots, \delta_{2k})$ with $|\delta_i| = 1$ and set
\begin{align*}
Z_g \left(1, \calD \right) = Z_g(\delta_1) \cdots Z_g(\delta_{2k})
.
\end{align*}
The second equality in Bump and Gamburd's product theorem (\textit{i.e.}\ Theorem~\ref{1_thm_products_BG}) allows us to infer the following expression for the $2k$-th shifted moment of a random characteristic polynomial from $U(N)$:
\begin{align*}
\int_{U(N)} Z_g(1, \calD)dg ={} & \int_{U(N)} \prod_{i = 1}^k Z_g(\delta_1) \prod_{i = k + 1}^{2k} \overline{Z(\delta_i)} dg
\\
={} & \prod_{i = 1}^k \delta_i^{-N/2} \prod_{i = k + 1}^{2k} \delta_i^{N/2} \int_{U(N)} \prod_{i = 1}^k \chi_g(\delta_i) \prod_{i = k + 1}^{2k} \chi_{g^{-1}}\left(\delta_i^{-1}\right) dg
\\
={} & \prod_{\delta \in \calD} \delta^{-N/2} \sum_{\calA, \calB} \frac{\prod_{\alpha \in \calA} \alpha^{k + N}}{\Delta(\calA; \calB)}
\end{align*}
where the second sum is over subsets $\calA$, $\calB$ of the shifts $\calD$ such that $|\calA| = k = |\calB|$ and
$\calA \cup \calB = \calD$. Following the method used in \cite{dehaye12}, we apply Lemma~\ref{1_lem_BG_laplace_Schur}
to express the sum over $\calA, \calB$ as a Schur function $\schur_\lambda(\calA \cup \calB) = \schur_\lambda(\calD)$, where the partition $\lambda$ is equal to $\left\langle N^k \right\rangle$. Finally, letting the shifts in $\calD$ go to 1 results in a combinatorial expression for $I_k(U(N))$, the $2k$-th moment of a random characteristic polynomial from $U(N)$.

Dehaye thus discovered combinatorial commonalities between a number theoretic conjectural derivation and a random matrix theory proof (on moments), while we do the same for a number theoretic proof and a combinatorial proof (on explicit formulae). \textbf{In short, Dehaye finds combinatorial commonalities between the structure of a proof and of a conjectural derivation, while we find
commonalities between two actual proofs.} Our discovery suggests that a better understanding of the combinatorial structure behind random matrix theory proofs might lead to rigorous number theoretic proofs -- as opposed to number theoretic conjectures. Moreover, the completed framework presented in Chapter~\ref{4_cha_mixed_ratios} for expressing these combinatorial commonalities covers both the product formula in \cite{bump06} and a formula for products of logarithmic derivatives of completed characteristic polynomials, which are related to moments and to explicit formulae (and thus correlations), respectively. 

On a side note, both the CFKRS recipe for conjecturing moments of $L$-functions and the combinatorial proof for moments of characteristic polynomials sketched above introduce pairwise distinct shifts $\calD$, without which the respective methods would not be applicable. From a purely combinatorial point of view this condition that the variables be pairwise distinct, which pervades the completed framework presented in Chapter~\ref{4_cha_mixed_ratios}, looks like an artefact of working with determinantal formulas for Schur and Littlewood-Schur functions -- rather than with combinatorial formulas, which describe the two types of functions as polynomials. However,  the fact that pairwise distinct shifts also play an important role in its number theoretical counterpart allows us to interpret this ``pairwise distinct'' condition as another indicator that the combinatorial structures underlying the two approaches are identical.

\section{On (to) other classic matrix groups}
In \cite{FamiliesLFunctions}, Katz and Sarnak introduce the theory of symmetry types associated to so-called families of $L$-functions. Examples of families of $L$-functions include (but are not limited to):
\begin{enumerate}
\item the set $\{L(s + it): t \geq 0\}$ where $L$ is a fixed $L$-function, 
\item the set of Dirichlet $L$-functions associated to quadratic Dirichlet characters. 
\end{enumerate}
Katz and Sarnak conjecture that the zero distribution within any family of $L$-functions coincides with the eigenvalue distribution of a classical compact matrix group (which is determined by its symmetry type). In this context, the compact groups of interest are the unitary group, the (special) orthogonal groups and the symplectic group. It is conjectured that our first example has a unitary symmetry type, while our second example has a symplectic symmetry type. 

In \cite{KS00L}, Keating and Snaith derive exact expressions for the moments of a random characteristic polynomial from both the special orthogonal and the symplectic groups, using techniques they have developed in \cite{KS00zeta} for computing moments of a random characteristic polynomial from the unitary group. They then go on to study the connection to the \emph{value} distribution of the $L$-functions within orthogonal and symplectic families at the so-called central point $s = 1/2$. More precisely, they compare the few known results for moments of families of $L$-functions at the central point to the moments of characteristic polynomials from the matrix groups corresponding to their symmetry type (at the point $z = 1$). 
In \cite{CF}, Conrey and Farmer conjecture that the moments of families of $L$-functions at the central point are asymptotically equivalent to products of two factors, namely a factor that depends on the specific family in question and a universal factor that only depends on its symmetry type. Building on this conjecture, their comparison leads Keating and Snaith to the discovery that the universal factor seems to be determined by random matrix theory. To sum up, the relationship between moments of families of $L$-functions and the matrix group corresponding to their symmetry type is analogous to the the relationship between moments of a single $L$-function and the unitary group. This comes as no surprise given that the moments of a single $L$-function can be viewed as moments of a particular family of $L$-functions at the central point $s = 1/2$. In fact, the moments of a single $L$-function are equal to the moments of our first example of a family of $L$-functions at $s = 1/2$.

In Chapter~\ref{1_cha_intro}, we have only presented Bump and Gamburd's combinatorial method for computing averages of characteristic polynomials over the unitary group. However, they also apply their method to other matrix groups of relevance to number theory in \cite{bump06}. This remark applies equally to our discussion of the CFKRS recipe for conjecturing the lower order terms in the moments of $L$-functions: their recipe is also applicable in the more general setting of moments of families of $L$-functions at the central point. Moreover, Dehaye informs me that he has extended his formalism for translating the CFKRS recipe into the language of symmetric functions to the orthogonal and the symplectic group (unpublished). Therefore, it should be possible to further expand the combinatorial framework presented in this thesis, unifying not only moments and explicit formulae for single $L$-functions, but also for families of $L$-functions.

\bibliographystyle{alpha}
\bibliography{bib_thesis}

\end{document}